\documentclass[12pt,a4paper,oneside,reqno]{book}

\usepackage{geometry}
\usepackage{fancyhdr}
\fancypagestyle{plain}{\fancyhf{}} 

\usepackage{setspace}
\usepackage[T1]{fontenc}
\usepackage[dvips]{graphicx}
\usepackage{epsfig}
\usepackage{amsmath,amssymb,amscd,amsthm}
\usepackage{subfigure}
\usepackage[latin1]{inputenc}
\usepackage{latexsym}
\usepackage{graphics}
\usepackage{colortbl} 
\usepackage{color}
\usepackage{ae}
\usepackage{enumitem}
\usepackage[all]{xy}
\usepackage{scalerel}
\usepackage{tikz}
\usepackage{cancel}
\usepackage{bbm,amsbsy}
\usepackage{mathrsfs}  
\usepackage[normalem]{ulem}
\usepackage{nicefrac}
\usepackage{xfrac}
\usepackage{faktor}
\usepackage{tikz-cd} 
\usepackage{dsfont}
\usepackage{tikz}
\usetikzlibrary{shapes,calc,cd,decorations.pathmorphing,hobby}

\usepackage{amsthm}

\makeatletter
\newtheorem*{rep@theorem}{\rep@title}
\newcommand{\newreptheorem}[2]{%
	\newenvironment{rep#1}[1]{%
		\def\rep@title{#2 \ref{##1}}%
		\begin{rep@theorem}}%
		{\end{rep@theorem}}}
\makeatother

\makeatletter
\newtheorem*{rep@cor}{\rep@title}
\newcommand{\newrepcor}[2]{%
	\newenvironment{rep#1}[1]{%
		\def\rep@title{#2 \ref{##1}}%
		\begin{rep@cor}}%
		{\end{rep@cor}}}
\makeatother

\makeatletter
\newtheorem*{rep@prop}{\rep@title}
\newcommand{\newrepprop}[2]{%
	\newenvironment{rep#1}[1]{%
		\def\rep@title{#2 \ref{##1}}%
		\begin{rep@prop}}%
		{\end{rep@prop}}}
\makeatother

\theoremstyle{definition}

\newtheorem{Definition}{Definition}[section]
\newtheorem{Def}[Definition]{Definition}

\newtheorem{Example}[Definition]{Example}
\newtheorem{example}[Definition]{Example}

\newtheorem{Remark}[Definition]{Remark}
\newtheorem{remark}[Definition]{Remark}
\theoremstyle{plain}
\newtheorem{Proposition}[Definition]{Proposition}
\newtheorem{Prop}[Definition]{Proposition}
\newtheorem{prop}[Definition]{Proposition}
\newreptheorem{prop}{Proposition}
\newtheorem{Lemma}[Definition]{Lemma}
\newtheorem{lemma}[Definition]{Lemma}
\newtheorem*{Theorem*}{Theorem}
\newtheorem{Theorem}[Definition]{Theorem}

\newreptheorem{theorem}{Theorem}
\newtheorem{Teo}[Definition]{Theorem}
\newtheorem{Cor}[Definition]{Corollary}
\newtheorem{cor}[Definition]{Corollary}
\newreptheorem{cor}{Corollary}

\newlist{steps}{enumerate}{1}
\setlist[steps, 1]{itemsep=8pt,leftmargin=0cm,itemindent=.5cm,labelwidth=\itemindent,labelsep=0cm,align=left,label = \textbf{\emph{Step \arabic*}:\,}}

\newcommand{\C}{{\mathbb C}}
\newcommand{\CC}{\mathds C}
\newcommand{\R}{{\mathbb R}}
\newcommand{\RR}{\mathds R}
\newcommand{\Z}{{\mathds Z}}

\newcommand{\hol}{\mathrm{hol}}
\newcommand{\Hyp}{\mathbb{H}}
\newcommand{\Sph}{\mathbb{S}}

\newcommand{\JJ}{\mathbbm{J}}
\newcommand{\JJJ}{\mathds J}

\newcommand{\ddt}{\left.\frac{d}{dt}\right|_{t=0}}
\newcommand{\Isom}{\mathrm{Isom}}

\newcommand{\V}{\mathcal{V}}
\newcommand{\VP}{\mathcal{V}^0}
\newcommand{\HH}{\mathcal{H}}
\newcommand{\HP}{\mathcal{H}^0}

\newcommand{\as}[1]{\textcolor{magenta}{#1}}

\newcommand{\inner}[1] {\langle #1 \rangle}
\newcommand{\inners}{\langle \cdot, \cdot \rangle}

\newcommand{\CP}{\mathbb C \mathbb P}
\newcommand{\PSL}{{\mathrm{PSL}(2,\mathds C)}}
\newcommand{\SL}{{\mathrm{SL}(2,\mathds C)}}
\newcommand{\asl}{\mathfrak{sl}(2,\CC)}
\newcommand{\PXX}{\mathbb P \mathbb X}

\newcommand{\Lieg}{\mathfrak g}
\newcommand{\Proj}{\mathbb P}
\newcommand{\Hy}{\mathbb H}
\newcommand{\XXX}{\mathbb X}
\newcommand{\blue}{\textcolor{blue}}

\newcommand{\phee}{\varphi}
\newcommand{\eps}{\epsilon}

\newcommand{\m}{\mathfrak m}

\newcommand{\SO}{\mathrm{SO}}
\newcommand{\GG}{{\mathbbm g}}
\newcommand{\GGG}{{\mathbb G}}
\newcommand{\G}[1]{\mathbb G_{#1}}
\newcommand{\gs}[1]{{\overline{Sas}_{#1}}}
\newcommand{\gstar}[1]{{\overline{Sas}_{#1} ^{\ *}}}
\newcommand{\xx}{{{x}}}

\newcommand{\vv}{{{v}}}

\newcommand{\ww}{{{w}}}

\theoremstyle{definition} 
\newtheorem{Claim}{Claim}

\newtheorem*{Notation*}{Notation}

\newcommand{\PML}{\XXX_{n+1}}
\newcommand{\ML}{\XXX_{n+1}}
\newcommand{\XX}{\mathtt{X}}
\newcommand{\YY}{\mathtt{Y}}
\newcommand{\ZZ}{\mathtt{Z}}

\newcommand{\CTM}{\C TM}
\newcommand{\CTS}{\C TM}
\newcommand{\io}{\sigma}
\newcommand{\NN}{\mathcal N}

\newcommand{\IIs}{\mathrm{I}\hspace{-0.04cm}\mathrm{I}}

\newcommand{\IIv}{\underline{\mathrm{I}\hspace{-0.04cm}\mathrm{I}}}

\newcommand{\Lieonn}{\mathfrak{o}(n+2,\C)}
\newcommand{\SOnn}{SO(n+2,\C)}
\newcommand{\upsi}{\underline \Psi}

\newcommand{\utheta}{\underline \theta}

\usepackage{marvosym}
\newcommand{\Tau}{\text{\Capricorn}}
\newcommand{\sslash}{\mathbin{/\mkern-6mu/}}

\newcommand{\I}{\mathrm{I}}
\newcommand{\II}{\mathrm{I}\hspace{-0.04cm}\mathrm{I}}
\newcommand{\III}{\mathrm{I}\hspace{-0.04cm}\mathrm{I}\hspace{-0.04cm}\mathrm{I}}
\newcommand{\arctanh}{\mathrm{arctanh}}

\newcommand{\Ham}{\mathrm{Ham}}
\newcommand{\Flux}{\mathrm{Flux}}

\usepackage[backend=biber,style=alphabetic
]{biblatex}
\usepackage[colorlinks=true]{hyperref}
\usepackage{caption}
\renewcommand{\maketitle}{
	
	\begin{titlepage}

		\thispagestyle{empty}
		\enlargethispage{20cm}  

		\begin{center}
			\vspace{1cm}
			Universit\`a degli Studi di Pavia \\
			\vspace{0.1cm}
			Dipartimento di Matematica Felice Casorati\\
			\vspace{0.1cm}
			Joint PhD Program in Mathematics Pavia - Milano Bicocca - INdAM
		\end{center}
		\vspace{2cm}
		\begin{center}
			\LARGE
			{\textbf{Immersions of surfaces into $\SL$ and into the space of geodesics of Hyperbolic space}}
		\end{center}
		
		\begin{flushright}
			\vspace{2cm}
			\textbf{PhD thesis of:} \\ Christian El Emam
		\end{flushright}
		
		\begin{flushleft}
			\vspace{0.5cm}
			\textbf{PhD Coordinator:}\\ Prof. Thomas Weigel\\
			\vspace{0.5cm}
			
			\textbf{Supervisor:} \\
			Prof. Francesco Bonsante

		\end{flushleft}

		\vspace{3cm}
		\begin{center}
			December 2020 - XXXIII Cycle
		\end{center}

\end{titlepage}}

\begin{document}
		\maketitle

	\tableofcontents
	
	\chapter*{Acknowledgements}
	
	There are several people I would like to thank for helping me get to this thesis.
	\vspace{6pt}
	
	Vorrei ringraziare con affetto Francesco Bonsante che, con molta disponibilit\`a, ha accettato tre anni fa di farmi da tutor per il dottorato. \`E stato una guida illuminante, paziente e generosa e gli sar\`o sempre grato per avermi insegnato a camminare nel mondo della ricerca. 
	
	Vorrei ringraziare Andrea Seppi, il mio "fratello maggiore" accademico, che ha investito tempo e pazienza per collaborare con un principiante. Insieme abbiamo trovato un po' di matematica interessante e io ho trovato un buon amico.
	
	I would like to thank Jean-Marc Schlenker for offering me a postdoc position in a tough moment for the academic job market. I thank him for his kindness and for the interesting conversations that we had and that we will have in Luxembourg.
	
	Un grazie profondo ai miei genitori. Mi hanno sempre lasciato libero di vivere il mio percorso e di fare le mie scelte (anche se forse a volte preferirebbero un figlio impiegato di banca che abiti a cinque minuti da casa). Vi voglio tanto bene e, anche quando non ci sono, sono sempre l\`i.
	
	Un grazie amorevole a Silvia, che accetta con pazienza di vivere una relazione quasi costantemente a distanza, ma che offre sempre ristoro al mio cuore vagabondo e che mi porta una casa ogni volta che mi viene a trovare.
	
	Ringrazio con tantissimo affetto tutti i miei compagni di dottorato. Grazie per le risate, per i pranzi, le pause t\`e e caff\` e per tutte le tante minipause fittamente distribuite durante il lavoro. \`E stato un bellissimo percorso di crescita condivisa, avete reso questi tre anni i pi\`u belli della mia formazione.
	
	Un grazie anche ai miei amici Miriam, Niyi, Oreste, Giorgio, Lorenzo, Martina, Rossana, il prof. Monti, Gabriele, Eugenia - e ne sto dimenticando qualcuno - per essere sempre una buona ragione in pi\`u per tornare a casa.
	\vspace{6pt}
	
	Being independent is beautiful, but always having someone to thank is probably even more so.

	\chapter{Introduction}
		It is a well known result in Riemannian and Pseudo-Riemannian Geometry that the theory of codimension-1 (admissible) immersions of a simply connected manifold $M$ into a space form is completely described by the study of two tensors on $M$, the \emph{first fundamental form} and the \emph{shape operator} of the immersion, corresponding respectively to the pull-back metric and to the derivative of a local normal vector field.
	
	On the one hand, for a given immersion the first fundamental form and the shape operator turn out to satisfy two equations known as Gauss and Codazzi equations. On the other hand, given a metric and a self-adjoint (1,1)-tensor on $M$ satisfying these equations, there exists one admissibile immersion into the space form with that pull-back metric and that self-adjoint tensor respectively as first fundamental form and shape operator; moreover, such immersion is unique up to post-composition with ambient isometries of the space form. This class of results are often denoted as \emph{Gauss-Codazzi Theorems}, \emph{Bonnet Theorems} or as \emph{Fundamental Theorem of hypersurfaces}, and can be extended for higher codimension. See for instance \cite{Kobayashi-Nomizu2, Gauss-Codazziinspaceforms}. We will recall in Chapter \ref{chapter: pseudoriem space forms} the Gauss-Codazzi theorems for pseudo-Riemannian space forms of constant sectional curvature $-1$, such as $\Hyp^n$ and $\widetilde {AdS^n}$.
	\vspace{3pt}

	Several remarkable results in the study of geometric structures on surfaces have been developed from the study of immersions of surfaces into $\Hyp^3$ and $AdS^3$. 
	A convex equivariant immersion of $\widetilde S$ into $\Hyp^3$ induces a projective structure on $S$ with developing map defined by the immersion $G^+_\sigma\colon \widetilde S \to \partial \Hyp^3$ with $G_\sigma^+(p)$ being the endpoint of the geodesic ray starting at $\sigma(p)$, orthogonal to $\sigma$ and directed toward the concave side: this has been a starting point for several results about projective structures \cite{Labourie, cyclic}. On the other hand, the Gauss map of an equivariant Riemannian immersion of $\widetilde S$ into $AdS^3$ points out a pair of points in the Teichm\"uller space of $S$ and minimal isometries are related to the problem of finding minimal Lagrangian maps between hyperbolic metrics, see \cite{KrasnovSchlenker, universal}.
	\vspace{5pt}
	
	In this thesis, we present two possible developments for this theory of immersions.
	\begin{itemize}
		\item Immersions of smooth manifolds into holomorphic Riemannian space forms of constant curvature $-1$, that we will denote as $\XXX_{n+1}$ (see Subsection \ref{subsec hRm intro}, complete definition in Chapter \ref {Chapter hRm}).
		
		A proper definition of holomorphic Riemannian metric (and of corresponding Levi-Civita connection and curvature) will be given in Chapter \ref {Chapter hRm} (see Definition \ref{def hRm}). Let us just mention now that holomorphic Riemannian metrics are a natural analogue of Riemannian metrics in the complex setting, and that holomorphic Riemannian space forms are simply-connected complex manifolds endowed with a holomorphic Riemannian metric with the property of being complete and with constant sectional curvature.

		The theory of immersions into $\XXX_{n+1}$ generalizes the one of immersions into hyperbolic pseudo-Riemannian space forms, since they all isometrically immerge into $\XXX_{n+1}$, leading to a more general version of the Gauss-Codazzi Theorem.
		
		This approach also allows to furnish a tool for constructing holomorphic variations of projective structures, with connections with the complex landslide. Moreover, this theory leads to the definition of a new geometric structure for surface, namely \emph{complex metrics}, for which we state a uniformization theorem that is deeply connected with Bers Simultaneous Uniformization Theorem. For $n=3$, we see that the space $\XXX_3$ is isometric to $\SL$ equipped with the global complex Killing form, up to a scalar.
		
		\item Immersions of smooth $n$-manifolds into the space $\G{n+1}$ of the geodesics of $\Hyp^{n+1}$. 
		
		In general an immersion $\sigma$ of a hypersurface $M=M^n$ in $\Hyp^{n+1}$ induces a Gauss map $G_\sigma=(G_\sigma^+, G_\sigma^-)$ into its space of geodesics, identified with $\partial \Hyp^{n+1}\times \partial \Hyp^{n+1}\setminus \Delta=:\G{n+1}$, defined by taking, at each point, the endpoints of the normal oriented unparametrized geodesic to $\sigma$. If $\sigma$ has principal curvatures in $(-1,1)$, both $G_\sigma^+$ and $G_\sigma^-$ are developing maps for projective structures.
		
		The space $\G n$ has an interesting para-K\"ahler structure (see Section \ref{sec:parakahler metric GG}) which encodes a lot of geometric aspects of the hyperbolic space.
		
		In this thesis we mainly question when an immersion $G\colon M^n\to \G{n+1}$ is the Gauss map of an immersion into $\Hyp^{n+1}$, distinguishing among local, global and equivariant integrability: in this sense we provide a characterization for integrability of Riemannian immersions in terms of their extrinsic geometry in relation with the para-K\"ahler structure of $\G{n+1}$. We will also focus on the symplectic geometry of $\G{n+1}$ and of (equivariant) Hamiltonian symplectomorphisms of this space, which turn out to send (equivariantly) integrable immersions into (equivariantly) integrable immersions.
			\end{itemize}

	 These two subjects have at least one point of contact: indeed, the space of geodesics of $\Hyp^3$ has a structure of holomorphic Riemannian space form of constant curvature $-1$.
	 
	 \vspace{10 pt}
	 This thesis draws its content mainly from the papers \cite{ioBonsante} and \cite{ioSeppi}, in which I have been working on together with Francesco Bonsante and Andrea Seppi respectively, jointly with some unedited contents and remarks.
	 
	 After the Introduction chapter, the thesis will proceed with the following structure.
	 \begin{itemize}
	 	\item Part \ref{part ambient spaces} is mainly devoted to the presentation of the spaces $\XXX_n$ and $\G {n}$. The part will start with a short chapter dedicated to pseudo-Riemannian space forms.
	 	\item Part \ref{part bonsante} is dedicated to the study of immersions of smooth $n$-manifolds into $\XXX_{n+1}$ and $\XXX_n$, and to the geometric consequences.
		\item Part \ref{parte Seppi} is devoted to the study of Gauss maps of hypersurfaces in $\Hyp^{n+1}$ and integrable immersions of $n$-manifolds into $\G{n+1}$.
	 \end{itemize}

	 \section{The spaces $\XXX_n$ and $\G n$}
	 
	 Let us give a first look at the spaces $\XXX_n$ and $\G n$.

	 	\subsection{Holomorphic Riemannian manifolds and the space $\XXX_n$}
	 	\label{subsec hRm intro}
	 	
	 	 The notion of \emph{holomorphic Riemannian metrics} on complex manifolds can be seen as a natural analogue of Riemannian metrics in the complex setting: a holomorphic Riemannian metric (also denoted as hRm in this thesis) on a complex manifold $\mathbb M$ is a holomorphic never degenerate section of $Sym_\C(T^*\mathbb M \otimes T^* \mathbb M)$, namely it consists pointwise of $\C$-bilinear inner products on the holomorphic tangent spaces whose variation on the manifold is holomorphic.
	 Such structures turn out to have several interesting geometric properties and have been largely studied (e.g. see the  works by LeBrun, Dumitrescu, Zeghib, Biswas as in \cite{holomorphicriemannian1}, \cite{holomorphicriemannian4} \cite{holomorphicriemannian2}, \cite{holomorphicriemannian3}). 
	 
	 In an attempt to provide a self-contained introduction to the aspects we will deal with, Chapter \ref{Chapter hRm} starts with some basic
	 results on holomorphic Riemannian manifolds. After a short overview in the general setting - where we recall the notions of Levi-Civita connections (with the corresponding curvature tensors) and sectional curvature in this setting - we will focus on \emph{holomorphic Riemannian space forms}, namely geodesically-complete simply-connected holomorphic Riemannian manifolds with constant sectional curvature.  
	 
	 Drawing inspitation from the proofs in the pseudo-Riemannian setting, we prove the following.
	 \begin{reptheorem}{Theorem space forms}
	 	For all integer $n\ge 2$ and $k\in \C$ there exists exactly one holomorphic Riemannian space form of dimension $n$ with constant sectional curvature $k$ up to isometry.
	 \end{reptheorem}
	 We will denote with $\mathbb X_n$ the holomorphic Riemannian space form of dimension $n$ and curvature $-1$.

	 The space $\mathbb X_n$ can be defined as 
	 \[\mathbb X_n= \{(z_1, \dots, z_{n+1}) \in \C^{n+1}\ |\ \sum_{i=1}^{n+1} z_1^2 + \dots + z_{n+1}^2 =-1 \}
	 \] with the metric inherited as a complex submanifold on $\C^{n+1}$ equipped with the standard $\C$-bilinear inner product.
	 
	 This quadric model of $\mathbb X_n$ may look familiar: it is definitely analogue to some models of $\Hy^n$, $AdS^n$, $dS^n$ and $S^n$. In fact, all the pseudo-Riemannian space forms of curvature $-1$ immerge isometrically in $\mathbb X_n$: as a result, $\Hy^n$, $AdS^n$ embed isometrically while $dS^n$ and $S^n$ embed anti-isometrically, i.e. $-dS^n$ and $-S^n$ (namely, $dS^n$ and $S^n$ equipped with the opposite of their standard metric) embed isometrically.
	 
	 For $n=1,2,3$, $\mathbb X_n$ turns out to be familiar also in another sense:
	 \begin{itemize}
	 	\item $\mathbb X_1$ is isometric to $\C^*$ equipped with the holomorphic Riemannian metric defined by the quadratic differential $\frac{dz^2}{z^2}$;
	 	\item $\mathbb X_2$ is isometric to the space $\GGG$ of oriented lines of $\Hy^3$, canonically identified with $\CP^1\times \CP^1\setminus \Delta$, equipped with the only $\PSL$-invariant holomorphic Riemannian metric of curvature $-1$, that we will denote as $\inners_\GGG$;
	 	\item $\mathbb X_3$ is isometric (up to a scale) to $\SL$ equipped with the Killing form. Chapter \ref{Chapter SL} will treat the particular case of $\SL$ and of its 2:1 quotient $\PSL$.
	 \end{itemize}

\subsection{The space $\G n$ of geodesics of $\Hyp^n$ and the bundle \\
	$\mathrm p\colon T^1\Hyp^n\to \G n$}
\label{sec intro 3}
	In the groundbreaking paper \cite{zbMATH03791301}, Hitchin observed the existence of a natural complex structure on the space of oriented geodesics in Euclidean three-space.  A large interest has then grown on the geometry of the space of oriented (maximal unparametrized) geodesics of Euclidean space of any dimension (see \cite{zbMATH02228497,zbMATH02232144,zbMATH05530808,zbMATH06300570}) and of several other Riemannian and pseudo-Riemannian manifolds (see \cite{zbMATH05988505,zbMATH06268759,zbMATH06828623,MR3888623,zbMATH07050784}). In this paper, we are interested in the case of hyperbolic $n$-space $\Hyp^n$, whose  space of oriented geodesics is denoted here by $\G{n}$. The  geometry of $\G{n}$ has been addressed in \cite{zbMATH05187891} and, for $n=3$, in \cite{zbMATH05821527,zbMATH05779824,zbMATH06173011,zbMATH06481975}.  For the purpose of this paper, the most relevant geometric structure on $\G{n}$ is a natural \emph{para-K\"ahler structure} $(\GG, \JJ, \Omega)$ (introduced in \cite{zbMATH05988505,zbMATH06268759}), 
	a notion that we will define in detail in Section \ref{sec:parakahler metric GG}.  A particularly relevant feature of such para-K\"ahler structure is the fact that the Gauss map of an oriented immersion $\sigma:M\to\Hyp^n$, which is defined as the map $G_\sigma\colon M\to \G n$ that associates to a point of $M$ the orthogonal geodesic of $\sigma$ endowed with the compatible orientation, is a Lagrangian immersion in $\G{n}$. 
	We will come back to this important point in Subsection \ref{sec intro 1}. Let us remark here that, as a consequence of the geometry of the hyperbolic space $\Hyp^n$, an oriented geodesic in $\Hyp^n$ is characterized, up to orientation preserving reparametrization, by the ordered couple of its "endpoints'' in the visual boundary $\partial \Hyp^n$: this gives an identification 
$
\G{n}\cong\partial \Hyp^n\times\partial \Hyp^n \setminus \Delta
$. 

\vspace{5pt}

 In this paper we give an alternative model for the para-K\"ahler structure of $\G{n}$ with respect to the previous literature (\cite{zbMATH05187891,zbMATH05779824,zbMATH05988505,zbMATH06268759}), which is well-suited for the problem of (equivariant) integrability.
%
The geodesic flow induces a natural principal $\R$-bundle structure {whose total space is $T^1\Hyp^{n+1}$ and whose bundle map is $\mathrm{p}:T^1\Hyp^{n+1}\to \G{n+1}$ defined in Equation \eqref{Def p intro}}, and acts by isometries of the \emph{para-Sasaki metric} $g$, which is a pseudo-Riemannian version of the classical Sasaki metric on $T^1\Hyp^{n+1}$. Let us denote by $\chi$ the infinitesimal generator of the geodesic flow, which is a vector field on $T^1\Hyp^{n+1}$ tangent to the fibers of $\mathrm p$. The idea is to define each element that constitutes the para-K\"ahler structure of $\G{n+1}$ (see the items below) by push-forward of certain tensorial quantities defined on the $g$-orthogonal complement of $\chi$, showing that the push-forward is well defined by invariance under the action of the geodesic flow. More concretely:

\begin{itemize}
	\item The \emph{pseudo-Riemannian metric} $\GG$ of $\G{n+1}$ (of signature $(n,n)$) is defined as push-forward of the restriction of $g$ to $\chi^\perp$;
	\item The \emph{para-complex structure} $\JJ$ (that is, a $(1,1)$ tensor whose square is the identity and whose $\pm 1$-eigenspaces are integrable distributions of the same dimension) is obtained from an endomorphism $J$ of $\chi^\perp$, invariant under the geodesic flow, which essentially switches {the} horizontal and vertical {distributions} in $T^1\Hyp^{n+1}$;
	\item The \emph{symplectic form} $\Omega$ arises from a similar construction on $\chi^\perp$, in such a way that $\Omega(X,Y)=\GG(X,\mathbb J Y)$.
\end{itemize}

The symplectic geometry of $\G{n+1}$ has a deep relation with the structure of $T^1\Hyp^{n+1}$. Indeed 
the total space of $T^1\Hyp^{n+1}$ is endowed with a connection form 
$\omega$, whose kernel consists precisely of $\chi^\perp$ (See Definition \ref{Def connection form}). In Proposition \ref{prop:identity curvature form} we prove that $\Omega$ is in fact the curvature form of the connection $\omega$ on $\mathrm p$, namely:
\begin{equation}\label{eq omega Omega intro}
	d\omega=\mathrm p^*\Omega~.
\end{equation}	 
\vspace{5pt}	 
	
	Finally, in Section \ref{subsection: G(3) modello Bonsante-Christian} we show that the holomorphic Riemannian metric $\inners_{\GGG}$  on $\G 3$ (mentioned above in the end of Subsection \ref{subsec hRm intro} and defined in more detail in Section \ref{section: G3 e hRm}) is related to the pseudo-Riemannian metric $\GG_3$ by
	\[
	Re \Big( \inners_{\GGG} \Big)= \GG_3\ .
	\]

	 \section{Immersions into $\XXX_{n+1}$ and geometric structures}

	In this thesis we approach the study of immersions of smooth manifolds into the spaces $\XXX_n$. 
	
	The study of this kind of immersions turns out to have some interesting consequences, including some remarks concerning geometric structures on surfaces. In order to give a general picture, here are a few facts.
	\begin{itemize}
		\item it provides, as we mentioned, a formalism for the study of immersions of surfaces into $\SL$ and into the space of geodesics of $\Hy^3$;
		\item it generalizes the classical theory of immersions into non-zero curvature space forms, leading to a model to study the transitioning of hypersurfaces among $\Hy^n$, $AdS^n$, $dS^n$ and $S^n$;
		\item it furnishes a tool to construct holomorphic maps from complex manifolds to the character variety of $SO(n, \C)$ (including $\PSL\cong SO(3,\C)$), providing also an alternative proof for the holomorphicity of the complex landslide map (see \cite{cyclic});
		\item it allows to introduce a notion of \emph{complex metrics} which extends Riemannian metrics and which turns out to be useful to describe the geometry of couples of projective structures with the same monodromy. We also show a uniformization theorem in this setting which is in a way equivalent to the classical Bers Theorem for quasi-Fuchsian representations.
	\end{itemize}

	\subsection{Immersions of manifolds into $\mathbb X_n$} In Chapter \ref{Chapter immersioni in Xn}, we approach the study of immersions of smooth manifolds into $\XXX_n$. The idea is to try to emulate the usual approach as in the pseudo-Riemannian case, but things turn out to be technically more complicated. 
	
	Given a smooth map $\sigma\colon M \to \XXX_n$, the pull-back of the holomorphic Riemannian metric is some $\C$-valued symmetric bilinear form, so one can extend it to a $\C$-bilinear inner product on $\CTM=TM \otimes \C$; it now makes sense to require it to be non-degenerate. 
	This is the genesis of what we define as \emph{complex (valued) metrics} on smooth manifolds, namely smoothly varying non-degenerate inner products on each $\C T_x M$, $x\in M$. We will say that an immersion $\sigma\colon M \to \XXX_n$ is \emph{admissible} if the pull-back metric is a complex valued metric. By elementary linear algebra, $\sigma$ is admissible only if $dim(M)\le n$: the real codimension is $n$ and, despite it seems high, it cannot be lower than that. It therefore makes sense to redifine the $\emph{codimension}$ of 
	$\sigma$ as $n-dim(M)$. In this paper we focus on immersions of codimension $1$ and $0$.
	
	It turns out that this condition of admissibility is the correct one in order to have extrinsic geometric objects that are analogue to the ones in the pseudo-Riemannian case. Complex metrics come with some complex analogue
	of the Levi-Civita connection, which, in turn, allows to define a curvature tensor and a notion of sectional curvature. In codimension 1, admissible immersions come with a notion of local normal vector field (unique up to a sign) that allows to define a second fundamental form and a shape operator. 
	
	Section \ref{sec Gauss-Codazzi eq first part} ends with Theorem \ref{Teo Gauss-Codazzi}, in which we deduce some analogue of Gauss and Codazzi equations. With a bit more work, we show in Theorem \ref{Teoremone} that immersion data satisfying Gauss-Codazzi equations can be integrated for simply-connected domains.
	Let us state here the theorem in the particular case of immersions of surfaces into $\mathbb X_3\cong \SL$.
	
	\begin{repcor}{Cor Gauss-Codazzi PSL}
		Let $S$ be a smooth simply connected surface. Consider a complex metric $g$ on $S$, with induced Levi-Civita connection $\nabla$, and a $g$-self-adjoint bundle-homomorphism 
		$\Psi\colon \CTS \to \CTS$.
		
		The pair $(g, \Psi)$
		\begin{align}
			1) &d^\nabla \Psi \equiv 0;\\
			2)	&K=-1+det(\Psi).
		\end{align}
		if and only if there exists an isometric immersion $\sigma \colon S \to \SL$ whose corresponding shape operator is $\Psi$. Moreover, such $\sigma$ is unique up to post-composition with elements in $\Isom_0(\SL)= \Proj(\SL \times \SL)$. 
	\end{repcor}

	The essential uniqueness of the immersion grants that if
	$\Gamma$ is a group acting on $S$ preserving $g$ and $\Psi$, then the immersion $\sigma$ is $(\Gamma,\Isom_0(\SL) )$-equivariant, namely there exists a representation \[mon_\sigma\colon \Gamma\to Isom_0(\SL)\] such that for all $\gamma\in \Gamma$
	\[
	mon_\sigma(\gamma)\circ \sigma = \sigma \circ\gamma. \]
	
	In other words, if $S$ is not simply connected, solutions of Gauss-Codazzi equations correspond to $(\pi_1(S), \Isom_0(\SL) )$-equivariant immersions of its universal cover into $\SL$.
	\vspace{5mm}

	\subsection{Immersions in codimension zero and into pseudo-Riemannian space forms}	 
	The study of immersions into $\mathbb X_n$ in codimension zero leads to an interesting result.

	\begin{reptheorem}{teoremone stessa dimensione}        
		Let $M$ be a smooth manifold of dimension $n$.
		
		Then, $g$ is a complex metric for $M$ with constant sectional curvature $-1$ if and only if there exists an isometric immersion
		\[(\widetilde M, \widetilde g)\to \mathbb{X}_n\] which is unique up to post-composition with elements in $Isom(\mathbb X _{n})$ and therefore is $(\pi_1(M), O(n+1,\C) )$-equivariant.
	\end{reptheorem}

	This result can be deduced by the Gauss-Codazzi Theorem: in fact, immersions in $\mathbb X_n$ of codimension $0$ correspond to  codimension $1$ totally geodesic immersions in $\mathbb X_{n+1}$, namely immersions with shape operator $\Psi=0$. A full proof is in Section \ref{section codimension zero}.
	
	As a result, every pseudo-Riemannian space form of constant sectional curvature $-1$ and dimension $n$ admits an essentially unique isometric immersion into $\mathbb X_n$.
	
	\vspace{5mm}
	
	In fact, the last remark and the similar description of Gauss-Codazzi equations for immersions into pseudo-Riemannian space forms lead to Theorem \ref{da X_n a space forms}, which, in the case of $\SL$, can be stated in the following way.
	
\begin{reptheorem}{teo GC passando da pseudoRiem SL}
	Let $\sigma\colon S\to \SL$ be an admissible immersion with pull-back metric $g$ and shape operator $\Psi$.
	\begin{itemize}
		\item $\sigma(S)$ is contained in the image of an isometric embedding of $\Hy^3$ if and only if $g$ is Riemannian and $\Psi$ is real.
		\item 
		$\sigma(S)$ is contained in the image of an isometric embedding of $AdS^3$ if and only if either $g$ is Riemannian and $i\Psi$ is real, or if $g$ has signature $(1,1)$ and $\Psi$ is real.
		\item $\sigma(S)$ is contained in the image of an isometric embedding of $-dS^3$ if and only if either $g$ has signature $(1,1)$ and $i\Psi$ is real, or if $g$ is negative definite and $\Psi$ is real.
		\item $\sigma(S)$ is contained in the image of an isometric embedding of $-\Sph^3$ if and only if $g$ is negative definite and $i\Psi$ is real. 
	\end{itemize}
\end{reptheorem}

	\subsection{Holomorphic dependence of the monodromy on the immersion data}	
	Given a smooth manifold $M$ of dimension $n$, we say that $(g, \Psi)$ is a couple of \emph{immersion data} for $M$ if there exists a $\pi_1(M)$-equivariant immersion $\widetilde M\to \mathbb X_{n+1}$ with pull-back metric $\widetilde g$ and shape operator $\widetilde \Psi$. As a result of the essential uniqueness of the immersion, each immersion data comes with a monodromy map $mon_\sigma\colon \pi_1(M)\to Isom_0(\mathbb X_{n+1})$.
	
	In Chapter \ref{Section dipendenza olo}, we consider families of immersion data $\{(g_\lambda, \Psi_\lambda)\}_{\lambda \in \Lambda}$ for $M$. 
	
	Let $\Lambda$ be a complex manifold. We say that the family $\{(g_\lambda, \Psi_\lambda)\}_{\lambda \in \Lambda}$ is  \emph{holomorphic} if for all $p\in M$, the functions
	\begin{align*}
		\Lambda &\to Sym^2(\C T^*_pM)\\ 
		\lambda &\mapsto g_\lambda(p)
	\end{align*}
	and 
	\begin{align*}
		\Lambda &\to End_\C(\C T_pM)\\ 
		\lambda &\mapsto \Psi_\lambda(p)
	\end{align*}
	are holomorphic. 
	
	For a fixed hyperbolic Riemannian metric $h$ on a surface $S$, an instructive example is given by the family $\{(g_z,\psi_z)\}_{z\in \C}$ defined by
	\[
	\begin{cases}
		g_z= \cosh^2(z) h; \\
		\psi_z= \tanh(z) id
	\end{cases} \ ,
	\] 
	whose monodromy is going to be the monodromy of an immersion into $\Hy^3$ for $z\in \R$ and the monodromy of an immersion into $AdS^3$ for $z\in i\R$. 
	
	The main result of Chapter \ref{Section dipendenza olo} is the following.
	
\begin{reptheorem}{Teo dipendenza olomorfa}
	Let $\Lambda$ be a complex manifold and $M$ be a smooth manifold of dimension $n$. 
	
	Let $\{(g_\lambda, \Psi_\lambda)\}_{\lambda\in \Lambda}$ be a holomorphic family of immersion data for $\pi_1(M)$-equivariant immersions $\widetilde M\to \mathbb X_{n+1}$. Then there exists a smooth map
	\[
	\sigma\colon \Lambda \times \widetilde M \to \mathbb X_{n+1}
	\]
	such that, for all $\lambda\in \Lambda$ and $p\in M$:
	\begin{itemize}
		\item $\sigma_\lambda:= \sigma(\lambda, \cdot)\colon \widetilde M \to \mathbb X_{n+1}$ is an admissible immersion with immersion data $(g_\lambda, \Psi_\lambda)$;
		\item $\sigma(\cdot,p)\colon \Lambda \to \mathbb X_{n+1}$ is holomorphic.
	\end{itemize}
	Moreover, for all $\alpha\in \pi_1(M)$, the monodromy map evaluated in $\alpha$
	\begin{align*}
		\Lambda &\to  SO(n+2, \C)\\
		\lambda &\mapsto mon(\sigma_\lambda) (\alpha)
	\end{align*}
	is holomorphic.
\end{reptheorem}

	As an application of Theorem \ref{Teo dipendenza olomorfa}, in Subsection \ref{Subsection complex landslide} we show an alternative proof of the holomorphicity of the complex landslide introduced in \cite{cyclic}, whose definition we recall in Section 
	\ref{sec complex landslide}. More precisely, denoting the complex landslide with initial data $(h, h(b\cdot, b\cdot))$ as the complex flow $z\mapsto P_{h,b}(z)$, we obtain the following.
	
	\begin{reptheorem}{Teo alternativo landslide}
	The holonomy of the projective structure $P_{h,h(b\cdot, b\cdot)}(z)$ is equal to the monodromy of the immersion $\widetilde S\to \GGG$ with induced complex metric 
	\[
	g_z = h \bigg( (\cos(z)id -  \sin(z) Jb)\ \cdot,  (\cos(z)id -  \sin(z) Jb)\cdot \bigg)
	\]
	with constant curvature $-1$.
	
	As a consequence of Theorem \ref{Teo dipendenza olomorfa} and Theorem \ref{Hol e diffeo}, $P_{h,b}$ is holomorphic.
	\end{reptheorem}

	\subsection{Uniformizing complex metrics and Bers Theorem}	
	
	In Chapter \ref{Section uniformization} we focus on complex metrics on surfaces. 
	
	Even in dimension 2, complex metrics can have a rather wild behaviour. Nevertheless, we point out a neighbourhood of the Riemannian locus whose elements have some nice features: we will call these elements \emph{positive complex metrics} (Definition \ref{Def positive metrics}).
	
	One of the main properties of positive complex metrics on orientable surfaces is that they admit a notion of \emph{Area form}, analogous to the one defined for Riemannian surfaces (see Theorem \ref{Teorema orientazione}).
	
	We prove that the standard Gauss-Bonnet Theorem also holds for positive complex metrics on closed surfaces:	
	\begin{reptheorem}{Gauss Bonnet}
		Let $S$ be a closed oriented surface and $g$ a complex metric of constant curvature $K_g$ and area form $dA_g$.
		Then \[
		\int_S K_g dA_g = 2 \pi \chi(S)
		\]
		where $\chi(S)$ is the Euler-Poincaré characteristic of $S$.
	\end{reptheorem}
	   
	After that, we prove the most relevant result in Chapter \ref{Section uniformization}, consisting in a version of the Uniformization Theorem for positive complex metrics (see Theorem \ref{Uniformization 1} for the complete version of the statement). 
	
	\begin{Theorem*}
		Let $S$ be a surface with $\chi(S)<0$. 
		
		For any positive complex metric $g$ on $S$ there exists a smooth function $f\colon S\to \C^*$ such that the positive complex metric $f\cdot g$ has constant curvature $-1$ and has quasi-Fuchsian monodromy.
	\end{Theorem*}

	The proof of this result uses Bers Simultaneous Uniformization Theorem (Theorem \ref{Bers} in this paper) and, in a sense, is equivalent to it.
	
	Indeed, by Theorem \ref{teoremone stessa dimensione}, complex metrics on $S$ with constant curvature $-1$ correspond to equivariant isometric immersions of $\widetilde S$ into $\GGG=\CP^1\times \CP^1\setminus\Delta$: hence, they can be identified with a couple of maps $\widetilde S\to \CP^1$ with the same monodromy. In this sense, Bers Theorem provides a whole group of immersions into $\CP^1\times \CP^1\setminus\Delta$: such immersions correspond to complex metrics of curvature $-1$ which prove to be positive. The proof of the uniformization theorem consists in showing that every complex positive metric is conformal to a metric constructed with this procedure.

	\section{Integrable immersions into $\G{n+1}$}
	
As we mentioned before, a hypersurface immersion $\sigma\colon M\to \Hyp^{n+1}$ defines a \emph{Gauss map} $G_\sigma=(G_\sigma^+, G_\sigma^-)\colon M\to \G {n+1}$ defined by $G_\sigma(p)$ being the orthogonal geodesic to $\sigma$ in $\sigma(p)$. The main interest lies in the study of immersions $\sigma$ which are equivariant with respect to some group representation $\rho:\pi_1(M)\to\Isom(\Hyp^{n+1})$, for an $n$-manifold $M$. 
	
 Uhlenbeck \cite{zbMATH03840752} and Epstein \cite{Epstein:1986aa,zbMATH03948676,zbMATH04015637} studied immersed hypersurfaces in $\Hyp^n$, mostly in dimension $n=3$, highlighting the relevance of hypersurfaces satisfying the geometric condition {for which principal curvatures are everywhere}  different from $\pm 1$, sometimes called \emph{horospherically convexity}: this is the condition that ensures that the components of the Gauss maps $G_\sigma^\pm$ are locally invertible. 
 
 As we will see, we will often focus on Gauss maps of immersions into $\Hyp^{n+1}$ \emph{with small principal curvatures}, i.e. with principal curvatures lying in $(-1,1)$.
	
The two main results of Part \ref{parte Seppi} are Theorem \ref{teorema hol H baby} and Theorem \ref{thm:second char ham}, which we will discuss in the following subsections.

	\subsection{Integrability of immersions in $\G n$}\label{sec intro 1}
	
	
	
	One of the main goals of this paper is to discuss when an immersion $G\colon M^n\to \G{n+1}$ is \emph{integrable}, namely when it is the Gauss map of an immersion $M\to \Hyp^{n+1}$, in terms of the geometry of $\G{n+1}$. We will distinguish three types of integrability conditions, which we list from the weakest to the strongest: \begin{itemize}
		\item An immersion $G\colon M\to \G {n+1}$ is \emph{locally integrable} if for all $p\in M$ there exists a neighbourhood $U$ of $p$ such that $G|_{U}$ is the Gauss map of an immersion $U\to \Hyp^{n+1}$;
		\item An immersion $G\colon M \to \G{n+1}$ is \emph{globally integrable} if it is the Gauss map of an immersion $M\to \Hyp^{n+1}$;
		\item Given a representation $\rho\colon \pi_1(M)\to \Isom (\Hyp^{n+1})$, a $\rho$-equivariant immersion $G\colon \widetilde M\to \G{n+1}$ is \emph{$\rho$-integrable} 
		if it is the Gauss map of a $\rho$-equivariant immersion $\widetilde M \to \Hyp^{n+1}$.
	\end{itemize}
	
	Let us clarify here that, since the definition of Gauss map requires to fix an orientation on $M$ (see Definition \ref{Def:lift and gauss}),  the above three definitions of integrability have to be interpreted as: ``there exists an orientation on $U$ (in the first case) or $M$ (in the other two) such that $G$ is the Gauss map of an immersion in $\Hyp^{n+1}$ with respect to that orientation''.
	
	We will mostly restrict to immersions $\sigma$ with small principal curvatures, which is equivalent to the condition that the Gauss map $G_\sigma$ is Riemannian, meaning that the pull-back by $G_\sigma$ of the ambient pseudo-Riemannian metric of $\G{n+1}$ is positive definite, hence a Riemannian metric (Proposition \ref{prop: small curv sse riemannian}). 
	
	\subsubsection{Local integrability} As it was essentially observed by Anciaux in \cite{zbMATH06268759}[Theorem 2.10], local integrability admits a very simple characterization in terms of the symplectic geometry of $\G{n+1}$.
	\begin{reptheorem}{cor: local integrability}
		Let $M^n$ be a manifold and $G\colon M\to \G {n+1}$ be an immersion. Then $G$ is locally integrable if and only if it is Lagrangian.
	\end{reptheorem}
	The methods of this paper easily provide a proof of Theorem \ref{cor: local integrability}, which is independent from the content of \cite{zbMATH06268759}.
	{Let us denote by $T^1\Hyp^{n+1}$  the unit tangent bundle of $\Hyp^{n+1}$ and {by}
		\begin{equation}\label{Def p intro}
			\mathrm p\colon T^1\Hyp^{n+1}\to \G{n+1}~
		\end{equation} the map such that $\mathrm{p}(x,v)$ is the oriented geodesic of $\Hyp^{n+1}$ tangent to $v$ at $x$. Then, if $G$ is Lagrangian, we prove that one can locally construct maps $\zeta:U\to T^1\Hyp^{n+1}$ (for $U$ a simply connected open set) such that $\mathrm p\circ\zeta=G$. 
		Up to restricting the domain again, one can find such a map $\zeta$ so that it projects to an immersion $\sigma$ in $\Hyp^{n+1}$ (Lemma \ref{lemma:desingularize for small t}), and the Gauss map of $\sigma$ is $G$ by construction.} 
	
	Our next results are, to our knowledge, completely new and give characterizations of global integrability and $\rho$-integrability under the assumption of small principal curvatures.

	\subsubsection{Global integrability} The problem of global integrability is in general more subtle than local integrability. As a matter of fact, in Example \ref{ex: Lagrangian not globally integrable} we construct an example of a locally integrable immersion $G:(-T,T)\to\G{2}$, defined for some $T\in \R^+$, that is not globally integrable. By taking a cylinder on this curve, one easily sees that the same phenomenon occurs in any dimension. We stress that in our example $M=(-T,T)$ (or the product $(-T,T)\times\R^{n-1}$ for $n>2$) is simply connected: {the key point in our example is that one can find globally defined maps $\zeta:M\to T^1\Hyp^{n+1}$ such that $G=\mathrm p\circ\zeta$, but no such $\zeta$ projects to an immersion in $\Hyp^{n+1}$. }
	
	Nevertheless, we show that this issue does not occur for Riemannian immersions $G$. In this case any immersion $\sigma$ whose Gauss map is $G$ (if it exists) necessarily has small principal curvatures. We will always restrict  to this setting hereafter. In summary, we have the following characterization of global integrability for $M$ simply connected.
	\begin{reptheorem}{prop: riemannian global integrability}
		Let $M^n$ be a simply connected manifold and $G\colon M\to \G {n+1}$ be a Riemannian immersion. Then $G$ is globally integrable if and only if it is Lagrangian.
	\end{reptheorem}
	
	We give a characterization of global integrability for $\pi_1(M)\ne \{1\}$ in Corollary \ref{cor hol H baby}, which is a direct consequence of our first characterization of $\rho$-integrability (Theorem \ref{teorema hol H baby}). Anyway, we remark that if a Riemannian and Lagrangian immersion  $G\colon M\to \G {n+1}$ is also complete (i.e. has complete first fundamental form), then $M$ is necessarily simply connected:
	
	
	\begin{reptheorem}{Cor G complete}
		Let $M^n$ be a manifold. If $G\colon M\to \G {n+1}$ is a complete Riemannian and Lagrangian immersion, then $M$ is diffeomorphic to $\R^n$ and $G$ is the Gauss map of a proper embedding $\sigma:M\to\Hyp^{n+1}$.
	\end{reptheorem}
	
	In Theorem \ref{Cor G complete} the conclusion that $G=G_\sigma$ for $\sigma$ a proper embedding follows from the fact that $\sigma$ is complete, which is an easy consequence of Equation \eqref{eq:fff gauss} relating the first fundamental forms of $\sigma$ and $G_\sigma$, and the non-trivial fact that complete immersions in $\Hyp^{n+1}$ with small principal curvatures are proper embeddings (Proposition \ref{prop injectivity}).
	

	\subsubsection{Equivariant integrability}  {Let us first observe that the problem of $\rho$-integrability presents some additional difficulties than global integrability}. Assume $G\colon \widetilde M \to \G{n+1}$ is a Lagrangian, Riemannian and $\rho$-equivariant immersion for some representation $\rho\colon \pi_1(M^n)\to \Isom(\Hyp^{n+1})$.
	Then, by Theorem \ref{prop: riemannian global integrability}, there exists $\sigma\colon \widetilde M\to\Hyp^{n+1}$ with Gauss map $G$, {but the main issue is that such a $\sigma$ will not be}  $\rho$-equivariant in general, as one can see in Examples \ref{ex: global integrable non equivariant} and \ref{ex: global integrable non equivariant2}.

	Nevertheless, $\rho$-integrability of Riemannian immersions into $\G{n+1}$ can still be characterized in terms of their extrinsic geometry. Let $\overline {\mathrm H}$ be the mean curvature vector of $G$, defined as the trace of the second fundamen{tal form, and let $\Omega$ be the symplectic form of $\G{n+1}$. Since $G$ is $\rho$-equivariant, the $1$-form $G^*(\Omega(\overline {\mathrm H},\cdot))$ on $\widetilde M$ is invariant under the action of $\pi_1(M)$, so it descends to a $1$-form on $M$. 
		One can prove that such $1$-form on $M$ is closed (Corollary \ref{cor:maslov closed}):
		we will denote with $\mu_G$ its cohomology class in $H^1_{dR}(M,\R)$ and we will call it the \emph{Maslov class} of $G$, in accordance with some related interpretations of the Maslov class in other geometric contexts (see, among others, \cite{zbMATH03730903,zbMATH00704931,zbMATH01523513,zbMATH01786838}). 
		The Maslov class encodes the existence of equivariantly integrating immersions, in the sense stated in the following theorem.

		\begin{reptheorem}{teorema hol H baby}
			Let $M^n$ be an {orientable} manifold, $\rho\colon \pi_1(M) \to \Isom(\Hyp^{n+1})$ be a representation and $G\colon \widetilde M \to \G{n+1}$ be a $\rho$-equivariant Riemannian and Lagrangian {immersion}. Then $G$   is $\rho$-integrable if and only if the Maslov class $\mu_G$ vanishes.
		\end{reptheorem}

		Applying Theorem \ref{teorema hol H baby} to a trivial representation, we immediately obtain a characterization of global integrability for Riemannian immersions, thus extending Theorem \ref{prop: riemannian global integrability} to the case $\pi_1(M)\neq \{1\}$.
		
		\begin{repcor}{cor hol H baby}
			Let $M^n$ be an {orientable} manifold and $G\colon M \to \G{n+1}$ be a Riemannian and Lagrangian immersion. Then $G$ is globally integrable if and only if its Maslov class $\mu_G$ vanishes.
		\end{repcor}

		\subsection{Nearly-Fuchsian representations}\label{sec intro 2}
		
		Let us now focus on the case for which $M$ is a closed oriented manifold. Although our results apply to any dimension, we borrow the terminology from the three-dimensional case (see \cite{zbMATH06204974}) and say that a representation $\rho\colon \pi_1(M)\to \Isom(\Hyp^{n+1})$ is \emph{nearly-Fuchsian} if there exists a $\rho$-equivariant immersion $\sigma\colon\widetilde M\to \Hyp^{n+1}$ with small principal curvatures. We show (Proposition \ref{prop:action free prop disc0}) that the action of a nearly-Fuchsian representation on $\Hyp^{n+1}$ is free, properly discontinuously and convex cocompact; the quotient of $\Hyp^{n+1}$ by $\rho(\pi_1(M))$ is called \emph{nearly-Fuchsian manifold}. 
		
		Moreover, the action of $\rho(\pi_1(M))$ extends to a free and {properly} discontinuous action on the complement of a topological $(n-1)$-sphere $\Lambda_\rho$ (the \emph{limit set} of $\rho$) in the visual boundary $\partial\Hyp^{n+1}$. Such complement is the disjoint union of two topological $n$-discs which we denote by $\Omega_+$ and $\Omega_-$.  It follows that there exists a maximal open region of $\G {n+1}$ over which a nearly-Fuchsian representation $\rho$ acts freely and properly discontinuously. This region is defined as the subset of $\G{n+1}$ consisting of oriented geodesics having either final endpoint in $\Omega_+$ or initial endpoint in $\Omega_-$. The quotient of this open region via the action of $\rho$, that we denote with $\mathcal G_\rho$, inherits a para-K\"ahler structure.
		
		
		Let us first state a uniqueness result concerning nearly-Fuchsian representations. A consequence of Theorem \ref{teorema hol H baby} and the definition of Maslov class is that if $G$ is a $\rho$-equivariant, Riemannian and Lagrangian immersion which is furthermore \emph{minimal}, i.e. with $\overline {\mathrm{H}}=0$, then it is $\rho$-integrable. Together with an application of a maximum principle in the corresponding nearly-Fuchsian manifold, we prove:
		
		\begin{repcor}{cor:uniqueness min lag}
			Given a closed orientable manifold $M^n$ and a representation $\rho:\pi_1(M)\to\Isom(\Hyp^{n+1})$, there exists at most one  $\rho$-equivariant {Riemannian} minimal Lagrangian immersion $G:\widetilde M\to\G{n+1}$ up to reparametrization. If such a $G$ exists, then $\rho$ is nearly-Fuchsian and $G$ induces a {minimal Lagrangian} embedding of $M$ in $\mathcal G_\rho$.
		\end{repcor}

		
		In fact, for any $\rho$-equivariant immersion $\sigma:\widetilde M\to\Hyp^{n+1}$ with small principal curvatures, the hyperbolic Gauss maps $G_\sigma^\pm$ are equivariant diffeomorphisms between $\widetilde M$ and $\Omega_\pm$. Hence up to changing the orientation of $M$, which corresponds to swapping the two factors {$\partial \Hyp^{n+1}$} in the identification $\G{n+1}\cong\partial\Hyp^{n+1}\times\partial\Hyp^{n+1}\setminus\Delta$, the Gauss map of $\sigma$ takes values in the maximal open region defined above, and induces an embedding of $M$ in $\mathcal G_\rho$. 
		
		This observations permits to deal (in the cocompact case) with embeddings in $\mathcal G_\rho$ instead of $\rho$-equivariant embeddings  in $\G{n+1}$. In analogy with the definition of $\rho$-integrability defined above, we will say that a $n$-dimensional
		submanifold $\mathcal L\subset \mathcal G_\rho$ is $\rho$-\emph{integrable} if it is the image in the quotient of a $\rho$-integrable embedding in $\G{n+1}$. Clearly such $\mathcal L$ is necessarily Lagrangian by Theorem \ref{cor: local integrability}. We are now ready to state our second characterization result for $\rho$-integrability.



		\begin{reptheorem}{thm:second char ham}
			Let $M$ be a closed orientable $n$-manifold, $\rho:\pi_1(M)\to\Isom(\Hyp^{n+1})$ be a nearly-Fuchsian representation and $\mathcal L\subset\mathcal G_\rho$ be a Riemannian $\rho$-integrable submanifold. Then a Riemannian submanifold $\mathcal L'$ is $\rho$-integrable if and only if there exists $\Phi\in \mathrm{Ham}_c(\mathcal G_\rho,\Omega)$ such that $\Phi(\mathcal L)=\mathcal L'$.
		\end{reptheorem}
		
		{In Theorem \ref{thm:second char ham} we denoted} by $\Ham_c(\mathcal G_\rho,\Omega)$ the group of compactly-supported \emph{Hamiltonian symplectomorphisms} of $\mathcal G_\rho$ with respect to its symplectic form $\Omega$. (See Definition \ref{Def Hamc}). The proof of Theorem \ref{thm:second char ham} in fact shows that if $\mathcal L$ is $\rho$-integrable and $\mathcal L'=\Phi(\mathcal L)$ for $\Phi\in \mathrm{Ham}_c(\mathcal G_\rho,\Omega)$, then $\mathcal L'$ is integrable as well, even without the hypothesis that $\mathcal L$ and $\mathcal L'$ are Riemannian submanifolds.
		
		If $\rho$ admits an equivariant Riemannian minimal Lagrangian immersion, then Theorem \ref{thm:second char ham} can be restated by saying that a Riemannian and Lagrangian submanifold $\mathcal L'$ is $\rho$-integrable if and only if it is in the $\Ham_c(\mathcal G_\rho,\Omega)$-orbit of \emph{the} minimal Lagrangian submanifold $\mathcal L\subset\mathcal G_\rho$, which is unique by Theorem \ref{cor:uniqueness min lag}.

		\subsection{Relation with geometric flows}
		
		Finally, in Section \ref{app:geometric flows} we apply these methods to study the relation between evolutions by geometric flows in $\Hyp^{n+1}$ and in $\G{n+1}$. More precisely, suppose that $\sigma_\bullet:M\times(-\epsilon,\epsilon)\to\Hyp^{n+1}$ is a smoothly varying family of Riemannian immersions that satisfy:
		$$\frac{d}{dt} \sigma_t  = f_t \nu_t$$
		where $\nu_t$ is the normal vector of $\sigma_t$ and $f_\bullet:M\times(-\epsilon,\epsilon)\to\R$ is a smooth function. Then the variation of the Gauss map $G_t$ of $\sigma_t$ is given, up to a tangential term, by the normal term  $-\JJ(dG_t({\overline\nabla}{}^t f_t))$, where ${\overline \nabla}{}^t f_t$ denotes the gradient with respect to the first fundamental form of $G_t$, that is, the Riemannian metric $G_t^*\GG$. 
		
		Let {us} consider the special case of the flow defined by
		$f_t:=\kappa_{\sigma_t}$, where $\kappa_{\sigma_t}$ is defined as the sum of the hyperbolic arctangent of the principal curvatures of $\sigma_t$ (see Equation \eqref{eq:aux function}). The study of the associated flow has been suggested in dimension three in \cite{zbMATH01789966}, by analogy with a similar flow on surfaces in the three-sphere. 
		Combining the aforementioned result of Section \ref{app:geometric flows} with
	Proposition \ref{Prop: formula H in G}, we obtain that such flow in $\Hyp^{n+1}$ induces the Lagrangian mean curvature flow in $\G{n+1}$ up to tangential diffeomorphisms. A similar result has been obtained in Anti-de Sitter space (in dimension three) in \cite{zbMATH06182655}. 

	\part{The ambient spaces}
	\label{part ambient spaces}

		\chapter{Pseudo-Riemannian space forms}
		\label{chapter: pseudoriem space forms}
	\section{Pseudo-Riemannian metrics}
		We recall some very basic concepts in differential geometry in order to fix the notation. The reader may refer to \cite{pseudoRiemanniangeometry, Kobayashi-Nomizu1} as good textbooks on these topics.
		
			Let $M=M^n$ be a smooth manifold of dimension $n$. Throughout all the paper, we will assume $M$ to be connected.
	
	A \emph{pseudo-Riemannian metric} $g$ on $M$ is a smooth section of the bundle 
	$Sym^2 (TM)$ such that the bilinear form $g_x \colon T_x M \times T_x M \to \R$ is non-degenerate. By connectedness of $M$, the bilinear form $g_x$ has constant signature for all $x\in M$.
	
	A \emph{Riemannian metric} $g$ on $M$ is a pseudo-Riemannian metric such that $g_x$ is positive-definite for all $x\in M$.  We say that $(M,g)$ is a \emph{(pseudo-)Riemannian} manifold if $g$ is (pseudo-)Riemannian. As a matter of notation, we will often denote pseudo-Riemannian manifolds with $M$ instead of $(M,g)$. 
	
	An \emph{affine connection} $\nabla$ on a manifold $M$ is a linear map
	\begin{align*}
	\nabla\colon \Gamma(TM) &\to End(TM)\\
	X &\mapsto \nabla X (\colon Y \mapsto \nabla_Y X).
	\end{align*}
	such that, for all $X,Y\in \Gamma(TM)$ and $f\colon M\to \R$, \[\nabla (f Y)=df \otimes Y + f \nabla Y.
	\]
A pseudo-Riemannian metric $g$ induces canonically an affine connection, called  \emph{Levi-Civita connection} of $g$, defined as the unique affine connection with the following properties:
	\begin{itemize}
		\item $\nabla$ is compatible with the metric, i.e. for all $X,Y\in \Gamma(TM)$, $d\big( g(X,Y)\big)= g(\nabla X, Y) + g(X, \nabla Y)$;
		\item $\nabla$ is torsion-free, i.e. for all $X,Y\in \Gamma(TM)$, $\nabla_X Y - \nabla_Y X = [X,Y]$.
	\end{itemize}
	
	A connection $\nabla$ comes with a notion of parallel transport, geodesics and exponential map. Since the exponential map at any point of $M$ is a local diffeomorphism, for every point $x$ of $M$ there exists a \emph{starred} neighborhood $U_x$, namely $U_x$ is such that any of its points can be connected to $x$ through a geodesic segment. 
	
	\begin{Lemma}
		\label{mappa differenziale determina isometria}
		Let $f_1,f_2\colon (M,g)\to (N,h)$ be two isometries between pseudo-Riemannian manifolds (with $M$ connected) such that, for some point $x\in \mathbb M$, $f_1(x)=f_2(x)$ and $d_x f_1=d_x f_2\colon T_x M \to T_{f_1(x)} N$. Then $f_1\equiv f_2$.
	\end{Lemma}
	\begin{proof}
		Being $f, g$ isometries, one has that 
		\[f \circ exp_x ^M =exp_{f(x)}^N \circ d_xf =exp_{g(x)}^N \circ d_xg= g\circ exp_x^M\ , \]
		so $f$ and $g$ coincide on any open neighbourhood of $x$ whose elements can be linked to $x$ by a geodesic segment. For all $y\in M$, one can take a connected compact subset of $M$ containg both $x$ and $y$ (e.g. a curve) and cover it by a finite amount of starred open subsets of $M$. By iteration of the argument above and by standard deductions, one gets that $f$ and $g$ coincide on $y$.
	\end{proof}

	We say that a pseudo-Riemannian manifold is \emph{complete} if its Levi-Civita connection is geodesically complete, namely if geodesics can be extended indefinitely. If the metric is not Riemannian, it is generally false that any two points on the manifold can be connected by a geodesic segment.
	
	An affine connection $\nabla$ defines a $(1,3)$-type \emph{curvature tensor} $R\in \Gamma(T^*M\otimes T^*M\otimes T^*M\otimes TM)$ as
	\[	
	R(X,Y)Z:= \nabla_X \nabla_Y Z - \nabla_Y \nabla_X Z - \nabla_{[X,Y]} Z.
	\] 
	If $(M,g)$ is a pseudo-Riemannian manifold, one also has a $(0,4)$-type curvature tensor, that we will still denote with $R$, defined as
	\[
	R(X,Y,Z, T):= -g(R(X,Y)Z, T).
	\]
	The $(0,4)$-type curvature tensor has several symmetries, generated by the equations \begin{equation}
	\label{eq: symmetries of R pseudo-Riem}
	\begin{split}
	&R(X,Y,Z,T)= -R(Y,X, Z,T)\\
	&R(X,Y, Z,T)=R(Z,T, X, Y)\\
	&R(X,Y,Z,T)+R(Y,Z,X,T)+R(Z, X, Y,T)=0.
	\end{split}
	\end{equation}
	
	For all $x\in M$, given a $2$-dimensional vector subspace $\mathcal V< T_x M$ such that ${g_x|}_{\mathcal V}$ is non degenerate, the \emph{sectional curvature of} $\mathcal V$ is defined as
	\[
	K(\mathcal V)= \frac{R(X,Y,X,Y)}{g(X,X)g(Y,Y)- (g(X,Y))^2}
	\]
	with $\mathcal V=Span_\R(X,Y)$. The multilinearity of $R$ and $g$ and the symmetries of $R$ ensure that the definition does not depend on the choice of the generators $X,Y$ for $\mathcal V$.

	\section{Pseudo-Riemannian space forms: uniqueness}
	\label{section: uniqueness pseudoRiem space forms}
	
	We say that the pseudo-Riemannian manifold $(M,g)$ has constant sectional curvature $k\in \R$ if, for all $x\in M$ and for all $g$-non-degenerate 2-dimensional vector subspace $\mathcal V<T_x M$, $K(\mathcal V)=k$.
	
	We will say that a pseudo-Riemannian manifold $(M,g)$ is a \emph{space form} if it is simply connected, complete and with constant sectional curvature. 
	\begin{Theorem}
		\label{Teo: unicita space form pseudo-Riem}
		For all $n,m\in \Z_{\ge 0}$, with $n\ge 2$ and $m\le n$, and $k\in \R$, there exists a unique pseudo-Riemannian space form of dimension $n$, signature $(m, n-m)$ and constant curvature $k$, up to isometries.
	\end{Theorem}
	
	Despite Theorem \ref{Teo: unicita space form pseudo-Riem} is a classical result in Differential Geometry, it will be convenient to recall one way to prove the uniqueness part of the statement because we will use the same idea for a main proof later on.
	\vspace{20pt}
	
	\emph{Step $1$:} The Cartan-Ambrose-Hicks Theorem for affine manifolds.
	
	Let $M$ and $N$ be two smooth manifolds of the same dimension $n$, let $\nabla^M$ and $\nabla^N$ be two linear connections on $TM$ and on $TN$ respectively. 
	The pair $(M, \nabla^M)$ (and $(N, \nabla^N)$) is often called an \emph{affine manifold}. 
	Assume that $\nabla^M$ and $\nabla^N$ are complete and that $M$ is simply connected.
	
	For convenience, assume both $\nabla^M$ and $\nabla^N$ are torsion-free. 
	Fix $x_0\in M$ and $y_0\in N$ and assume $L\colon T_{x_0}M \to T_{y_0}N$ is a linear isomorphism. 
	For all $X\in T_{x_0} M$ and for any given $\tau> 0$, consider the geodesic $\gamma(t)= exp^M(tX)$ with $t\in [0, \tau]$ and the geodesic curve $\gamma_L(t)=exp^N(tL(X))$ with $t\in [0, \tau]$: the composition of $L$ with the parallel transports via the geodesics $\gamma$ and $\gamma_L$ induces a linear isomorphism
	\[
	L_\gamma\colon T_{\gamma(\tau)}M\to T_{\gamma_L(\tau)}N.\]
	
	By iterating this argument, one gets that any piecewise geodesic curve $\gamma	\colon [0, \tau]\to M$ with starting point $x_0$ induces a piecewise geodesic curve $\gamma_L\colon [0, \tau]\to N$ with starting point $y_0$ and the composition of $L$ with the parallel transports defines once again a linear isomorphism $L_\gamma\colon T_{\gamma(\tau)}M\to T_{\gamma_L(\tau)}N$.
	
	Recall that, being $M$ (and $N$) connected, any two of its points can be linked by a piecewise geodesic curve.

	\begin{Theorem}[Ambrose-Cartan-Hicks Theorem]
		\label{Teo: Ambrose-Cartan-Hicks}
		Under the notations and the assumptions above, if for all piecewise geodesic curve $\gamma\colon [0,\tau]\to M$ with $\gamma(0)=x_0$ the corresponding isomorphism $L_\gamma\colon T_{\gamma(\tau)}M\to T_{\gamma_\phi(\tau)}N$ preserves the curvature tenors, namely
		\begin{equation} 
		\label{eq: R in Cartan-Ambrose-Hicks}
		L_\gamma (R^M (X_1, X_2) X_3 )= R^N (L_\gamma (X_1), L_\gamma (X_2) ) (L_\gamma (X_3) )
		\end{equation}
		for all $X_1, X_2,X_3\in T_{\gamma(T)}M$, then there exists a unique covering map $f\colon M\to N$ such that:
		\begin{itemize}
			\item $f$ is affine, i.e. $f^* \nabla^N= \nabla^M$;
			\item $f(x_0)=y_0$ and $d_{x_0}f=L$;
			\item for all piecewise geodesic curve $\gamma\colon [0, \tau]\to M$ starting at $x_0$,  $f(\gamma(\tau))= \gamma_L(\tau)$ with $d_{\gamma(\tau)} f = L_\gamma$.
		\end{itemize}
	\end{Theorem}
	For a proof of Theorem \ref{Teo: Ambrose-Cartan-Hicks} see \cite{Piccione}.
	\vspace{25pt}
	
	\emph{Step $2$:} The curvature tensor of manifolds with constant sectional curvature.
	
	\begin{Prop}
		\label{prop: R determined by sectional curvatures}
		Let $M$ be a manifold and let $R, R'\in \Gamma(T^*M ^{4\otimes} )$ be two $(0,4)$-type tensors both satisfying Equations \eqref{eq: symmetries of R pseudo-Riem}. If
		\[
		R(X,Y, X,Y)=R'(X,Y,X,Y)
		\]
		for all $X, Y\in \Gamma(TM)$, then $R\equiv R'$.
	\end{Prop}
	\begin{proof}
		Observe that the tensor $R- R'$ satisfies Equations \eqref{eq: symmetries of R pseudo-Riem}, being the equations linear, and that it is such that $(R-R')(X,Y,X,Y)=0$ for all $X,Y\in\Gamma(TM)$. For all $W,X,Y,Z \in \Gamma(TM)$, we have
		\begin{align*}
		&(R-R')(X,Y+Z, X,Y+Z)= 0 \Longrightarrow (R-R')(X,Y,X,Z)= 0;\\
		&(R-R')(X+W,Y, X+W, Z)=0 \Longrightarrow (R-R')(X,Y,W,Z)=(R-q)(Y,W,X,Z);\\
		&(R-R')(X,Y,W,Z)+(R-R')(W,X,Y,Z)+(R-R')(Y,W,X,Z)=0 \Longrightarrow \\
		&\qquad	\qquad \Longrightarrow 2(R-R')(X,Y,W,Z)= (R-R')(X,W,Y,Z);
		\end{align*}
		then $4(R-R')(X,Y,W,Z)=2(R-R')(X,W,Y,Z)=(R-R')(X,Y,W,Z)$, hence $R-R'\equiv 0$.
	\end{proof}

	\begin{Cor}
		\label{cor: R di space form pseudo-Riem}
		If $(M, g)$ is a pseudo-Riemannian manifold of constant sectional curvature $k\in \R$, then, for any $X,Y,Z, W\in \Gamma(T\mathbb M)$,
		\[R(X,Y,Z,W)= k (g(X,Z) g(Y,W) - g(Y,Z)g(X,W)),\]
		or equivalently
		\[
		R(X,Y)Z= -k (g(X,Z)Y - g(Y,Z)X).
		\]
		In particular, $R(X,Y)Z\in Span(X,Y)$.
	\end{Cor}
	\begin{proof}
		The tensor $R'(X, Y, Z, W)= k (g(X,Z)g(Y,W) - g(Y,Z)g(X,W))$ satisfies Equations \eqref{eq: symmetries of R pseudo-Riem} and is such that $R'(X,Y,X,Y)=R(X,Y,X,Y)$ for all $X,Y\in \Gamma(TM)$.
	\end{proof}
	\vspace{15pt}
	\emph{Step $3$}: Conclusion.
	\begin{proof}[Proof of Theorem $\ref{Teo: unicita space form pseudo-Riem}$]
		Let $(M,g)$ and $(N, h)$ be two simply connected, complete pseudo-Riemannian manifolds of dimension $n$, signature $(m, n-m)$ and with constant sectional curvature $k$. Denote with $\nabla^M$ and $\nabla^N$ respectively their Levi-Civita connections. Fix $x\in M$ and $y\in N$. 
		
		By standard linear algebra, there exists a linear isometry $L\colon (T_x M, g_x)\to (T_y N, g_y)$ because $T_xM$ and $T_y N$ have the same dimension, and $g_x$ and $h_y$ have the same signature. With notations as in \emph{Step 1}, let $\gamma\colon [0, \tau]\to M$ be a geodesic with $\gamma(0)= x$ and let $\gamma_L\colon [0, \tau] \to N$ such that $\dot \gamma_L (0)= L(\dot \gamma(0))$. Since parallel transport along $\gamma$ defines an isometry between the tangent spaces $T_{\gamma(0)} M \to T_{\gamma(\tau)}N$, and likewise for $\gamma_L$, the induced map
		\[
		L_\gamma\colon (T_{\gamma(\tau)}M, g) \to (T_{\gamma_L(\tau)}N, h)
		\]
		is a linear isometry.
		By iteration, the same holds for $L_\gamma$ when $\gamma$ is a piecewise geodesic curve starting at $x$. 
		
		By Corollary \ref{cor: R di space form pseudo-Riem} applied to both $M$ and $N$ and by the fact that $L_\gamma$ is an isometry, for all $X,Y, Z\in T_{\gamma_L(\tau)}M$ one gets that
		\begin{align*}
		L_\gamma(R^M(X,Y)Z)&= - k\ g(X,Z)L_\gamma(Y) + k\ g(Y,Z)L_\gamma(X) =\\ 
		&= - k\ h(L_\gamma(X), L_\gamma(Z)) L_\gamma(Y) + k\ h(L_\gamma(Y), L_\gamma(Z))L_\gamma(X)=\\
		&=  R^N(L_\gamma(X), L_\gamma(Y)) (L_\gamma(Z)).
		\end{align*}
		By Theorem \ref{Teo: Ambrose-Cartan-Hicks}, there exists a covering map $f\colon M \to N$ such that for every piecewise geodesic curve $\gamma\colon [0, \tau] \to M$ starting at $x$, $d_{\gamma(\tau)}f=L_\gamma$, hence $f$ is a local isometry. Since $M$ and $N$ are both simply connected, $f$ is an isometry.
	\end{proof}
	
	\begin{Remark}
		The proof of \ref{Teo: unicita space form pseudo-Riem} leads to another observation: if $(M,g)$ is a pseudo-Riemannian space form, then, for any given $x\in M$ and any given linear isometry $L\colon (T_x M, g_x)\to (T_xM, g_x)$, there exists an isometry $f\colon (M,g)\to (M,g)$ such that $f(x)=x$ and $d_x f=L$. 
	\end{Remark}

	We will denote with $\mathbb F_k^{m_1, m_2}= (\mathbb F_k^{m_1, m_2}, \inners)$ the pseudo-Riemannian space form of dimension $m_1+m_2\ge 2$, signature $(m_1,m_2)$ and constant sectional curvature $k\in \R$.
	
	We will discuss existence of pseudo-Riemannian space forms in Section \ref{sec existence pseudoRiem space forms}, after recalling the Gauss-Codazzi Theorem. The reader might find quite unusual to work with space forms before proving they exist, but hopefully the reason behind this choice will be clear later on.

	\section{Immersions of hypersurfaces into pseudo-Riemannian space forms}
	\label{section: Gauss-Codazzi pseudoRiem}
	
	There exists a general theory of immersions of manifolds into pseudo-Riemannian space forms. In the following we stick to the theory of hypersurfaces, where the theory has a more simple description. The reader might refer for instance to \cite{Gauss-Codazziinspaceforms} which treats the topic in the general setting of immersions of any codimension into products of space forms. 
	
	The setting is very classical and in the following we will skip the main proofs. Nevertheless, in Chapter \ref{Chapter immersioni in Xn}, we will handle a similar formalism to study immersions of smooth manifolds of dimension $n$ into $(n+1)$-dimensional holomorphic Riemannian space forms and we will provide all the major proofs: \emph{mutatis mutandis}, proofs in the pseudo-Riemannian setting work in the same fashion with some more attention to signatures but with generally fewer technical precautions.

	Let $M=M^n$ be an orientable $n$-dimensional manifold, let $m_1, m_2$ be two integers with $m_1+m_2=n+1$, and denote $\mathbb F= (\mathbb F_k^{m_1, m_2}, \inners)$. Equip both $\mathbb F$ and $M$ with an orientation. 
	
	We say that an immersion  $\sigma\colon M\to \mathbb F$ is \emph{admissible} if the pull-back metric $\I=\sigma^*\inners$ is a (non-degenerate) pseudo-Riemannian metric on $M$. The pull-back metric $\I$ falls within one of the following two cases:
	\begin{itemize}
		\item[Case ($a$):] $\I$ has signature $(m_1-1,m_2)$;
		\item[Case ($b$):] $\I$ has signature $(m_1, m_2-1)$.
	\end{itemize}

	The map $d\sigma$ induces a canonical bundle inclusion of $TM <\sigma^* (T \mathbb F)$. In order to relax the notation, we denote with $\inners$ also the pull-back metric on $\sigma^*(T \mathbb F)$. By construction, $\inners_{|TM}= \I$.
	
	There exists a unique section $\nu\in \Gamma(\sigma^*(T\mathbb F))$ such that $\inner{\nu, \nu}=1$ if in case $(a)$ and such that $\inner{\nu,\nu}=-1$ if in case $(b)$ and so that $(d\sigma(e_1), \dots, d\sigma(e_n), d\sigma(\nu) )$ is an oriented frame for $\mathbb{F}$ whenever $(e_1, \dots, e_n)$ is an oriented frame for $M$. Such $\nu$ is a \emph{(global) normal vector field}.
	
	Denote with $\nabla$ the Levi-Civita connection on $TM$ with respect to $\I$ and with $\overline \nabla$ the Levi-Civita connection on $\sigma^*(T\mathbb F)$ with respect to $\inners$, namely the pull-back of the Levi-Civita connection of $\inners$ on $T\mathbb F$. One has that, for all $X,Y\in \Gamma(TM)$,
	\[
	\overline \nabla_X Y= \nabla_X Y + \II(X,Y) \nu\ ,
	\]
	where $\II$ is called \emph{second fundamental form} of $\sigma$ and is a symmetric $(0,2)$-tensor on $TM$.
	
	By raising an index and changing a sign, one can define the \emph{shape operator} $B$ for $\sigma$ as the $(1,1)$-tensor on $TM$ defined by
	\[
	\I(B(X),Y):= \II(X,Y).
	\]
	Since $\II$ is symmetric, $B$ is $\I$-self adjoint.
	
	Define the \emph{exterior covariant derivative} of a $(1,1)$-tensor $\beta$ with respect to the connection $\nabla$ as the skew-symmetric $(1,2)$-tensor defined by
	\[
	d^\nabla \beta (X,Y)= \nabla_X(\beta(Y) )- \nabla_Y(\beta(X))- \beta([X,Y]).
	\]

	The following result is known as \emph{Gauss-Codazzi Theorem} (also \emph{Bonnet Theorem} or also \emph{Fundamental theorem of hypersurfaces}).
	
	\begin{Theorem}[Gauss-Codazzi Theorem]
		\label{Gauss-Codazzi pseudo-Riem}
		Let $\I$ be a pseudo-Riemannian metric on $M$ with signature $(m_1- \frac{\delta-1}{2},m_2 + \frac{\delta-1}{2})$ with $\delta\in \{\pm 1\}$. Let $B\colon TM\to TM$ be a $\I$-self adjoint $(1,1)$-form. The pair $(\I,B)$ satisies
		\begin{align*}
		1) &d^{\nabla} B \equiv 0, \qquad \text{and};\\
		2)	&R(e_i, e_j, \cdot, \cdot) + \delta B^i \wedge B^j= k\ \theta^i\wedge \theta^j\\
		&\text{for some constant $k\in \R$ and for all $i,j =1, \dots n$, where $(e_i)_{i=1}^n$ is any}\\ &\text{ $\I$-orthonormal frame with corresponding coframe $(\theta^i)_{i=1}^n$ and with $B= B^i \otimes e_i$},
		\end{align*}
		if and only if there exists a $\pi_1(M)$-equivariant isometric immersion $\sigma\colon (\widetilde M,\widetilde \I)\to \mathbb F_k^{m_1,m_2}$ with shape operator $\widetilde B$, with $\widetilde \I$ and $\widetilde B$ denoting the natural lifts of $\I$ and $B$ on the universal cover.
		
		Such immersion is unique up to post-composition of $\sigma$ with ambient isometries of $\mathbb F_k^{m_1, m_2}$.
	\end{Theorem}
	
	We recall that $\sigma\colon \widetilde M \to \mathbb F$ is said to be $\pi_1(M)$-equivariant, with $\pi_1(M)$ denoting the deck group of the covering map $\widetilde M \to M$, if for all $\alpha \in \pi_1(M)$ there exists $\phi\in \Isom(\mathbb F)$ such that $\sigma \circ \alpha= \phi \circ \sigma$.
	
	Gauss-Codazzi Theorem allows to see immersions into pseudo-Riemannian space forms as a pair of metric and shape operator, translating extrinsic geometry into intrinsic geometry.
	
	Chapter \ref{Chapter immersioni in Xn} will be devoted to proving a similar result for immersions into holomorphic Riemannian space forms, and the pseudo-Riemannian version of the Gauss-Codazzi theorem will allow to see immersions into pseudo-Riemannian space forms as particular cases of immersions into holomorphic Riemannian space forms (see Section \ref{sec pseudoRiemannian hrm}).
	
	\section{Pseudo-Riemannian space forms: existence}
	\label{sec existence pseudoRiem space forms}
	Let $m_1$, $m_2$ be non negative integers with $n:=m_1+m_2\ge2$. Denote with $\inners_{m_1, m_2}$ both the bilinear form on $\mathbb R^n$ defined by 
	\[\inner{x,y}_{m_1, m_2}= x_1y_1 + \dots + x_{m_1} y_{m_1} - x_{m_1+1}y_{m_1+1}- \dots - x_n y_n\ ,\]
	and the corresponding pseudo-Riemannian metric induced by the canonical identification $T \mathbb R^n = \mathbb R^n \times \mathbb R^n$, namely
	\[
	\inners_{m_1, m_2}= dx_1^2 + \dots + dx_{m_1}^2 - dx_{m_1+1}^2 -\dots - d x_n^2. 
	\]
	The pseudo-Riemannian manifold $\mathbb R^{m_1, m_2}= (\mathbb R^n, \inners_{m_1, m_2})$ is often denoted as the \emph{Minkowski space} of signature $(m_1, m_2)$.
	
	For all $n=m_1+m_2\ge 2$, $0\le m_1,m_2$ and $k\in \mathbb R$, the space form $\mathbb F_k^{m_1, m_2}$ exists, and it admits a very simple description.

	\begin{Prop}
		\label{Prop: quadriche pseudo-Riem}
		\begin{itemize}
			\item[$1)$]  $\mathbb F_0^{m_1, m_2}= \mathbb R^{m_1, m_2}$.
			\item[$2)$] For $k\ne 0$, $\mathbb F_k^{m_1, m_2}$ is the isometric universal cover of any connected component of
			\[
			Q_k^{m_1, m_2}=\bigg\{x\in \mathbb R^{m_1+m_2 +1}\ |\ \inner{x,x}_{m_1+1, m_2}=\frac{1}{k} \bigg\} \subset \mathbb R^{m_1+1, m_2} \qquad \text{if $k>0$}
			\]
			and
			\[
			Q_k^{m_1, m_2}=\bigg\{x\in \mathbb R^{n+1}\ |\ \inner{x,x}_{m_1, m_2+1}=\frac{1}{k} \bigg\} \subset \mathbb R^{m_1, m_2+1} \qquad \text{if $k<0$,}
			\]
			both endowed with the submanifold metric.
		\end{itemize}	
	\end{Prop}
	\begin{proof}
		\begin{itemize}
			\item[1)] In the canonical identification $\Gamma(T \mathbb R^n)=C^\infty(\mathbb R^n, \mathbb R^n)$, one can check explicitly on the canonical orthonormal frame generated by the canonical basis that the Levi-Civita connection for $\mathbb R^{m_1, m_2}$ is given by the standard affine connection
			\[
			\mathrm{d}_X Y:= Jac(Y) \cdot X.
			\]
			Schwarz Theorem on commutativity of partial derivatives in $\mathbb R^n$ involves $R\equiv 0$, hence the sectional curvature is constantly zero.
			
			Geodesics in this model are straight lines $t\mapsto tx + y$, hence $\mathbb R^{m_1,m_2}$ is complete.	
			
			\item[2)] Denote with $Q_k^{m_1, m_2}$ the quadrics in the statement endowed with the submanifold metric, namely \[Q_k^{m_1, m_2}=\big\{x\in \mathbb R^{m_1, m_2+ 1}\ |\ \inner{x,x}_{m_1, m_2+1}=\frac{1}{k} \big\} \qquad \text{if $k<0$,}\]  \[Q_k^{m_1, m_2}=\big\{x\in \mathbb R^{m_1+1, m_2}\ |\ \inner{x,x}_{m_1+1, m_2}=\frac{1}{k} \big\} \qquad \text{if $k>0$}. \] 
			In the following, in order to ease the notation, we denote with $\inners$ both the induced metric on $Q_k^{m_1, m_2}$ and the corresponding ambient metric for $\mathbb R^{n+1}$ in which $Q_k^{m_1, m_2}$ is embedded, bearing in mind that the latter depends on the sign of $k$.
			
			For any curve $\gamma$ on $Q_k^{m_1, m_2}$,  $\inner{\gamma(t),\gamma(t)}=\frac 1 k$ implies $\inner{\gamma(t), \dot \gamma(t)}=0$: as a result, we deduce that $T_x Q_k^{m_1, m_2}= x^{\bot_{(m_1,m_2)}}$, that the normal vector field along (the standard inclusion into $\R^{n+1}$ of) $Q_k^{m, n-m}$ is given by $\nu(x)= \sqrt{|k|}\ x$, and that the retriction of $\inners$ to $Q_k^{m, n-m}$ is a pseudo-Riemannian metric.
			
			Denote with $D$ the Levi-Civita connection of $Q_k^{m_1, m_2}$, then $D \nu$ is equal to the orthogonal projection onto $TQ_k^{m_1, m_2}$ of \[\mathrm d \nu= \sqrt{|k|}\ \mathrm d x= \sqrt{|k|}\ id, \]
			hence $D\nu= \sqrt{|k|}\ id$ as well. As a result, one gets that the second fundamental form of the inclusion $Q_k^{m_1, m_2}\hookrightarrow \R^{n+1}$ is $\II= \sqrt{|k|}\ \inners$ and that the shape operator is $B= \sqrt{|k|}\ id$. 
			
			By Gauss-Codazzi Theorem \ref{Gauss-Codazzi pseudo-Riem}, one concludes that, for local orthonormal frame $e_1,e_2$ (of any signature), the curvature tensor $R$ for $Q_k^{m, n-m}$ satisfies
			\[
			R(e_1, e_2, \cdot, \cdot ) = sign(k)|k|\  \inner{e_1, \cdot}\wedge \inner{e_2, \cdot}
			\]
			involving $K(Span(e_1, e_2))= \frac{R(e_1, e_2, e_1, e_2)}{\inner{e_1, e_1}\inner{e_2, e_2} }=k$. This proves that $Q_k^{m_1, m_2}$ has constant sectional curvature $k$.
			
			Geodesics of $Q_k^{m_1, m_2}$ have a simple description. A curve $\gamma\colon \mathbb R \to Q_k^{m_1, m_2}$ is a geodesic if $D \dot \gamma\equiv 0$, namely if $\ddot \gamma (t)= \lambda(t) \gamma(t)$ for some $\lambda\colon \mathbb R\to \mathbb R$. As a consequence, $\gamma$ has constant speed, i.e. $\inner{\dot\gamma, \dot \gamma}\equiv \inner{\dot\gamma(0), \dot \gamma(0)}$. By differentiating the equation $\inner{\gamma, \dot \gamma}\equiv 0$, one gets that $\gamma$ satisfies 
			\begin{equation}
			\label{eq: geodetiche}
			\ddot \gamma (t)= - k \inner{\dot \gamma(0), \dot \gamma(0)} \gamma(t)
			\end{equation}
			for all $t$.
			
			Equation \eqref{eq: geodetiche} has solution $\gamma\colon \mathbb R \to Q_k^{m_1, m_2}$ for any initial data $\gamma(0)$ and $\dot \gamma(0)$:
			\begin{align*}
			\text{if }& -k \inner{\dot \gamma(0), \dot \gamma(0)} >0, \\ 
			&\  \gamma(t)= \cosh\big( \sqrt{-k \inner{\dot \gamma(0), \dot \gamma(0)}} t \big) \gamma(0) + \frac{\sinh\big( \sqrt{-k \inner{\dot \gamma(0), \dot \gamma(0)}} t \big)} { \sqrt{-k \inner{\dot \gamma(0), \dot \gamma(0)}}} \dot \gamma(0);\\
			\text{if }& - k \inner{\dot \gamma(0), \dot \gamma(0)} <0, \\ 
			&\ \gamma(t)= \cos\big( \sqrt{k \inner{\dot \gamma(0), \dot \gamma(0)}} t \big) \gamma(0) + \frac{\sin\big( \sqrt{k \inner{\dot \gamma(0), \dot \gamma(0)}} t \big)} { \sqrt{k \inner{\dot \gamma(0), \dot \gamma(0)}}} \dot \gamma(0);\\
			\text{if }& \inner{\dot \gamma(0), \dot \gamma(0)} =0, \qquad \quad \gamma(t)= \gamma(0)+ t\dot \gamma(0).
			\end{align*}
			In conclusion, $Q_k^{m, n-m}$ has constant sectional curvature and is geodesically complete: its universal cover endowed with the pull-back metric is a space form.
		\end{itemize}

	\end{proof}
	\begin{Remark}
		\label{rmk: Q e quasi sempre s.c.}
		We remark that $Q_k^{m_1, m_2}$ is connected and simply connected for most $m,n$ and $k$.
		
		First, observe that, for all $k<0$, $Q_k^{m_1, m_2}$ is diffeomorphic to $Q_{-1}^{m_1, m_2}$ and, by swapping signs and rearranging the variables, to $Q_{+1}^{m_2,m_1}$. Focusing on $Q_{-1}^{m_1, m_2}$, one can construct the diffeomorphism
		\begin{equation*}
		\begin{split}
		\mathbb R^{m_1} \times S^{m_2} &\to Q_{-1}^{m_1, m_2}\\
		(u_1, \dots, u_m, v_1, \dots, v_{n-m+1}) & \mapsto \big(u_1, \dots, u_m, (\sqrt{1 + \|u \|^2})v_1, \dots, (\sqrt{1 + \|u \|^2})v_{n-m+1}  \big)  
		\end{split}
		\end{equation*}
		where $\|u\|^2= u_1^2 + \dots + u_m^2$.
		
		As a result: 
		\begin{itemize}
			\item[\_] $Q_{-1}^{n,0}$ is disconnected and each of its two connected components of its is contractible;
			\item[\_] $Q_{-1}^{n-1, 1}$ is connected and has fundamental group isomorphic to $\Z$;
			\item[\_] $Q_{-1}^{m_1,m_2}$ is connected and simply connected in all the other cases.
		\end{itemize}
		In fact, the former two cases correspond to two very important pseudo-Riemannian manifolds which will occur quite often in this paper:
		\begin{itemize}
			\item $\Hyp^n= \mathbb F_{-1}^{n,0}=  Q^{n,0}_{-1} \cap \{x_{n+1}>0\}$ is the \emph{hyperbolic space};
			\item $AdS^n= Q^{n-1,1}_{-1}$ is the \emph{Anti-de Sitter space}.
		\end{itemize}
	\end{Remark}

\section{Models of  $\Hyp^n$}	
\label{sec model Hn}
	We quickly recall some aspects of the geometry of the hyperbolic space $\Hyp^n$. The reader might look at \cite{martelli} for a more complete survey.
	
	In this thesis we will refer to the following three models for the hyperbolic space.
	\begin{itemize}
	\item The \emph{hyperboloid model} described above, canonically embedded in $\R^{n,1}$. The isometry group of $Q^{n,0}_{-1}$ coincides with the group of isomorphisms of $\R^{n,1}$ that preserve the bilinear form, namely $O(n,1)$. The subgroup of $O(n,1)$ that fixes $\Hyp^{n}$ (which is a connected component of $Q^{n,0}_{-1})$ is often denoted with $O^+(n,1)\cong \Isom(\Hyp^{n})$.
	
	The group $Isom_0(\Hyp^n)$ of orientation-preserving isometries of $\Hyp^n$ is connected and corresponds to 
	\[
	SO(n,1) \cap O^+(n,1) = SO_0 (n,1) 
	\]
	which is the connected component of $SO(n,1)$ and of $O(n,1)$ containing the identity.
	
	\item The \emph{half-space model} corresponds to the manifold $\R^{n-1}\times \R^+$ equipped with the Riemannian metric
	\[
	\frac 1 {x_n^2} (dx_1^2 + \dots, dx_n^2).
	\]
	For $n=2$, the group of orientation preserving isometries of the half-space model is $\Isom_0(\Hyp^2)\cong \mathrm{PSL}(2,\R)$ acting on the open subset $\R\times \R^+\subset \C$ by M\"obius transformations; the whole isometry group is generated by orientation preserving isometries and by the isometry $z \mapsto -\overline z$.
	
	\item The Poincaré \emph{disc} model is the manifold given by the open ball $\{x\in \R^n\ |\ \|x\|^2_{\R^n} <1  \}$ endowed with the Riemannian metric
	\[
	\frac{4}{(1- \|x\|^2_{\R^n})^2}  (dx_1 + \dots + dx_n). 
	\]
	For $n=2$, regarding the disc as an open ball in $\C$, the group of orientation preserving isometries of the disc model corresponds to the group $PSU(1,1)$ acting by M\"obius transformation, while an orientation-reversing isometry is given by the map $z\mapsto \overline z$ is 
	\end{itemize} 
	
	As a negatively-curved, complete and simply-connected space, $\Hyp^n$ admits a notion of visual boundary $\partial \Hyp^n$ at infinity. The visual boundary $\partial \Hyp^n$ can be defined as the quotient of $T^1 \Hyp^n$ under the equivalence relation defined by
	\[
	(x,v)\sim (x',v') \quad \text{if and only if} \sup_{t\in[0,+\infty)} dist_{\Hyp^n}(exp_x(tv), exp_{x'}(tv')) <\infty.
	\]
	
	By Cartan-Hadamard Theorem, one has that for all $x\in \Hyp^n$ the projection $T^1_x \Hyp^n\to \partial \Hyp^n$ is a bijection. With respect to the half-space model and to the Poincaré disc model, the boundary $\partial \Hyp^n$ can be identified as the boundary of the corresponding open subsets of $\R^{n}\cup \{\infty\}$, namely as $\R^{n-1}\cup\{\infty\}$ in the half-space model, and as the sphere $S^{n-1}$ in the disc model: in both cases the correspondence is obtained by simply taking
	\begin{align*}
		[(x,v)]\mapsto \lim_{t\to +\infty} exp_x(tv) \ .
	\end{align*}
	
	As a consequence of the regularity of $\Hyp^n$, one has that for any two distinct limit points $p_+, p_- \in \partial \Hyp^n$ there exists a unique maximal geodesic $\gamma$ up to positive reparametrization such that $\lim_{t\to \pm \infty} \gamma(t)= p_{\pm}$. This is a key aspect to define the space of geodesics of $\Hyp^n$ in Chapter \ref{chapter spazio geodetiche}.
	\vspace{3pt}
	
	Finally, let us recall that for $n=3$ the boundary has a conformal structure. Regarding the half-space in $\R^3$ as $\C\times \R^+$, its boundary can be seen as $\overline \C=\C\cup \{\infty\}$, which is naturally a complex manifold. On the other hand, also the boundary $S^2$ of the disc in $\R^3$ has a natural complex structure for which it is biholomorphic to $\overline \C$: this allows us to say that $\partial \Hyp^3$ has a complex structure, equivalently defined in the two models.
	
	Now, every isometry of $\Hyp^n$ clearly defines a homeomorphism of its boundary: for $n=3$, if the isometry preserves the orientation, this homeomorphism is in fact a biholomorphism of the boundary, and this correspondence provides the isomorphism
\[
Isom_0(\Hyp^3)\cong Bihol(\partial \Hyp^2)\ .
\]
For $n=3$, this furnishes another useful description for the isometry group of $\Hyp^3$. Indeed, the biholomorphisms of $\overline \C$ are given exactly by the M\"obius transformation, leading, in the half-space model, to the isomorphism 
\[
Isom_0(\Hyp^3)\cong \PSL\ .
\]

	\chapter{Geometry of holomorphic Riemannian space forms}
	\label{Chapter hRm}
	\section{Holomorphic Riemannian metrics}
	
	Let $\mathbb M$ be a complex analytic manifold, with complex structure $\JJJ$, let $n=dim_\C \mathbb M$ and $T\mathbb M\to \mathbb M$ be the tangent bundle. 
	
	We recall that local coordinates $(x_1, y_1, \dots, x_n, y_n)\colon U \to \R^{2n}\equiv \C^n$, $U\subset \mathbb M$, are \emph{holomorphic} if \[\JJJ\bigg(\frac{\partial}{\partial x_k}\bigg)=\frac{\partial}{\partial y_k}\qquad
	\JJJ\bigg(\frac{\partial}{\partial y_k}\bigg)= -\frac{\partial}{\partial x_k}.
	\]
	A function $f\colon U \to \C^N$ is \emph{holomorphic} if, in local holomorphic coordinates, it can be seen as a holomorphic function from an open subset of $\C^n$ to $\C^N$. Equivalently, $f$ is holomorphic if and only if $df\circ \JJJ= i\ df$.
	
	A local vector field $X$ on $\mathbb M$ is \emph{holomorphic} if $dz_k (X)$ is a local holomorphic function for all $k\in\{1, \dots, n\}$, where $dz_k=dx_k + i dy_k$.

	\begin{Definition}
		\label{def hRm}
		
		A \emph{holomorphic Riemannian metric} (also \emph{hRm}) on $\mathbb M$ is a symmetric $2$-form $\inner{\cdot, \cdot}$ on $T\mathbb M$, i.e. a section of $Sym^2 (T^*\mathbb M)$, such that: \begin{itemize}
			\item $\langle \cdot, \cdot \rangle$ is $\C$-bilinear, i.e. for all $X,Y \in T_x \mathbb M$ we have $\langle \JJJ X, Y \rangle= \langle X, \JJJ Y\rangle =i \langle X, Y\rangle$;
			\item $\inners$ is non-degenerate at each point as a complex bilinear form;
			\item for all $X_1, X_2$ local holomorphic vector fields, $\inner{X_1, X_2}$ is a holomorphic function; equivalently, for all local holomorphic coordinates $(x_1, y_1, \dots, x_n, y_n)$, the functions $
			\inner{\frac{\partial}{\partial x_k}, \frac{\partial}{\partial x_h}}$ (or equivalently the functions $ \inner{\frac{\partial}{\partial x_k}, \frac{\partial}{\partial y_h}}$ and $\inner{\frac{\partial}{\partial y_k}, \frac{\partial}{\partial y_h}}$)
			are all holomorphic. 
		\end{itemize}
		
		We also denote $\| X\|^2:= \inner{X,X}$.
	\end{Definition}

	Taking inspiration from basic (Pseudo-)Riemannian Geometry, one can define several constructions associated to a holomorphic Riemannian metric, such as a Levi-Civita connection - leading to notions of curvature tensor, (complex) geodesics and completeness - and sectional curvatures. We recall some basic notions, the reader might find a more detailed treatment in \cite{holomorphicriemannian4}.
	
	\begin{Remark}Observe that, for a given holomorphic Riemannian metric $\inner{\cdot, \cdot}$, both the real part $Re\inner{\cdot, \cdot}$ and the imaginary part $Im\inner{\cdot, \cdot}$ are pseudo-Riemannian metrics on $\mathbb M$ with signature $(n,n)$. Indeed, if $X_1, \dots, X_n\in T_x \mathbb M$ are $\inners_x$-orthonormal, then $(X_1, \dots, X_n, \JJJ X_1, \dots, \JJJ X_n)$ is a real orthonormal basis for $Re\inners_x$ with signature $(n,n)$; similarly for $Im\inners= Re(i\inners)$.
	\end{Remark}

	There exists an analogous result to the Levi-Civita Theorem.
	
	\begin{Proposition} [See \cite{holomorphicriemannian4}]
		\label{prop: Levi Civita hRm}
		Given a holomorphic Riemannian metric $\inner{\cdot, \cdot}$ on $\mathbb M$, there exists a unique affine connection $D$ over $T\mathbb M$, that we will call \emph{Levi-Civita connection}, such that for all $X,Y \in \Gamma(T \mathbb M)$ the following conditions hold:
		\begin{align}
		\label{compatibilita con metrica}
		d \inner{X, Y} &= \inner{D X, Y}+ \inner{X, D Y} \qquad && \text{($D$ is compatible with the metric)};\\
		\label{torsion free}
		[X,Y]&= D_X Y - D_Y X \qquad &&\text{($D$ is torsion free)}.
		\end{align}
		Such connection coincides with the Levi-Civita connections of $Re\inner{\cdot, \cdot}$ and $Im\inner{\cdot, \cdot}$ and $D\JJJ=0$.
	\end{Proposition}

	\begin{proof}
		Let $D^1$ and $D^2$ be the Levi-Civita connections for $Re\inner{\cdot, \cdot}$ and $Im\inner{\cdot, \cdot}$ respectively. Let $(x_1, y_1, \dots, x_n, y_n)$ be local holomorphic coordinates for $\mathbb M$. 
		
		A straightforward calculation shows that $D^1$ and $D^2$ are characterised by the fact that for all $X,Y,Z\in \{\frac{\partial}{\partial x_1}, \frac{\partial}{\partial y_1}, \dots, \frac{\partial}{\partial x_n}, \frac{\partial}{\partial y_n} \}$ (therefore $X,Y,Z$ are pairwise commuting) we have
		\begin{align*}
		&Re\langle D^1_X Y, Z \rangle= \frac 1 2 \Big( \partial_X (Re\inner{Y,Z}) + \partial_Y (Re\inner{Z,X}) - \partial_Z (Re\inner{X, Y}) \Big) \\
		&Im\inner{D^2_X Y, Z}= \frac 1 2 \Big( \partial_X (Im\inner{Y,Z}) + \partial_Y (Im\inner{Z,X}) - \partial_Z (Im\inner{X, Y}) \Big).
		\end{align*}

		Recalling that, for such $X,Y, Z$, we have $\partial_{\JJJ Z} (\langle X, Y \rangle)= i \partial_Z (\langle X, Y \rangle)$, observe that
		\begin{align*}
		&Im \langle D^1_X Y, Z\rangle= - Re (i\langle D^1_X Y, Z\rangle)= - Re (\langle D^1_X Y, \JJJ Z\rangle)= \\
		&=-\frac 1 2 \Big( \partial_X (Re\langle Y, \JJJ Z \rangle) + \partial_Y (Re\langle \JJJ Z, X \rangle) - \partial_{\JJJ Z} (Re\langle X, Y \rangle) \Big)= \\
		&=-\frac 1 2 \Big( Re (\partial_X \langle Y, \JJJ Z \rangle) + Re(\partial_Y \langle \JJJ Z, X \rangle) - Re (\partial_{\JJJ Z} \langle X, Y \rangle )\Big)= \\
		&= \frac 1 2 \Big(  Im (\partial_X \langle Y, Z \rangle) + Im(\partial_Y \langle Z, X \rangle) - Im ( \partial_Z\langle X, Y \rangle) \Big)=\\
		&= Im \langle D^2_X Y, Z\rangle.
		\end{align*}
		We conclude that $D^1=D^2=:D$ and both equations (\ref{compatibilita con metrica}) and (\ref{torsion free}) of the statement hold since they hold for both the real and the imaginary part of the metric.
		
		Finally, for all $X, Y\in \Gamma(T\mathbb M)$, we have
		\[
		d\inner{ X, Y}= -i\ d\inner{\JJJ X, Y}= -i \inner{D( \JJJ X),Y}+\inner{X, DY},
		\]
		so, by (\ref{compatibilita con metrica}), $\langle D(\JJJ X), Y\rangle=i\langle DX, Y\rangle= \langle \JJJ DX, Y\rangle$.
		Hence $D (\JJJ X)=\JJJ DX$ since the bilinear form is non-degenerate.
	\end{proof}

	\begin{Remark} A direct computation shows that, exactly as in Pseudo-Riemannian Geometry, the Levi-Civita connection $D$ for a hRm $\langle \cdot, \cdot \rangle$ is explicitly described by
		\begin{equation}
		\label{Levi-Civita}
		\begin{split}
		\langle D_X Y, Z\rangle= \frac 1 2 \Big( X\inner{Y,Z} + Y\inner{Z,X} - Z\inner{X, Y} + \\
		+\inner{[X,Y],Z} -\inner{[Y,Z],X}+ \inner{[Z,X],Y} \Big)
		\end{split}
		\end{equation}
		for all $X,Y, Z\in \Gamma(T\mathbb M)$.
	\end{Remark}

	The notion of Levi-Civita connection $D$ for the metric $\inners$ leads to the standard definition of the $(1,3)$-type and $(0,4)$-type \emph{curvature tensors}, that we will denote with $R$, defined by
	\[
	R(X,Y,Z, T):= -\inner{R(X,Y)Z,T}=- \inner{ \nabla_X \nabla_Y Z - \nabla_Y \nabla_X Z - \nabla_{[X,Y]} Z, T }
	\]
	for all $X,Y,Z, T \in \Gamma(T\mathbb M)$.

	Since $D$ is the Levi-Civita connection for $Re\inner{\cdot, \cdot}$ and for $Im\inner{\cdot, \cdot}$, it is easy to check that all of the standard  symmetries of curvature tensors for (the Levi-Civita connections of) pseudo-Riemannian metrics hold for (the Levi-Civita connections of) holomorphic Riemannian metrics as well, namely all Equations \eqref{eq: symmetries of R pseudo-Riem} hold. So, for instance, 
	\[
	R(X,Y,Z,T)= -R(X, Y, T, Z)= R(Z, T, X, Y)= -R(Z, T, Y, X).
	\]
	Since the $(0,4)$-type $R$ is obviously $\C$-linear on the last component, we conclude that it is $\C$-multilinear.
	
	\begin{Definition}	
		\label{def sectional curvature hRm}
		A \emph{non-degenerate plane} of $T_p \mathbb M$ is a complex vector subspace $\mathcal V< T_p \mathbb M$ with $dim_\C\mathcal V=2$ and such that $\inner{\cdot, \cdot}_{|\mathcal V}$ is a non degenerate bilinear form.
		
		For holomorphic Riemannian metrics, we can define the {complex sectional curvature} of a nondegenerate complex plane $\mathcal V= Span_\C (V,W)<T_{p}M$ as
		\begin{equation}
		\label{def curvatura}
		K(Span_\C (V,W))=\frac{-\inner{R(V,W)V,W}}{\|V\|^2 \|W\|^2 - \inner{V,W}^2}\newline.
		\end{equation}
		This definition of $K(Span_\C (V,W))$ is well-posed since $R$ is $\C$-multilinear and satisfies Equations \eqref{eq: symmetries of R pseudo-Riem}.
	\end{Definition}

	\begin{Example} 
		\label{Esempio gruppi di Lie}
		The Killing form on complex semisimple Lie groups.
		
		Let $G$ be a complex Lie group with unit $e$. 
		Consider on $T_e G\cong Lie(G)$ the \emph{Killing form} \[Kill\colon T_e G \times T_e G \to \C,\] defined by $Kill(u,v):= tr(ad(u)\circ ad(v) )$. By standard Lie Theory, the Killing form is $\C$-bilinear and symmetric. $Kill$ is also \emph{$Ad$-invariant}, i.e. for all $g\in G$ \[Kill(Ad(g) \cdot, Ad(g) \cdot)=Kill,\] and \emph{$ad$-invariant}, i.e. for all $v\in Lie(G)$ \[Kill(ad(v)\cdot, \cdot)+ Kill(\cdot, ad(v)\cdot)=0.\]
		
		Assume that $G$ is \emph{semisimple}, namely that $Kill$ is a non-degenerate bilinear form.
		
		For all $g\in G$ one can push-forward $Kill$ via $L_g$ to define a non degenerate $\C$-bilinear form on $T_g G$, namely 
		\[Kill_g (X,Y):= \Big((L_g)_* Kill \Big)(X,Y)=  Kill\bigg( \big(d_g (L_g^{-1})\big)(X), \big(d_g (L_g^{-1})\big)(Y) \bigg)\]
		for all $X,Y\in T_g G$.
		By $Ad$-invariance, the analogous bilinear form $(R_g)_* Kill$ is such that $(R_g)_* Kill=(L_g)_* Kill$.
		
		This defines globally a nowhere-degenerate section $Kill_{\bullet} \in \Gamma\big(Sym^2(T^*G)\big)$ such that, for all $X,Y$ left-invariant vector fields, hence holomorphic vector fields, $Kill_{\bullet}(X,Y)$ is constant, hence a holomorphic function. Since any other holomorphic vector field can be seen as a combination of left-invariant vector fields with holomorphic coefficients, we conclude that $Kill_{\bullet}$ is a holomorphic Riemannian metric.
		
		Let $X, Y, Z$ be left-invariant vector fields for $G$, then $Kill_\bullet(Y,Z)$ and $Kill_\bullet(X,Z)$ are constant functions and, by $ad$-invariancy, $Kill_\bullet({[Z,Y],X})+Kill_\bullet({[Z,X],Y})=0$. In conclusion if $D$ is the Levi-Civita connection for $Kill_\bullet$, then, by the explicit expression (\ref{Levi-Civita}), we get that 
		\[
		D_X Y =\frac 1 2 [X,Y]
		\]
		for all $X,Y$ left-invariant vector fields.

		We can also explicitly compute the curvature tensor.
		
		Let $[\cdot,\cdot]$ denote the Lie bracket on $T_e G=\Lieg$. For all $X_0, Y_0, Z_0\in \Lieg$, let $X, Y, Z$ be the corresponding left-invariant vector fields. Then,
		\begin{equation}
			\label{eq: R in gruppi Lie complessi}
		\begin{split}
		R(X_{0},Y_{0}) Z_{0} :=&( D_X D_Y Z - D_Y D_X Z - D_{[X,Y]} Z)_{|_{e}}=\\
		=& \frac 1 4 \Big( [X, [Y,Z]]- [Y, [X,Z]]\Big)_{|_{e}}- \frac 1 2 [[X,Y],Z]_{|_{e}}\\
		=&\frac 1 4 \Big( [X, [Y,Z]]+[Y,[Z,X]] + [Z,[X,Y]] \Big)_{|_{e}} - \frac 1 4 [[X,Y],Z]_{|_{e}}=\\
		=& -\frac 1 4 [[X_{0},Y_{0}],Z_{0}].
		\end{split}
		\end{equation}

	\end{Example}

	\section{Holomorphic Riemannian space forms}
	
	We will say that a connected holomorphic Riemannian manifold $\mathbb M=(\mathbb M, \inner{\cdot, \cdot})$ is \emph{complete} if the Levi-Civita connection is complete, namely if geodesic curves can be extended indefinitely, equivalently if the exponential map is defined on the whole $T\mathbb M$.
	
	We will call $\emph{holomorphic Riemannian space form}$ a complete, simply connected holomorphic Riemannian manifold with constant sectional curvature.
	
	\begin{Theorem}
		\label{Theorem space forms}
		For all integer $n\ge 2$ and $k\in \C$ there exists exactly one holomorphic Riemannian space form of dimension $n$ with constant sectional curvature $k$ up to isometry.
	\end{Theorem}
	
	We first prove uniqueness, then existence will follow from an explicit description of the space forms. 
	
	\subsection{Uniqueness}
	
	The proof of uniqueness follows exactly the same idea as in Section \ref{section: uniqueness pseudoRiem space forms}.

	\begin{Lemma}
		\label{lemma space form}
		If $(\mathbb M,\inners)$ is a manifold of constant sectional curvature $k\in \C$, then, for any $W,X,Y,Z\in \Gamma(T\mathbb M)$,
		\[R(X,Y,Z,W)= k (\inner{X,Z} \inner{Y,W} - \inner{Y,Z}\inner{X,W}),\]
		or equivalently
		\[
		R(X,Y)Z= -k (\inner{X,Z}Y - \inner{Y,Z}X).
		\]
		In particular, $R(X,Y)Z\in Span(X,Y)$.
	\end{Lemma}
	\begin{proof}
		Exactly as in the pseudo-Riemannian case, the tensor $R'(X, Y, Z, W)= k (\inner{X,Z}\inner{Y,W} - \inner{Y,Z}\inner{X,W})$ satisfies Equations \eqref{eq: symmetries of R pseudo-Riem} and is such that $R'(X,Y,X,Y)=R(X,Y,X,Y)$ for all $X,Y\in \Gamma(TM)$, so the thesis follows from \ref{prop: R determined by sectional curvatures}.
	\end{proof}

	We can now prove uniqueness in Theorem \ref{Theorem space forms}.
	\begin{proof}[Proof of Theorem $\ref{Theorem space forms}$ -Uniqueness]
		Let $({\mathbb M},\inner{\cdot, \cdot}_{\mathbb M}),(\mathbb M',\inner{\cdot,\cdot}_{\mathbb M'})$ be two holomorphic Riemannian space forms with the same dimension $n$ and constant sectional curvature $k\in \C$. Fix any $x\in \mathbb M$ and $y\in \mathbb M'$. Since all the non-degenerate complex bilinear forms on a complex vector space are isomorphic, there exists a linear isometry $L\colon (T_x \mathbb M, \inner{\cdot, \cdot}_{\mathbb M})\to (T_y \mathbb M', \inner{\cdot, \cdot}_{\mathbb M'})$. 
		
		With exactly the same notation as in Section \ref {section: uniqueness pseudoRiem space forms}, every piecewise geodesic curve $\gamma\colon[0,\tau]\to \mathbb M$ starting at $x$ induces a piecewise geodesic curve $\gamma_L$ on $\mathbb M$. 
		Since the Levi Civita connection of a hRm is the Levi-Civita connection for both the real and the imaginary part of the metric, the parallel transport along $\gamma$ defines an isomorphism $T_{\gamma(0)}\mathbb M \to T_{\gamma(\tau)}\mathbb M$ which is an isometry w.r.t. both $Re\inners_{\mathbb M}$ and $Im\inners_{\mathbb M}$, hence it is a linear isometry for $\inners_{\mathbb M}$. The same for $\inners_{\mathbb N}$. The composition of $L$ with the parallel transports alogn $\gamma$ and $\gamma_L$ defines an isometry
		\[
		L_\gamma\colon T_{\gamma(\tau)}{\mathbb M}\to T_{\gamma_L(\tau)} \mathbb M'.\]
		The explicit description of the curvature tensor in Lemma $\ref{lemma space form}$ and the fact that $L_\gamma$ is an isometry imply that 
		\[
		L_{\gamma(\tau)} ^* (R^{\mathbb M'})=  R^{\mathbb M}
		\] 
		both when $R$ is meant as a $(0,4)$-tensor and as a $(1,3)$ tensor.
		
		By Theorem \ref{Teo: Ambrose-Cartan-Hicks}, there exists a covering map $f\colon M \to N$ such that for every piecewise geodesic curve $\gamma\colon [0, \tau] \to M$ starting at $x$, $d_{\gamma(\tau)}f=L_\gamma$, hence $f$ is a local isometry. Since $M$ and $N$ are both simply connected, $f$ is an isometry.
	\end{proof}

	\subsection{Existence - the spaces $\mathbb{X}_n \cong \faktor {SO(n+1, \C)}{SO(n, \C)}$}
	
	The simplest example of holomorphic Riemannian manifold is $\C^n$ with the usual inner product \[\inner{ z, w}_0= \sum_{i=1}^n z_i w_i.\]

	In this thesis, we will focus on another important class of examples.
	
	Consider the complex manifold
	\[\mathbb{X}_n=\{ z\in \C^{n+1}\ |\  \inner{ z,  z}_0= \sum z_i^2=-1 \}.\]
	
	The restriction to $\mathbb X_n$ of the metric $\inners_0$ of $\C^{n+1}$ defines a holomorphic Riemannian metric. Indeed,
	\[
	T_{ z} \mathbb{X}_n =  z ^\bot= \{ v\in \C^{n+1}\ |\ \inner{ v, z}_0=0\}
	\]
	and the restriction of the inner product to $ z^\bot$ is non degenerate since $\inner{ z, z}_{0}\ne 0$; finally, since $\mathbb X_n\subset \C^{n+1}$ is a complex submanifold, local holomorphic vector fields on $\mathbb X_n$ extend to local holomorphic vector fields on $\C^{n+1}$, and this allows to prove that the inherited metric is in fact holomorphic.

	Since $SO(n+1,\C)$ acts by isometries on $\C^{n+1}$ with its inner product, it acts by isometries on $\mathbb{X}_n$ as well and the action on $\mathbb X_n$ is transitive by linear algebra. Moreover, for $ e=(0, \dots, 0, i)\in \mathbb{X}_n$, \[Stab(e)=\begin{pmatrix}
	SO(n, \C) & \underline 0\\
	^t\underline 0 & 1
	\end{pmatrix}\cong SO(n,\C);\] we conclude that $\mathbb X _n$ has a structure of homogeneous space \[\mathbb{X}_n\cong \faktor {SO(n+1, \C)}{SO(n, \C)}.\]
	
	\begin{Theorem}
		\label{Theorem explicit space forms}
		The $n-$dimensional space form with constant sectional curvature $k\in \C$ is:
		\begin{itemize}
			\item $\C^n$ with the usual inner product $\inners_0$ for $k=0$;
			\item $(\mathbb{X}_n, -\frac {1} k\inner{\cdot, \cdot})$ for $k \in \C^*$.
		\end{itemize}	
	\end{Theorem}
	
	It is clear that $\C^n$ is the flat space form: its real and imaginary parts are indeed the pseudo-Riemannian spaces $\R^{n,n}$ which are flat pseudo-Riemannian space forms and for which the curvature tensor is constantly zero.
	
	It is also clear that, if we prove that $\mathbb{X}_n=(\mathbb{X}_n, \inner{\cdot, \cdot})$ is the space form of constant sectional curvature $-1$, then, for all $\alpha \in \C^*$, $(\mathbb{X}_n, -\frac {1} \alpha\inner{\cdot, \cdot})$  has the same Levi-Civita connection and is the space form of constant sectional curvature $\alpha$.
	
	We give a proof of the fact that $\mathbb{X}_n=(\mathbb{X}_n, \inner{\cdot, \cdot})$ is the space form of constant sectional curvature $-1$ among the following remarks on the geometry of this space.
	
	\begin{Remark}		
		\label{rmk Super hRm}
		\begin{enumerate}

			\item \label{rmk Super hRm orientation}
			For future use, we will need to define a sort of extension of the concept of "oriented frame" for the manifold $\XXX_n$.
			
			Define a \emph{complex orientation} for $\mathbb{X}_n$ as a holomorphic non-zero $n$-form $\omega_0$ with the property of being $SO(n+1,\C)$-invariant. At least one complex orientation exists: any $\C$-multilinear $n$-form on $T_e \mathbb{X}_n$ is $SO(n, \C)$-invariant (indeed, it is $SL(n, \C)$-invariant), so the action of $SO(n+1,\C)$ defines a well-posed $SO(n+1,\C)$-invariant $n$-form on $\mathbb{X}_n$. 
			
			By Lemma \ref{mappa differenziale determina isometria}, we have that \[SO(n+1,\C)\cong \Isom_0 (\mathbb{X}_n)\cong \Isom(\mathbb{X}_n, \omega_0).\]
			
			Moreover, $\Isom_0 (\mathbb{X}_n)$ is a subgroup of index $2$ of $\Isom(\mathbb{X}_n)\cong O(n+1,\C)$.

			\item The Levi-Civita connection $D$ for $\mathbb{X}_n$ is the tangent component of the canonical connection $\mathrm d$ for $\C^{n+1}$: in other words, seeing $T \mathbb X_{n}$ as a subbundle of $T\C^{n+1}_{|\mathbb  X_{n} } \equiv \mathbb X_n \times\C^{n+1}$,  smooth vector fields on $T \mathbb X_{n}$ can be seen as smooth functions $\mathbb X_n\to \C^{n+1}$, then 
			
			\[D_{X( z)} Y = (\mathrm d_{X( z)} Y)^T \] 
			
			where $(\mathrm d_{X( z)} Y)^T $ is the tangent component of the vector $ (\mathrm d_{X( z)} Y)^T$ along $T_p \mathbb{X}_n$.
			
			Exactly as in pseudo-Riemannian geometry, this follows by observing that $\mathrm d^T$ is a linear connection for $\mathbb{X}_n$ which satisfies the same properties as the Levi-Civita connection.
			
			\item 
			Consider the problem of finding a geodesic $\gamma\colon \R\to \XXX_n$ with initial data $\gamma(0)$ and $\dot \gamma(0)$. The condition $D_{\dot\gamma(t)}\dot \gamma(t)= (\frac{d \dot\gamma(t)}{dt})^T=0$ leads to $\ddot \gamma(t)\in Span_\CC(\gamma(t))$,
			hence differentiating the condition $\inner{\dot \gamma, \dot\gamma}\equiv \inner{\dot \gamma(0), \dot\gamma(0)}$ one gets the ODE \[\ddot \gamma(t)=  \inner{\dot \gamma(t), \dot \gamma(t)} \gamma(t).\]
			The ODE admits a solution for all initial data $\gamma(0)$ and $\dot \gamma (0)$ and leads to the follwing explicit description of the exponential map:
			\begin{equation}
			\label{descrizione geodetiche}
			\begin{split}
			\exp_{ z} \colon T_{ z}\mathbb{X}_n &\to \mathbb{X}_n\\
			v & \mapsto
			\begin{cases}
			\cosh(\sqrt{\inner{ v,  v}} )  x + \frac{\sinh ({\sqrt{\langle  v,  v\rangle}})}{\sqrt{\langle  v,  v\rangle}}  v \quad &\text{if $\langle  v,  v \rangle\ne 0$}\\
			 x +  v \quad &\text{if $\langle  v,  v \rangle =0$ .}
			\end{cases}
			\end{split}
			\end{equation}
			Notice that the description of the exponential map is independent from the choice of the square root.
			
			Moreover, this explicit description allows to see that for all $ z\in \XXX_n$ and $ v\in T_{ x} \XXX_n$, the map
			\begin{align*}
			\C &\to \mathbb X_n\\
			\lambda& \mapsto exp_{ z}(\lambda  v)
			\end{align*}
			is holomorphic. We will refer to these maps $\lambda \mapsto exp_{ z}(\lambda  v)$ as \emph{complex geodesics}.

			\item The space $\mathbb{X}_n$ is diffeomorphic to $T S^n$. 
			
			Regard $T S^n$ as $\{( u,  v)\in \R^{n+1}\times \R^{n+1}\ |\ \| u\|_{\R^{n+1}}=1, \inner{ u,  v}_{\R^{n+1}} =  0 \}$. Then a diffeomorphism $\mathbb{X}_n \xrightarrow{\sim} T S^n$ is given by
			\begin{equation*}
			 z =  x + i  y \mapsto  (\frac{ 1}{\| y\|_{\R^{n+1}}} y,  x),
			\end{equation*}
			which is well-posed since \[\inner{ x + i  y,  x +i y}_0=-1 \Longleftrightarrow \begin{cases}
			\| y\|^2_{\R^{n+1}}= \| x \|^2_{\R^{n+1}}+1>0 &\\
			\inner{ x,  y}_{\R^{n+1}}=0
			\end{cases}.\]
			
			In particular, $\mathbb{X}_n$ is simply connected for $n \ge 2$.
			
			\item For $n\ge 2$, $\mathbb{X}_n$ has constant sectional curvature $-1$. 
			
			It is clear that it has constant sectional curvature since $SO(n,\C)$ acts transitively on nondegenerate complex planes of $(T_{ e} \mathbb{X}_n,\inners)$. In order to compute the value of the sectional curvature, observe that the embedding 
			\begin{align*}
			\R^{2,1} &\hookrightarrow \C^{n+1}\\
			(x_1, x_2, x_3) &\mapsto (x_1, x_2, 0, \dots, 0, ix_3)
			\end{align*}
			induces an isometric embedding 
			\[
			\Hy^2 \hookrightarrow \mathbb X_n
			\]
			which is admissible and totally geodesic (i.e. it sends geodesics into geodesics or equivalently the pull-back of the Levi-Civita connection is the Levi-Civita connection) by Formula \eqref{descrizione geodetiche}.
			
			Finally, one computes the sectional curvature of a complex $2$-plane tangent generated by the image of $\Hy^2$ by choosing a basis lying in the real tangent space to the image of $\Hy^2$, implying that the sectional curvature is $-1$.
			
			\item
			\label{rmk Super hRm immersion pseudo-Riemannian}
			More generally, one gets that $Q_{-1}^{m,n-m}$ defined as in Proposition \ref{Prop: quadriche pseudo-Riem} isometrically embeds in $\mathbb X_n$:
			\begin{equation}
			\label{eq: quadrica pseudo-Riem dentro X_n}
			\begin{split}
			Q_{-1}^{m,n-m} &\to \mathbb X_{n}\\
			 x=(x_1, \dots, x_{n+1}) &\mapsto (x_1,\dots, x_m, ix_{m+1}, \dots, ix_{n+1}),
			\end{split}
			\end{equation}
			hence $\mathbb F_{-1}^{m,n-m}$ admits an isometric immersion into $\mathbb X_n$.
			
			Clearly these embeddings are planar, i.e. they are the restrictions of a totally geodesic embedding of $\R^{n+1}$ into $\C^{n+1}$.
			In fact, the embeddings of $Q_{m,n,-1}$ are totally geodesic as well: in general if $\mathcal W<\C^{n+1}$ is a real vector subspace such that ${\inners_0}_{|\mathcal W}$ is a real bilinear form, then Formula \eqref{descrizione geodetiche} shows that $\mathcal W\cap \XXX_n$ is totally geodesic.
			
			\item In the above remark, the condition for which ${\inners_0}_{|W}$ is real is essential.
			
			For instance, let $ v_1, v_2\in \C^{n+1}$ be such that $\inner{ v_i,  v_j}_0 =\delta_{ij}$ and define $\mathcal W=Span_\R (i v_1, \sqrt i  v_2)$. Then, $\mathcal W\cap \XXX_n$ has real dimension $1$ and passes by $i v_1\in \mathbb X_n$ where it is tangent to the vector $\sqrt i  v_2$. 
			By Formula \eqref{descrizione geodetiche} it is clear that the geodesic of $\XXX_n$ by $i v_1$ and tangent to $\sqrt i  v_2$ is not contained in $\mathcal W$. 
			
			This example shows both that the intersection of $\XXX_n$ with a generic real vector subspace of $\C^{n+1}$ need not be totally geodesic and that smooth totally geodesic submanifolds of $\XXX_n$ need not be planar.

			\item \label{rmk Super hRm isometries PX}
		 Consider the projective quotient $\Proj \colon \C^{n+1}\setminus 0 \to \CP^n$.
			
			Then ${\Proj}_{|\mathbb{X}_n}$ is a two-sheeted covering on its image $\PXX_n$, which corresponds to the complementary in $\CP^n$ of the non-degenerate hyperquadric \[\partial \PXX_n=\{z_1^2+ \dots +z_{n+1}^2=0\}.\]
			
			Since $-id\in O(n+1,\C)$, the hRm on $\XXX_n$ descends to a hRm on $\PXX_n$. The $O(n+1,\C)$-invariancy of $\mathbb{X}_n$ and its hRm implies that the action of $SO(n+1,\C)$ on $\CP^n$ fixes $\PXX_n$ globally (hence the complementary hyperquadric) and acts by isometries on it.

			The group of isometries of  $\PXX_n$ is given by $PO(n+1,\C)$ acting on the whole $\CP^n$ as a subgroup of $PGL(n+1,\C)$. Conversely, it is simple to check that $PO(n+1,\C)$ coincides exactly with the subgroup of elements of $PGL(n+1,\C)$ that fix $\PXX_n$ globally (or, equivalently, that fix $\partial \PXX_n$ globally).

			\item From the expression of the exponential map given in Formula \eqref{descrizione geodetiche} it is immediate to see that, if $\mathcal W$ is a complex subspace of $\mathbb C^{n+1}$ then $\mathcal W\cap\XXX_n$ is a totally geodesic complex submanifold. Equivalently, the intersection of $\PXX_n$ with a projective subspace of $\CP^n$ is a totally geodesic complex submanifold. 
			
			It turns out that the intersection of $\PXX_n$ with a complex projective line $\ell$ concides with the exponential of a complex subspace of (complex) dimension $1$ of the tangent plane at any point of $\ell\cap \PXX_n$. As a result, any totally geodesic complex submanifold is in fact the intersection of $\PXX_n$ with a complex projective subspace.
			
			We remark that if ${\inners_0}_{|\mathcal W}$ is not degenerate then $\mathcal W\cap\XXX_n$ is isometric to $\XXX_{k}$, where $k=\dim_{\mathbb C}(\mathcal W)-1$.
			
			\item If $\ell=\mathbb P(W)$ is a projective line in $\CP^n$, then $\ell\cap \partial \PXX$ can contain either one point (if $\ell$ is tangent to $\partial \PXX$) or two points (if the intersection is transverse). In the latter case the restriction of the product
			$\inners$ to $\mathcal W$ admits two isotropic directions, and in particular is not degenerate. In the former case there is only one isotropic direction that in fact is contained in the orthogonal subspace of $\mathcal W$. The product in this case is degenerate and the restriction of the metric on $\ell\cap \partial\PXX$ is totally isotropic.

		\end{enumerate}

	\end{Remark}

We will show that the spaces $\mathbb X_2$ and $\mathbb X_3$ have an interesting geometric description in relation with $\Hyp^3$: we will devote Chapter \ref{Chapter SL} to showing that $\mathbb X_3$ is isometric to $\SL$ with its Killing form up to a scale, and Section \ref{section: G3 e hRm} to showing that $\mathbb X_2$ can be seen as the space of geodesics of $\Hyp^3$. 

Before that, let us give a quick look at $\mathbb X_1$.

	\section{$\mathbb X_1$ as $\C^*$ endowed with a holomorphic quadratic differential}
	
	By regarding $\mathbb X_1$ as $\{(z_1,z_2)\in \C^2\ |\ z_1^2 +z_2^2=-1\}=\{(i\cos(z),i\sin(z))\ |\ z\in \C\}$, one can define the map
	\begin{align*}
	F_1\colon \mathbb X_1 &\xrightarrow{\sim} \C^*\\
	(z_1,z_2)=(i\cos(z),i\sin(z))&\mapsto z_2-iz_1=e^{iz}
	\end{align*}
	which is in fact a biholomorphism.
	
	Denote with $\inners$ also the push-forward Riemannian holomorphic metric on $\C^*$.
	
	A straightforward calculation shows that the automorphisms of $(\C^*,\inners)$ correspond to multiplications by a constant: indeed, for all $\begin{pmatrix}
	\cos(\alpha) & - \sin(\alpha)\\
	\sin(\alpha) & \cos(\alpha)
	\end{pmatrix}\in SO(2,\C)$, $\alpha\in \C$, we have 
	\[
	F_1\ ^t\bigg( \begin{pmatrix}
	\cos(\alpha) & - \sin(\alpha)\\
	\sin(\alpha) & \cos(\alpha)
	\end{pmatrix} \cdot \begin{pmatrix}
	i\cos(z)\\
	i\sin(z)
	\end{pmatrix} \bigg)= F_1( (i\cos(z+\alpha), i\sin(z+\alpha))= e^{i\alpha}e^{iz}.
	\]
	
	Using this isotropy, we can compute explicitly the holomorphic Riemannian metric. At each point $z\in \C^*$, we have 
	\[\inners_z= \lambda(z) dz^2
	\]
	for some holomorphic function $\lambda\colon \C^*\to \C$. 
	
	Invariance by constant multiplication implies $\lambda(z)=\frac {\lambda(1)} {z^2}$. 
	In order to compute $\lambda (1)$, set the condition for which the vector $F_{*(0,i)} \bigg(\begin{pmatrix} 1 \\ 0 \end{pmatrix}\bigg)= -i \in T_i \C^*$ has norm $1$ to get $\lambda(1)=1$: we conclude that 
	\[
	\inners= \frac {dz^2}{z^2}.
	\]
	
	In general, we can see $\mathbb X_1$ as $\CP^1\setminus \{ p_1,p_2\}$ and the holomorphic Riemannian metric as some holomorphic quadratic differential on $\CP^1$ with exactly two poles of order $2$ in $p_1$ and $p_2$.
	
	Finally, a direct computation via $F$ shows that geodesics in $(\C^*, \inners)$ are all of the form $t\mapsto \mu_1 e^{t\mu_2}$ with $\mu_1\in \C^*$, $\mu_2\in \C$: the images correspond either to a circumference with center $0$, a straight ray connecting $0$ and $\infty$ or a spiraling ray connecting $0$ and $\infty$. As a result, the geodesics for $\inners$ coincide with the geodesics for the flat structure $\frac{|dz|^2}{|z|^2}$ induced by the metric seen as a holomorphic quadratic differential. 
	\vspace{6mm}

	\begin{tikzpicture}[scale=0.4]
	\draw[domain=-500:-200, hobby,densely dashed] plot ({-exp(\x/200)*sin(\x)},{exp(\x/200)*cos(\x)});
	\draw[domain=-200:400, hobby] plot ({-exp(\x/200)*sin(\x)},{exp(\x/200)*cos(\x)});
	\draw[domain=400:410, hobby, densely dashed] plot ({-exp(\x/200)*sin(\x)},{exp(\x/200)*cos(\x)});
	\draw (12,0) ellipse (1 and 1); 
	\draw (24,0.5)--(24,5);
	\draw[densely dashed] (24,5)--(24,6);
	\draw[densely dashed] (24,0)--(24,0.5);
	\coordinate [label=above:$i$] (A) at (0,1); \fill (A) circle (3pt);
	\coordinate [label=above:$i$] (B) at (12,1); \fill (B) circle (3pt);
	\coordinate [label=right:$i$] (C) at (24,1); \fill (C) circle (3pt);
	\coordinate [label=below:$0$] (D) at (0,0); \fill (D) circle (1.5pt);
	\coordinate [label=below:$0$] (E) at (12,0); \fill (E) circle (1.5pt);
	\coordinate [label=below:$0$] (F) at (24,0); \fill (F) circle (1.5pt);
	\end{tikzpicture}

	\chapter{$\SL$ as a holomorphic Riemannian space form}
	\label{Chapter SL}

		\section{$\mathbb X_3$ and $\SL$ with its Killing form} 
		We show that $\mathbb{X}_3$ is isometric (up to a scale) to the complex Lie group $\SL=\{A\in Mat(2,\CC)\ |\ det(A)=1 \}$ endowed with the holomorphic Riemannian metric given by the Killing form, globally pushed forward from $I_2$ equivalently by left or right translation.
	
	Consider on $Mat(2,\CC)$ the non-degenerate quadratic form given by $M\mapsto -det(M)$, which corresponds to the complex bilinear form 
	\[
	\inner{M,N}_{Mat_2}=\frac{1}{2}\bigg(tr(M\cdot N)- tr(M)\cdot tr(N) \bigg).
	\]
	In the identification $T Mat(2,\CC)=Mat(2,\CC)\times Mat(2,\CC)$, this complex bilinear form induces a holomorphic Riemannian metric on $Mat(2,\CC)$. 
	
	Also notice that the action of $\SL\times \SL$ on $Mat(2,\CC)$ given by $(A,B)\cdot M:= A 
	M B^{-1}$ is by isometries, because it preserves the quadratic form.

	Since all the non-degenerate complex bilinear forms on complex vector spaces of the same dimension are isomorphic, there exists a linear isomorphism \[F\colon (\C^4,\inners_0) \to (Mat(2,\CC),\inners_{Mat_2}),\]
	such as, for instance,
		\begin{equation}
	\label{iso C4 Mat(2,C)}
	F(z_1, z_2, z_3, z_4)=
	\begin{pmatrix}
	-z_1- iz_4 &  -z_2-iz_3\\
	-z_2 + iz_3 & z_1-iz_4
	\end{pmatrix}.
	\end{equation}
	
	  As a linear isometry, $F$ is an isometry between the natural induced holomorphic Riemannian metrics on $\C^4$ and $Mat(2,\CC)$: such isometry $F$ restricts to an isometry between $\mathbb X_3$ and $\SL$, where $\SL$ is equipped with the submanifold metric.

	For all $A\in \SL$,  $T_A \SL=   A\cdot T_{I_2} \SL$ and \[T_{I_2} \SL=\asl= \{M\in Mat(2,\C) \ |\ tr(M)=0 \}\] can be endowed with the structure of complex Lie algebra associated to the complex Lie group $\SL$.
	
	The induced metric on $\SL$ at a point $A$ is given by
	\[
	\inner{AV,AW}_A= \inner{V,W}_{I_2}= \inner{V,W}_{Mat_2}= \frac{1}{2} tr(V\cdot W),
	\]
	
	as a consequence we have $\inners_{I_2}=\frac{1}{8} Kill$ where $Kill$ is the Killing form of $\asl$. Since $\asl$ is semisimple, its Killing form is $Ad$-invariant, and $\inners_A$ is the symmetric form on $T_A \SL$ induced by pushing forward $\inners$ equivalently by right or left translation by $A$.
	
	We will often focus on \[\PSL=\faktor{\SL}{\{\pm I_2\}}\cong \Proj \mathbb X_3\] as well, which is a 2:1 isometric quotient of $\SL$.

\subsection{The group of automoporphisms}

	\begin{Proposition}
		\label{prop: Isom of SL}
		Define $Aut(\SL):=\Isom_0( \SL, inners)$. Then
		\[Aut(\SL)\cong \faktor{\SL \times \SL}{\{\pm(I_2,I_2)\}}\] where the action of $\SL\times \SL$ is given by
		\begin{equation}
		\label{automorfismi SL}
		\begin{split}
		{\SL \times \SL}\times \SL \to \SL\\
		(A,B)\cdot C := A\ C\ B^{-1}
		\end{split}.
		\end{equation}.
		In particular, the $Ad$ representation $Ad\colon \SL \to Isom_0(\asl, \inners)\cong SO(3,\C)$ is surjective.
	\end{Proposition}
	\begin{proof}
		Since $Isom_0(\mathbb X_3)\cong Isom_0(\C^{4}, \inners_0)\cong SO(4, \C)$, one has  $Isom_0(\SL)\cong Isom_0(Mat(2,\C))\cong SO(4,C)$. Since, for all $A, B\in \SL$, $det(A\ M\ B^{-1} )= det(M)$, the action above preserves the quadratic form, hence it is by isometries. We therefore have a homomorphism $\SL\times \SL \to Isom(\SL)$ whose kernel is $\{\pm(I_2,I_2)\}$. Finally, since \[dim_\C \bigg(\faktor{\SL\times \SL}{ \{\pm(I_2,I_2)\} }\bigg)=6= dim_\C Isom(\SL),\] we conclude that $Isom_0 (\SL)\cong \faktor{\SL\times \SL}{ \{\pm(I_2,I_2)\}}$
		
		The same dimensional argument shows that $Ad$ is surjective.
	\end{proof}

	\begin{Remark}
		As a consequence of Proposition \ref{prop: Isom of SL}, one gets that
		\[
		\Isom_0(\PSL. \inners) \cong \PSL\times \PSL
		\]
		where $\PSL\times \PSL$ acts on $\PSL$ by left and right multiplication. Moreover, $Ad\colon \PSL\to \Isom_0 (\asl)$ is an isomorphism.
	\end{Remark}
 
\subsection{Cross-product}

 \begin{Lemma}
 	For all $V,W\in \asl$,
 	\[
 	\| [V, W] \|^2 =-4 \|V\|^2 \|W\|^2 + 4 \inner{V,W}^2.
 	\]
 	As a result, if $V, W$ are orthonormal, then $ (V, W, \frac 1 {2i} [V,W] )$ is an orthonormal basis.
 \end{Lemma}
 \begin{proof}
 	It is enough to prove the formula for $V$ and $W$ non-isotropic, since non-isotropic vectors are dense in $\asl$.
 	
 	Observe that for $V \in Mat(2,\C)$ the following equations stand:
 	\begin{gather}
 	\label{formula1} Tr(V^2)= (Tr(V))^2 -2 det(V); \\
 	\label{formula2}det(V- I)= 1 -Tr(V)+det(V); \\
 	\label{formula3}\text{if $V\in \asl$, then} \qquad V^2=det(V) I_2
 	\end{gather}
 	Hence, with a straightforward computation, if $V,W \in \asl$ are non isotropic, we get that
 	\begin{align*}
 	\|[V,W]\|^2 = &-det(V\ W\ -\ W\ V)= -det(V\ W\ V^{-1}W^{-1} -\ I)det(V)det(W)= \\
 	&= -det(V) det(W)(2 - Tr(V\ W\ V^{-1} W^{-1}))=\\
 	&= -det(V)det(W) \Bigg(2- \frac 1 {det(V)det(W)} Tr\Big( (V\ W)^2\Big) \Bigg) =\\
 	&=-2 det(V)det(W) +Tr ( (V\ W)^2 ) - 2 det(V)det(W)=\\
 	&= - 4 \|V\|^2 \|W\|^2 + 4 \inner{V,W}^2.
 	\end{align*}
 	
 	If $V, W$ are orthonormal, then $\|[V, W]\|^2=-4$. Moreover, since $\inners$ is $ad$-invariant on $\asl$, we have $\inner{[V,W],V}= \inner{[V,W],W}=0$, proving that $(V, W, \frac 1 {2i} [V,W])$ is an orthonormal basis.
 \end{proof} 
 
 The previous Lemma suggests the definition of a \emph{cross-product} 
 \[
 \times \colon \asl \times \asl \to \asl
 \] given by
 \[
 V \times W := \frac{1}{2i} [V,W].
 \]
 
 Observe that, for all $A\in \PSL$, \[Ad(A)(V) \times Ad(A) (W)= Ad(A) (V\times W)\]
 since $Ad(A)$ is a Lie-algebra endomorphism of $\asl$. As a result, the cross product can be extended to a $(1,2)$ tensor on $\PSL$, equivalently by left or right pull-back.
 
 We also say that a basis $(V_1, V_2, V_3)$ for $\asl$ is \emph{positive} if $V_3 = V_1 \times V_2$.

 We can fix a positive orthonormal basis given by $\mathbb {B}^0 =(V_1 ^0, V_2^0, V_3^0)$, where
 \begin{equation}
 \label{base positiva}
 \quad V_1 ^0 = \begin{pmatrix}
 1 & 0 \\
 0 & -1
 \end{pmatrix}, \quad V_2 ^0 =\begin{pmatrix}
 0 & 1 \\
 1 & 0
 \end{pmatrix}, \quad V_3 ^0 =\begin{pmatrix}
 0 & -i \\
 i & 0
 \end{pmatrix}. 
 \end{equation}

 Equivalently, one can consider the complex $3$-form $\omega_0$ on $\asl$ defined by $\omega_0 (V_1^0, V_2^0, V_3^0)=1$. Such $\omega$ is $Ad$-invariant and hence can be extended to a global $3$-form $\omega_0$ on $\PSL$: one can therefore check that, for all $U,V\in T_A \PSL$, the cross product $U\times V$ is the unique vector in $T_A \PSL$ such that, $\forall W\in T_A \PSL$, $\omega_0(U,V,W) = \inner{U\times V, W}$. So, a orthonormal frame ${V_1, V_2, V_3}$ is positive if and only if $\omega_0(V_1, V_2, V_3)=1$.

 Exactly as it happened for $\mathbb X_3$ in Remark \ref{rmk Super hRm}.\ref{rmk Super hRm orientation},
 \[
 Aut(\SL)=Isom_0( \SL, \inners)=Isom(\SL, \inners, \omega_0).
 \]

\subsection{$\asl$ and isometries of $\Hy^3$}

As we showed in Section \ref{sec model Hn}, the group $\PSL$ acts on $\overline \C$ via M\"obius maps: the action furnishes an isomorphism with the group of biholomorphisms of the 2-sphere $\PSL\cong Aut(\C \Proj^1)$, which induces a canonical isomorphism $\PSL\cong {Isom_0} (\Hy^3)$.

There are several interesting correspondences between the Killing form of $\asl$, equal to $\inners$ defined above up to a scale, and the geometry of $\Hyp^3$. Recall that $\inner{V,V}=-det(V)$ for all $V\in \asl$.

Since the isometry $F$ defined above is a linear map, by Equation \eqref{descrizione geodetiche}, we get that the exponential map at $I_2\in \PSL$ is given by 
\begin{equation}
\label{eq geodetiche SL}
\exp(V)= \cosh(\sqrt{\inner{V,V}}) I_2 + \frac {\sinh (\sqrt{\inner{V,V}})} {\sqrt{\inner{V,V}}}  V,
\end{equation}
for all $V\in \asl$, where we denote $\frac{\sinh (\sqrt{\inner{V,V}})} {\sqrt{\inner{V,V}}}=1$ if $\inner{V,V}=0$.

As a result, $Tr(exp(V))= 2\cosh(\|V\|)$ and, if we see $\PSL$ as the automorphisms group of $\Hy^3$, we can determine the nature of $exp(V)$ by the norm of $V$. Since the type of the isometry $A\in \PSL$ can be deduced by $tr(A)$ (See \cite{martelli}), we have the following.

\begin{Prop}
	\label{prop: tipo isometria H3}
	Define \[\Lambda:=\{W\in \asl \ |\ \|W\|\in 2i \pi \Z, \|W\|\ne 0\}\cup \{0\}.\] Then, for $V\in \asl$:
	\begin{itemize}
		\item $exp(V)=id$ if and only if $V\in \Lambda$;
		\item $exp(V)$ is elliptic if and only if $V\notin \Lambda$ and $\| V\| \in i \R \setminus \{0\}$;
		\item $exp(V)$ is parabolic if and only if $\|V\|=0$ and $V\ne 0$;
		\item $exp(V)$ is (purely) hyperbolic if and only if $\|V\| \in \R + 2i\pi \Z$ with $\nobreak{Re(\|V\|)\ne 0}$;
		\item $exp(V)$ is loxodromic elsewise.
	\end{itemize}
\end{Prop}

Now, define the set
\begin{gather*}
\overline{\mathbb{G}}_\sim= \faktor{\partial\Hyp^3 \times \partial \Hyp^3}{\sim: (x,y)\sim (y,x)}
\end{gather*} 
and consider the map 
\begin{equation}
	\label{eq: mappa Fix su bar G}
Fix\colon \PSL \setminus\{I_2\} \to \overline{\mathbb G}_\sim
\end{equation}
that assigns to each isometry of $\Hy^3$ different from $I_2$ its fixed points, counted as a double point if the isometry is parabolic. If $A\in \PSL \setminus \{I_2\}$ is not parabolic, we can also see $Fix(A)$ as an unoriented unparametrized maximal geodesic of $\Hyp^3$.

The group $\PSL$ acts on $\asl \setminus \Lambda$ via $Ad$, acts on $\PSL \setminus \{I_2\}$ by conjugation and acts on $\overline {\mathbb G}_\sim$ as $\PSL$ is the isometry group of $\Hy^3$: the maps 
\[
\asl \setminus \Lambda \xrightarrow{exp} \PSL\setminus \{I_2\} \xrightarrow{Fix} \overline{\mathbb G}_\sim
\] 
are all equivariant with respect to these actions.

\begin{Proposition}
	\label{Prop: geodetiche complesse stesso asse}
	Let $V, W \in \asl \setminus \Lambda$.
	\begin{itemize}
		\item[1)] $Fix\circ \exp (V)= Fix\circ exp(W)$ if and only if $Span_\C(V)= Span_\C(W)$. As a result
		\[
		Fix \circ exp \colon \mathbb P \asl\to \overline{\mathbb G}_\sim 
		\]
		is a bijection.
		\item[2)] If $\inner{V,W}=0$ with $\inner{V,V}, \inner{W,W}\ne 0$, then the two geodesics $Fix(exp(V))$ and $Fix(exp(W))$ meet orthogonally.
		\item[3)] If $\inner{V,W}=0$ with $\inner{V,V}=0$, then the point fixed by $exp(V)$ is also fixed by $exp(W)$. 
	\end{itemize}
\end{Proposition}
\begin{proof}
	\begin{itemize}
		\item[1)] Since the $Ad$ representation is surjective by Proposition \ref{prop: Isom of SL}, the diagonal in $\SL\times \SL$ acts transitively on the vectors of $\asl$ with squared-norm $-1$, and on the isotropic ones. As a consequence, it is sufficient to check that the statement holds for $V=V^0_1$ and for $V= V^0_2- i V^0_3$, defined as in \eqref{base positiva}. 
		
		Pick the model $\partial\Hyp^3\cong\overline \C$. We use extensively Equation \eqref{eq geodetiche SL}.
		
		For all $\lambda\in \C$, $exp(\lambda V^0_1)=\begin{pmatrix}
		e^\lambda &0\\
		0 & e^{-\lambda}
		\end{pmatrix}$ has fixed points $0$ and $\infty$ for all $\lambda$ and one can immediately check that an element of $\PSL$ fixing $0$ and $\infty$ is of this form.
		
			In the same fashion, $exp(\lambda(V_2^0-iV_3^0))= \begin{pmatrix}
			1 & 2\lambda \\
			0 &1
		\end{pmatrix}$ has $\infty$ as unique fixed point, and this is the description of every other automorphism with this property.
		
		This proves that $Fix\circ exp$ is injective. It is also surjective because both $Fix$ and $exp$ are.
		
		\item[2)] The group $SO(3,\C)$ acts transitively on couples of orthonormal vectors, so the statement holds if and only if it works for $V=V_1^0$ and $V=V_2^0$. 
		
		An explicit computation (composed with a suitable stereographic projection) shows that, in the model $\partial \Hyp^3\cong \mathbb S^2$,
		 \begin{align*}
		&Fix \circ exp(Span_\C (V_1^0))= \{(0,1,0), (0,-1,0)\}\\
		&Fix \circ exp(Span_\C (V_2^0))= \{(0,0,1), (0,0,-1)\}
		\end{align*}
		and the thesis follows.
		\item[3)] Once again, $SO(3,\C)$ acts transitively on couples of orthogonal vectors $(V,W)$ with $V$ isotropic and $W$ with norm $1$, so the thesis holds as it holds for $(V^0_2- i V^0_3, V^0_1)$.
		\end{itemize}
\end{proof}

	\subsection{A boundary for $\PSL$}	
	The linear isomorphism $F$ defined in Equation \eqref{iso C4 Mat(2,C)} induces a map $F\colon \C\mathbb P^3 \xrightarrow{\sim} \mathbb P Mat(2,\C)$ which restricts to $F\colon \mathbb P \XXX_3 \xrightarrow{\sim} \PSL$. We can define the \emph{boundary} of $\PSL$ as 
	\[\partial \PSL= \mathbb P Mat(2,\C) \setminus \PSL= \{[M]\in \mathbb P Mat(2, \C)\ |\ det(M)=0\},\]
	 in particular $\PSL$ is a Zariski-open subset of $\mathbb P Mat(2,\C)$.

Observe that the map
	\begin{align*}
	\partial \PSL&\cong \C\mathbb P^1 \times \C \mathbb P^1=\partial \Hy^3 \times \partial \Hy^3\\
	[M] &\mapsto (Ker M, Im M)
	\end{align*}	
is a biholomorphism, with inverse
\[
((z_1 \colon z_2), (w_1 \colon w_2)) \mapsto \begin{pmatrix} w_1 z_2 & -w_1z_1\\
w_2 z_2 & -w_2 z_1
\end{pmatrix}
\]

As we noticed in Remark \ref{rmk Super hRm}.\ref{rmk Super hRm isometries PX}, every isometry of $\mathbb {P}\XXX_3$ extends to a linear isomorphism of $\C \mathbb P^3$, and clearly the same holds for	$\PSL$ inside $\mathbb P Mat(2,\C)$.

	\begin{Proposition}
		\label{Prop geodetiche sul bordo}
		Let $\gamma$ be the closure in $\mathbb P Mat(2,\C)$ of a nonconstant maximal (real or complex) geodesic in $\PSL$ with $\gamma(0)=I_2$. Then, via the identification above,
		\begin{itemize}
			\item if $\dot\gamma(0)$ is isotropic, then $\gamma \cap \partial \PSL =\{(z,z)\}$ where $Fix(\gamma(t))=(z,z)\in \overline{\mathbb G}_\sim$ for all $t\ne 0$;
			\item if $\dot\gamma(0)$ is not isotropic, then $\gamma \cap \partial \PSL =\{(z_1,z_2), (z_2, z_1)\}$, where $Fix(\gamma(t))=[(z_1,z_2)]$ for all $t$ for which $\gamma(t)\ne I_2$.
		\end{itemize}
		
	\end{Proposition}
	\begin{proof}
		See $\PSL= PGL(2,\CC)$.
		
		If $V\in \asl$, with $\|V\|=0$ and $V\ne 0$, we have
		\[
		exp(tV)= [I+tV]=[V+ \frac 1 t I] ,
		\]
		hence, $\partial \{exp(tV)\ |\ t\in \C \}= \{(Ker(V),Im(V))\}$. Since $det(V)=0$, $V^2=0$ by Formula \eqref{formula3}, hence $Ker(V)=Im(V)$. Moreover, for all $ z \in Ker(V)$, $exp(tV)( z)=  z$ for all $t$, proving that $z:=[ z]\in \C \mathbb P^1$ is the fixed point of $\gamma(t)$.
		
		If $V\in \asl$ with $\|V\|^2=1$ we have
		\[
		exp(tV)= [\cosh(t) I + \sinh(t) V]=[I+ \tanh(t) V],
		\]
		hence, $\partial \{exp(tV)\ |\ t\in \C\} =\{(Ker(I+V), Im(I+V)), (Ker(I-V), Im(I-V))\}$. Observe that $(I+V)(I-V)=(I-V)(I+V)={0}$, hence $Ker(I\pm V)=Im(I\mp V)$. Clearly, $Ker(I+V)$ and $Ker(I-V)$ correspond to eigenvectors of $V$, thus, by the explicit description, they correspond to eigenvectors of $exp(tV)$ for all $t$, proving that they are fixed points of the corresponding M\"obius maps.
	\end{proof}

\section{Some interesting immersions of $\mathbb F_{-1}^{m,3-m}$ in $\SL$}

As we observed in Remark \ref{rmk Super hRm}.\ref{rmk Super hRm immersion pseudo-Riemannian}, $\mathbb F_{-1}^{m,3-m}$ isometrically immerges into $\mathbb{X}_3$ in a totally geodesic way. Let us see that, by post-composing Equation \eqref{eq: quadrica pseudo-Riem dentro X_n} with the map $F$ in Equation \eqref{iso C4 Mat(2,C)}, $\Hyp^3$, $AdS^3$ and $-\mathbb S^3$ admit some interesting embeddings into $\SL$.

\begin{itemize}
	\item  Let us start from $AdS^3$. Recall that $AdS^3$ can be seen as the quadric $Q_{-1}^{2,1}\subset \R^{2,2}$ which embeds into $\mathbb X_3$ via
	\begin{align*}
		AdS^3 &\hookrightarrow \mathbb X_3\\
		(x_1,x_2,x_3,x_4)&\mapsto (x_1, x_2, i x_3, ix_4),
	\end{align*}
	and composing with $F$ we get the isometric embedding
	\begin{align*}
		F_1\colon AdS^3 &\to \SL \\
		(x_1,x_2,x_3,x_4)&\mapsto \begin{pmatrix} -x_1+x_4 &-x_2+x_3\\
			-x_2 - x_3 & x_1+x_4
		\end{pmatrix}.
	\end{align*}
	Observe that $F_2(AdS^3)$ is exactly the subset of the elements of $\SL$ whose entries are real, namely $F_2$ gives the isometry
	\[
	F_2\colon AdS^3 \to \mathrm{SL}(2,\RR).
	\]
	Since the Killing form of $\mathfrak{sl}(2,\RR)$ is the restriction of the 
	Killing form of $\asl$, $F_1$ provides an isometry between $AdS^3$ and the global Killing form of $\SL$ up to a scalar.

\item	
Replying the same constructions as above, the isometric immersion of $-\mathbb S^3\subset \R^{0,4}$ into $\mathbb X_3$ given by
\[
(x_1,x_2,x_3,x_4)\mapsto (ix_1, ix_2,ix_3,ix_4)
\]
provides an isometric immersion
\begin{align*}
	F_2\colon -\mathbb S^3 &\to \SL \\
	(x_1,x_2,x_3,x_4)&\mapsto \begin{pmatrix}
		x_4 -ix_1  & +x_3 -ix_2 \\
		-x_3 -ix_2  & x_4 -ix_1  
	\end{pmatrix}
\end{align*}
One gets that the image of $F_2$ is in fact the Lie group $SU(2)$, which corresponds to the stabilizer of $\underline 0\in \Hyp^3$ in the Poincar\'e disk model. One has once again that the Killing form of ${\mathfrak{su}(2)}$ is the restriction of the Killing form of $\asl$, therefore $-\mathbb S^3$ is isometric up to a scalar factor to $SU(2)$ with its global Killing form.

\item For $\Hyp^3$ in the hyperboloid model, let us slightly modify $F$ with $-F$ to get the explicit isometric immersion
\begin{equation}
	\begin{split}
		F_3\colon \Hyp^3 &\to \SL\\
		(x_1,x_2,x_3,x_4)&\mapsto \begin{pmatrix}
			x_1-x_4 & x_2 +ix_3\\
			x_2-ix_3 & x_1 +x_4
		\end{pmatrix}
	\end{split}.
\end{equation}
Since the immersion is totally geodesic,  $F_3 (exp^{\Hyp^3})= exp^{\SL}( dF_3(v))$, hence the image of $F_3$ is the image via $exp^{\mathrm{SL}}$ of $d F_3 (T_{(0, 0, 0,1)}\Hyp^n)$. Observe that
\begin{align*}
d F_3 (T_{(0, 0, 0,1 )}\Hyp^n) &=\{V\in \asl\ |\ ^tV= \overline V \}=\\
&= i \cdot \{V\in \asl\ |\ ^t V= - \overline V\} = i \mathfrak{su}(2).
\end{align*}
Now, $\Hyp^3$ has a natural structure of symmetric space for $\SL$ when you see it in the disk model: the stabilizer of $0\in \Hyp^3\subset \R^3$ is $SU(2)$ and the corresponding Cartan decomposition is given by $\asl=\mathfrak {su}(2)\oplus i \mathfrak{su}(2)$ (see also Section \ref{Section symm spaces}).
 By the classical theory of symmetric spaces, $exp(\mathfrak{su}(2))\subset \SL$ corresponds to purely hyperbolic maps of $\Hyp^3$ with axis passing by $0$: more precisely, the projection of $exp(\mathfrak{su}(2))$ in $\PSL$ contains all and only the isometries of this type. 
\end{itemize}

\chapter{The space of geodesics of the hyperbolic space}
\label{chapter spazio geodetiche}
As a consequence of the geometry of the hyperbolic space $\Hyp^n$, a maximal non-constant geodesic in $\Hyp^n$ is characterized, up to orientation-preserving reparameterization, by the ordered couple of its "endpoints" in its visual boundary at infinity $\partial \Hyp^n$: by means of this identification, from now on we will refer to the manifold
\[
\G n= \partial \Hyp^n\times\partial \Hyp^n \setminus \Delta
\]
as the \emph{space of (maximal, oriented, unparameterized) geodesics} of $\Hyp^n$, namely, the space of non-constant geodesics $\gamma:\R\to\Hyp^n$ up to monotone increasing reparameterizations. The space $\G n$ is a manifold of dimension $2(n-1)$.

 We recall that, in the disk model $\partial \Hyp^n$ is naturally identified with the unitary sphere $S^{n-1}$, in the hyperplane model with $\R^{n-1}\cap \{\infty\}$, while, in the hyperboloid model, $\partial\Hyp^n$ can be identified to the projectivization of the null-cone in Minkowski space:
\begin{equation}\label{eq:bdy hyperboloid model}
	\partial\Hyp^n=\faktor{\{x\in\R^{n,1}\,|\,\langle x,x\rangle =0,\,x_{n+1}>0\}}{\R_{>0}}~.
\end{equation}

The aim of this chapter is to present a model for a para-K\"ahler structure for $\GGG_n$, which is going to encode a lot of informations about the geometry of $\Hyp^n$ and of $T^1\Hyp^n$.

\section{The geodesic flow and the para-Sasaki metric on $T^1 \mathbb X_n$}
\label{sec geodesic flow on T1Xn}
As we have already observed, the tangent bundle to $\XXX_n$ can be seen as
\[
T \XXX_n = \{( z,{ v})\in \C^{n+1}\times \C^{n+1}\ |\  z\in \XXX_n \ |\ \inner{ z,  z}_0=-1,\, \inner{ z,  v}_0=0  \},
\]
where the obvious projection map is 
$$\bar\pi:T\XXX_n\to\XXX_n\qquad \pi( z,{ v})= z~.$$

We will be particularly interested in the \emph{unit tangent bundle} of $\XXX_n$, namely the bundle of unit tangent vectors
\begin{equation}\label{eq:modelT}
	T^1\XXX_n=\{( z,{ v})\in T\XXX_n \,|\, \inner{ v,  v}_0=1\}
\end{equation}
and let us denote $\pi= \bar\pi |_{T^1\XXX_n}$
Notice that both $T \XXX_n$ and $T^1 \XXX_n$ are complex submanifolds of $\C^{n+1}\times \C^{n+1}$ and that $\pi$ and $\bar\pi$ are both holomorphic.

In this model, we can describe the tangent spaces to $T\XXX_n$ and to $T^1\XXX_n$ in a point $(z,v)$ as
\begin{equation}
	\label{eq:modelTT}
	\begin{split}
		T_{( z,{ v})}T\XXX_n &=\{(\dot { z},\dot { v})\in\C^{n+1}\times \C^{n+1}\,|\,\langle  z,\dot { z}\rangle_0= \langle  z,\dot { v}\rangle_0+\langle { v},\dot { z}\rangle_0=0\}~\\ 
		T_{( z,{ v})}T^1\XXX_n &=\{(\dot { z},\dot { v})\in\C^{n+1}\times \C^{n+1}\,|\,\langle  z,\dot { z}\rangle_0=\langle  z,\dot { v}\rangle_0+\langle { v},\dot { z}\rangle_0=\langle { v},\dot { v}\rangle_0=0\}~
	\end{split}
\end{equation}

We observed that $\XXX_n$ has both a notion of real geodesic and a notion of complex geodesic. In the following, $\lambda$ can be seen both as a real or as a complex number. 

Define the real (resp. complex) \emph{geodesic flow} on $T\XXX_n$ as
\begin{equation}
	\label{eq: geodesic flow}
	\lambda\cdot ( z,{ v})=\varphi_\lambda( z,{ v})= (\gamma(\lambda),\gamma'(\lambda))~,
\end{equation}
where $\gamma$ is the unique parameterized real (resp. complex) geodesic such that $\gamma(0)= z$ and $\gamma'(0)={ v}$. 

Both the real and the complex geodesic flow restrict to flows on $T^1\XXX_n$. By Formula \eqref{descrizione geodetiche} the flow $\varphi_\lambda\colon T^1\Hyp^n\to T^1\Hyp^n$ can be written explicitly as
\begin{equation}\label{eq:geodflow}
	\varphi_\lambda( z,  v) = (\cosh(\lambda) z + \sinh(\lambda) v,\sinh(t) z+\cosh(t){ v})~.
\end{equation}

We also denote the vector field $\chi$ on $T \XXX_n$ defined by
\[
\chi_{( z_0,  v_0)} := \frac{d}{dt}_{|t=0} \phee_{t}( z_0,  v_0)= \frac{d}{dz}_{|z=0} \phee_z( z_0,  v_0)
\] 
as the \emph{infinitesimal generator} of the geodesic flow.

In the following, we will use the expression \emph{geodesic flow} to refer to the \emph{real} geodesic flow.
\vspace{5pt}

We shall now introduce a holomorphic Riemannian metric on $T\XXX_n$ and $T^1\XXX_n$ for which elements in $\Isom(\XXX_n)$ and the geodesic flow act by isometries. For this purpose, let us construct {horizontal} and {vertical} distributions on $T\XXX_n$ and $T^1\XXX_n$. This construction imitates the more standard construction in the Riemannian setting. 
\vspace{5pt}

Given $( z,{ v})\in T\XXX_n$, the \emph{vertical subspace} at $( z, v)$ is defined as:
$$\V_{( z,{ v})}=T_{( z,{ v})} \big(\bar\pi^{-1}( z)\big)\cong T_{ z}\XXX_n.$$

Hence given a vector $ w\in T_{ z}\XXX_n$ orthogonal to ${ v}$, we can define its \emph{vertical lift}
$ w^\V\in \V_{({ z},{ v})}$, and 
vertical lifting gives a map from ${ v}^\perp$ to $\V_{({ z},{ v})}< T_{({ z},{ v})}T\XXX_n$ which is simply the identity map under the above identification. More concretely, in the model for $T_{({ z},{ v})}T\XXX_n$ defined in \eqref{eq:modelTT},
$$ w^\V=(0, w)\in \C^{n+1}\times \C^{n+1}~.$$

For $T^1\XXX_n$, the induced vertical distribution is defined by
\begin{equation}
	\label{eq: vertical lift}
	\VP_{( z,{ v})}=T_{( z,{ v})} \big(\pi^{-1}( z)\big)\cong  v^\bot < T_{ z}\XXX_n.
\end{equation}
with the vertical lifting defined as for $T\XXX_n$.

\vspace{5pt}
Let us now define the horizontal lifting. Given $ u\in T_{ z}\XXX_n$, let us consider the parameterized geodesic $\gamma:\R\to\XXX_n$ with $\gamma(0)={ z}$ and $\gamma'(0)= u$, and let ${v}(t)$ be the parallel transport of ${ v}$ along $\gamma$. 
Then ${ u}^\HH$ is defined as the derivative of {$(\gamma(t),{ v}(t))$} at time $t=0$; notice that if $ v\in T^1_{z}\XXX_n$, then $v(t)$ has norm $1$ too, and the curve $(\gamma(t),v(t))$ lies in $T^1\XXX_n$.

Let us compute the horizontal lift on $T\XXX_n$ in the hyperboloid model by distinguishing two different cases.

First, let us consider the case of ${ u}\in { v}^\perp\subset T_{ z}\XXX_n$. In the model \eqref{eq:modelTT}, using that the image of the parameterized geodesic $\gamma$ is the intersection of $\XXX_n$ with a plane in $\C^{n+1}$ {orthogonal to ${ v}$}, the parallel transport of ${ v}$ along $\gamma$ is the  vector field constantly equal to ${ v}$, and therefore
\begin{equation}
	\label{eq: horizontal lift}
	{ u}^\HH=\ddt(\gamma(t),{ v})=({ u},0)\in \C^{n+1}\times \C^{n+1}~.
\end{equation}
We shall denote by $\HP_{({ z},{ v})}$ the subspace of horizontal lifts of this form, which is therefore a horizontal subspace in $T_{({ z},{ v})}T^1\XXX_n$ isomorphic to ${ v}^\perp$.

This gives an injective linear map from $T_{ z}\XXX_n$ to $T_{({ z},{ v})}T^1\XXX_n$ that we define as the \emph{horizontal lift} and whose image is the \emph{horizontal subspace} $\HH_{({ z},{ v})}< T_{({ z},{ v})} T\XXX_n$, with \[\HH_{({ z},{ v})}< T_{({ z},{ v})} T^1\XXX_n<T_{({ z},{ v})} T\XXX_n\]
if $ v$ has norm $1$.

There remains to understand the case of ${ u}={ v}$. 

\begin{Lemma} \label{lemma:generator geoflow hor lift}
	Given $({ z},{ v})\in T\XXX_n$, the horizontal lift ${ v}^\HH$ coincides with the infinitesimal generator $\chi_{({ z},{ v})}$ of the geodesic flow, and has the expression:
	$$\chi_{({ z},{ v})}=({ v},{ z})\in\C^{n+1}\times\C^{n+1}~.$$

\end{Lemma}
\begin{proof}
	Since the tangent vector to a parameterized geodesic is parallel along the geodesic itself, $\varphi_t({ z},{ v})$ also equals {$(\gamma(t),{ v}(t))$, for ${ v}(t)$ the vector field} used to define the horizontal lift. 
	Hence clearly 
	$${ v}^\HH=\chi_{({ z},{ v})}=\ddt \varphi_t({ z},{ v})~.$$ 
	Differentiating Equation \eqref{eq:geodflow} at $t=0$ we obtain the desired expression.
\end{proof}

In summary,
\begin{equation}
	\label{eq:direct sum}
	\begin{split}
		T_{({ z},{ v})}T^1\XXX_n&=\HH_{({ z},{ v})}\oplus\VP_{({ z},{ v})}=\HP_{({ z},{ v})}\oplus\mathrm{Span}_{\C}(\chi_{({ z},{ v})})\oplus\VP_{({ z},{ v})}~\\
		T_{({ z},{ v})}T\XXX_n&=\HH_{({ z},{ v})}\oplus\V_{({ z},{ v})}=\HP_{({ z},{ v})}\oplus\mathrm{Span}_{\C}(\chi_{({ z},{ v})})\oplus\VP_{({ z},{ v})}\oplus \mathrm{Span}_\C ((0,  v))~.
	\end{split}
\end{equation}

{We are now able to} introduce the para-Sasaki metric on the unit tangent bundle.

\begin{Def}\label{Def:parasasaki}
	Define the \emph{para-Sasaki metric} on $T\XXX_n$ as the metric $\gs n$ defined by
	$$\gs{n}(V_1,V_2)=\begin{cases} 
		+\langle { u}_1,{ u}_2\rangle_0 & \text{if }V_1,V_2\in\HH_{({ z},{ v})}\text{ and }V_i={ u}_i^\HH \\
		-\langle { w}_1,{ w}_2\rangle_0 & \text{if }V_1,V_2\in\V_{({ z},{ v})}\text{ and }V_i={ w}_i^\V \\
		0 & \text{if }V_1\in\HH_{({ z},{ v})}\text{ and }V_2\in\V_{({ z},{ v})}~.
	\end{cases}$$
\end{Def}

Clearly the metric  $\gs n$ is non-degenerate. Moreover it is holomorphic  Riemannian: if $(v_1,\dots, v_n)$ is a holomorphic orthonormal frame for $U\subset\XXX_n$, then $(v_1^\mathcal H,\dots, v_n^\mathcal H, v_1^\mathcal V, \dots, v_n^\mathcal V)$ is a local orthonormal frame in $\bar \pi^{-1}(U)$, and, by the explicit description we have computed in the model \eqref{eq:modelTT}, the elements of this frame are local holomorphic vector fields on $T\XXX_n$. 
Moreover, the restriction of $\gs n$ to $T^1\XXX_n$, that we will denote with $\gs n$, is a holomorphic Riemannian metric since it is non degenerate.

It is also worth observing that, from Definition \ref{Def:parasasaki} and Lemma \ref{lemma:generator geoflow hor lift}, that
\begin{equation}\label{eq:generator geoflow is unit}
	\gs{n}(\chi_{({ z},{ v})},\chi_{({ z},{ v})})=1~,
\end{equation}
and that $\chi_{({ z},{ v})}$ is orthogonal to both $\VP_{({ z},{ v})}$ and $\HP_{({ z},{ v})}$.

Clearly the obvious action of the isometry group $\Isom(\XXX_n)$ on $T^1\XXX_n$ preserves the para-Sasaki metric, since all the ingredients involved in the definition are invariant by isometries. The same is also true for the action of the complex geodesic flow, and this fact is much more peculiar of the choice we made in Definition \ref{Def:parasasaki}.

\begin{Lemma}\label{lemma:geodflow isometric}
	The $\C$-action of the complex geodesic flow on $T^1\XXX_n$ is isometric for the para-Sasaki metric, and commutes with the action of $\Isom(\XXX_n)$.
\end{Lemma}
\begin{proof}
	Let us first consider the differential of $\varphi_\lambda$, for a given $\lambda\in\C$. Since the expression for $\varphi_\lambda$ from Equation \eqref{eq:geodflow} is linear in ${ z}$ and ${ v}$, we have: 
	\begin{equation}\label{eq:geodflow linearized}
		d\varphi_\lambda(\dot { z},\dot { v})=(\cosh(\lambda)\dot { z}+\sinh(\lambda)\dot { v},\sinh(\lambda)\dot { z}+\cosh(\lambda)\dot { v})~,
	\end{equation}
	for $X=(\dot { z},\dot { v})$ as in \eqref{eq:modelTT}.
	Let us distinguish three cases. 
	
	If $X={ w}^\HH=({ w},0)$ for ${ w}\in { v}^\perp\subset T_{ z}\XXX_n$, then
	\begin{equation}\label{eq:diff geoflow1}
		d\varphi_\lambda({ w}^\HH)=(\cosh(\lambda){ w},\sinh(\lambda){ w})=\cosh(\lambda){ w}^\HH+\sinh(\lambda){ w}^\V~.
	\end{equation}
	For $X={ w}^\V=(0,{ w})$ a completely analogous computation gives
	\begin{equation}\label{eq:diff geoflow2}
		d\varphi_\lambda({ w}^\V)=(\sinh(\lambda){ w},\cosh(\lambda){ w})=\sinh(\lambda){ w}^\HH+\cosh(\lambda){ w}^\V~.
	\end{equation}
	Finally, for $X=\chi_{({ z},{ v})}$, by construction 
	\begin{equation}\label{eq:diff geoflow3}
		d\varphi_\lambda(\chi_{({ z},{ v})})=\chi_{\varphi_\lambda({ z},{ v})}~.
	\end{equation}
	Now using \eqref{eq:diff geoflow1} and \eqref{eq:diff geoflow2}, and Definition \ref{Def:parasasaki}, we can check that that 
	$$\gs{n}(d\varphi_\lambda({ w}_1^\HH),d\varphi_\lambda({ w}_2^\HH))=(\cosh^2(\lambda)-\sinh^2(\lambda))\langle { w}_1,{ w}_2\rangle=\langle { w}_1,{ w}_2\rangle=\gs{n}({ w}_1^\HH,{ w}_2^\HH)~.$$
	A completely analogous computation shows 
	$$\gs{n}(d\varphi_\lambda({ w}_1^\V),d\varphi_\lambda({ w}_2^\V))=-\langle { w}_1,{ w}_2\rangle=\gs{n}({ w}_1^\V,{ w}_2^\V)~.$$ Also, one checks that
	$\gs{n}(d\varphi_\lambda({ w}_1^\HH),d\varphi_\lambda({ w}_2^\V))=0=\gs{n}({ w}_1^\HH,{ w}_2^\V)$.
	By  \eqref{eq:generator geoflow is unit} and \eqref{eq:diff geoflow3}, the norm of vectors proportional to $\chi_{({ z},{ v})}$ is preserved. Together with \eqref{eq:diff geoflow1} and \eqref{eq:diff geoflow2}, vectors of the form $d\varphi_\lambda({ w}^\HH)$ and $d\varphi_\lambda({ w}^\V)$ are orthogonal to $d\varphi_\lambda(\chi_{({ z},{ v})})=\chi_{\varphi_\lambda({ z},{ v})}$.
	This concludes the first part of the statement.
	
	Finally, since isometries map parameterized geodesics to parameterized geodesics, it is straightworward to see that the $\C$-action commutes with $\Isom(\XXX_n)$.
\end{proof}

\subsection{The para-Sasaki metric on $T^1\Hyp^n$}
As we anticipated, the discussion above was only meant to present the geodesic flow and the para-Sasaki metric on $T^1\Hyp^n$ in the most general setting: indeed, all the constructions and the staments above for $\XXX_n$ can be repeated for the unitary tangent space of any space form of constant curvature $-1$ via the fact that $\mathbb F_{-1}^{m,n-m}$ isometrically immerges into $\XXX_n$.

The horizontal and vertical decomposition $T T^1\mathbb F_{-1}^{m,n-m}= \HH \oplus \VP$ depends uniquely on the Levi-Civita connection on $\mathbb F_{-1}^{m,n-m}$ and is defined as for $T^1\XXX_n$. Similarly we can define the \emph{para-Sasaki metric} on $T^1\mathbb F_{-1}^{m,n-m}$ as 	$$\gs{m,n-m}(V_1,V_2)=\begin{cases} 
	+\langle { u}_1,{ u}_2\rangle_{m,n-m} & \text{if }V_1,V_2\in\HH_{({ z},{ v})}\text{ and }V_i={ u}_i^\HH \\
	-\langle { w}_1,{ w}_2\rangle_{m,n-m} & \text{if }V_1,V_2\in\V_{({ z},{ v})}\text{ and }V_i={ w}_i^\V \\
	0 & \text{if }V_1\in\HH_{({ z},{ v})}\text{ and }V_2\in\V_{({ z},{ v})}~
\end{cases}.$$

One can immediately check that the totally geodesic isometric immersion $\mathbb F_{-1}^{m,n-m}\to \XXX_n$, defined as in Equation \eqref{eq: quadrica pseudo-Riem dentro X_n},  induces an isometric immersion $T^1 \mathbb F_{-1}^{m,n-m}\to T^1\XXX_n$ equipped with the respective para-Sasaki metrics. Also, this immersion is equivariant with respect to the geodesic flow on $T^1 \mathbb F_{-1}^{m,n-m}$ (defined wordly as in Equation \eqref{eq: geodesic flow}) and the one on $T^1\XXX_n$.

As a consequence of Lemma \ref{lemma:geodflow isometric} one also has the following.

\begin{Prop}
	The geodesic flow acts by isometries for $T^1 \mathbb F_{-1}^{m,n-m}$ with the para-Sasaki metric. 
	
	In particular, it is true for $T^1\Hyp^n$, $T^1 AdS^n$, and $T^{-1} dS^n$.
\end{Prop}

Finally, let us summarize the description of horizontal and vertical lifts for $T^1\Hyp^n$ in the hyperboloid model. 

Consider the hyperboloid model \[\Hyp^n=\{ x=(x_1,\dots, x_{n+1})\in \R^{n,1} |\ \inner{ x,  x}_{n,1}=-1 ,\ x_{n+1}>0 \}\subset \R^{n,1}.\] Then \[T^1\Hyp^n =\{( x,  v)\in \R^{n,1}\times \R^{n,1}\ |\ \inner{ x, x}_{n,1}=-1,\ \inner{ x, v}_{n,1}=0,\ \inner{ v, v}_{n,1}=1 \},\]
and
\[
T_{( x,  v)} T^1\Hyp^n= \{(\dot{ x},\dot{  v})\ |\ \inner{ x,\dot{ x}}_{n,1}= \inner{ v,\dot{ v}}_{n,1}= \inner{ x,\dot{ v}}_{n,1}+ \inner{\dot{ x},  v}_{n,1}=0  \}.
\]
In the hyperboloid model, for all $ w\in  v^\bot <T_{ x}\Hyp^n$, the vertical and horizontal lifts are respectively
\begin{align*}
	 w^\V &=(0, w)\in \V_{( x, v)},\\
	 w^\HH &=( w, 0)\in \HH_{( x,  v)}.
\end{align*}
while 
\[
 v^\HH= \chi_{( x,  v)}= ( v,  x)\in \HH_{( x,  v)},
\]
and the para-Sasaki metric of $T^1\Hyp^n$, that for simplicity we will denote by $\gs n$, is
\begin{equation}
	\label{eq: Sas di T1Hn}
\gs n ( u_1^\HH+  w_1 ^\V,  u_2^\HH +  w_2^\V)= \inner{ u_1,  u_2}-\inner{ w_1, w_2}.
\end{equation}

\section{A para-K\"ahler structure on $\G n$}\label{sec:parakahler metric GG}

The aim of this section is to introduce a para-K\"ahler structure on the space $\G n= \partial \Hyp^n\times \partial \Hyp^n \setminus\Delta$ of the geodesics of $\Hyp^n$.

Before that, let us start by quickly recalling the basic definitions of para-complex and para-K\"ahler geometry. First introduced by Libermann in \cite{libermann}, the reader can refer to the survey \cite{zbMATH00894673} for more details on para-complex geometry. 
\vspace{5pt}
Given a manifold $\mathcal M$ of dimension $2n$,  an \emph{almost para-complex structure} on $\mathcal M$ is a tensor $\JJ$ of type $(1,1)$ (that is, a smooth section of the bundle of endomorphisms of $T\mathcal M$) such that $\JJ^2=\mathbbm 1$ and that at every point $p\in \mathcal M$ the eigenspaces {$T_p^\pm \mathcal M=\ker(\JJ\mp \mathbbm 1)$} have dimension $n$. The almost para-complex structure $\JJ$ is a \emph{para-complex structure} if the distributions $T_p^\pm \mathcal M$ are integrable. 

A \emph{para-K\"ahler structure} on $\mathcal M$ is the datum of a para-complex structure $\JJ$ and a pseudo-Riemannian metric $\GG$ such that $\JJ$ is $\GG$-skewsymmetric, namely
\begin{equation}\label{eq:JGcompatible}
	\mathbbm g(\JJ X,Y)=-\mathbbm g(X,\JJ Y)
\end{equation}
for every $X$ and $Y$, and such that the corresponding \emph{fundamental form}, namely
\begin{equation}
	\label{eq: Definizione Omega}
	\Omega(X,Y):=\mathbbm g(X,\JJ Y),
\end{equation} 
is closed. 

Observe that Equation \eqref{eq:JGcompatible} is equivalent to the condition that $\JJ$ is anti-isometric for $\mathbbm g$, namely:
\begin{equation}
	\label{eq:JGantiisometry}
	\mathbbm{g}(\JJ X,\JJ Y)=-\mathbbm{g}(X,Y)
\end{equation}
which implies immediately that the metric $\mathbbm g$ is necessarily neutral (that is, its signature is $(n,n)$) and that $\Omega$ is a 2-form, hence a symplectic form on $\mathcal M$.

\vspace{10pt}

In order to introduce the para-K\"ahler structure on the space of geodesics $\G{n}$, consider the map
\[
\mathrm p\colon T^1 \Hyp^{n} \to \G {n}
\]
that sends $( x, v)$ to the geodesic passing by $ x$ and tangent to $ v$. Observe that $\mathrm p$ is a $\R$-principal bundle via the action of the geodesic flow.

The para-K\"ahler structure on $\G n$ will be defined by introducing some structures on $T^1\Hyp^{n}$ and showing that they can be pushed forward to $\G{n}$.
Clearly the group $\Isom(\Hyp^{n})$ acts both on $T^1\Hyp^{n}$ and on $\G {n}$, in the obvious way, and moreover the two projection maps $\pi:T^1\Hyp^{n}\to\Hyp^{n}$ and $\mathrm p:T^1\Hyp^{n}\to\G {n}$ are equivariant with respect to these actions.

The kernel of the differential map $d\mathrm p$ is generated by the infinitesimal generator of the geodesic flow. As a consequence, given a point $(\xx,\vv)\in T^1\Hyp^{n}$, the decomposition \eqref{eq:direct sum} 
shows that the tangent space $T_\ell \G{n}$ identifies to $\chi_{(\xx,\vv)}^\perp=\HP_{(\xx,\vv)}\oplus\VP_{(\xx,\vv)}$, where $\ell\in \G{n}$ is the oriented unparameterized geodesic going through $\xx$ with speed $\vv$, and the orthogonal subspace is taken with respect to the para-Sasaki metric $\gs n$: in other words, the differential of $\mathrm p$ induces a vector space isomorphism 
\[d\mathrm p|_{\chi^\perp_{(x,v)} }\colon \chi_{(\xx,\vv)}^\perp \to T_\ell \G{n}.\]

Now, let us define ${\mathrm J}\in\mathrm{End}(\chi_{(\xx,\vv)}^\perp)$ by the following expression:
$${\mathrm J}(\dot \xx,\dot \vv)=(\dot \vv,\dot \xx)~.$$
In other words, recalling that $\HP_{(\xx,\vv)}$ consists of the vectors of the form $(\ww, 0)$, and $\VP_{(\xx,\vv)}$ of those of the form $( 0,\ww)$, for $\ww\in \vv^\perp$, ${\mathrm J}$ is defined by
\begin{equation}\label{eq:defiJ2}
	{\mathrm J}(\ww^\HH)=\ww^\V\qquad \text{and}\qquad {\mathrm J}(\ww^\V)=\ww^\HH~.
\end{equation}

\begin{Lemma}\label{lemma:paracomplex structure}
	The endomorphism ${\mathrm J}$ induces an almost para-complex structure $\JJ$ on $T_\ell \G{n}$, which does not depend on the choice of $(x,v)\in \mathrm p^{-1}(\ell)$.
\end{Lemma}
\begin{proof}
	By definition of the $\R$-principal bundle structure $\mathrm p:T^1\Hyp^{n}\to \G{n}$ and of the geodesic flow $\varphi_t$, $\mathrm p\circ \varphi_t=\mathrm p$ for every $t\in\R$. Moreover $\varphi_t$ preserves the infinitesimal generator $\chi$ (Equation \eqref{eq:diff geoflow3}) and acts isometrically on $T^1\Hyp^{n}$ by Lemma \ref{lemma:geodflow isometric}, hence it preserves the orthogonal complement of $\chi$.
	Therefore, given vectors $X,Y\in T_\ell \G{n}$, any two lifts of $X$ and $Y$ on $T^1\Hyp^{n}$ orthogonal to $\mathrm p^{-1}(\ell)$ differ by push-forward by $\varphi_t$.
	
	However, it is important to stress that the differential of $\varphi_t$ does \emph{not} preserve the distributions $\HP$ and $\VP$ {individually}  (see Equations \eqref{eq:diff geoflow2} and \eqref{eq:diff geoflow3}).
	Nevertheless, by Equation \eqref{eq:geodflow linearized}:
	\begin{align*}(\varphi_t)_*({\mathrm J}(\dot \xx,\dot \vv))=(\varphi_t)_*(\dot \vv,\dot \xx)&=(\cosh(t)\dot \vv+\sinh(t)\dot \xx,\sinh(t)\dot \vv+\cosh(t)\dot \xx) \\
		&={\mathrm J}(\sinh(t)\dot \vv+\cosh(t)\dot \xx,\cosh(t)\dot \vv+\sinh(t)\dot \xx)={\mathrm J}(\varphi_t)_*(\dot \xx,\dot \vv)
		~,
	\end{align*}
	which shows that $\JJ$ is well-defined on $T_\ell \G{n}$. It is clear from the expression of $J$ that $\JJ^2=\mathbbm 1$, and moreover that the $\pm 1$-eigenspaces of $\JJ$ both have dimension $n$, since the eigenspaces of ${\mathrm J}$ consist precisely of the vectors of the form $(\ww,\ww)$ (resp. $(\ww,-\ww)$) for $\ww\in v^\perp\subset T_\xx\Hyp^{n+1}$.
\end{proof}

Let us now turn our attention to the construction of the neutral metric $\GG$, which will be defined by a similar construction. In fact, given $(\xx,\vv)\in \mathrm p^{-1}(\ell)$, we simply define $\m$ on $T_\ell \G{n}$ as the push-forward of the restriction $\gs{n}|_{\chi^\perp_{(\xx,\vv)}}$ by the isomorphism $$d\mathrm p|_{\chi^\perp_{(\xx,\vv)}}:\chi^\perp_{(\xx,\vv)}\to T_\ell \G{n}~.$$
Well-posedness of this definition follows immediately from Equation \eqref{eq:diff geoflow3} and Lemma \ref{lemma:geodflow isometric}.

\begin{Lemma}
	The restriction of $\gs{n}$ to $\chi^\perp_{(\xx,\vv)}$ induces a neutral metric $\GG$ on $T_\ell \G{n}$, which does not depend on the choice of $(\xx,\vv)\in \mathrm p^{-1}(\ell)$, such that $\JJ$ is $\GG$-skewsymmetric.
\end{Lemma}
\begin{proof}
	It only remains to show $\GG$-skewsymmetry, namely Equation \eqref{eq:JGcompatible}. The latter is indeed equivalent to Equation \eqref{eq:JGantiisometry}, which is immediate from Definition \ref{Def:parasasaki} and the definition of ${\mathrm J}$, see in particular \eqref{eq:defiJ2}. 
\end{proof}

There is something left to prove in order to conclude that the constructions of $\JJ$, $\GG$, and $\Omega:=\GG(\cdot, \JJ\cdot)$ induce a para-K\"ahler structure on $\G{n}$ but we defer the remaining checks to the following sections: in particular, we are left to prove that the almost para-complex structure $\JJ$ is integrable (it will be a consequence of Example \ref{ex:horospheres}) and that the  2-form $\Omega$ is closed (which is the content of  Corollary \ref{cor:omega closed}).

{
	\begin{Remark}
		\label{rmk: Isom preserva Omega e J}
		The group $\Isom(\Hyp^{n})$ acts naturally on $\G{n}$ and the map $\mathrm p\colon T^1 \Hyp^n\to \G{n}$ is equivariant, namely $\mathrm p (\psi\cdot (\xx,\vv))= \psi \cdot \mathrm p(\xx,\vv)$ for all $\psi\in \Isom(\Hyp^n)$. As a result, by construction of $\GG$ and $\JJ$, the action of $\Isom(\Hyp^n)$ on $\G{n}$ preserves both $\GG$ and $\JJ$.
\end{Remark}}

\begin{Remark}\label{rmk other metric1}
	Of course some choices have been made in the above construction, in particular in the expression of the para-Sasaki metric of Definition \ref{Def:parasasaki}, which has a fundamental role when introducing the metric $\GG$. The essential properties we used are the naturality with respect to the isometry group of $\Hyp^{n}$ and to the action of the geodesic flow (Lemma \ref{lemma:geodflow isometric}).
	
	Some alternative definitions for $\gs{n}$ would produce the same expression for $\GG$.
	For instance one can define for all $c\in \R^+$ a metric ${\overline{Sas}_{n} ^{\ *c}}$ on $T^1 \Hyp^{n}$ so that, with respect to the direct sum decomposition \eqref{eq:direct sum}:
	\begin{itemize}
		\item ${\overline{Sas}_{n} ^{\ *c}}(\ww_1^\HH,\ww_2^\HH)=-\gs n (\ww_1^\V,\ww_2^\V)=\langle \ww_1,\ww_2\rangle$ for any $\ww_1,\ww_2\in v^\perp\subset T_\xx\Hyp^{n}$,
		\item ${\overline{Sas}_{n} ^{\ *c}} (\chi_{(\xx,\vv)},\chi_{(\xx,\vv)})=c$,
		\item $\mathrm{Span}(\chi_{(\xx,\vv)})$, $\HP_{(\xx,\vv)}$ and $\VP_{(\xx,\vv)}$ are mutually ${\overline{Sas}_{n} ^{\ *c}}$-orthogonal.
	\end{itemize}
	
	Replacing $\gs{n}$ with such a ${\overline{Sas}_{n} ^{\ *c}}$, one would clearly obtain the same metric $\GG$ since it only depends on the restriction of ${\overline{Sas}_{n} ^{\ *c}}$ to the orthogonal complement of $\chi$.
	{Moreover, ${\overline{Sas}_{n} ^{\ *c}}$ is invariant under the action of $\Isom(\Hyp^n)$ and under the geodesic flow.}
	
\end{Remark}
\begin{Remark}\label{rmk other metric2}
	It will  be convenient to use Remark \ref{rmk other metric1} in the following, by considering $T^1\Hyp^{n}$ as a submanifold of $\R^{n,1}\times \R^{n,1}$, and taking the metric given by the Minkowski product on the first factor, and its opposite on the second factor, restricted to $T^1\Hyp^{n}$, i.e.
	\begin{equation}\label{eq:metric minkxmink}
		\gstar n ((\dot \xx_1,\dot \vv_1),(\dot \xx_2,\dot \vv_2))=\langle \dot \xx_1,\dot \xx_2\rangle-\langle \dot \vv_1,\dot \vv_2\rangle~.
	\end{equation}
	In fact, it is immediate to check that $\gstar n(\ww_1^\HH,\ww_2^\HH)=\gstar n ((\ww_1,0),(\ww_2,0))=\langle \ww_1,\ww_2\rangle$ for $\ww_i\in \vv^\perp$, that similarly $\gstar n(\ww_1^\V,\ww_2^\V)=-\langle \ww_1,\ww_2\rangle$, and that 
	$$\gstar n (\chi_{(\xx,\vv)},\chi_{(\xx,\vv)})=\gstar n ((\vv,\xx),(\vv,\xx))=\langle \vv,\vv\rangle-\langle \xx,\xx\rangle =2~.$$
	Finally elements of the three types are mutually orthogonal, and therefore
	$\gstar n ={\overline{Sas}_{n} ^{\ *2}}$ with ${\overline{Sas}_{n} ^{\ *2}}$ as in Remark \ref{rmk other metric1}.
\end{Remark}

\section{The space $\G 3$ as $\mathbb X_2$}
\label{section: G3 e hRm}
For $n=3$, the space $\G 3$ of the geodesics of $\Hyp^3$ also has a structure of holomorphic Riemannian manifold: in fact, it is isometric to $\XXX_2$. 

From now on we will denote $\GGG= \G 3$

Let us first observe that $\GGG$ has a natural complex structure. Indeed, $\partial \Hyp^3$ has a natural complex structure for which it is biholomorphic to $\CP^1$: this structure coincides with the standard complex structure on $S^2$ when you see $\Hyp^3$ in the disk model, or equivalently to the one of $\overline \C$ seen as the boundary of $\Hyp^3$ in the half-space model $\C\times \R^+$. The action of the group $\PSL\cong \Isom_0(\Hyp^3)$ on $\Hyp^3$ extends to an action of $\partial \Hyp^3$ which coincides with the action via M\"obius maps in the identification $\partial \Hyp^3\cong \overline \C$, hence it is by biholomorphisms.

As a result, the space $\GGG = \partial \Hyp^3 \times \partial \Hyp^3\setminus \Delta \cong \overline \C\times \overline \C \setminus \Delta$ inherits a natural complex structure. The natural action of $\PSL\cong \Isom_0(\Hyp^3)$ on $\GGG$ is diagonal and corresponds to the M\"obius maps action on each component, namely
\begin{align*}
	\PSL\times \GGG &\to \GGG\\
	A \cdot (z_1, z_2)&= (A\cdot z_1, A\cdot z_2).
\end{align*}

\begin{Theorem}
	\label{teo: metrica su G}
	The space $\GGG$ admits a complete holomorphic Riemannian metric $\inners_\GGG$ of constant curvature $-1$ such that $\Isom_0(\inners_\GG)\cong\PSL$ under the natural action. As a consequence, one has $(\GGG, \inners_\GGG) \cong (\XXX_2, \inners)$.
	
	Explicitly, if $(U,z)$ is an affine chart for $\CP^1$, then in the chart $(U\times U\setminus \Delta, z\times z)$ for $\GGG$ the metric is described by  \begin{equation}
		\label{eq: metrica su G}
		\inners_\GGG= -\frac{4}{(z_1-z_2)^2} dz_1 \cdot dz_2.
	\end{equation}
	Finally, $\inners_\GGG$ is the unique hRm on $\GGG$ with the property of being $\PSL$-invariant and having constant curvature $-1$.
\end{Theorem}	

\begin{proof}
	Let us start with showing that the metric locally defined in Equation \eqref{eq: metrica su G} does not depend on the choice of the affine chart. Clearly, since \eqref{eq: metrica su G} gives a local hRm, independence on the chart would prove that this local description defines a global hRm on $\GGG$.
	
	Assume $(U,z)$ and $(U',z')$ be two overlapping affine charts for $\CP^1$, then $\psi=z' \circ z^{-1}$ (defined on $z(U\cap U')$) is a M\"obius map.
	As a result, in the local charts $(U\times U\setminus \Delta,z\times z)$ and $(U'\times U'\setminus \Delta, z'\times z')$ for $\GGG$, denoting $\psi(w)=\frac{aw+b}{cw+d}$ with $ad-bc=1$, we compute that
	\begin{align*}
		-\frac{4}{(z'_1-z'_2)^2} dz'_1 \cdot dz'_2&= -\frac{4}{(\psi(z_1)-\psi(z_2))^2} d(\psi\circ z_1) \cdot d(\psi \circ z_2)=\\
		&= -\frac{4}{\big(\frac{az_1+b}{cz_1+d}-\frac{az_2+b}{cz_2+d} \big)^2} \frac 1 {(cz_1+d)^2} \frac 1 {(cz_2 +d)^2} dz_1 \cdot dz_2=\\
		&= -\frac 4 {\big((az_1+b)(cz_2+d)-(az_2+b)(cz_1+d)\big)^2} dz_1 \cdot dz_2=\\
		&=-\frac{4}{(z_1-z_2)^2} dz_1 \cdot dz_2
	\end{align*}
	proving that Equation \eqref{eq: metrica su G} does not depend on the affine chart: as a result it defines a well posed hRm on $\GGG$.
	
	Let us prove that $\inners_\GGG$ is $\PSL$-invariant. Let $\ell\in \GGG$ and $\psi\in \PSL$. Let $(U,z)$ be an affine chart on $\CP^1$ such that the endpoints of $\ell$ and of $\psi(\ell)$ lie on $U$. By definition of the action of $\PSL$ on $\GGG$, locally around $\ell$ one has that $(\psi\times \psi)\circ(z\times z)= (z\times z)\circ \psi$. Using this and recalling that $(\psi\times \psi)\circ(z\times z)$ is a local affine coordinate around $\ell$, we get (here $\inners=\inners_\GGG$):
	\begin{align*}
		\psi^* (\inners|_{\psi(\ell)})&= \psi^* \bigg( (- \frac 4 {(z_1-z_2)^2} dz_1 \cdot dz_2)\bigg)|_{\psi(\ell)}=\\
		&= -\frac{4}{(\psi(z_1) - \psi(z_2))^2} d(\psi\circ z_1)\cdot d(\psi\circ z_2) |_{\ell}= \\
		&=\inners|_\ell,
	\end{align*}
	and the $\PSL$ invariancy follows.
	
	Since $\PSL$ acts transitively on $\GGG$, the Levi-Civita connection of $\inners_\GGG$ endows $\GGG$ with a structure of homogeneous space, hence the connection is geodesically complete (See \cite{Kobayashi-Nomizu2}). We conclude that $(\GGG, \inners_\GGG)$ is a holomorphic Riemannian space form.

	We finally show that $\inners_\GGG$ has constant curvature $-1$. Clearly, it has constant sectional curvature since $\PSL$ acts transitively.
	Emulating our proof of the fact that $\mathbb{X}_2\subset \C^3$ has constant curvature $-1$, it is sufficient to show that there exists an isometric embedding $\Hy^2\to (\GGG,\inners_\GGG)$ which is totally geodesic.

	Consider the immersion
	\begin{align*}
		\sigma \colon H&\to \GGG \\ 
		z &\mapsto (z,\overline z)
	\end{align*}
	where $H=\{z\in \C\ |\ Im(z)>0\}$ is the upper half plane.
	
	By Equation \eqref{eq: metrica su G}, the pull-back metric on $H$ is
	\[
	- \frac 4 {- 4 (Im(z))^2} dz\cdot d\overline z=\frac 1 {(Im(z))^2} dz d\overline z,
	\]
	so $\sigma$ is an isometric immersion of the hyperbolic plane $\Hy^2$ in its upper half plane model.
	
	Moreover, observe that $\sigma(H)$ is the fixed locus of the involution $(z, w)\mapsto (\overline w, \overline z)$. A direct computation shows that this involution is an isometry for the pseudo-Riemannian metric given by $Re\big( \frac{1}{(z_1-z_2)^2} dz_1 dz_2\big)$, hence it is an isomorphism for the induced Levi-Civita connection which coincides with the one induced by $\inners_\GGG$ by Proposition \ref{prop: Levi Civita hRm}: we conclude that $\sigma$ is totally geodesic.
	
	Being $\GGG$ simply connected and being $\inners_{\GGG}$ complete of constant curvature $-1$, by Theorem \ref{Theorem space forms}, $(\GGG,\inners_{\GGG})$ is isometric to $\XXX_2$.

	Since $\Isom_0(\GGG,\inners_\GGG)\cong \Isom_0(\XXX_2)\cong \SO(3,\C)$ and since $\PSL$ acts faithfully on $\GGG$, the induced map $\PSL\to \Isom_0(\GGG, \inners_\GGG)$ is an injective homomorphism of Lie groups of the same dimension, hence an isomorphism.
	
	In conclusion, let us prove that $\inners_\GGG$ is the unique hRm on $\GGG$ being $\PSL$-invariant and with constant curvature $-1$. Assume $\tau$ is another such hRm on $\GGG$, then $\tau$ coincides with $\inners_\GGG$ if and only if they coincide in one point since $\PSL$ acts transitively. 
	
	Now, observe that if a bilinear form on $\C^2$ is $\SO(2,\CC)$-invariant, then, since every element of $\SO(2,\CC)$ has eigenspaces $\{(z,z)\ |\ z\in \C\}$ and $\{(z,-z)\ |\ z\in \C\}$ with eigenvalues generally different from $\pm 1$, one deduces that its eigenspaces coincide with the isotropic directions for the metric. With the same argument applied to $(T_\ell \GGG, \inners_\GGG|_\ell)$ (for any $\ell$) and to its stabilizer $Stab(\ell)\cong SO(2,\CC)$, one gets that $\inners_\GGG$ and $\tau$ have the same isotropic direction in $T_\ell \GGG$, therefore $\tau=\lambda \cdot \inners_\GGG$ for some $\lambda\in \C^*$, but then $\lambda=1$ because by ipothesis $\inners_\GGG$ and $\tau$ have the same curvature.
\end{proof}

\begin{Notation*}
	With a little abuse, in the following we will use $\GGG$ to denote the space $\G 3$ equipped with the hRm $\inners_\GGG$. 
\end{Notation*}

We anticipate that we will prove later on in this thesis (Proposition {prop: pull-back Re hRm}) that $\GG_3= Re(\inners_\GGG)$. 
\vspace{10pt}

It might turn out interesting to produce an explicit isometry between $\GGG$ and $\XXX_2$. Hereafter, we present two explicit isometries which carry some geometric meaning.
\vspace{8pt}

\emph{Model} 1. 

Recall that we defined an isometry $F\colon(\C^4, \inners_0)\to(Mat(2,\C), \inners_{Mat_2})$ by 
\[F(z_1, z_2, z_3, z_4)=
\begin{pmatrix}
	-z_1- iz_4 &  -z_2-iz_3\\
	-z_2 + iz_3 & z_1-iz_4
\end{pmatrix}.\]

As we previously observed, it restricts to an isometry $F\colon \mathbb X_3 \to \SL$. 

On the other hand, by regarding $\mathbb X_2$ as $\{(z_1, z_2, z_3,0)\in \mathbb X_3\}$, we have that $F$ restricts to an isometry \[F\colon \mathbb X_2 \to\asl \cap \SL=\{M\in Mat(2,\C)\ |\ det(M)=1,\ tr(M)=0 \}.\] 
As we already observed in Equation \eqref{formula3}, for all $M\in \asl \cap \SL$ one has that $M^2=I_2$, clearly with $M\ne -I_2$: as a result, with respect to the standard action of $\SL$ on $\Hy^3$, we get that $\asl \cap \SL$ corresponds to orientation-preserving isometries of order $2$, i.e. to rotations of angle $\pi$ around some axis in $\Hy^3$. We therefore have a 2-sheeted covering map 
\[
Axis\colon \SL\cap \asl\to \GGG_\sim =\nicefrac{(\CP^1\times \CP^1\setminus \Delta)}{(x,y)\sim(y,x)}
\]
analogous to the one seen in Equation \eqref{eq: mappa Fix su bar G},
that sends each matrix to the fixed unoriented geodesic of the corresponding isometry in $\Hy^3$.

Since $\XXX_2$ is simply connected, by uniqueness of the universal cover, $Axis$ lifts to some diffeomorphism \[\widetilde{Axis}\colon \SL\cap \asl \to \GGG.\] By precomposing with $F$, we get a diffeomorphism $\widetilde {Axis} \circ F \colon \mathbb X_2 \to \GGG$. Observe that explicitly $\widetilde {Axis}$ is given by:
\begin{equation}
	\label{eq: mappa Axis esplicita}
	\begin{pmatrix}
		a &b \\
		c &-a
	\end{pmatrix} \mapsto \Big(\frac{a+i} c= - \frac b {a-i}, \frac{a-i} c= - \frac b {a+i}  \Big),\footnote{This notation means that, on each component, at least one between RHS and LHS makes sense and they coincide if they both make sense.}
\end{equation}
in particular it is a biholomorphism.

Let us show that $\widetilde{Axis}$ is also an isometry of holomorphic Riemannian manifolds, i.e. that the push-forward of $\inners_{\SL}$ via $\widetilde{Axis}$ coincides with $\inners_\GGG$. Indeed, the map $\widetilde Axis$ is equivariant with respect to the action of $\SL$ via conjugation on $\SL\cap \asl$ and by M\"obius maps on $\GGG$, hence the push-forward metric is $\SL$-invariant; since the pull-back metric also has curvature $-1$, it coincides with $\inners_\GGG$ by the uniqueness part of Theorem \ref{teo: metrica su G}.

\vspace{15pt}
\emph{Model 2.}

Let us construct a second isometry $\mathbb X_2\to \GGG$.

Recall that, as we observed in \ref{rmk Super hRm}.\ref{rmk Super hRm isometries PX}, we defined $\mathbb{P X}_2\subset \CP^2$ as the complementary of the projective conic $\partial \mathbb{P X}_2= Q:= \{z_1^2 +z_2^2 +z_3^2=0\}$. 

Notice that $Q$ can be seen as the image of some adapted version of the Veronese embedding
\begin{align*}
	v\colon \CP^1 &\to  Q\subset \CP^2\\
	(t_1 \colon t_2 ) &\mapsto (i(t_1^2 +t_2^2)\colon 2t_1 t_2\colon t_1^2-t_2^2)
\end{align*} 

We can therefore describe a holomorphic 2-sheeted covering, hence a universal covering, of $\Proj \mathbb X_2$ via the map
\begin{equation}
	\label{Mappa 2:1}
	\begin{split}
		u \colon \GGG:=(\CP^1 \times \CP^1) \setminus \Delta &\to  \mathbb{P X}_2 \\
		(p, q) &\mapsto \ell_{v(p)}  \cap \ell_{v(q)}
	\end{split}
\end{equation}
where $\ell_{v(x)}$ denotes the tangent line to the quadric $Q$ in the point $v(x)$. The fact that $u$ is a $2$-sheeted covering follows from the fact that $Q$ is a conic, hence, for every fixed external point, there are exactly two lines passing by that point and tangent to $Q$. 

By uniqueness of the universal cover, we conclude that $u$ lifts to a biholomorphism between $\XXX_{2}$ and $\GGG=\CP^1\times \CP^1\setminus \Delta$.
\vspace{3pt}

One can check explicitly that for all $A\in \PSL$ there exists $\delta(A)\in SO(3,\C)$ such that $u(A\cdot p, A\cdot q)= \delta(A) \cdot u(p, q)$, with $\delta\colon \PSL \to SO(3,\C)$ an isomorphism. As a result, the pull-back metric via $u$ is $\PSL$-invariant, in addition to having clearly curvature $-1$, hence it coincides with $\inners_\GGG$ by Theorem \ref{teo: metrica su G}.

\section{$\Hy^3$ and $\GGG$ as symmetric spaces of $\PSL$}
\label{Section symm spaces}
As we remarked previously, $
Isom_0(\Hy^3)\cong Isom_0(\GGG)\cong \PSL.$

In fact, both $\Hy^3$ and $\GGG$ can be seen as symmetric spaces (in the sense of affine spaces) associated to $\PSL$:
\begin{itemize}
	\item[a)]  $\Hy^3 \cong \faktor {\PSL}{SU(2)}$, where the symmetry at $0\in \Hy^3\subset \R^3$ in the disk model is given by the map $ x\mapsto - x$;\\
	\item [b)]  $\GGG=\faktor {\PSL}{SO(2,\C)}$, where the symmetry at a geodesic $\gamma\in \GGG$ is given by the rotation of angle $\pi$ around $\gamma$.
\end{itemize}
For a complete survey on symmetric spaces see \cite{Kobayashi-Nomizu2}

Our aim in this section is to show that the metrics on $\Hy^3$ and $\GGG$ are in fact related to the hRm on $\PSL$ through their structures of symmetric spaces.

\vspace{5mm}

Fix $x_1 \in \Hy^3$ and $x_2 \in \GGG$ and, as a matter of convenience, denote in this section $X_1:= \Hy^3$ and $X_2:=\GGG$.

For $k=1,2$, define the evaluation map related to $(X_k, x_k)$ as
\begin{equation}
	\begin{split}
		\beta_k\colon \PSL &\to X_k\\
		A &\mapsto A\cdot x_k.
	\end{split}
\end{equation}

The two marked symmetric spaces $(X_1,x_1)$ and $(X_2,x_2)$ induce two Cartan decompositions of the Lie algebra $\asl$, namely
\begin{equation}\begin{split}
		\text{for $(X_1, x_1)$: } &\quad \asl= \mathfrak u (2) \oplus i \mathfrak u(2)=: \mathfrak h_1 \oplus \mathfrak m_1; \\
		\text{for $(X_2, x_2)$: } &\quad \asl= \mathfrak o(2,\C) \oplus (Sym(2,\C)\cap \asl)=: \mathfrak h_2 \oplus \mathfrak m_2.
	\end{split}
\end{equation}
We recall the following facts (see \cite{Kobayashi-Nomizu2} \S X-XI):
\begin{itemize}
	\item $\mathfrak h_i =Lie(Stab(x_i))$, and the $Ad$ action of $Stab(x_i)$ on $\asl$ globally fixes $\mathfrak m_i$;
	\item $[\mathfrak h_i, \mathfrak m_i]\subset \mathfrak m_i$ and $[\mathfrak m_i, \mathfrak m_i]\subset \mathfrak h_i$;
	\item the map 
	\[
	d_{I_2}\beta_i\colon \mathfrak m_i \to T_{x_i}  X_i
	\]
	is a linear isomorphism. Let us denote $d_{I_2}\beta_i (V)=: V_{x_i}$ for all $V\in \mathfrak m_i$;
	
	\item For all $A\in \PSL$, one has \begin{equation}
		\label{eq 1}
		\beta_i \circ L_A= A\circ \beta_i.\end{equation}
	As a corollary, for all $M\in Stab(x_i)$,\begin{equation}
		\label{eq 2}
		d_{I_2}\beta_i\circ Ad(A)= d_{x_i} A \circ d_{I_2} \beta_i.
	\end{equation}
	
	\item 	\label{Prop curvatura spazio simm}
	For all $U, V, W\in \mathfrak m_i$, 
	\begin{equation}
		\label{eq curvatura spazio simmetrico}
		R^{X_i}(U_{x_i}, V_{x_i})W_{x_i}= - [[U,V], W]
	\end{equation}
	where $R^{X_i}$ is the curvature tensor of $X_i$.
\end{itemize}

Define on $T\PSL$ the distribution \[\mathcal D_i|_{A} := (L_A)_* \mathfrak m_i < T_A \PSL.\]
\begin{Prop}
	For both $i=1,2$, the restriction to $\mathcal D_i$ of the differential of $\beta_i$ is a linear isometry at each $A\in \PSL$ up to a constant. Namely, for all $A\in \PSL$, 
	\[
	d_A \beta_i \colon (\mathcal D_i|_A, 4 \inners) \xrightarrow{\sim} T_{A \cdot x_i} X_i.
	\]
\end{Prop}
\begin{proof}
	Both the Riemannian metric of $\Hy^3$ and the hRm of $\GGG$ are uniquely determined by being $\PSL$-invariant metrics (Riemannian and holomophic Riemannian resp.) with constant sectional curvature $-1$. It is therefore enough to show that push-forward bilinear forms $(\beta_{i |\mathcal D_i})_* (\inners)$ define two well-posed, $\PSL$-invariant metrics of constant sectional curvature $-4$.
	
	By standard Lie theory, ${Kill^{\asl}}_{|\mathfrak u(2)}= Kill^{\mathfrak u(2)}$ which is a real negative-definite bilinear form being $U(2)$ compact and semisimple. By $\C$-bilinearity of $\inners=\frac 1 8 Kill^{\asl}$,  $\inners_{|i \mathfrak u(2)}$ is real and positive-definite. 
	
	On the other hand, $\inners_{|\mathfrak m_i}$ is a non-degenerate $\C$-bilinear form with orthonormal basis $\bigg( \begin{pmatrix} 0 & 1\\
		1 &0
	\end{pmatrix},  \begin{pmatrix} 0 & -i\\
		i & 0
	\end{pmatrix} \bigg)$.
	
	As we mentioned, consider on each $X_i$ the metric $g_i$ defined so that for all $A\in \PSL$ and for all $V,W\in \mathfrak m_i$
	\[
	(g_i)_{A\cdot x_i} \bigg((d_{x_i}A) (V_{x_i}), (d_{x_i} A) (W_{x_i})\bigg):= \inner{V, W}_{\asl}.
	\]
	\begin{itemize}
		\item The definition of $(g_i)$ is well-posed. Indeed, if $A\cdot x_i=B\cdot x_i$, $(d_{x_i}A) V_{x_i}= (d_{x_i}B) V'_{x_i}$ and $(d_{x_i}A) W_{x_i}=(d_{x_i}B) W'_{x_i}$, then using that $\inners$ is $Ad$-invariant and equation \eqref{eq 2} one has
		\begin{align*}
			(g_i)_{A\cdot x_i} ( (d_{x_i}A) V_{x_i}, (d_{x_i}A) W_{x_i}):=& g_{x_i}( V_{x_i}, W_{x_i})=\\
			=& g_{x_i}\bigg(\big(d_{x_i}(A^{-1}B) \big) V_{x_i}', \big(d_{x_i}(A^{-1}B) \big) W_{x_i}'\bigg)=\\
			=& \inner{Ad(A^{-1}B) V',Ad(A^{-1}B) W'}_\SL=\\
			=& \inner{V', W'}_\SL=\\
			=& (g_i)_{B\cdot x_i} ((d_{x_i}B) V'_{x_i}, (d_{x_i}B) W'_{x_i}).
		\end{align*}
		Since $\PSL$ acts on $X_i$ transitively and by diffeomorphisms, $g_i$ is uniquely defined.
		
		By construction $\PSL$ acts by isometries on both $g_1$ and $g_2$.  
		
		\item The metric $(g_i)_{A\cdot x_i}$ is the push-forward via $\beta_i$ of $\inners_{|\mathcal D_i}$, i.e. \[d_A\beta_i\colon (D_i (A), \inners_\SL) \to (T_{A\cdot x_i} X_i, g_i)\] is a linear isometry for all $A$. Indeed, by equation $\eqref{eq 2}$, for all $V, W\in \mathfrak m_i$,
		\begin{equation}
			\begin{split}
				&(g_i)_{A\cdot x_i} \big( (d_A\beta_i)(d_{I_2} L_A) (V),(d_A\beta_i)(d_{I_2} L_A) (W)\big)=\\
				&=(g_i)_{A\cdot x_i} \big( (d_{x_i}A) (d_{I_2}\beta_i)(V),(d_{x_i}A) (d_{I_2}\beta_i) (W)\big) =\\
				&= \inner{V, W}_\SL = \inner{(d_{I_2} L_A)(V), (d_{I_2} L_A)(W) }_\SL.
			\end{split}
		\end{equation}
		
		\item Since $\beta_1$ is smooth and $\beta_2$ is holomorphic, one can easily see that $g_1$ is Riemannian and $g_2$ is holomorphic Riemannian.
		
		\item We compute the sectional curvature of $g_i$. 
		
		First, let $V, W, Z\in \mathfrak m_i$ be orthonormal with respect to $\inners$,then by Equation \ref{eq: R in gruppi Lie complessi} and Equation \eqref{eq curvatura spazio simmetrico} we have
		\begin{align*}
			&-1= K(Span_\C (V, W))= \inner{R^{\PSL}(V, W) W, V}= -\frac 1 4 \inner{  [[V,W], W], V}= \\
			& =- \frac 1 4 \inner{R^{X_i} (V_{x_i}, W_{x_i})W_{x_i}, V_{x_i} }= \frac 1 4 K^{X_i} (Span_\C (V_{x_i},W_{x_i}) ).
		\end{align*}
		We conclude that $(X_i, g_i)$ has constant sectional curvature $-4$ and the proof follows.
	\end{itemize}

\end{proof}

\part{Immersions into $\XXX_{n+1}$ and geometric consequences}	
\label{part bonsante}

\chapter{Immersed hypersurfaces in $\PML$}    
\label{Chapter immersioni in Xn}

In this Chapter we study the geometry of smooth immersions of the form 
\[M\to \XXX_{n+1}\]where $M=M^n$ is a smooth manifold of (real) dimension $n$ and $\XXX_{n+1}$ is the Riemannian holomorphic space form of constant sectional curvature $-1$ and complex dimension $n+1$. 

One can immediately notice that, as an immersion between smooth manifolds, it has very high codimension. Nevertheless, we can define a suitable class of immersions for which we can translate in this setting some aspects of the classical theory of immersions of hypersurfaces.
In order to do it, we will introduce a new structure on manifolds that extends the notion of Riemannian metric: \emph{complex valued metrics}. 

In this Chapter, we will use $\XX, \YY, \ZZ$ to denote elements (and sections) of $TM$ and $X, Y, Z$ to denote elements (and sections) of the complexified tangent bundle $\CTM:= TM \oplus i TM= \C \otimes_\R TM$ whose elements can be seen as complex derivations of germs of complex-valued functions.

Let $M$ be a smooth manifold of (real) dimension $m$ and $\sigma \colon M\to \mathbb X_{n+1}$, with $n+1\ge m$, be a smooth immersion. Since $\mathbb X_{n+1}$ is a complex manifold, the differential map $d\sigma$ extends by $\C$-linearity to a map 
\begin{align*}
	d\sigma \colon \CTM &\to T\mathbb X_n\\
	X= \XX + i \YY &\mapsto d\sigma(\XX)+ \JJJ(d\sigma)(\YY)=: d\sigma(X).
\end{align*}

Now, consider the $\C-$bilinear pull-back form $\sigma^*\inner{\cdot,\cdot}$ on $\CTM$ defined by 
\begin{align*}
	\sigma^*\inner{\cdot, \cdot}_p\colon \C T_pM \times \C T_pM &\to \C\\
	(X, Y) &\mapsto \inner{d\sigma(X), d\sigma(Y)}.
\end{align*}
$\sigma^*\inner{\cdot,\cdot}_p$ is $\C$-bilinear and symmetric since $\inner{\cdot, \cdot}_{\sigma(p)}$ is $\C$-bilinear and symmetric. 

\begin{Definition}
	\begin{itemize}
		\item A \emph{complex (valued)} metric $g$ on $M$ is a non-degenerate smooth section of the bundle $Sym(\C T^*M\otimes \C T^*M)$, i.e. it is a smooth choice at each point $p\in M$ of a non-degenerate symmetric complex bilinear form
		\[
		g_p \colon \C T_pM\times \C T_pM \to \C.
		\]
		
		\item A smooth immersion $\sigma \colon M \to \PML$ is \emph{admissible} if $g=\sigma^*\inner{\cdot, \cdot}$ is a complex valued metric for $M$, i.e. if $g_p$ is non-degenerate. 
		
		\item If $g$ is a complex metric on $M$, an immersion $\sigma\colon (M, g)\to \PML$ is \emph{isometric} if $\sigma^*\inner{\cdot, \cdot}=g$. 
	\end{itemize}
\end{Definition}

\begin{Remark}
	If $\sigma$ is an admissible immersion, then $d_p\sigma \colon \C T_p M\to T_{\io(p)} \PML$ is injective. Indeed, if $d_{p}\sigma(X)=0$ then clearly $\sigma^* \inner{X, \cdot }\equiv 0$, hence $X=0$.
\end{Remark}

\section{Levi-Civita connection and curvature for complex metrics}

Let $M$ be a manifold of dimension $m$ and $g$ be a complex valued metric on $\CTM$. 
Despite looking quite weak as geometric structures, complex metrics admit a notion of Levi-Civita connection analogous to the one in pseudo-Riemannian geometry.

Recall that for sections $X, Y$ of $\CTM$ we have a well-posed Lie bracket $[X,Y]$ which coincides with the $\C$-bilinear extension of the usual Lie bracket for vector fields.

Define a \emph{connection on $\C TM$} as the $\C$-linear application
\begin{align*}
	\nabla \colon \Gamma (\C TM) &\to \Gamma(Hom_{\C} (\C TM, \C TM))\\
	\alpha &\mapsto \nabla \alpha (\colon X \mapsto \nabla_X \alpha)
\end{align*}
such that, for all $f\in C^{\infty}(M, \C)$, $\nabla_X (f \alpha)= f \nabla_X \alpha + X(f) \alpha$.

\begin{Proposition}
	\label{prop: LC di complex metrics}
	For every complex valued metric $g$ on $M$, there exists a unique connection $\nabla$ on $\CTM$, that we will call \emph{Levi-Civita connection}, such that for all $X,Y \in \Gamma(\CTM)$ the following conditions hold:
	\begin{align*}
		d ( g(X, Y) )&= g(\nabla X, Y)+ g(X, \nabla Y) \qquad & & \text{($\nabla$ is compatible with the metric)};\\
		[X,Y]&= \nabla_X Y - \nabla_Y X \qquad & & \text{($\nabla$ is torsion free)}.
	\end{align*}
\end{Proposition}

Observe that if $g$ is obtained as a $\C$-bilinear extension of some pseudo-Riemannian metric, then the induced Levi-Civita connection is the complex extension of the Levi-Civita connection for the (pseudo-)Riemannian metric on $M$.

We can also define the $(0,4)$-type and the $(1,3)$-type \emph{curvature tensors} for $g$ defined as 
\[
R(X,Y,Z, T):= -g(R(X,Y)Z,T):=- g\Big( \nabla_X \nabla_Y Z - \nabla_Y \nabla_X Z - \nabla_{[X,Y]} Z, T \Big)
\]
with $X,Y,Z, T \in \Gamma(\CTS)$.
The curvature tensor is $\C$-multilinear and has all of the standard symmetries of the curvature tensors induced by Riemannian metrics.

Finally, for every complex plane $Span_\C (X,Y)\in \C T_p M$ such that $g|_{Span_\C (X,Y)}$ is non-degenerate, we can define the {sectional curvature} $K(X,Y):=K(Span_\C (X,Y) )$ as
\begin{equation}
	\label{def curvatura}
	K(X,Y)=\frac{-g(R(X,Y)X,Y)}{g(X,X)g(Y,Y) -g(X,Y)^2}
\end{equation}
where the definition of $K(X,Y)$ is independent from the choice of the basis $\{X,Y\}$ for $Span_\C (X,Y)$.\newline

It is simple to check, via the Gram-Schmidt algorithm, that in a neighbourhood of any point $p\in M$ it is possible to construct a local $g-$orthonormal frame $(X_j)_{j=1}^m$ on $M$. We show it explicitly.

Fix a orthonormal basis $(W_j (p))_{j=1}^m$ for $\C T_p M$ and locally extend, in a neighbourhood of $p$, each $W_j(p)$ to a complex vector field $W_j$. Up to shrinking the neighbourhood in order to make the definition well-posed, define by iteration the local vector fields $Y_j$ by
\[
Y_j:= W_j - \sum_{k=1}^{j-1} \frac{g(W_j, Y_k)}{g(Y_k, Y_k)}Y_k
\]
which are such that $Y_j(p)= W_j(p)$ and by construction the $Y_j$'s are pairwise orthogonal.
Finally, up to shrinking the neighbourhood again, the vectors $X_j= \frac{Y_j}{\sqrt{g(Y_j,Y_j)}}$ (defined for any local choice of the square root) determine a local $g$-orthonormal frame around $p$. Similarly, every set of orthonormal vector fields can be extended to an orthonormal frame.\newline

Let $(X_j)_{j=1}^m$ be a local orthonormal frame for $g$, with $X_j\in \CTM$. Let $(\theta^j)_{j=1}^m$ be the correspondent coframe, $\theta^i\in \C T^*M$, defined by $\theta^i= g(X_i, \cdot)$.

We can define the \emph{Levi-Civita connection forms} $\theta^i_j$ for the frame $(X_i)_i$ by
\[
\nabla X_i = \sum_h \theta_i ^h \otimes X_h,
\] 
or, equivalently, by the equations
\[\begin{cases}
	d \theta^i =- \sum_j \theta_j ^i \wedge \theta^j\\
	\theta_j^i=-\theta_i ^j
\end{cases}.
\]

\section{Extrinsic geometry of hypersurfaces in $\PML$}

From now on, assume $dim (M)=n$.

Let $\sigma\colon M \to \PML$ be an admissible immersion and $g=\sigma^*\inner{\cdot, \cdot}$ be the induced complex metric.
Denote with $D$ be the Levi-Civita connection on $\PML$ and with $\nabla$ be the Levi-Civita connection for $g$. We want to adapt the usual theory of extrinsic geometry for immersed hypersurfaces to our setting.

Define the pull-back vector bundle $\Lambda=\io^* (T\PML)\to M$, which is a complex vector bundle, defined by $i\sigma^*(v):=\sigma^*(\mathbb J v)$ for all $v\in T\PML$, and is endowed on each fiber with the pull-back complex bilinear form that we still denote with $\inners$.

If $U$ is an open subset of $M$ over which $\io$ restricts to an embedding, then $\Lambda|_{U} \cong T\PML|_{ \io (U)}$.

Since the map $d \io \colon \C TM \to \PML$ is injective, $\C TM$ can be seen canonically as a complex sub-bundle of $\Lambda$ via the correspondence
\begin{align*}
	\C TM &\hookrightarrow \Lambda\\
	X &\mapsto \io^* ({d\io (X)}).
\end{align*}
Moreover, the complex bilinear form on $\Lambda$ corresponds to the one on $\C TM$ when restricted to it, since $\io$ is an isometric immersion.

We pull back the Levi-Civita connection $D$ of $T\PML$ in order to get a $\R$-linear connection $\overline \nabla$ on $\Lambda$,
\[
\overline \nabla :=\sigma^*D  \colon \Gamma(\Lambda) \to \Gamma \Big( Hom_{\R}(TM, \Lambda) \Big).
\] Observe that $\overline \nabla$ is completely defined by the Leibniz rule and by the condition
\[
\overline \nabla_\XX \ \io^*\xi := \io^*\Big( D_{ \io_* \XX}\ \xi   \Big) \qquad \forall \xi \in \Gamma(T \PML), \XX\in \Gamma(TM).
\]
By $\C$-linearity, we can see $\overline \nabla$ as a map 
\[ \overline \nabla \colon \Gamma(\Lambda) \to \Gamma \Big( Hom_{\C}(\C TM, \Lambda) \Big).  \]
by defining \[
\overline \nabla_{\XX_1 + i \XX_2} \hat \xi := \overline \nabla_{\XX_1} \hat \xi + i \overline \nabla_{\XX_2} \hat \xi.
\]

Via the canonical immersion of bundles $\CTM\hookrightarrow \Lambda$, it makes sense to consider the vector field $\overline \nabla _X Y$ with $X,Y\in \Gamma(\CTM)$.

Since $D\JJJ=0$ on $\PML$, for all $\XX \in \Gamma(TM)$ and $Y\in \Gamma(\C TM)$ we have that
\begin{align*}
	\overline{\nabla}_\XX (i  \sigma^*(\xi)) &= \sigma^* \Big(D_{d\sigma (\XX)} { (\mathbb J \xi)} \Big) =  \sigma^* \Big(D_{d\io(\XX)} \JJJ\ \xi \Big)=\\
	&= \sigma^* \Big(\JJJ D_{d\io(\XX)} \ \xi \Big)= i\   \overline \nabla_\XX \sigma^*\xi,
\end{align*}
and we conclude that $\overline \nabla$ is $\C$-bilinear.

We observed that $\C TM$ is a sub-bundle of $\Lambda$ over which the restriction of the complex bilinear form of $\Lambda$ is non-degenerate. Hence, we can consider the \emph{normal bundle} $\NN= \C TM^{\bot}$ over $M$ defined as the orthogonal complement of $\C TM$ in $\Lambda$. $\NN$ is a rank-1 complex bundle on $M$. 

For all local fields $X,Y \in \Gamma( \C TM)$ we can define $\IIv (X,Y)$ as the component in $\NN$ of $\overline \nabla_X Y$. 

\begin{Proposition}
	For all $X,Y \in \Gamma(\C TM)$, the component in $\C TM$ of $\overline \nabla_X Y$ is $\nabla_X Y$, where $\nabla$ is the Levi-Civita connection on $M$. In other words:
	\[
	\overline \nabla_X Y = \nabla_X Y + \IIv (X,Y).
	\]
	Moreover, $\IIv$ is a symmetric, $\C$-bilinear tensor.
\end{Proposition}
\begin{proof}
	By $\C$-linearity of $\nabla$ and $\overline \nabla$, it is enough to prove it for $X,Y\in \Gamma(TM)$. The proof just follows the standard proof in the Riemannian case: defining $A_X Y:= \overline{\nabla}_X Y - \IIv(X,Y)$, one can show that $A$ is a connection on $\CTM$, that it is torsion free and compatible with the metric, hence $A=\nabla$ by Proposition \ref{prop: LC di complex metrics}.
\end{proof}

For all $p\in M$, consider on a suitable neighbourhood $U_p \subset M$ a norm-1 section $\nu$ of $\NN$: we call such $\nu$ a local \emph{normal vector field} for $\sigma$. 

A local normal vector field fixed, we can locally define the \emph{second fundamental form} of the immersion $\sigma$ as the tensor  
\[
\IIs:=\inner{\IIv, \nu}= \inner{\overline \nabla, \nu}.
\]

Since there are two opposite choices for the section $\nu$, $\IIs$ is defined up to a sign.

We define the \emph{shape operator} $\Psi$ associated to the immersion $\io\colon M \to \PML$ as the tensor
\begin{align*}
	\Psi \in \Gamma \Big( Sym(\C T^*M \otimes_{\C} \C TM)\Big)
\end{align*}
such that, $\forall p\in M$ and $\forall\ \XX,\YY \in T_p M$,  \[g(\Psi(\XX),\YY)= \inner{\IIv(\XX,\YY), \nu} =\IIs(\XX,\YY).\]
As $\IIs$ is defined up to a sign, $\Psi$ is defined up to a sign as well.

We will say that $\io$ is \emph{totally geodesic} if and only if $\IIv\equiv 0$, i.e. $\Psi \equiv 0$.

\begin{Remark}
	One can see that the curvature tensor $\overline R$ of the connection $\overline\nabla$ is the pull-back of the curvature tensor $D$. That is:
	\begin{align*}
		\overline R(\XX,\YY)Z = \io^*\Big( R^D (d\io(\XX), d\io (\YY) )d\io(Z) \Big)
	\end{align*}
	
	Hence, since $\XXX_{n+1}$ has sectional curvature $-1$,
	\begin{equation}
		\label{eq: bar R curvature -1}
		\begin{split}
			g(\overline R (\XX,\YY)\YY, \XX)= \inner{R(d\sigma( \XX), d\io( \YY)) d\io(\YY), d\io(\XX)}\\
			=-(g(\XX,\XX)^2 g(\YY,\YY)^2-g(\XX, \YY)^2).
		\end{split}
	\end{equation}
\end{Remark}
\noindent

\begin{Proposition}
	Let $\nu$ be an normal vector field and let $\Psi$ be the corresponding shape operator. Then,
	\[
	\Psi =-\overline\nabla \nu
	\]
\end{Proposition}
\begin{proof}
	For all $\XX \in \Gamma(TM)$ and $Y\in \Gamma(\C TM)$,
	\begin{align*}
		g(\Psi(\XX), Y)&=\IIs(\XX,Y) =\inner{\overline \nabla_\XX Y, \nu} =\\
		&= -{\inner{ Y, \overline \nabla_\XX \nu}} + d(\inner{Y,\nu})(\XX)=\\
		&= -\inner{ Y, \overline \nabla_\XX \nu}.
	\end{align*}
	Moreover, 
	\begin{equation*}
		\inner{\overline\nabla_\XX \nu, \nu}=  \frac 1 2 \partial_{\XX}(\inner{\nu, \nu})=0.
	\end{equation*}
	Hence, $\overline \nabla_\XX \nu \in \C TM$ and $\Psi(\XX)=- \overline\nabla_\XX \nu$. The proof follows by $\C$-linearity.
\end{proof}

\section{Gauss and Codazzi equations}	
\label{sec Gauss-Codazzi eq first part}
Consider the exterior covariant derivative $d^\nabla$ associated to the Levi-Civita connection $\nabla$ on $\C TM$, namely
\[
(d^\nabla \Psi)(X,Y)= \nabla_X (\Psi (Y)) - \nabla_Y (\Psi(X)) - \Psi([X,Y])
\]
for all $X,Y\in \Gamma(\C TM)$.

Fix a local orthonormal frame $(X_i)_{i=i}^n$ with coframe $(\theta^i)_{i=1}^n$.  We can see the shape operator $\Psi$ in coordinates:
\[
\Psi=: \Psi^i_j \cdot \theta^j \otimes X_i =: \Psi^i \otimes X_i,
\]
where $\Psi^i_j \in \C^{\infty}(U,\C)$, with $\Psi^i_j=\Psi^j_i$, and $\Psi^i\in \Omega^1(\CTM)$. Notice that
\begin{equation}
	\label{equazione Psi_i}
	\IIs(X_i, \cdot)= \inner{\Psi(X_i), \cdot}=  \inner{\Psi^j_i X_j, \cdot}= \Psi^j_i \theta^j=\Psi_j^i \theta^j=\Psi^i
\end{equation}

We are ready to state the first part of an adapted version of Gauss-Codazzi Theorem.

\begin{Theorem}[Gauss-Codazzi, first part]
	\label{Teo Gauss-Codazzi}
	Let $\io \colon M^n \to \PML$ be an admissible immersion and $g=\sigma^* \inner{\cdot, \cdot}$ be the induced complex metric.
	
	Let $\nabla$ be the Levi-Civita connection on $M$, $R^M$ the curvature tensor of $g$ and $\Psi$ the shape operator associated to $\sigma$. 
	
	For any $g$-orthonormal frame $(X_i)_i$ with corresponding coframe $(\theta^i)_i$, the following equations hold: 
	\begin{align}
		\label{GC1}
		&1)d^\nabla \Psi \equiv 0 \qquad &&\text{(Codazzi equation)};\\
		\label{GC2}
		&2)    R^M(X_i, X_j, \cdot, \cdot) - \Psi^i \wedge \Psi^j= -\theta^i\wedge \theta^j \qquad &&\text{(Gauss equation)}. 
	\end{align}
\end{Theorem}\vspace{-3ex}
\begin{proof}
	In a neighbourhood $U$ of a point $p\in M$, fix a local normal vector field $\nu$. Let $Z\in \Gamma(\C TU)$, let $\XX(p),\YY(p) \in T_p M$ and let $\XX,\YY\in \Gamma(TU)$ be local extensions of $\XX(p)$ and $\YY(p)$ such that $[\XX, \YY]=0$ (such extensions can be constructed through a local chart). 
	
	By definition,
	\[
	\overline \nabla_ { \YY} Z= \nabla_{\YY} Z + \IIv (Z, \YY), 
	\]
	hence, by second derivation, 
	\begin{align*}
		\overline \nabla_\XX \overline{\nabla}_{\YY} Z&= \overline \nabla_\XX \nabla_{\YY} Z + \overline\nabla_\XX (\IIs(Z,\YY)\nu)=\\
		&= \nabla_\XX \nabla_{ \YY} Z + \IIs(\XX, \nabla_\YY Z)\nu + \overline \nabla_\XX (\IIs(Z, \YY) \nu)=\\
		&=\nabla_\XX \nabla_{\YY} Z + \IIs(\XX, \nabla_{\YY} Z)\nu + \partial_\XX (\IIs(Z,{\YY})) \nu + \IIs(Z,\YY) \overline\nabla_\XX \nu=\\
		&= \nabla_\XX \nabla_{\YY} Z -  \IIs(Z,\YY)\Psi(\XX) + \big( \IIs(\XX, \nabla_\YY Z) + \partial_\XX(\IIs(Z,\YY)) \big)\nu.
	\end{align*}
	
	We compute $\overline R (\XX,\YY)Z$.  Since $[ \XX,  \YY]=0$, we have
	\begin{align*}
		\overline R (\XX,\YY)Z =& R^M (\XX,\YY) Z - g(\Psi(\YY),Z)\Psi(\XX)+g(\Psi(\XX),Z)\Psi(\YY) + \\
		&+ \Big( \IIs(\XX, \nabla_\YY Z)+ \partial_\XX(\IIs(Z,  \YY))- \IIs(\YY, \nabla_\XX Z) - \partial_\YY(\IIs(Z,  \XX))\Big)\nu.
	\end{align*}
	Recall that $\overline R (\XX,\YY)Z= \io^* \Big( R^D(d\io(\XX), d\io(\YY)) d\io(Z) \Big)$. 
	
	Also recall that by Lemma \ref{lemma space form}, for all $V_1,V_2,V_3\in T_{ z} \PML$
	\[R^D(V_1, V_2)V_3 \in Span_\C(V_1, V_2)\]
	thus $\inner{\overline R(\XX,\YY)Z , \nu}=0$.
	
	As a result, we have the equalities 
	\begin{align*}
		\text{(a)}:&\   \IIs(\XX, \nabla_\YY Z)+ \partial_\XX(\IIs(Z,  \YY))- \IIs(\YY, \nabla_\XX Z) - \partial_\YY(\IIs(Z,  \XX)) =0;\\
		\text{(b)}:&\  \overline R (\XX,\YY)Z =  R^M (\XX,\YY) Z - g(\Psi(\YY),Z)\Psi(\XX)+g(\Psi(\XX),Z)\Psi(\YY).
	\end{align*}
	We deduce (\ref{GC1}) from (a) and (\ref{GC2}) from (b).
	
	By manipulation of (a), we get:
	\begin{align*}
		\IIs(\XX, \nabla_\YY Z)+& \partial_\XX(\IIs(Z,  \YY))- \IIs(\YY, \nabla_\XX Z) - \partial_\YY(\IIs(Z,  \XX))=\\
		&=\partial_\XX (g(Z, \Psi(\YY) ))- g(\nabla_\XX Z, \Psi(\YY))- \partial_\YY (g(Z, \Psi(\XX))) - g(\nabla_\YY Z, \Psi(\XX))=\\
		&= g(Z, \nabla_\XX(\Psi(\YY)))- g(Z, \nabla_\YY(\Psi(\XX)) )=\\
		&	= g(Z, d^{\nabla} \Psi (\XX,\YY));
	\end{align*}
	
	since this holds for all $Z$ in $\Gamma(\C TM)$, $d^{\nabla} \Psi (\XX,\YY)=0$ for all $\XX,\YY \in T_p M$, hence $d^{\nabla} \Psi=0$.

	In order to prove $(\ref{GC2})$, observe that by $\C$-linearity (b) is equivalent to
	\[
	b': \overline R (X,Y)Z =  R^M (X,Y) Z - g(\Psi(Y),Z)\Psi(X)+g(\Psi(X),Z)\Psi(Y)
	\]
	for all $X,Y,Z\in \Gamma(\CTM)$. 
	
	Recalling Lemma \ref{lemma space form} and equation (\ref{equazione Psi_i}), in the orthonormal frame $(X_i)_i$ we have
	\begin{align*}
		\overline R (X_i, X_j, X_k, X_h)=& R^M(X_i, X_j, X_k, X_h)+\\
		&+g(\Psi(X_i), X_h)g(\Psi(X_j), X_k) - g(\Psi(X_j), X_h)g(\Psi(X_i),X_k),
	\end{align*}
	so, by Equation \eqref{eq: bar R curvature -1}
	\[-g(X_i, X_k)g(X_j, X_h)+ g(X_i, X_h)g(X_j, X_k)=
	R^M(X_i, X_j, X_k, X_h) + (\Psi^i\wedge \Psi^j) (X_h, X_k)\]
	and
	\[
	-\theta^i\wedge \theta^j= R^M(X_i, X_j, \cdot, \cdot)- \Psi^i\wedge \Psi^j.
	\]
	The proof follows.

\end{proof}

For $n=2$, the Gauss equation $(\ref{GC2})$ can be written more simply. Fixed an orthonormal frame $\{X_1,X_2\}$ on the surface $M$, the curvature tensor $R$ is completely determined by its value $R(X_1,X_2,X_1,X_2)$ which is the curvature of $M$. Similarly, $\Psi^1\wedge \Psi^2= (\Psi_1^1 \Psi_2^2- \Psi^2_1\Psi^1_2)\theta^1\wedge \theta^2= det(\Psi)\theta^1\wedge \theta^2$. Therefore, equation $(\ref{GC2})$ is equivalent to 
\begin{equation}
	\label{GC2superfici}
	\text{Gauss equation for surfaces in $\PSL$:}\qquad \qquad K- det(\Psi)=-1
\end{equation}

\section{Integration of Gauss-Codazzi equations}
\label{Section integrazione GC}
The main aim of this section is to show that the converse of Theorem \ref{Teo Gauss-Codazzi} is also true for simply connected manifolds, in the way expressed in Theorem \ref{Teoremone}. 

Let us first give a definition inspired from Theorem \ref{Teo Gauss-Codazzi}.
\begin{Definition}
	An \emph{immersion data} for a smooth manifold $M=M^n$ is a pair $(g, \Psi)$ with $g$ a complex metric and $\Psi\in \Gamma(\CTM \otimes CT^*M)$ such that:
	\begin{itemize}
		\item $d^\nabla \Psi=0$, with $\nabla$ the Levi-Civita connection of $g$ ;
		\item for all $g$-orthonormal frame $(X_i)_{i=1}^n$ with correspondent coframe $(\theta^i)_{i=1}^n$, 
		\[R(X_i, X_j, \cdot, \cdot) - \Psi^i \wedge \Psi^j= -\theta^i\wedge \theta^j
		\]	
		with $R$ the curvature tensor of $g$ and the $\Psi^j$'s defined by $\Psi=\Psi^i\otimes X_i$. 
	\end{itemize}	
	Equivalently, we will say that $(g, \Psi)$ \emph{satisfy the Gauss-Codazzi equations}, (i.e. Equations \eqref{GC1}, \eqref{GC2}).
\end{Definition}

\begin{Teo} [Gauss-Codazzi, second part]
	\label{Teoremone} 
	Let $M$ be a smooth simply connected manifold of dimension $n$. Then, for all immersion data $(g,\Psi)$, there exists an isometric immersion $\sigma \colon (M,g) \to \ML$ with shape operator $\Psi$.
	
	Such $\sigma$ is unique up to post-composition with an element in $\Isom_0(\XXX_{n+1})$: i.e., if $\sigma'$ is another isometric immersion with the same shape operator, there exists a unique $\phi\in \Isom_0(\XXX_{n+1})$ such that $\sigma'(x)= \phi \cdot \sigma(x)$ for all $x\in M$.	
\end{Teo}

We state the case of surfaces in $\SL$ as a corollary.
\begin{Cor}
	\label{Cor Gauss-Codazzi PSL}
	Let $S$ be a smooth simply connected surface. Consider a complex metric $g$ on $S$, with induced Levi-Civita connection $\nabla$, and a $g$-self-adjoint bundle-homorphism 
$\Psi\colon \CTS \to \CTS$.

The pair $(g, \Psi)$ satisfies
\begin{align}
	1) &d^\nabla \Psi \equiv 0;\\
	2)	&K=-1+det(\Psi)\quad, 
\end{align}
if and only if there exists an isometric immersion $\sigma \colon S \to \SL$ whose corresponding shape operator is $\Psi$. Moreover, such $\sigma$ is unique up to post-composition with elements in $Isom_0(\SL)= \Proj(\SL \times \SL)$. 
\end{Cor}

The proof of Theorem \ref{Teoremone} will require some Lie theory.

Let $G$ be a Lie group. Recall that the \emph{Maurer-Cartan form} of $G$ is the $1-$form $\omega_G\in \Omega^1(G, Lie(G))$ defined by
\[
(\omega_G)_g (\dot g)= d(L_{g}^{-1}) (\dot g), 
\]	hence it is invariant by left translations. Moreover, the Maurer-Cartan form $\omega_G$ is completely characterized as the unique solution to the differential equation
\[
d\omega + [\omega, \omega]=0, \qquad \omega\in \Omega^1(G, Lie(G)).
\]

The proofs of both existence and uniqueness of Theorem \ref{Teoremone} are based on the following result on Lie groups.

\begin{Lemma}
	\label{Lemma chiave unicita}
	Let $M$ be a simply connected manifold and $G$ be a Lie group.
	
	Let $\omega \in \Omega^1 (M, Lie(G))$.
	
	Then 
	\begin{equation}
		\label{differenziale MC}
		d \omega+[\omega,\omega]=0
	\end{equation}
	if and only if there exists a smooth $\Phi\colon M \to G$ such that $\omega=\Phi^* \omega_G$, where $\omega_G$ denotes the Maurer-Cartan form of $G$.
	
	Moreover, such $\Phi$ is unique up to post-composition with some $L_g$, $g\in G$.
\end{Lemma}

A proof of Lemma \ref{Lemma chiave unicita} follows by constructing a suitable Cartan connection - depending on $\omega$ - on the trivial principal bundle $\pi_M \colon M\times G \to M$ so that its curvature being zero is equivalent to condition \eqref{differenziale MC}. 
In Chapter \ref{Section dipendenza olo} we will prove Proposition \ref{Teorema esistenza foliazione} which can be seen as a more general version of Lemma \ref{Lemma chiave unicita}.

Lemma \ref{Lemma chiave unicita} provides a remarkable one-to-one correspondence between immersions $M\to SO(n+2,\C)$ up to translations and elements in $\nobreak{\Omega^1(M, \Lieonn)}$ satisfying $(\ref{differenziale MC})$.

\vspace{10pt}

We start the proof of Theorem \ref{Teoremone} providing a way, which will turn out to be very useful for our pourposes, to construct (local) immersions into $SO(n+2,\C)$ from immersions into $\mathbb{X}_{n+1}$.

Let $dim(M)=n$ and let $\sigma \colon (M,g) \to \PML$ be an isometric immersion with shape operator $\Psi$.

Let $(X_i)_{i=1}^n$ be a local orthonormal frame for $\CTM$ on some open subset $U\subset M$ and $\nu$ be a normal vector field for the immersion $\sigma$. 
Recalling that $\mathbb X_{n+1}\subset \C^{n+2}$, notice that  $( d\sigma(X_1), \dots, d\sigma (X_n), d\sigma(\nu(x)), -i\sigma(x) )$ is an orthonormal basis for $\C^{n+2}$. Up to switching $\nu$ with $-\nu$, we can assume that this basis of $\C^{n+2}$ lies in the same $\SOnn$-orbit as the canonical basis $( v_1^0, \dots,  v_{n+2}^0)$ of $\C^{n+2}$. Recall that we defined \[ e=\begin{pmatrix}
	0\\
	\dots\\
	0\\
	i
\end{pmatrix}=i v^0_{n+2}\in \mathbb X_{n+1}.\]

Given an immersion $\sigma\colon U\to \mathbb X_{n+1}$, we construct the smooth map \[\Phi\colon U \to \SOnn\] defined, for all $x\in U$, by
\begin{equation}
	\label{definizione Phi}
	\begin{split}
		\Phi(x) (  v_i) &= d_x\sigma (X_i) \qquad \qquad i=1,\dots n\\
		\Phi(x) (  v_{n+1}) &= d_x\sigma (\nu),\\	\Phi(x)( v_{n+2})&=-i\Phi(x) ( e) =   -i\sigma(x) 	
	\end{split}
\end{equation}
In the following, we denote by $\alpha, \beta, \gamma$ indices in $\{1, \dots, n+2\}$ and by $i, j, k, h$ indices in $\{1, \dots, n\}$

Let $\omega_G$ be the Maurer-Cartan form for $\SOnn$ and define \[\omega_\Phi:= \Phi^* \omega_G\in \Omega^1(\C TU,\Lieonn),\] where we consider the canonical identification $\Lieonn=Skew({n+2}, \C)$.

By definition of the Maurer-Cartan form, $\omega_\Phi$ satisfies
\[
\omega_\Phi (x) (\XX)= d(L_{\Phi(x)} ^{-1}) \circ d_x\Phi (\XX)= \Phi(x)^{-1} \cdot \big(d_x\Phi (\XX)\big) 
\]

We are now able to prove the following.
\begin{Prop} [Uniqueness in Theorem \ref{Teoremone}]
	\label{Teo unicita}
	Let  $M$ be a connected smooth manifold of dimension $n$. Let $\sigma, \sigma'\colon M \to \ML$ be two admissible immersions of a hypersurface with the same induced complex metric $g=\sigma^*\inners=(\sigma')^*\inners$ and the same shape operator $\Psi$. Then, there exists a unique $\phi\in \Isom_0(\ML)$ such that $\sigma' (x)=\phi \cdot \sigma(x)$ for all $x\in M$.
\end{Prop}
\begin{proof}We first prove that $\phi$ is unique. Assume $\phi_1 \circ \sigma=\phi_2 \circ \sigma$ for some $\phi_1,\phi_2\in SO(n+2, \C)$, then, for any $x\in U$, $\phi_2^{-1} \circ \phi_1$ coincides with the identity on $Span_\C(\sigma(x), d\sigma(T_x M)  )\subset \C^{n+2}$ which is a complex vector subspace of dimension $(n+1)$, therefore, since both $\phi_1$ and $\phi_2$ are orientation-preserving, we conclude that $\phi_1=\phi_2$.
	
	We now prove that $\phi$ exists. Observe that it is enough to show that the statement holds for any open subset of $M$ that admits a global $g$-orthonormal frame, then the thesis follows by the uniqueness of $\phi$ and by the fact that $M$ is connected. We therefore assume in the proof that $M$ admits a global ortonormal frame without loss of generality.
	
	Consider an isometric immersion $\sigma\colon M \to \ML$ with immersion data $(g,\Psi)$, fix a $g$-frame $(X_i)_i$, and construct the lifting $\Phi\colon M \to Isom_0(\ML)$ as in Equation \eqref{definizione Phi}.
	
	In order to prove the statement, it is enough to show that the form $\omega_\Phi$ depends on $\sigma$ only through $g$ and $\Psi$. Indeed, assume that $\sigma$ and $\sigma'$ are two isometric immersions with the same immersion data $(g,\Psi)$: the induced maps $\Phi, \Phi'\colon M \mapsto G= SO(n+2,\C)$ would be such that $\Phi^* \omega_G = (\Phi')^* \omega_G$, hence, by Lemma \ref{Lemma chiave unicita}, there exists $\phi\in \SOnn$ such that $\Phi'(x)=\phi\cdot \Phi(x)$ for all $x\in M$, hence $\sigma'(x)=\phi\cdot \sigma(x)$. \newline

	For $\alpha=1, \dots, n+2$, define \[\Phi_\alpha (x):= \Phi(x) v^0_\alpha.\]
	Notice that  $i\Phi_\alpha\in \mathbb{X}_{n+1}$ for all $\alpha$, and that $\Phi_\beta(x)\in T_{i\Phi_\alpha(x)} \mathbb{X}_{n+1}$ for all $\alpha\ne \beta$ (see Equation \eqref{definizione Phi}). In particular, for $\beta \ne n+2$,  $\Phi_\beta$ can be seen as a vector field on $\mathbb X_{n+1}$ along $\sigma(U)$.
	
	As usual, denote with $\theta^i_j$ the Levi-Civita connection forms for $(X_i)_{i=1}^n$, so $
	\nabla X_i = \theta_i ^j \otimes X_j$. Also recall the notation $\Psi= \Psi^i \otimes X_i$. 
	
	We compute the form $\omega_\Phi$ explicitly using the extrinsic geometry of $\sigma$. In the following computations, elements such as $\omega_\Phi(\YY)$, $\Phi(x)$, and $(d_x \Phi)(\YY)$ are to be seen as matrices, while elements such as $\sigma(x)$ and $(d_x \sigma)(\YY)$ are to be seen as vectors.

	For all $x\in U$ and $\YY\in T_x U$,
	\begin{align*}
		\Phi(x)\ \omega_\Phi (\YY)\  v_\alpha^0 &= (d_x\Phi) (\YY)\   v_\alpha^0 = d_x\big(\Phi(\cdot) \  v_\alpha^0 \big) (\YY)= (d_x\Phi_{\alpha}) (\YY).
	\end{align*}
	In particular, \begin{align*}
		\omega_\Phi(\YY) \cdot  v_{n+2}^0 &= (\Phi(x))^{-1}\ d_x(\Phi(\cdot)  v^0_{n+2}) (\YY) =-i(\Phi(x))^{-1}\ d_x\sigma(\YY)=\\
		&= \sum_{k=1}^n -i\inner{\YY, X_k} \cdot (\Phi(x))^{-1}\ (d_x\sigma(X_k))=\\
		&= \sum_{k=1}^n -i\inner{\YY, X_k}  v_k^0= \sum_{k=1}^n -i\theta^k(\YY)  v_k^0
	\end{align*}
	
	We now compute $\omega_\Phi \cdot  v_j^0$ with $j=1, \dots, n$. Recall that the Levi-Civita connection on $\mathbb X_{n+1}\subset \C^{n+2}$ is the tangent component of the standard differentiation on $\C^{n+2}$ to deduce that
	\begin{align*}
		\Phi(x)\ \omega_\Phi (\YY)\   v_j^0&= (d_x\Phi_j)(\YY)= \partial_\YY \Phi_j =\\
		&= (d_x\sigma) (\overline \nabla_ \YY X_j) + \inner{ (d_x \Phi_j)( \YY), \Phi_{n+2}} \Phi_{n+2}=\\
		&=(d_x\sigma) (\nabla_ \YY X_j) + \IIs (\YY, X_j) d\sigma(\nu) + \inner{ (d_x \Phi_j)( \YY), \Phi_{n+2}} \Phi_{n+2}=\\
		&=(d_x\sigma) (\nabla_ \YY X_j) + \IIs (\YY, X_j) d\sigma(\nu) -\inner{ \Phi_{j}, (d_x \Phi_{n+2})( \YY)} \Phi_{n+2}=\\
		&=\sum_k \theta_j ^k (\YY) \Phi_k +\IIs (\YY, X_j) \Phi_{n+1}+ i\inner{ \Phi_j, d\sigma(\YY)} \Phi_{n+2}.
	\end{align*}
	
	As a result, 
	\begin{align*}
		\omega_\Phi(\YY)  v_j^0&= \theta_j^k  v_k^0 + \IIs (\YY, X_j) v_{n+1}^0 +i \inner{X_j, \YY}  v_{n+2}^0=\\
		&= \theta_j^k  v_k^0 + \Psi^j(\YY) v_{n+1}^0 +i \theta^j(\YY) \underline v_{n+2}^0.
	\end{align*}
	
	We therefore have a complete description of $\omega_\Phi \in \Gamma(\CTM, \Lieonn)= Skew(n+2, \Gamma(\CTM))$ in terms of $g$ (and of the induced connection $\nabla$) and $\Psi$ only:
	
	\begin{equation}
		\label{eq: forma omega phi}
		\omega_\Phi=
		\begin{pmatrix}
			{\scalebox{2} {$\Theta$} }& \begin{matrix}
				\Psi^1 & -i\theta^1 \\
				\dots &\dots \\
				\Psi^n & -i\theta^n
			\end{matrix} \\
			\begin{matrix}
				-\Psi^1 & \dots & -\Psi^n\\
				i\theta^1 & \dots & i\theta^n
			\end{matrix} & 
			\begin{matrix}0\text{    } & 0\text{    } \\ 
				0\text{    } & 0\text{    }
			\end{matrix}
		\end{pmatrix}  
	\end{equation}
	where $\Theta=(\theta^i_j)_{i,j}$. The proof follows.
\end{proof}

\begin{Cor}
	\label{Corollario Teo unicita}
	Let $\sigma, \sigma'\colon (M^n,g) \to \ML$ be two isometric immersions with the same shape operator $\Psi$. Assume $\sigma(x)=\sigma'(x)$ and $d_x\sigma=d_x\sigma'$, then $\sigma\equiv\sigma'$.
\end{Cor}
\begin{proof}
	By Proposition \ref{Teo unicita}, there exists $\phi \in \SOnn$ such that $\phi \circ \sigma=\sigma'$. Since there exists at most one matrix in $\SOnn$ sending $n+1$ given vectors in $\C^{n+2}$ into other $n+1$ given vectors, we have $d_{\sigma(x)} \phi=id$, hence $\phi=id$.
\end{proof}

We are finally able to prove existence.

\begin{proof} [Proof of Existence in Theorem $\ref{Teoremone}$] We prove the statement in two steps.
	
	\emph{Step 1.}     
	Assume $M=U$ is such that there exists a globally defined $g$-orthonormal frame $(X_i)_{i=1}^n$ in $\C T U$ with dual frame $(\theta^i)_i$.
	
	Our aim is to construct a suitable form $\omega\in \Omega^1(U, \Lieonn)$ satisfying \eqref{differenziale MC} in order to be able to apply Lemma \ref{Lemma chiave unicita} to obtain an immersion $\Phi\colon U\to\SOnn$. Then, drawing inspiration from Equations \eqref{definizione Phi}, we will prove that the immersion $\sigma(x)=i \Phi(x)\cdot \underline v_{n+2}$ into $\XXX_{n+1}$ satisfied the expected conditions. 
	
	Let $\Theta=(\theta^i_j)$ be the skew-symmetric matrix of the Levi-Civita connection forms for $(X_i)_i$ and recall the notation
	$\Psi= \Psi^i \otimes X_i=\Psi^i_j \cdot \theta^j \otimes X_i$.
	Also define \[\upsi:= \begin{pmatrix}
		\Psi^1 \\
		\Psi^2\\
		\dots\\
		\Psi^n
	\end{pmatrix} \quad \text{and}\quad \utheta:= \begin{pmatrix}
		\theta^1 \\
		\theta^2\\
		\dots\\
		\theta^n
	\end{pmatrix} .\]

	Equation \eqref{eq: forma omega phi} in the proof of Proposition \ref{Teo unicita} suggests to define a form $\omega\in \Omega^1(U, \Lieonn)$ as
	\[		
	\omega=
	\begin{pmatrix}
		{\scalebox{2} {$\Theta$} }& \begin{matrix}
			\Psi^1 & -i\theta^1 \\
			\dots &\dots \\
			\Psi^n & -i\theta^n
		\end{matrix} \\
		\begin{matrix}
			-\Psi^1 & \dots & -\Psi^n\\
			i\theta^1 & \dots & i\theta^n
		\end{matrix} & 
		\begin{matrix}0\text{    } & \text{    } 0 \\ 
			0\text{    } & \text{    } 0
		\end{matrix}
	\end{pmatrix}  
	\]
	
	We want to prove that $d\omega + [\omega, \omega]=0$. \newline
	
	By the explicit computation, we get that 
	\[
	d\omega + [\omega, \omega] =
	\begin{pmatrix}
		d\Theta + \Theta\wedge \Theta - \upsi \wedge ^t\upsi+ \utheta \wedge ^t\utheta & d\upsi + \Theta \wedge \upsi & -i d\utheta -i \Theta \wedge \utheta \\
		-d\upsi - \Theta \wedge \upsi	 &  0 & i^t\upsi\wedge \underline \theta \\
		i ^td\utheta+ i ^t \utheta \wedge \Theta & i ^t\utheta \wedge \upsi & 0
	\end{pmatrix}
	\]
	We now observe the following.
	\begin{itemize}
		\item By definition of the Levi-Civita connection forms,
		\[
		d\theta^j +\theta^i_j\wedge \theta^j=0,\]
		hence $d\utheta+ \Theta\wedge \utheta=0$.
		\item Since $\Psi$ is symmetric, we have 
		\[^t\upsi\wedge \utheta= \sum_{j,k} \Psi^j_k \theta^j\wedge \theta^k=0.
		\]
		\item Expanding $d^{\nabla} \Psi$, one gets
		\begin{align*}
			g\big((d^\nabla \Psi) (X_i, X_j), X_h\big) =& \big(\nabla_{X_i} (\Psi^k(X_j) X_k), X_h\big)-\\
			&-\big(\nabla_{X_j} \Psi^k(X_i)X_k, X_h\big) - \Psi^h([X_i, X_j])=\\
			=& X_i (\Psi^h (X_j)) + \Psi^k (X_j)\theta_k^h (X_i)-\\
			&- X_j(\Psi^h(X_i)) - \Psi^k (X_i) \theta_k ^h (X_j) - \Psi^h([X_i, X_j])=\\
			=& (d\Psi^h)(X_i, X_j) +  (\Theta\wedge \upsi)^{(h)} (X_i, X_j).
		\end{align*}
		which leads to the standard formula
		\[
		\underline {d^\nabla \Psi}= d\upsi +  \Theta \wedge \upsi,
		\]
		(see also \cite{Kobayashi-Nomizu1}). Thus, by Codazzi equation,
		\[
		d\upsi +  \Theta \wedge \upsi=0
		\]

		\item  By a straightforward computation \[ d\theta_j^i+ \sum_t \theta_k^i\wedge \theta^k_j =R(X_i, X_j, \cdot, \cdot).\]
		Now use Gauss equation to get
		\[d\theta_j^i + \sum_k \theta^i_k\wedge \theta_j^k - \psi^i \wedge \psi^j+ \theta^i \wedge \theta^j=0 \]
		for all $i, j=1,\dots, n$, i.e. \[d\Theta +\Theta \wedge \Theta-\upsi\wedge^t\upsi + \utheta\wedge^t\utheta=0.\]

	\end{itemize}
	We can finally conclude that $d\omega + [\omega, \omega]=0$.
	
	By Lemma \ref{Lemma chiave unicita} there exists $\Phi\colon U \to \SOnn$ such that $\omega= \omega_\Phi= \Phi^* \omega_{G}$ where $\omega_{G}$ is the Maurer-Cartan form of $\SOnn$.
	
	Define $\sigma\colon M \to \ML$ as $\sigma(x):= \Phi(x)\cdot  e=i\Phi(x) v_{n+2}^0$.
	
	Now, we note that the way $\omega$ is defined allows exactly to compute the following:
	\begin{align*}
		d_x\sigma(X_j) &= i \big(d_x ( \Phi(\cdot)  v_{n+2}^0 )\big)(X_j)=i (d_x\Phi)(X_j)\  v_{n+2}^0=\\
		&=i \Phi(x)\ \omega(X_j)\  v^0_{(n+2)}= \theta^k(X_j)\ \Phi(x)\  v_k^0= \Phi(x) v_j^0.
	\end{align*}
	We conclude that $\sigma$ is an immersion at every point and that it is an isometric immersion since its differential sends an orthonormal basis into orthonormal vectors. 
	
	Observe that this construction is in fact inverse to construction $(\ref{definizione Phi})$, i.e. $\Phi$ is the map one gets if they started from the immersion $\sigma$ in the first place. By the proof of Lemma \ref{Lemma chiave unicita}, one gets that the penultimate column of $\omega=\Phi^*\omega_G$ has as non-zero entries the components of the shape operator w.r.t. the frame $(X_1, \dots, X_n)$, hence $\Psi$ is the shape operator of $\sigma$.
	\vspace{10pt}

	\emph{Step 2.} We extend the result for simply connected manifold $M$ with immersion data $(g,\Psi)$. We will say that an open subset of $M$ is \emph{immersible} if it admits a $g$-isometric immersion into $\ML$ with shape operator $\Psi$.
	
	Fix a point $x\in M$ and an immersible open neighbourhood $U_0$ of $x$ and an isometric immersion $\sigma_0 \colon U_0 \to \ML$. In fact, one can just fix the germ of an immersion.
	
	For any $y\in M$, let $\alpha\colon [0,1] \to M$ be a simple path connecting $x$ to $y$. Consider a collection of open subsets $\{U_i\}_{i=0}^m$ of $\ML$ with the property of being a \emph{good cover} for $\alpha([0,1])$, i.e. such that: \begin{itemize}
		\item[\_] $\alpha([0,1]) \subset \bigcup_{i=0}^m U_i$;
		\item[\_] $U_i$ is immersible for all $i$;
		\item[\_] $\alpha^{-1}(U_i)$ is a connected interval;
		\item[\_] $U_{i}\cap U_j\cap \alpha([0,1]) \ne \emptyset$ iff $|i-j|\le 1$ and $U_i\cap U_{j}$ is either empty or connected for all $i,j$.
	\end{itemize}
	
	By Corollary \ref{Corollario Teo unicita}, we can construct a unique family $\{\sigma_i \}_{i=0}^m$ such that $\sigma_i \colon U_i \to \ML$ is an isometric immersion and ${\sigma_i}_{|U_i \cap U_{i-1}}= {\sigma_{i-1}}_{|U_i \cap U_{i-1}}$ for all $i=1, \dots m$.
	
	We prove that the germ of $\sigma_m$ around $y$ does not depend on the choice of the good cover for $\alpha$.
	Let $\{U'_j\}_{j=0}^m$, with $U'_0=U_0$, be another good cover for $\alpha$ with associated set of immersions $\{\sigma'_j\}_{j=0}^p$, with $\sigma'_0=\sigma_0$. For all $i\in \{0, \dots, m\}$ and $j\in \{0, \dots, p \}$, $U_i\cap U'_j\cap \alpha([0,1])$ is either empty or connected, since it is the image of a connected interval; we define $C_{i,j}$ as the connected component of $U_i \cap U'_j$ which intersects $\alpha([0,1])$.
	
	By contradiction, assume that $(i_0, j_0)\in \{0, \dots, m\}\times\{0,\dots, p\}$ is such that $U_{i_0}\cap U'_{j_0}\cap \alpha(I)\ne \emptyset$ and such that $\sigma_{i_0}$ and $\sigma'_{j_0}$ do not coincide over $C_{i_0, j_0}$; also pick $(i_0, j_0)$ so that this does not hold for any other couple $(i, j)$ with $i\le i_0$ and $j\le j_0$. One can see (because of the topology of $[0,1]$) that, for such $(i_0, j_0)$, either $U_{i_0 -1}\cap U_{i_0}\cap U'_{j_0}\ne \emptyset$ or $U'_{j_0-1}\cap U_{i_0}\cap U_{j_0}\ne \emptyset$; assume the former without loss of generality. Then, by minimality $\sigma_{i_0-1}$ and $\sigma'_{j_0}$ coincide over $C_{i_0-1, j_0}$, while by construction $\sigma_{i_0-1}$ and $\sigma_{i_0}$ coincide over $U_{i_0}\cap U_{i_0-1}$: as a result $\sigma_{i_0}$ and $\sigma'_{j_0}$ coincide on $C_{i_0,j_0} \cap U_{i_0}\cap U_{i_0-1}$ which is nonempty, hence they coincide on $C_{i_0, j_0}$ by Corollary \ref{Corollario Teo unicita}, which is a contradiction.
	
	Finally, we use the simply-connectedness of $M$ to prove that the germ of $\sigma_m$ over $y$ does not depend on the path $\alpha$ either. Indeed, any two paths from $x$ to $y$ are linked by some homotopy $H\colon [0,1]\times [0,1]\to M$; by compactness, it is clear that if $\{U_i\}_{i=0}^m$ is a good cover for the path $\alpha_t:=H(t,\cdot)$, then there exists $\eps_t$ such that $\{U_i\}_i$ is a good cover for $\alpha_s$ for all $|s-t|<\eps$; as a result, the function that assigns to each time $t\in [0,1]$ the germ of the corresponding $\sigma_n^t$ in $q$ constructed via $\alpha_t$ is locally constant, hence constant.
	
	We can therefore extend $\sigma_0$ to an isometric immersion $\sigma\colon M \to \ML$.
\end{proof}

\begin{Remark}
	Let $M$ be a manifold with immersion data $(g,\Psi)$ satisfying the Gauss-Codazzi equation. 
	
	Consider its universal cover $(\widetilde M, \widetilde g)$, over which $\pi_1(M)$ acts by isometries, and the lifting $\widetilde \Psi$ of $\Psi$, which is $\pi_1(M)$-invariant; by the previous result, there exists an isometric immersion \[\sigma\colon (\widetilde M,\widetilde g) \to \PML\] with shape operator $\widetilde \Psi$, unique up to an ambient isometry. It is now trivial to check that $\sigma$ is $(\pi_1 (M), \SOnn)$-equivariant. Indeed, for all $\alpha \in \pi_1 (M)$, $\sigma \circ \alpha$ is a new isometric embedding with shape operator $\Psi$, hence, by uniqueness of Theorem \ref{Teoremone}. there exists a unique element $mon(\alpha)\in \SOnn$ such that \[\sigma \circ \alpha= mon(\alpha) \circ \sigma.\]
\end{Remark}

\section{Totally geodesic hypersurfaces in $\PML$}
\label{section codimension zero}
A particular case of immersions $M\to \mathbb X_{n+1}$ is given by totally geodesic immersions, namely immersions with $\Psi=0$. The study of this case leads to several interesting results.

\begin{Lemma}
	If $g$ is a complex metric on a smooth manifold $M$ with constant sectional curvature $k\in \C$, then, for any $X,Y,Z,W\in \Gamma(\C TM)$,
	\[
	R(X,Y,Z,W)= k (\inner{X,Z}\inner{Y,W} - \inner{Y,Z}\inner{X,W}).
	\]
	In particular, $R(X,Y)Z\in Span_{\C}(X,Y)$.
\end{Lemma}
\begin{proof}The proof is exactly as in Lemma \ref{lemma space form}.
\end{proof}

\begin{Theorem}
	\label{teoremone stessa dimensione}        
	Let $M$ be a smooth manifold of dimension $n$.
	
	Then, $g$ is a complex metric for $M$ with constant sectional curvature $-1$ if and only if there exists an isometric immersion
	\[(\widetilde M, \widetilde g)\to \mathbb{X}_n\] which is unique up to post-composition with elements in $Isom(\mathbb X _{n})$ and therefore is $(\pi_1(M), O(n+1,\C) )$-equivariant.
\end{Theorem}
\begin{proof}
	Let $\iota \colon \mathbb \C^{n+1}\hookrightarrow \mathbb \C^{n+2}$ be the immersion
	$\iota(z_1,\dots z_{n+1})=(z_1,\dots, z_{n+1}, 0)$.
	
	Assume there exists an isometric immersion $\sigma\colon (M,g)\to \mathbb X_n$, then the map $\overline \sigma=\iota \circ \sigma\colon (M,g) \to \mathbb X_{n+1}$ is an isometric immersion as well, and it has $\overline \sigma^*  v^{n+2}_0$ as a global normal vector field. The induced shape operator is therefore \[\Psi=\overline \nabla (\overline \sigma^* v^{n+2}_0)=0.\] By Gauss equation, this means that, for every local orthonormal frame $(X_i)_i$, $R(X_i, X_j, X_i, X_j)=-1$, i.e. $(M,g)$ has constant sectional curvature $-1$.
	\vspace{5pt}
	
	Conversely, assume $(M,g)$ has constant sectional curvature $-1$. Then, by taking $\Psi\equiv 0$, the couple $(g,\Psi)$ trivially satisfies the Codazzi Equation \eqref{GC1}, and, by the previous lemma, we have $R(X_i, X_j,\cdot, \cdot)=-\theta^1\wedge \theta^j$, so the Gauss equation holds as well.
	By Theorem \ref{Teoremone}, there exists an isometric $\pi_1(M)$-equivariant immersion $\overline \sigma\colon (M,g)\to \mathbb{X}_{n+1}$. 
	
	Let $\nu=\overline \sigma^* \nu_0$ be a normal local vector field w.r.t. $\overline\sigma$, with $\nu_0(x)\in T_{\sigma(x)}\mathbb X_{n+1}$. Let $\gamma(t)$ be a curve on $M$ with $\gamma(0)=x$ over which $\nu((\gamma(t)))=:\nu(t)$ is well-defined. Also denote $\nu_0(t)= d\overline{\sigma}(\nu(t))$.
	
	Since $\nu_0$ is unitary it turns out that $\dot\nu_0$ is orthogonal to $\nu_0$. On the other hand, differentiating $0=\inner{\nu_0(t),\overline \sigma(\gamma(t))}$ and using that $\nu_0$ is orthogonal to $d\overline{\sigma}(\dot\gamma)$, we deduce that $\dot\nu_0$ is orthogonal to the vector $\overline{\sigma}(\gamma(t))$. Thus $\dot\nu_0$ is contained in $d\overline{\sigma}(\CTM)$, and 
	\[
	\dot\nu_0= d\overline\sigma(D_{\dot \gamma} \nu) =-d\overline\sigma(\Psi(\dot\gamma))=0.
	\]
	
	We conclude that $\dot \nu_0(t)\equiv0$, hence $\nu_0$ is a constant vector and $Im(\overline\sigma)\subset \nu_0^\bot$. Up to composition with elements in $\SOnn$, we can assume $\nu_0= (0, \dots, 0, 1)$, so $Im (\overline \sigma) \subset \mathbb X_n$.
	
	Finally, an isometric immersion of $(\widetilde M, \widetilde g)$ in $\mathbb{X}_n$ is unique up to composition with elements in $\SOnn$ that stabilize $\mathbb X_n$, namely up to elements in $O(n+1,\C)$.
\end{proof}

\begin{Remark}
	\label{rmk: pseudo riemannian immerge in Xn}
	By Theorem \ref{teoremone stessa dimensione}, every pseudo-Riemannian space form of constant curvature $-1$ of dimension $n$ admits an essentially unique isometric immersion into $\mathbb{X}_{n}$.
\end{Remark}

An interesting case we are going to treat in Chapter \ref{Section uniformization} is the case $n=2$. In this setting, the previous theorem can be stated in the following way.
\begin{Proposition}
	\label{surfaces in G first part}	
	Let $(S,g)$ be a surface equipped with a complex metric, denote with $(\widetilde S, \widetilde g)$ the universal covering. Then:
	\begin{itemize}
		\item $(S,g)$ has constant curvature $-1$ if and only if there exists an isometric immersion 
		\[
		\sigma=(f_1,f_2) \colon (\widetilde S,\widetilde g)\to {\mathbb G}= (\CP^1 \times \CP^1 \setminus \Delta, -\frac{4}{(z_1 - z_2)^2}dz_1 dz_2) 
		\]
		which is $(\pi_1 (S), Isom(\GGG))$-equivariant. In particular, $g$ induces a monodromy map 
		\[
		mon_g \colon \pi_1 (S) \to \PSL \times \Z_2
		\] defined up to conjugation. Being $\sigma$ admissible, the maps $f_j\colon \widetilde S \to \CP^1$ are such that $rk( d f_j) \geq 1$.
		\item
		By composing with some affine chart $(U\times U\setminus \Delta, z\times z)$ of $\GGG$, a complex metric $g$ with constant curvature $-1$ can be locally expressed as
		\[
		g= -\frac{4}{(f_1 - f_2)^2}df_1 df_2.
		\]
		\item The maps $f_1$, $f_2$ are local diffeomorphisms if and only if no real vector $\XX \in TS\setminus \{0\}$ is isotropic for $g= \sigma^*\inners= (f_1,f_2)^*\inners$.
	\end{itemize}
\end{Proposition}
\begin{proof}
	We only need to prove the last part of the proposition.
	There exists $\XX\in T_x S$ such that $g(\XX,\XX)=0$ if and only if $df_1(\XX)\cdot df_2(\XX)=0$ for some $v$, which holds if and only if one between $f_1$ and $f_2$ is not a local diffeomorphism.
\end{proof}
\begin{Example}
	The hyperbolic plane $\Hy^2$ in the upper half-plane model admits the isometric immersion $z \mapsto (z,\overline z)$ into $\GGG$ as described in the proof of Theorem \ref{teo: metrica su G}.
	
	Consider the immersion of $S^2$ given by \begin{align*}
		(f_1,f_2)\colon S^2 \approx \overline \C &\to \GGG\\
		z &\mapsto (z,-\frac{1}{\overline z})
	\end{align*}
	which in fact embeds $S^2$ into the graph of the antipodal map. The pull-back metric is given by
	\[(f_1,f_2)^* \inners= - \frac{4}{(1+ |z|^2)^2} dzd\overline{z}
	\]
	which coincides with the negative definite space form of curvature $-1$, namely $-S^2$, the sphere equipped with the opposite of the standard elliptic metric.
	
	Another example is given by the 2-dimensional Anti-de Sitter space ${AdS^2}$ which is isometric to $S^1\times S^1\setminus \Delta\subset \GGG$. 
\end{Example}

\section{Connections with immersions into pseudo-Riemannian space forms}
\label{sec pseudoRiemannian hrm}

Recall that in Section \ref{section: uniqueness pseudoRiem space forms} we defined $\mathbb F_{-1}^{m_1,m_2}$ as the pseudo-Riemannian space form of dimension $n:=m_1+m_2$, constant sectional curvature $-1$ and signature $(m_1,m_2)$. From now on, let us denote $\mathbb F^{m_1,m_2}:= \mathbb F_{-1}^{m_1,m_2}$.

As we already observed in Remark \ref{rmk: pseudo riemannian immerge in Xn}, there exists a unique isometric immersion of $\mathbb{F}^{m_1,m_2}$ into $\mathbb X_{n}$ up to composition with ambient isometries; moreover, we constructed in \ref{rmk Super hRm orientation}.\ref{rmk Super hRm immersion pseudo-Riemannian} an explicit isometric immersion $\iota\colon\mathbb{F}^{m_1,m_2} \to \mathbb X_{n}$.

As we mentioned in Section \ref{section: Gauss-Codazzi pseudoRiem}, there exists a general theory of immersions of hypersurfaces into $\mathbb F^{m_1,m_2}$, leading to the Gauss-Codazzi Theorem \ref{Gauss-Codazzi pseudo-Riem}. By composition with the isometric immersion $\iota\colon \mathbb F^{m_1,m_2}\to \mathbb X_n$, one can see an immersion $\sigma$ into $\mathbb F^{m_1,m_2}$ as an immersion $\overline \sigma=\iota\circ \sigma$ into $\mathbb X_n$ and, in a sense, the extrinsic geometry of the latter extends the one of the former. Let us look at this with more attention.

\begin{itemize}
	\item The map $d\iota$ induces a canonical bundle inclusion of $\sigma^* T \mathbb F^{m_1,m_2}$ into $\overline\sigma^*T \mathbb X_{n}$: in fact, since $Span_\C (d\iota(T_{ x} \mathbb F^{m_1,m_2}) )= T_{\iota( x)} \mathbb X_n$, one has
	\[
	\sigma^* T \mathbb F^{m_1,m_2} \otimes_\R \C \cong \bar \sigma^* T \mathbb X_n.
	\] Since $\iota$ is isometric, the pull-back Levi-Civita connections coincide on $\sigma^* T \mathbb F^{m_1,m_2}$.
	
	\item Let $g=\overline\sigma^*\inners_{\XXX_n}=\sigma^*\inners_{m_1,m_2}$. Recall that in Section \ref{section: Gauss-Codazzi pseudoRiem} we distinguished cases $(a)$ and $(b)$ according to the signature of $g$, respectively of type $(m_1-1, m_2)$ or $(m_1, m_2-1)$. A local normal vector field $\nu$ has squared norm $+1$ in case $(a)$ and squared norm $-1$ in case $(b)$. 
	
	Observe that in case $(a)$, $d\iota(\nu)=:\nu'$ is a normal vector field for $\overline \sigma$. Similarly, in case $(b)$, $i(d\iota)(\nu)=:\nu'$ is norm-$1$ and is a normal vector field for $\overline \sigma$.
	
	\item The exterior derivative of the local normal vector field defines a shape operator $B$, which coincides to the shape operator $\Psi$ induced by $\overline \sigma$ in case $(a)$, and which satisfies $\Psi=i B$ in case $(b)$.

	\item Recall Theorem \ref{Gauss-Codazzi pseudo-Riem} and the notation $\delta=-1$ in case $(a)$ and $\delta=+1$ in case $(b)$.	
	
	If $\sigma\colon M\to \mathbb F^{m_1, m_2}$ is an admissible immersion with pull-back metric $\I$ and shape operator $B$, then $\overline\sigma\colon M\to \mathbb X_n$ is an admissible immersion with immersion data $(g,\Psi)=(g,\delta B)$. On the other hand, the data $(g, \Psi)$ uniquely determines the immersion of $\widetilde M$ into $\mathbb X_n$ up to ambient isometry. Since also $\iota$ is the unique isometric immersion of $\mathbb F^{m_1,m_2}$ into $\mathbb X_n$ up to ambient isometry, one can conclude the following.
\end{itemize}
\begin{Theorem}
	\label{da X_n a space forms}
	Let $M=M^{n}$, $(g,\Psi)$ be immersion data for a $\pi_1(M)$-equivariant immersion of $\widetilde M$ into $\mathbb X_{n+1}$.
	Assume that $g$ is real, namely that $g|_{TM}$ is pseudo-Riemannian, and has signature $(m_1,m_2)$, with $m_1+m_2=n$. 
	
	Then, if $\Psi$ is real, i.e. if $\Psi$ restricts to a bundle homomorphism $\Psi\colon TM\to TM$, there exists an isometric $\pi_1(M)$-equivariant immersion $\sigma\colon \widetilde M\to \mathbb X_n$ such that $\sigma( \widetilde M)\subset \iota(\mathbb F^{m_1+1,m_2})$.
	
	Similarly, if $i\Psi$ is real, then there exists an isometric $\pi_1(M)$-equivariant immersion $\sigma\colon \widetilde M\to \mathbb X_n$ such that $\sigma( \widetilde M)\subset \iota(\mathbb F^{m_1,m_2+1})$.
\end{Theorem}

Let us state, as a corollary, how the last theorem can be restated for immersions into $\SL$.

\begin{Theorem}
	\label{teo GC passando da pseudoRiem SL}
	Let $\sigma\colon S\to \SL$ be an admissible immersion with pull-back metric $g$ and shape operator $\Psi$.
	\begin{itemize}
		\item $\sigma(S)$ is contained in the image of an isometric embedding of $\Hy^3$ if and only if $g$ is Riemannian and $\Psi$ is real.
		\item 
		$\sigma(S)$ is contained in the image of an isometric embedding of $AdS^3$ if and only if either $g$ is Riemannian and $i\Psi$ is real, or if $g$ has signature $(1,1)$ and $\Psi$ is real.
		\item $\sigma(S)$ is contained in the image of an isometric embedding of $-dS^3$ if and only if either $g$ has signature $(1,1)$ and $i\Psi$ is real, or if $g$ is negative definite and $\Psi$ is real.
		\item $\sigma(S)$ is contained in the image of an isometric embedding of $-S^3$ if and only if $g$ is negative definite and $i\Psi$ is real. 
	\end{itemize}
\end{Theorem}

\section{From immersions into $\Hy^3$ to immersions into $\GGG$}
\label{sec metrica complessa indotta in G}
Given $\sigma\colon \widetilde S\to \Hy^3$ a
$\pi_1(S)$-equivariant immersion with corresponding immersion data $(\I, B)$ on $S$ and normal field $\nu$, one can define the immersion 
\[
G_\sigma\colon \widetilde S\to \GGG
\]
where $G_\sigma(x)$ is the oriented maximal geodesic of $\Hyp^3$ tangent to $\nu(x)$. In the identification $\GGG=\CP^1 \times \CP^1\setminus \Delta$, one has $G_\sigma= (G_\sigma^+, G_\sigma^-)$ corresponding to the endpoints of the geodesic rays starting at $\sigma(x)$ with tangent directions respectively $\nu(x)$ and $-\nu(x)$. We will call $G_\sigma$ the \emph{Gauss map} of $\sigma$. 
We notice that the map $G_\sigma$ is $\pi_1(S)$-equivariant with the same monodromy as $\sigma$. 

We will extensively discuss Gauss maps of hypersurfaces in Part \ref{parte Seppi}. We anticipate the discussion here in order to get several examples of immersions into $\GGG\cong \XXX_2$, and, as a consequence, of complex metrics of constant curvature $-1$ on surfaces.

The aim of this section is to prove the following formula for the pull-back metric for $G_\sigma$.
\begin{Prop}
	\label{Prop metrica in G}
	Let $J$ be the complex structure induced by $h$ and inducing the same orientation as $\nu$. Under the notations above, the pull-back $h=G_\sigma^* \inners_\GGG$ descends to the complex bilinear form on $\C  TS$ defined by
	\begin{align*}
		h &= \I((id - i  JB)\cdot, (id - i  JB)\cdot )=\\
		&= \I- \I(B\cdot, B\cdot)+ i \I( (   JB- B   J)\cdot, \cdot)
	\end{align*}
	which is non degenerate (i.e. a complex metric) in $x$ if and only if $K_{\I} (x)\ne 0$ 
\end{Prop} 	

In order to prove the proposition, we regard $\overline\sigma$ as the 
composition of the map $\zeta_\sigma=(\sigma,\nu):\widetilde S\to T^1\Hy^3$ with the natural projection $\mathrm p:T^1\Hy^3\to\GGG$ that sends each vector to the unique tangent geodesic.

First of all, recall that in Section \ref{sec geodesic flow on T1Xn} we observed that the tangent space to $T_{(x,v)}(T^1\Hy^3)$ is naturally identified with $T_x\Hy^3\oplus v^\perp= \mathcal H\oplus \mathcal V^0$. Under this decomposition, one has that for a given  path $\alpha:(-\epsilon, \epsilon)\to T^1\Hy^3$, which can be written as $\alpha(t)=(x(t),v(t))$ with $v(t)$ being a unit vector field along the path $x(t)$, the identification is given by 
\[\dot\alpha(0) = (\dot x(0),\frac{Dv}{dt}(0)).\] Note that, since $v$ is unitary, by differentiating $\inner{v(t), 
	v(t)}_{\Hy^3}=1$ one gets that the vector $\frac{Dv}{dt}(0)$ is orthogonal to 
$v$, so the correspondence is well-posed.

\begin{Lemma}
	\label{Lemma 1 metrica pull back G}
	Under the above identification, for any $p\in\widetilde S$ and $X\in T_p\widetilde S$ we have 
	\[
	(d_p\zeta_\sigma)(X)=(X,-B(X))\,.
	\]
\end{Lemma}
\begin{proof}
	The proof is trivial by definition of the shape operator. Also see the proof of Proposition \ref{sec:gauss map}.
\end{proof}
\begin{Lemma}
	\label{Lemma 2 metrica pull back G}
	Let us fix $(x,v)\in T^1\Hy^3$ and $w_1,w_2\in v^\perp$. Then
	\begin{align*}
		\big\langle(d_{(x,v)}\mathrm p)(w_1,w_2), &(d_{(x,v)}\mathrm p)(w_1,w_2)\big \rangle _{\GGG}=\\
		&\inner{w_1, w_1}_{\Hyp^3}-\inner{w_2,w_2}_{\Hyp^3}+i (\inner{w_1, v\times w_2}_{\Hyp^3}-\inner{w_2, v\times w_1}_{\Hy^3}),
	\end{align*}
	where $\times$ is the vector product on $T\Hy^3$.  
\end{Lemma}
\begin{proof}
	We consider the half-space model of $\Hy^3=\C\times\R^{+}$. Set 
	$x=(0,1)$ and, in the identification $T_x\Hy^3=\C\times\R$, set $v=(1,0)$. As a result, $v^\perp= Span_\R((i,0),(0,1)) <T_x \Hy^3$, we can canonically see $\partial \Hy^3=\overline \C$ and $\mathrm p(x,v) =(1,-1)$. The tangent space $T_{(1,-1)} \GGG$ can be trivially identified with $T_1 \C\times T_{-1}\C\cong \C \times \C$.

	Let us consider the following $1$-parameter groups of isometries of $\Hy^3$
	\[
	\begin{matrix}
		a(t)= \exp(tV), & b(t)=\exp(itV),\\
		c(t)=\exp(tW), & d(t)=\exp(itW)\,.
	\end{matrix}
	\]
	where $V,W\in\asl$ are defined as  $V=\begin{pmatrix}1/2 & 0\\0&-1/2\end{pmatrix}$, and $W=\begin{pmatrix}0 & -i/2\\i/2 & 0\end{pmatrix}$.
	
	Notice that $a(t)$ and $c(t)$ are groups of hyperbolic transformations with axis respectively $(0,\infty)$ and $(i,-i)$. On the other hand, $b(t)$ and $d(t)$ are pure rotations around the corresponding axes. We have chosen the normalization so that the translation lengths of $a(t)$ and $c(t)$ equal $t$, and so that the rotation angle of $b(t)$ and $d(t)$ is $t$.
	Observe that $x$ lies on all the axes of  $a(t), b(t), c(t), d(t)$.
	
	It follows that $t\to a(t)\cdot v$ is a parallel vector field along the axis of $a(t)$, so  the derivative of $a(t)\cdot(x, v)$ corresponds under the natural identification to the vector $((0,1), (0,0))$. On the other hand, the variation of the endpoints of the family of geodesics $a(t)\cdot (-1,1)\in \GGG$ is given by $(-1,1)\in T_{(-1,1)}\GGG$. Using that the map ${\mathrm p}$ is equivariant under the action of $\PSL$, we conclude that
	\[
	d_{(x,v)}{\mathrm p}((0,1),(0,0))=(-1,1)\,.
	\]
	
	In the same fashion, using $c(t)$ we deduce that
	\[
	d_{(x,v)}{\mathrm p}((-i,0), (0,0))=(-i,-i)\,.
	\]
	
	On the other side, $b(t) \cdot x=x$ for all $t$, therefore one can explcitly compute that
	\[b(t)\cdot v=(d_x b(t)) (v)=\cos t \ v+\sin t\ (0,1)\times v\in T_{x} \Hy^3.\] It follows that the derivative at $t=0$ of $b(t)\cdot(x,v)$ corresponds to $((0,0),(i,0))$.
	We conclude as above that
	\[
	d_{(x,v)}{\mathrm p}((0,0), (i,0))=(-i,i)
	\]
	and analogously for $d(t)$ we get
	\[
	d_{(x,v)}{\mathrm p}((0,0), (0,1))=(1,1).
	\]
	Finally, we can explcitly compute \[d_{(x,v)}{\mathrm p}_{|v^\perp \oplus v^\perp}\colon v^\perp \oplus v^\perp \cong (i\R \times \R^+)\times(i \R \times \R^+)  \to T_{(1,-1)}\GGG\cong \C \times \C \] as
	\begin{equation}
		\label{eq mappa differenziale di d}
	d{\mathrm p}((i\alpha,\beta), (i\gamma, \delta))=((\delta-\beta)+i(\alpha-\gamma),(\delta+\beta)+i(\alpha+\gamma)).
	\end{equation}
	Using the description of the metric as in Equation \ref{eq: metrica su G}, we get that
	\begin{align*}
		||d{\mathrm p}((i\alpha,\beta), (i\gamma, \delta))||_\GG^2&=
		-4\frac{[(\delta-\beta)+i(\alpha-\gamma)][(\delta+\beta)+i(\alpha+\gamma))]}{(1-(-1))^2}=\\
		&= \alpha^2+\beta^2-\gamma^2-\delta^2 - 2i(\alpha\delta-\beta\gamma)=\\
		&= ||(i\alpha,\beta)||^2-||(i\gamma, \delta))||^2 +2i \inner{(i\alpha,\beta), (1,0)\times (i\gamma, \delta))}_{\Hyp^3}
	\end{align*}
	and the thesis follows.
\end{proof}

\begin{proof} [Proof of Proposition $\ref{Prop metrica in G}$]
	The proof follows directly by Lemmas \ref{Lemma 1 metrica pull back G} and \ref{Lemma 2 metrica pull back G}.

	If $(X_1,X_2)$, with $X_2=  J X_1$ is a $\I$-orthonormal frame of eigenvectors for $B$, with corresponding eigenvalues $\lambda_1$ and $\lambda_2$ respectively, the pull-back bilinear form via $G_\sigma$ is described by \[
	h\leftrightarrow \begin{pmatrix}1- \lambda^2_1 & i(\lambda_1 -\lambda_2)\\
		i(\lambda_1-\lambda_2) &1-\lambda^2_2
	\end{pmatrix}
	\]	
	whose determinant is $(1-\lambda_1\lambda_2)^2$: hence, by Gauss equation, $\I$ is a complex metric at $x$ if and only if $K_{\I}(x)\ne0$.
\end{proof}

\chapter{Holomorphic dependence on the immersion data}
\label{Section dipendenza olo}

\section{The main result}
In this chapter we discuss  holomorphic dependence on the immersion data for immersions into $\mathbb X_{n+1}$ and their monodromy. 

Given a smooth manifold $M$  of dimension $n$ and any point $p\in M$, $\C T_p M$ is a complex vector space and provides a natural complex structure to the manifolds $\C T_p ^*M$, $Sym^2 (\C T_p^*M)$ and $End(\C T_p M)=\C T_p M \otimes \C T_p^* M$. 

\begin{Definition}
	\label{def: holomorphic immersion data}
	Given a complex manifold $\Lambda$, we will say that a family of immersion data $\{(g_\lambda, \Psi_\lambda)\}_{\lambda\in \Lambda}$ for $\pi_1(M)$-equivariant immersions of $\widetilde M$ into $\mathbb X_{n+1}$ is \emph{holomorphic} if, for all $p\in M$, the maps
	\begin{equation}
		\label{eq: def g lambda olo}
		\begin{split}
			\Lambda &\to Sym^2 (\C T_p^*M) \\
			\lambda &\mapsto g_\lambda(p) 
		\end{split}
	\end{equation}
	and
	\begin{equation}
		\begin{split}
			\Lambda &\to End(\C T_p M)\\
			\lambda &\mapsto \Psi_\lambda (p)
		\end{split}
	\end{equation}
	are both holomorphic in $\lambda$. 
\end{Definition}

We remark that this definition does not require any complex structure on $M$.

This section is devoted to the proof of the following theorem.
\begin{Theorem}
	\label{Teo dipendenza olomorfa}
	Let $\Lambda$ be a complex manifold and $M$ be a smooth manifold of dimension $n$. 
	
	Let $\{(g_\lambda, \Psi_\lambda)\}_{\lambda\in \Lambda}$ be a holomorphic family of immersion data for $\pi_1(M)$-equivariant immersions $\widetilde M\to \mathbb X_{n+1}$. Then there exists a smooth map
	\[
	\sigma\colon \Lambda \times \widetilde M \to \mathbb X_{n+1}
	\]
	such that, for all $\lambda\in \Lambda$ and $p\in M$:
	\begin{itemize}
		\item $\sigma_\lambda:= \sigma(\lambda, \cdot)\colon \widetilde M \to \mathbb X_{n+1}$ is an admissible immersion with immersion data $(g_\lambda, \Psi_\lambda)$;
		\item $\sigma(\cdot,p)\colon \Lambda \to \mathbb X_{n+1}$ is holomorphic.
	\end{itemize}
	Moreover, for all $\alpha\in \pi_1(M)$, the monodromy map evaluated in $\alpha$
	\begin{align*}
		\Lambda &\to  SO(n+2, \C)\\
		\lambda &\mapsto mon(\sigma_\lambda) (\alpha)
	\end{align*}
	is holomorphic.
\end{Theorem}

We state as a Corollary the analogous result for immersions $\widetilde M^n\to \XXX_n$.

\begin{Cor}
	\label{cor: dipendenza olomorfa stessa dim}
	Let $\Lambda$ be a simply complex manifold, and $M$ be a smooth manifold of dimension $n$. 
	If $\{g_\lambda\}_\lambda$ is a family of complex metrics of constant sectional curvature $-1$ which is holomorphic in $\lambda$ (namely if the function in Equation \ref{eq: def g lambda olo} is holomorphic), 
	then there exists a smooth map
	\[
	\sigma\colon \Lambda \times \widetilde M \to \mathbb X_{n}
	\]
	such that $\sigma(\lambda, \cdot)$ is $\pi_1(M)$-equivariant with pull-back metric $\widetilde g_\lambda$  for all $\lambda\in \Lambda$, and such that $\sigma(\cdot, p)$ is holomorphic for all $p\in \widetilde M$. 
	
	Moreover, the monodromy of $\sigma(\lambda, \cdot)$ is holomorphic in $\lambda$.
\end{Cor}

\begin{Example}
	Let $h$ be a hyperbolic metric on a closed surface $S$ and let $b\colon TS\to TS$ be a $h$-self-adjoint (1,1)-form such that $d^{\nabla_h}b=0$ and $det(b)=1$; one may choose for instance $b=id$. Then, the family $\{(g_z,\Psi_z)\}_{z\in \C}$ defined by
	\[
	\begin{cases}
		g_z= \cosh^2(z) \widetilde h; \\
		\Psi_z= -\tanh(z) \widetilde b
	\end{cases}
	\] 
	is a holomorphic family of $\pi_1(S)$-equivariant immersion data for constant curvature immersions $\widetilde S\to \SL$, with $K_{g_z} =- \frac{1}{\cosh(z)^2}$. Observe that $z\in \R$ corresponds to an immersion data into $\Hy^3$, while $z\in i\R$ corresponds to an immersion data into $AdS^3$ by Theorem \ref{da X_n a space forms}. 
	
	By Theorem \ref{Teo dipendenza olomorfa}, there exists a family of immersions $\sigma_z\colon \widetilde S\to \SL$ with data $(g_z, \Psi_z)$ whose monodromy is a holomorphic function in $z$. 
\end{Example}

\section{Some Lie Theory lemmas}

We start the proof of Theorem \ref{Teo dipendenza olomorfa}.

A great part of the proof of this result is a matter of integrating a distribution
on a manifold. We start from a technical result which extends Lemma \ref{Lemma chiave unicita}.

\begin{Proposition}
	\label{Teorema esistenza foliazione}
	Let $M$ and $\Lambda$ be two simply connected manifolds, $G$ a Lie group with Lie algebra $\Lieg$. 
	
	Consider a smooth family of forms $\{\omega_\lambda\}_{\lambda\in \Lambda}\subset \Omega^1(M, \Lieg)$, namely a smooth map $\Lambda\to \Omega^1 (M, \Lieg)$.
	
	The following are equivalent:
	\begin{itemize}
		\item for all $\lambda\in \Lambda$	\begin{equation}
			\label{eq MC parametrica}
			d\omega_\lambda + [\omega_\lambda, \omega_\lambda]=0;
		\end{equation}
		\item there exists a smooth map $\Phi \colon \Lambda \times M\to G$ such that, for all $\lambda\in \Lambda$, \begin{equation} 
			\label{eq pull-back MC}
			(\Phi(\lambda, \cdot))^* \omega_G= \omega_\lambda,
		\end{equation}
		where $\omega_G$ is the Maurer-Cartan form of $G$.
	\end{itemize}
	Moreover, both $\Phi$ and $\Phi'$ satisfy equation $\eqref{eq pull-back MC}$ if and only if
	\[
	\Phi' (\lambda, p)= \psi(\lambda) \cdot \Phi(\lambda,p)
	\]
	for some smooth $\psi\colon \Lambda\to G$.
	
	In other words, for every fixed $p_0\in M$, smooth $\psi_0\colon \Lambda \to G$ and for every smooth collection of $\Lieg$-valued 1-forms $\{\omega_\lambda\}_{\lambda\in \Lambda}$, there exists a unique $\Phi\colon \Lambda\times M\to G$ such that 
	\[
	\begin{cases}
		&\omega_\lambda= \Phi(\lambda, \cdot)^* \omega_G \\
		&\Phi(\lambda, p_0)=\psi_0(\lambda)
	\end{cases}
	\]
	
	For $\Lambda=\{pt\}$ one has Lemma \ref{Lemma chiave unicita}.
\end{Proposition}

\begin{proof}
	It is clear by the differential equation of the Maurer-Cartan form that Equation \eqref{eq pull-back MC} in the latter statement implies Equation \eqref{eq MC parametrica} in the former one. We prove the opposite implication. 
	
	Endow the trivial bundles $\pi_{\Lambda\times M}\colon \Lambda\times M\times G\to \Lambda \times M$ and $\pi_M\colon  M\times G\to M$ with the natural left action of $G$ given by left translations on the last component. Also let $\pi_G\colon M\times G\to G$ denote the standard projection.
	
	On the tangent bundle of $\Lambda\times M\times G$, let $D=\{D_{(\lambda,p,g)}\}_{(\lambda,p,g)}$ be the distribution
	\begin{equation}
		\label{eq: definizione D}
		D_{(\lambda, p,g)}=\{(\dot \lambda,\dot p, \dot g)\in T_{(\lambda, p,g)} (\Lambda \times M\times G) \ |\ \omega_\lambda(\dot p)= \omega_G(\dot g)  \}= T_{(\lambda,p,g)} \Lambda \oplus D^0_{(\lambda, p, g)} .
	\end{equation}
	where
	\[
	D^0_{(\lambda, p, g)} =\{(0,\dot p, \dot g)\in T_{(\lambda, p,g)} \Lambda \times M\times G \ |\ \omega_\lambda(\dot p)= \omega_G(\dot g)  \}.
	\]
	Both $D$ and $D^0$ are invariant under the left action of $G$ because the Maurer-Cartan form on $G$ is left invariant.
	
	\begin{itemize}
		\item [Step 1.]
		We prove that the distribution $D^0$ is integrable. 
		
		For all $\lambda\in \Lambda$, consider on $M\times G$ the $\Lieg$-valued form
		\[
		\widetilde \omega_\lambda = \pi_G^* \omega_G - \pi_M^* \omega_\lambda \in  \Omega^1(M\times G, \Lieg).
		\]

		Clearly, defining for all $\lambda\in \Lambda$ the map $\iota_\lambda\colon M\times G\to \Lambda \times M\times G$ as the inclusion $(p,g)\mapsto (\lambda, p, g)$, then \[
		D^0= \bigcup_{\lambda\in \Lambda} d(\iota_\lambda)\big( Ker (\widetilde\omega_\lambda) \big).
		\]
		
		Now, for all $X, Y\in \Gamma(Ker({\widetilde\omega_\lambda}))$, one has
		\[
		d \widetilde \omega_\lambda(X, Y)= X(\widetilde \omega_\lambda (Y))- Y(\widetilde \omega_\lambda (X))- \widetilde \omega_\lambda ([X,Y])=-\widetilde \omega_\lambda ([X,Y]),
		\]therefore
		\begin{align*}
			d\widetilde \omega_\lambda (X,Y)&= d \omega_G ({d\pi_G}(X),{d\pi_G}(Y) )- d\omega_\lambda ({d\pi_M}(X), {d\pi_M}(Y))=\\
			&= -[\omega_G({d\pi_G}(X) ),\omega_G({d\pi_G}(Y))]+[\omega_\lambda({d\pi_M}(X) ),\omega_\lambda({d\pi_M}(Y))]=\\
			&= -[\omega_G({d\pi_G}(X)),\omega_G({d\pi_G}(Y))]+[\omega_G({d\pi_G}(X) ),\omega_G({d\pi_G}(Y))]=0
		\end{align*}
		so $[X,Y]\in Ker (\widetilde \omega_\lambda)$: by Frobenius Theorem, the distribution $Ker(\widetilde\omega_\lambda)$ is integrable, call  $\mathcal F_\lambda$ the integral foliation.
		As a result, $D^0$ is integrable with integral foliation
		\[
		\bigcup_{\lambda\in \Lambda} \iota_\lambda (\mathcal F_\lambda).
		\]

		\item[Step 2.] The distribution $D$ is integrable. 
		
		Recall that $D=T\Lambda \oplus D^0$ and that $D^0$ is integrable. Then, clearly, for all $X_1, Y_1\in \Gamma(T\Lambda)$ and $X_2,Y_2\in \Gamma(D^0)$, 
		\[
		[(X_1,X_2), (Y_1,Y_2)]=([X_1,Y_1], [X_2,Y_2]) \in \Gamma(T\Lambda) \oplus \Gamma(D^0)=\Gamma(D)
		\]
		and the integrability of $D$ follows by Frobenius Theorem. Denote with $\mathcal F$ the integral foliation. 
		
		\item[Step 3.] Each leaf of $\mathcal F$ projects diffeomorphically onto $\Lambda\times M$.
		
		Since the distribution $D$ on $\Lambda\times M\times G$ is invariant under the left action of $G$, $G$ has a well-posed left action on the foliation, namely, given a maximal leaf $\Sigma\in \mathcal F$ passing by the point $(\lambda_0, p_0,g_0)$, $L_g(\Sigma)=g \cdot \Sigma$ is the maximal leaf passing by the point $(\lambda_0, p_0, g g_0)$.

		Since $\omega_G(\dot g)=0$ if and only if $\dot g=0$, by Equation \eqref{eq: definizione D} the fibers of $\pi_{\Lambda\times M}$ are transverse to the distribution $D$, hence transverse to any maximal leaf of $\mathcal F$. 
		As a result, for each leaf $\Sigma\in \mathcal F$, the projection map $(\pi_{\Lambda\times M})_{|\Sigma}\colon \Sigma \to \Lambda \times M$ is a local diffeomorphism. 
		
		We show that it is also a proper map. Assume $\{(\lambda_n, p_n, g_n)\}_n\subset \Sigma$ is a sequence such that $(\lambda_n, p_n)\in \Lambda\times M$ converges to some $(p,\lambda)$ in $\Lambda \times M$, then, for any choice of $\bar g\in G$, the maximal leaf $\hat\Sigma$ passing by $(\lambda, p, \bar g)$ projects via a local diffeomorphism to $\Lambda\times M$: so there exists $\hat g$ such that $(\lambda_n, p_n, \hat g g_n)\in \Sigma'$ definitely for $n$, therefore $\Sigma= \hat g^{-1} \hat \Sigma$ and the sequence $(\lambda_n, p_n, g_n)\in \Sigma$ converges in $\Sigma$ to $(\lambda, p, \hat g ^{-1} \bar g)$.
		
		Since $(\pi_{\Lambda \times M})_{|\Sigma}\colon \Sigma \to \Lambda \times M$ is a local diffeomorphism and a proper map, it is a covering map, hence a diffeomorphism being $M$ simply connected. 
	\end{itemize}
	
	For any fixed leaf $\Sigma\in \mathcal F$, one can finally conclude that the map
	\[
	\Phi= \pi_G \circ (\pi_{\Lambda\times M})_{|\Sigma}^{-1}\colon \Lambda\times M \to G
	\]
	is smooth and such that $\Sigma=Graph(\Phi)$, therefore  \[\mathcal F=
	\{g\cdot \Sigma\}_{g\in G}=\{g\cdot Graph(\Phi) \}_{g\in G}= \{Graph(L_g \circ \Phi) \}_{g\in G}. \]
	As a result, for all $(\dot \lambda, \dot p)\in T_{(\lambda, p)}\Lambda\times M$, one has that 
	\[
	\omega_\lambda (\dot p)= \omega_G( d\Phi (\dot \lambda,\dot p) ),
	\]
	in particular, \[\omega_\lambda (\dot p)= \omega_G (d\Phi (0, \dot p) )= \omega_G(d(\Phi(\lambda, \cdot)) (\dot p) ).\]
	\vspace{5mm}
	
	We prove uniqueness. 
	
	Assume $\Phi_1, \Phi_2\colon \Lambda\times M\to G$ both satisfy 
	\[
	(\Phi_1(\lambda, \cdot))^*\omega_G= (\Phi_2(\lambda, \cdot))^*\omega_\lambda
	\]
	for all $\lambda\in \Lambda$, i.e.
	\[
	\omega_\lambda(\dot p)= \omega_G( d(\Phi_1 (\lambda, \cdot))(\dot p))= \omega_G( d(\Phi_2 (\lambda, \cdot))(\dot p)).
	\]
	Then, the graphs of $\Phi_1(\lambda, \cdot)$ and $\Phi_2 (\lambda, \cdot)$ are both integral manifolds for the distribution $D^0$. As a result, since $D^0$ is left invariant, for each $\lambda\in \Lambda$ there exists $\psi(\lambda)\in G$ such that
	\[
	\Phi_2(\lambda,p)=\psi(\lambda)\cdot \Phi_1(\lambda,p)
	\]
	for all $p\in M$. A posteriori, $\psi$ is smooth.
\end{proof}

With the same notations as in the previous theorem, let $\{\omega_\lambda \}_{\lambda\in \Lambda}\subset \Omega^1(M, \Lieg)$ be a smooth family of $\Lieg$-valued 1-forms, fix $p_0\in M$ and let $e\in G$ be the unity. Let $\Phi\colon \Lambda\times M\to G$ be such that
\begin{equation}
	\label{Phi integra omega base}
	\begin{cases}
		&\omega_\lambda= \Phi(\lambda, \cdot )^* \omega_G \\
		&\Phi(\lambda, p_0)=e
	\end{cases}.
\end{equation}
for all $\lambda$.

Denote $\Phi_\lambda=\Phi(\lambda, \cdot)\colon M\to G$.

The following technical Lemma will be useful to compute the derivatives of $\Phi$ with respect to the parameter $\lambda$.

\begin{Lemma}
	\label {Lemma CP}
	With the above notation, assume (\ref{Phi integra omega base}) holds. Let $\lambda_0\in \Lambda$ and $\dot \lambda\in T_{\lambda_0} \Lambda$.
	
	The Cauchy problem 
	\begin{equation}
		\label{Problema di Cauchy chiave}
		\begin{cases}
			&df = Ad(\Phi_{ \lambda_0} ) \circ \big(  \partial_{\dot \lambda} \omega_\bullet(\cdot) \big) \colon TM\to \Lieg \\
			&f(p_0)= 0
		\end{cases}
	\end{equation}
	with unknown quantity $f \colon M \to \Lieg$ has  
	\[f(p)= \partial_ {\dot \lambda} \big(\Phi(\cdot,p) \cdot (\Phi(\lambda_0,p))^{-1} \big)\]
	as unique solution.
\end{Lemma}
\begin{proof}Assume $f_1$ and $f_2$ are both solutions to $(\ref{Problema di Cauchy chiave})$, then $d(f_1-f_2)=0$, hence $f_1-f_2$ is a constant, and $f_1(p_0)=f_2(p_0)$, hence $f_1\equiv f_2$. This proves uniqueness.

	Define $ \xi:\Lambda\times M\to G $ as
	\begin{equation}
		\label{def Sigma}\xi(\lambda, p):=\xi_\lambda(p)=\Phi(\lambda, p)(\Phi(\lambda_0, p))^{-1}.
	\end{equation}
	We need to prove that $\partial_{\dot \lambda} \xi\colon M\to \Lieg$ satisfies (\ref{Problema di Cauchy chiave}).
	
	Clearly,  $\xi(\lambda_0,\cdot)= \xi(\cdot, p_0)\equiv e$, therefore $\partial_{\dot \lambda} \xi (p_0)=0$.

	Differentiating
	\[
	\Phi_\lambda(p)= \xi_\lambda(p) \cdot \Phi_{\lambda_0}(p).
	\]
	one gets
	\[
	d_p(\Phi_\lambda)= d(L_{\xi(\lambda, p)}) \circ d_p(\Phi_{\lambda_0}) + d(R_{\Phi_{\lambda_0}(p)})\circ d_p(\xi_\lambda)
	\]
	Recalling that $\omega_G$ is left-invariant, we deduce that the $\omega_\lambda$'s satisfy
	\begin{equation}
		\label{eq 5.2}
		\begin{split}
			\omega_\lambda-\omega_{\lambda_0}= \Phi_\lambda^* (\omega_G) -\omega_{\lambda_0}&=  \Phi_{\lambda_0}^*(\omega_G) + (\xi_\lambda)^* \Big( Ad(\Phi_{\lambda_0}(p) ^{-1}) (\omega_G) \Big)-\omega_{\lambda_0}=\\
			&= \Phi_{\lambda_0}^*(\omega_G) +  Ad(\Phi_{\lambda_0}(p) ^{-1})\Big( (\xi_\lambda)^*  (\omega_G) \Big)-\omega_{\lambda_0}=\\
			&=Ad(\Phi_{\lambda_0}(p) ^{-1})\Big( (\xi_\lambda)^*  (\omega_G) \Big).
		\end{split}
	\end{equation}
	We will deduce the first equation of $(\ref{Problema di Cauchy chiave})$ with $f=\partial_{\dot \lambda} \xi$ by differentiating Equation $(\ref{eq 5.2})$ in the direction $\dot \lambda$. 
	
	In order to properly differentiate the RHS of Equation $(\ref{eq 5.2})$, we need to extend $\xi$ to
	\begin{align*}\Xi\colon \Lambda\times M\times G&\to M\times G\\
		(\lambda, p, g)&\mapsto (p,\  g\cdot \xi(\lambda, p) ),
	\end{align*}
	so that $\Xi(\lambda, \cdot, \cdot )$ can be seen as a family of diffeomorphisms of $M\times G$ depending on $\lambda$. 
	
	Clearly $(\lambda, p, e)= (p, \xi_\lambda (p))$ and
	\[
	\partial_{\dot \lambda} \Xi\in  \Gamma (T(M\times G))
	\]
	satifies
	\[
	(\partial_{\dot \lambda} \Xi)_{(p,g)}=  (0, (dL_g)(\partial_{\dot \lambda} \xi)_p  ), 
	\]
	hence, defining $\widetilde \omega:= \pi_G^* (\omega_G)$, 
	\begin{equation}
		\label{eq olo 1}
		\widetilde \omega( \partial_{\dot \lambda} \Xi)=  \partial_{ \dot \lambda} \xi\,.
	\end{equation}
	
	We are finally able to differentiate Equation (\ref{eq 5.2}). 
	
	Let $j_e\colon M\to M\times G$ be the section $j_e(p)=(p,e)$, hence $\pi_G\circ \Xi\circ j_e= \xi$.
	Then, by Equation \eqref{eq olo 1}, the derivative of Equation \eqref{eq 5.2} in the direction ${\dot \lambda}$ gives
	\begin{equation}
		\label{eq 5.3}
		\partial_{\dot\lambda} \omega = Ad\Big((\Phi_{\lambda_0} (p) )^{-1}\Big) \big(j_e ^*(\mathcal L _{\partial_{\dot  \lambda} \Xi} \ \widetilde \omega) \big)
	\end{equation}
	where $\mathcal L$ is the Lie derivative. 
	
	Recalling the properties of the Lie derivative with respect to differentiation and contraction, for all $\dot p\in TM$,
	\begin{align*}
		j_e^* \big(\mathcal L_{\partial_{\dot  \lambda}\Xi}\ \widetilde \omega \big) (\dot p) &= \Big(\mathcal L_{\partial_{\dot  \lambda}\Xi}\ \widetilde \omega\Big) (\dot p, 0)= \\
		&=d \widetilde \omega (\partial_{\dot  \lambda}\Xi,  (\dot p, 0)) + d\big( \widetilde \omega (\partial_{\dot  \lambda}\Xi) \big)({( \dot p,0)})=\\
		&=d \omega_G (\partial_{\dot \lambda} \xi, 0)+ d(\ \partial_{\dot  \lambda} \xi ) (\dot p)= \\
		&=d( \partial_{\dot  \lambda} \xi ) (\dot p).
	\end{align*}
	From the last equation and $(\ref{eq 5.3})$, the thesis follows.
\end{proof}

\section{Proof of Theorem \ref{Teo dipendenza olomorfa}}

We are now able to discuss the holomorphic dependence on the immersion data for immersions into $\mathbb X_{n+1}$.

\begin{proof}[Proof of Theorem $\ref{Teo dipendenza olomorfa}$]
	
	We prove that for all $(\lambda_0, p_0)\in \Lambda \times M$ there exists an open neighbourhood $U\times V$ and a $\sigma\colon U\times V\to \mathbb X_{n+1}$ as in the statement of Theorem \ref{Teo dipendenza olomorfa} and that, for any other $\sigma'$ satisfying the thesis like $\sigma$, $\sigma'(\lambda, p)= \phi(\lambda)\circ \sigma(\lambda, p)$ with $\phi\colon U \mapsto Isom(\mathbb X_n)$ holomorphic. Then, since  $M\times \Lambda$ is simply connected, the proof of the Theorem follows by the standard analytic continuation method as in "Step 2" of the proof of Theorem \ref{Teoremone}.

	Choose the neighbourhood $U\times V$ so that there exists on $V$ a $g_{\lambda_0}$-orthonormal frame $(X_1,\dots X_n)$, then the Gram-Schmidt algorithm applies on $U\times V$ to this frame in order to have for all $\lambda\in U$ a $g_\lambda$-orthonormal frame $(e_{1;\lambda}, \dots, e_{n;\lambda})$ on $V$ in the form 
	\[
	e_{i;\lambda}(p)=\sum_j T_{i}^j(\lambda, p) X_j(p)
	\]
	
	where \[T=(T^j_i)_{i,j}\colon U\times V\to Mat(n, \C)
	\] is such that $T(\cdot, p)$ is holomorphic for all $p\in V$.
	Define $\{\theta_\lambda^i\}_i$ as the dual coframe. Both the $e_{i; \lambda}$'s and the $\theta^i_\lambda$'s are (pointwise) holomorphic in $\lambda$.
	
	We recall the steps in the construction of the immersions from the immersion data, as we showed in Section \ref{Section integrazione GC}, and show the holomorphic dependence is preserved at each step.
	
	\begin{itemize}
		\item From $(g_\lambda, \Psi_\lambda)$, locally construct the form $\omega_\lambda \in \Omega^1(V, \Lieonn)$  
		\[
		\omega_\lambda=
		\begin{pmatrix}
			{\scalebox{2} {$\Theta_\lambda$} }& \begin{matrix}
				-\Psi_\lambda^1 & -i\theta_\lambda^1 \\
				\dots &\dots \\
				-\Psi_\lambda^n & -i\theta_\lambda^n
			\end{matrix} \\
			\begin{matrix}
				\Psi_\lambda^1 & \dots & \Psi_\lambda^n\\
				i\theta_\lambda^1 & \dots & i\theta_\lambda^n
			\end{matrix} & 
			\begin{matrix}0\text{    } & 0\text{    } \\ 
				0\text{    } & 0\text{    }
			\end{matrix}
		\end{pmatrix}  
		\]
		where $(\theta^i_\lambda)_{i=1}^n$ is the dual of $(e_{i; \lambda})_{i=1}^n$, $\Theta_\lambda=(\theta^i_{j; \lambda})_{i,j=1}^n$ is the matrix of the Levi-Civita connection forms w.r.t. the $\theta_\lambda^i$'s and $\Psi_\lambda= \Psi_\lambda^i \otimes e_{i; \lambda}$.
		
		Let us show that the map $\lambda\mapsto\omega_\lambda$ is holomorphic in $\lambda$. 
		
		One can easily check in coordinate charts that the $d\theta^i_\lambda$'s are holomorphic in $\lambda$. Defining 
		\[
		d\theta^i=: \alpha^i_{j, k; \lambda} \theta^j \wedge \theta^k,
		\]
		the coefficients $\alpha^i_{j, k; \lambda}$'s are holomorphic in $\lambda$, hence, since one can straightforwardly show that
		\[
		\theta^i_{k;\lambda} := -\alpha^i_{j,k; \lambda} \theta^j,
		\]
		we conclude that $\Theta_\lambda$ is holomorphic in $\lambda$; finally, the $\Psi_\lambda^i$'s are holomorphic in $\lambda$ since $\Psi_\lambda= \Psi_\lambda^i\otimes e_{i;\lambda}$.
		
		\item Up to shrinking $U$ and $V$, assume they are simply connected. By Proposition \ref{Teorema esistenza foliazione}, there exists a smooth map $\Phi\colon U\times V\to \SOnn$ such that
		$\Phi(\cdot, p_0)\equiv e$ and $\Phi(\lambda,\cdot)^*\omega_G= \omega_\lambda$ where $\omega_G$ is the Maurer-Cartan form of $G$. 
		We prove that $\Phi$ is holomorphic w.r.t. $\lambda$. 
		
		Define $\xi$ as in Equation $(\ref{def Sigma})$. For all $\dot \lambda \in T_{\lambda_0}\Lambda$, using Lemma $\ref{Lemma CP}$ and the fact that the $\omega_\lambda$'s vary in a holomorphic way in $\lambda$, one has
		\[
		d(\partial_{i \dot \lambda} \xi) = Ad(\Phi_{\lambda_0})\circ \partial_{i\dot \lambda} {\omega} = Ad(\Phi_{\lambda_0})\circ \big(i\partial_{\dot \lambda} {\omega} \big)= i\  d(\partial_{\dot \lambda} \xi).
		\]
		
		As a result, both $\partial_{i \dot \lambda}\xi$ and $i\partial_{ \dot \lambda}\xi$ solve the Cauchy problem (\ref{Problema di Cauchy chiave}) for $i\dot \lambda$, hence they coincide. 
		
		By $(\ref{def Sigma})$, one has that for all $p\in V$ and $\lambda\in U$
		\begin{align*}
			(\partial_{i\dot \lambda}\Phi) (\lambda_0, p) &= d(L_{\Phi(\lambda_0,p)})( (\partial_{i\dot \lambda}\xi)(\lambda_0,p)) = d(L_{\Phi(\lambda_0,p)})( i(\partial_{\dot \lambda} \xi)(\lambda_0,p))  = \\
			&=i (\partial_{\dot \lambda}\Phi) (\lambda_0, p).
		\end{align*} 
		
		\item By the proof of Theorem \ref{Teoremone}, defining the map
		\[
		\sigma\colon 
		U\times V\to \mathbb X_{n+1}
		\]
		by \[ \sigma (\lambda, p)= \Phi(\lambda,p) \cdot e, \]
		one has that $\sigma(\lambda, \cdot):V \to \mathbb X_{n+1}$ is an immersion with immersion data $(g_\lambda, \Psi_\lambda)$. Finally, $\sigma$ is holomorphic with respect to $\lambda$ since $\Phi$ is.
		\item By Theorem $\ref{Teoremone}$, if $\sigma'$ satisfies the statement of Theorem \ref{Teo dipendenza olomorfa} like $\sigma$, then $\sigma'(\lambda,p)= \phi(\lambda)\cdot \sigma(\lambda,p)$ for a unique function $\phi\colon U\to Isom_0(\mathbb X_n)\cong SO(n+2,\C)$. 
		
		We prove that $\phi$ is holomorphic.
		
		Fix $p\in V$ and a basis $X_1,\dots, X_n\in T_p V$. Then, the functions $A, A'\colon U\to GL(n+2, \C)$ defined by
		\begin{equation*}
			\begin{split}
				A(\lambda)&= \begin{pmatrix} d_p(\sigma(\lambda, \cdot)) (X_1)\ \ | & \dots \ \ |& d_p(\sigma(\lambda, \cdot)) (X_n)\ \  |& d_p(\sigma(\lambda, \cdot)) (\nu_{\sigma(\lambda, \cdot)})\ \ |& \sigma(p) \end{pmatrix}\\ &\text{and}\\ A'(\lambda)&= \begin{pmatrix} d_p(\sigma'(\lambda, \cdot))(X_1)\ \ |& \dots\ \  |& d_p(\sigma'(\lambda, \cdot)) (X_n)\ \ |& d_p(\sigma'(\lambda, \cdot))(\nu_{\sigma'(\lambda, \cdot)})\ \ |& \sigma'(p) \end{pmatrix} 
			\end{split}
		\end{equation*}
		are holomorphic with respect to $\lambda$: this is an elementary consequence of the assumptions on $\sigma$ and $\sigma'$ and of the fact that vector fields on $U$ and on $V$ commute with each other when seen as vector fields on $U\times V$. By definition of $\phi$, $\phi(\lambda)= A'(\lambda) (A(\lambda))^{-1}$, hence, by the chain rule, $\phi$ is holomorphic.

	\end{itemize}
	
	By the analytic continuation argument mentioned above, one gets that for all $\lambda$ there exists a neighbourhood $U\subset \Lambda$ such that there exists $\sigma\colon U\times \widetilde M\to \mathbb X_{n+1}$ as in the statement. 
	
	Moreover, by fixing a point $p_0\in \widetilde M$ and setting $\Phi(\lambda, p_0)\equiv e_G$ for all $\lambda$, one gets that the 1-jet of the immersion $\sigma(\lambda, \cdot)$ at $p_0$ is the same for all $\lambda$. As a result, different $\sigma$'s defined on domains of the form $U\times\widetilde M$ with the same 1-jet on $p_0$ glue together, proving that $\sigma$ extends globally on $\Lambda \times \widetilde M$.
	
	It is now simple to check that the monodromy of $\sigma(\lambda, \cdot)$ is holomorphic with respect to $\lambda$. Indeed,
	for all $\alpha\in \pi_1(M)$, the map 
	\[mon_\alpha \colon \Lambda\to SO(n+2,\C) \]
	is defined by 
	\[\sigma( \lambda, \alpha(p))= mon_\alpha (\lambda)\circ \sigma (\lambda,p):\]
	since both $\sigma$ and $\sigma \circ \alpha$ satisfy the conditions in the statement of the Theorem for the data $(g,\Psi)$, by the previous step $mon_\alpha$ is holomorphic in $\lambda$.

\end{proof}

\begin{proof}[Proof of Corollary $\ref{cor: dipendenza olomorfa stessa dim}$]
	Since $g_\lambda$ is holomorphic in $\lambda$, the set $\{(g_\lambda, 0)\}_{\lambda\in \Lambda}$ is a holomorphic family of immersion data for $M$. Then, Theorem \ref{Teo dipendenza olomorfa} ensures that there exists $\overline\sigma\colon \Lambda\times \widetilde M\to \XXX_{n+1}$ with $\overline \sigma(\cdot, p)$ holomorphic and with $\overline \sigma(\lambda, \cdot)$ an immersion with immersion data $(g_\lambda, 0)$. 
	
	Denoting $\nu_\lambda$ as the normal vector field of $\sigma_\lambda$, one has that $\nu_\lambda$ is a constant independent from $p\in \widetilde M$, as we showed in the proof of Theorem \ref{teoremone stessa dimensione}.
	
	Now, one can construct a map $\phi\colon \Lambda \to \SO(n+2,\C)$ such that $\phi(\lambda) (\nu_\lambda)\equiv \begin{pmatrix}
	 0 & \dots & 0 & 1
	\end{pmatrix}$: indeed, clearly one can construct $\phi$ locally on $\Lambda$ by extending $\nu_\lambda$ to a basis of $\C^{n+2}$, then applying Gram-Schmidt algorith to get a orthonormal basis, and then choosing $\phi(\lambda)$ as the unique isometry of $(\C^{n+2}, \inners_0)$ that sends this basis to the canonical basis (one has $\phi(\lambda)\in \SO(n+2)$ up to rearranging the basis); to see that $\phi$ can be chosen globally, apply the standard analytic continuation methods we already used in the proofs of Theorem \ref{Teo dipendenza olomorfa} and of Theorem \ref{Teoremone}.
	
	Then, under the canonical inclusion $\XXX_n\to \XXX_{n+1}$ given by $p\mapsto (p,0)$, the map $\sigma\colon \Lambda\times \widetilde M \to \XXX_n$ defined by $\sigma(\lambda, p)= \psi(\lambda)\circ \sigma (\lambda, p)$ is the researched function. The monodromy of $\sigma(\lambda, \cdot)$ is holomorphic in $\lambda$ since $SO(n+1, \C)$ is a complex submanifold of $SO(n+2, \C)$.
\end{proof}

\section{An application: the complex landslide is holomorphic}
\label{Subsection complex landslide}

In this section, we would like to provide a direct application of Theorem \ref{Teo dipendenza olomorfa}: we show an alternative proof for the fact that the holonomy of the complex landslide, defined in \cite{cyclic}, is holomorphic. 

We briefly recall some basic notions on projective structures, the notions of landslide, smooth grafting and complex landslide. One can use as main references \cite{Dumas:complexprojectivestructures} and \cite{cyclic}.
\vspace{5mm}

\subsection{Projective structures and holonomy}
\label{subsec proj structures and holonomy}
Let us quickly recall the definition of $\PSL$-character variety, projective structures and of their holonomy. For a more complete survey on these topics, see \cite{Dumas:complexprojectivestructures}.

Let $S$ be an oriented closed surface of genus $\mathrm g\ge 2$.

Let us define the character variety of $\PSL$. The general definition of character variety is more complex but for immersions of $\pi_1(S)$ into $\PSL$ it has a simple description. The reader might refer to \cite{charctervariety1, charactervariety2, charactervariety3} for more detailed treatments.

Fixing a set of generators $\{\alpha_1, \dots, \alpha_{2\mathrm g}\}$, the set \[\mathcal R(\pi_1(S), \PSL):=\mathrm{Hom}(\pi_1(S), \PSL)\] can be seen naturally as a subset of $\PSL^{\times(2\mathrm g)}$, and inherits an affine complex structure.

Define on $\mathcal R$ the equivalence relation 
\[
\rho_1\sim \rho_2 \qquad \text{if and only if} \qquad tr^2(\rho_1(\alpha))=tr^2(\rho_2(\alpha)) \quad \text{for all $\alpha\in \pi_1(S)$}
\]
with $tr^2$ denoting the square of the trace of a matrix. This equivalence relation induces a quotient 
\begin{equation}
	\label{eq: def character variety}
	\mathcal X (S) := {\mathcal R(\pi_1(S), \PSL)} \sslash {\PSL}
\end{equation}
that we denote as \emph{the character variety of} $\PSL$. For a suitable choice of a set of generators of $\pi_1(S)$, the evaluation of $tr^2$ provides an injective map into $\C^N$: the image is an affine variety, and this allows us to say that $\mathcal X(S)$ is an algebraic variety.

\vspace{5pt}
Let us now focus on complex projective structures.
In terms of $(G, X)$-structures, a \emph{(complex) projective structure} on $S$ is a $(\PSL, \CP^1)$-structure on $S$. A projective structure induces a complex structure and an orientation, we will stick to projective structures compatible with the orientation on $S$. As a $(\PSL, \CP^1)$-structure, a projective structure is determined by a $(\pi_1(S), \PSL)$-equivariant developing map $\widetilde S \to \CP^1$, which is unique up to post-composition with elements in $\PSL$. We therefore define
\begin{equation*}
	\begin{split}
		\widetilde{\mathcal P}(S)&:= \{\text{projective structures on $S$}\}=\\
		& =\faktor{\left \{(f, \rho)\ \bigg|\ \begin{split}
				&\rho\colon \pi_1(S)\to \PSL,\\
				&f\colon \widetilde S \to \CP^1 \text{ $\rho$-equiv. orientation-preserving local diffeo} 
			\end{split}
			\right \}}{\PSL} \\
		\mathcal P(S) &:=\faktor{\widetilde{\mathcal P} (S)}{\text{Diff}_0(S)}
	\end{split}
\end{equation*}

Equip the set of pairs $\{(f,\rho)\}$ with the compact-open topology, and $\widetilde {\mathcal P}(S)$ and $\mathcal P(S)$ with the quotient topology. The space $\mathcal P(S)$ is a topological manifold, over which one can define  a smooth and complex structure - compatible with the topology - via the Schwarzian parametrization (e.g. see \cite{Dumas:complexprojectivestructures}).

Now, one can define the holonomy map
\begin{align*}
	\widetilde {Hol} \colon \widetilde{\mathcal P}(S)&\to \mathcal X(S)\\
	[(f,\rho)] &\mapsto [\rho]\ ,
\end{align*}
which passes to the quotient to a map
\[Hol\colon 	{\mathcal P}(S) \to 	\mathcal{X}(S).\]
In fact, one can prove that the image of $Hol$ lies into a Zariski-open subset of $\mathcal X (S)$ corresponding to representations $\pi_1(S)\to \PSL$ whose action on $\Hyp^3$ and $\partial\Hyp^3$ does not fix neither an internal point, nor a point in the boundary, nor an unoriented geodesic (\cite{charactervariety2}). As a result, the image of $Hol$ is contained in a smooth complex region of $\mathcal X(S)$. In this context, the following Theorem holds.

\begin{Theorem} (Hejhal \cite{Hejhal}, Hubbard \cite{Hubbard},  Earle \cite{Earle})
	\label{Hol e diffeo}
	The holonomy map $Hol$ is a local biholomorphism. 
\end{Theorem}

\subsection{The complex landslide}
\label{sec complex landslide}
Let $\Tau (S)$ be the Teichm\"uller space of $S$. 

In \cite{cyclic}, the authors define two smooth functions satisfying several interesting properties: the \emph{Landslide map}
\[
L\colon \Tau(S)\times \Tau(S)\times \R \to \Tau(S) \times \Tau(S)
\]
and the \emph{Smooth Grafting map}
\[
SGr\colon \Tau(S) \times \Tau(S)\times \R^+ \to \mathcal P(S)
\]
which turn out to have several interesting geometric properties and are related respectively to the earthquake map and to the grafting map. (For brevity we are not going to discuss earthquakes and grafting in this thesis.)

There are several equivalent definitions of $L$ and $SGr$, we mention the ones most fitted to our pourposes.

Let $h$ be a hyperbolic metric. As a consequence of some works by Schoen \cite{Schoen} and Labourie \cite{Labourie}, it has been proved that the map
\begin{equation}
	\label{eq: tensori b h-regular}
\begin{split}
	\left\{
	\begin{array}{l}
		b\colon TS\to TS\\
		\text{bundle isomorphism}\\
		\text{such that:}
	\end{array}\Bigg|
	\begin{array}{l}
		\text{ $h$-self adjoint,} \\
		d^{\nabla_h} b=0, \\
		det(b)=1
	\end{array} \right\}
	&\xrightarrow{\sim}  \Tau(S) \\
	b &\mapsto [h(b\cdot, b\cdot)]
\end{split} 
\end{equation}
is well-posed and bijective. We will denote the tensors in the domain as \emph{$h$-regular}.

\begin{itemize}
	\item Given $([h], [h(b\cdot, b\cdot)])\in \Tau(S)\times \Tau(S)$ with $b$ $h$-regular, define for all $t\in \R$ the operator \[\beta_t:= \cos(\frac t 2)id + \sin(\frac t 2)J b\] where $J$ is the complex structure induced by $h$. The \emph{landslide map} is defined as
	\[
	L([h], [h(b\cdot, b\cdot)], t)=
	\bigg( \Big[h\Big(\beta_t\ \cdot, \beta_t \cdot \Big)\Big], \Big[h\Big(\beta_{t+\pi}\ \cdot, \beta_{t+\pi}\ \cdot \Big)\Big]  \bigg)\]
	\item Given $([h], [h(b\cdot, b\cdot)])\in \Tau(S)\times \Tau(S)$, observe that for all $s\in\R$ the pair
	\begin{equation}
		\label{eq dati immersione SGr}
		(\I_s, B_s):=\bigg(\cosh^2\Big(\frac s 2\Big) h, - \tanh\Big(\frac s 2\Big) b \bigg)
	\end{equation}
	satisfies Gauss-Codazzi equations for $\Hyp^3$. As a result, for all $s$ there exists a unique $\pi_1 (S)$-equivariant isometric immersion
	\[
	\sigma_s\colon (\widetilde S, \widetilde g_s) \to \Hy^3 
	\]
	with shape operator $-\tanh\big( \frac s 2\big) \widetilde{B_s}$.

	We will discuss convex hypersurfaces in $\Hyp^{n+1}$ later on in Definition \ref{Def:convex immersion}. For now, let us just mention that $\sigma_s$ is convex, i.e.  the tensor $b_s$ is negative definite, if and only if $s> 0$, and, as a result of convexity (see Lemma \ref{lemma:convex gauss map diffeo}), for $s>0$ the map $\sigma_s$ induces a map
	\[
	G^+_{s}\colon \widetilde S \to \partial \Hy^3 \cong \CP^1
	\]
	which is a diffeomorphism onto its image and which is defined by $G^+_{s} (p)$ being the endpoint of the geodesic ray starting at $\sigma(p)$ with tangent direction $\nu(p)$. The map $G^+_{s}$ is equivariant inducing the same monodromy into $\PSL$ as $\sigma_s$, thus it defines a projective structure on $S$. 
	
	In conclusion, the \emph{smooth grafting} is defined as the map
	\[
	Sgr([h], [h(b\cdot, b\cdot)], s)= [G_s^+].
	\]
\end{itemize}
Finally, in the same article, the authors define a \emph{complex landslide} as the composition of landslides and smooth grafting: more precisely, for every hyperbolic metric $h$ and $h$-regular tensor $b$, they define a map
\begin{align}
	P_{h,b}\colon H&\to \mathcal P(S)\\
	z=t+is &\mapsto Sgr\bigg( L\Big([h], [h(b\cdot, b\cdot)], t\Big), s \bigg)
\end{align}
where $H\subset \C$ is the upper half-plane. 

\begin{Theorem}[Bonsante - Mondello - Schlenker]
	$P_{h,b}$ is holomorphic.
\end{Theorem}
\vspace{5mm}
The argument for the proof in \cite{cyclic} uses analytic methods. Here we provide an alternative proof using Theorem \ref{Teo dipendenza olomorfa} and Proposition \ref{Prop metrica in G}.

\begin{Theorem} 
	\label{Teo alternativo landslide}
	The holonomy of the projective structure $P_{h,b}(z)$ is equal to the monodromy of the immersion $\widetilde S\to \GGG$ with induced complex metric 
	\[
	g_z = h \bigg( (\cos(z)id -  \sin(z) Jb)\ \cdot,  (\cos(z)id -  \sin(z) Jb)\cdot \bigg)
	\]
	with constant curvature $-1$.
	
	As a consequence of Theorem \ref{Teo dipendenza olomorfa} and Theorem \ref{Hol e diffeo}, $P_{h,b}$ is holomorphic.
\end{Theorem}

\begin{proof}
	In order to make the notation more simple, observe that there exists a natural right action of the group $\Gamma( GL( TS))$, with the operation of composition pointwise, over the space $\Gamma(Sym^2(T^*S))$ given by
	\begin{equation}
		\label{eq: notazione twist tensori}
		\begin{split}
		\Gamma(Sym^2(T^*S)) \times\Gamma( GL( TS)) &\to	\Gamma(Sym^2(T^*S))\\
		\tau \ast \alpha &= \tau(\alpha \cdot, \alpha \cdot)\ ,
		\end{split}
	\end{equation}
	and the same holds replacing $TS$ with $\C TS$ and $T^*S$ with $\C T^* S$, and we will adopt the same notation.

	Let $z=t+is$ and let us try to compute the holonomy of $P_{h,b}(z)= Sgr\bigg( L\Big([h], [h(b\cdot, b\cdot)], t\Big), s \bigg)$. In order to do it, we need to compute some elements.
	
	\begin{itemize}
		\item 
		For all $x\in \R$, define $\bar h_x:= h \ast \beta_x$. Then
		\[L\Big([h], [h(b\cdot, b\cdot)], t\Big)= \bigg([\bar h_{t}], [\bar h_{t+\pi}] \bigg). \] 
		
		\item 
		By a straightforward computation on a $h$-orthonormal frame diagonalizing $b$, one gets that $JbJb=-id$, and as a consequence that
		\begin{align*}
			\beta_t \circ \beta_s&= ( \cos\Big(\frac t 2\Big) id +sin\Big(\frac t 2\Big) Jb )\circ ( \cos\Big(\frac s 2\Big) id +sin\Big(\frac s 2\Big) Jb )=\\
			&=\cos\Big(\frac{t+s} 2 \Big)id + \sin\Big(\frac{t+s} 2\Big)Jb=\\
			&= \beta_{t+s}
		\end{align*}
		for all $t,s\in \R$. Hence, one has that
		\begin{align*}
			\bar h_{t+\pi}&= h \ast \beta_{t+\pi}= (h \ast \beta_\pi)\ast \beta_t=\\
			&= (h\ast b) \ast \beta_t= (h\ast \beta_t)\ast (\beta_{-t}b \beta_t)=\\
			&=\bar h_t \ast  (\beta_{-t}b \beta_t).
		\end{align*}
		Let us define $\bar b_t:= (\beta_{-t}b \beta_t)$. One can directly prove that $\bar b_t$ is self-adjoint and Codazzi for $\bar h_t$ and has determinant $1$.
		\item By a direct computation, one gets that the complex structure of $\bar h_t$ is \[
		\bar J_{t}= \beta_{-t}J\beta_{t}.
		\]

	\end{itemize}
	In order to compute the monodromy of $Sgr([\bar h_{-t}], [\bar h_{-t+\pi}], s)$, consider, as in Equation \eqref{eq dati immersione SGr}, the immersion into $\Hy^3$ with immersion data $(\I_s, B_s):=(cosh^2(s)\bar h_{t}, \tanh(s)\bar b_{t})$ and apply Proposition \ref{Prop metrica in G} to conclude that it has the same monodromy as the induced immersion into $\GGG$ with pull-back metric
	\begin{align*}
		h_z:=\I_s \ast( id- i J_{\I_s} \psi_s)&=\cosh^2\Big(\frac s 2\Big)\  \overline h_t \ast \Big(id + i \tanh\Big(\frac s 2 \Big) \overline J_t \overline b_t\Big)=\\
		&= \bar h_t \ast\Big(\cosh\Big(\frac s 2\Big) id + i \sinh\Big(\frac s 2 \Big) \overline J_t \overline b_t \Big)=\\
		&= \bar h_t \ast\Big(\cosh\Big(\frac s 2\Big) id + i \sinh\Big(\frac s 2 \Big) \beta_{-t}Jb \beta_t\Big)=\\
		&=\bar h_t \ast \bigg(\beta_{-t} \circ \Big(\cosh\Big(\frac s 2\Big) id + i \sinh\Big(\frac s 2\Big) Jb \Big)\circ \beta_t  \bigg)=\\
		&=h\ast \bigg( (\cosh\Big(\frac s 2\Big) id + i \sinh\Big(\frac s 2\Big) Jb )\circ \beta_t \bigg)=\\
		&= h \ast \bigg( \Big( \cos\Big(i\frac s 2 \Big) id + \sin\Big(i\frac s 2 \Big) Jb \Big)\circ \Big( \cos\Big(\frac t 2 \Big) id + \sin\Big(\frac t 2 \Big) Jb \Big) \bigg) =\\
		&= h\ast \Big( \cos\Big(\frac z 2 \Big)id + \sin\Big(\frac z 2 \Big) Jb \Big).
	\end{align*}
	
	By construction, each $g_z$ is a complex metric of curvature $-1$, corresponding to an equivariant immersion of $\widetilde S$ into $\mathbb \GGG$ for all $z$.
	
	Clearly, for any $X, Y\in \Gamma(TS)$, \[g_z(X,Y)= \cos^2(z)h(X,X)+ \sin^2(z) h(bX,bY)+ \sin z \cos z ( h(X, JbY)+ h(JbX, Y ))\] is holomorphic in $z$, hence the family of complex metrics $\{g_z\}_{z\in \C}$ is holomorphic. 
	By Corollary \ref{cor: dipendenza olomorfa stessa dim}, there exists a map $\sigma\colon \C\times \widetilde S\to \GGG$ with $\sigma=\sigma(z,p)$ holomorphic in $z$ and with $\sigma_z$ providing pull-back metric $g_z$. 
	The monodromy of $g_z$, hence the holonomy of $P_{h,b}$, is holomorphic in $z$; by Theorem \ref{Hol e diffeo}, the projective structure $P_{h,b}(z)$ depends on $z$ holomorphicly.
\end{proof}

\chapter{A Gauss-Bonnet Theorem and a Uniformization Theorem for complex metrics}
\label{Section uniformization}

\section{Positive complex metrics}

In this section we suggest an intrinsic study of complex metrics on surfaces. To this aim, we will introduce a natural generalization of the concept of complex structure for surfaces and we will present a uniformization theorem in this setting.

In the following, $S$ denotes an oriented surface.

The natural inclusion $TS\hookrightarrow \C TS$ factors to a bundle inclusion
\[
\begin{tikzcd}
	\Proj_\R (TS) \arrow [rr, hook]\arrow[rd] & & \Proj_{\C} (\C TS) \arrow[ld]\\
	& S & 
\end{tikzcd}.\]

Fiberwise, for every $p\in S$, $\Proj_\R (T_p S)$ is mapped homeomorphicly into a circle in $\Proj_\C(\C T_p S)\cong S^2$ whose complementary is the disjoint union of two open discs. The conjugation map on $\C TS$ descends to a bundle isomorphism on $\Proj_\C(\C T S)$ that fixes $\Proj_\R(TS)$ and that swaps the two discs fiberwise. 
\begin{Definition}	
	\label{Def bicomplex structure}
	A \emph{bicomplex structure} on a surface $S$ is a tensor \textbf{J} $\in \Gamma(\C T^* S \otimes \C TS)$ such that
	\begin{itemize}
		\item $\textbf{J}^2 = \textbf{J} \circ \textbf{J}=-id_{\C TS}$; as a result, $\textbf{J}$ is diagonalizable with eigenvalues $\pm i$ and eigenspaces $V_i (\textbf{J}), V_{-i}(\textbf{J})\subset \C TS$ with complex dimension $1$.
		\item for all $p\in S$, the eigenspaces $V_i (\textbf J_p)$ and $V_{-i}(\textbf J_p)$ of $\textbf J_p$ have trivial intersection with $T_p S$ and are such that the points $\Proj_\C(V_{i}(\textbf J_p))$ and $\Proj_\C(V_{-i}(\textbf J_p))$ lie in different connected components of $\Proj_\C(\C T_p S)\setminus \Proj_\R (T_pS)$.
	\end{itemize}
\end{Definition}

\begin{Remark}
	Clearly complex structures extend to bicomplex structures. Observe that a bicomplex structure $\textbf{J}$ is not a complex structure for $S$ in general, since it may not restrict to a section of $TS \otimes T^*S$. In other words, $\textbf J(\overline X)\ne \overline {\textbf J(X)}$ in general.
	
	Nevertheless, $\textbf{J}$ induces two complex structures $J_1, J_2$ defined by the conditions
	\[
	V_i (\textbf{J})= V_i (J_2) =  \overline{V_{-i} (J_2)} \qquad \text{and} \qquad V_{-i} (\textbf{J})=V_{-i}(J_1)= \overline{V_{i} (J_1)}
	\] 
	which totally characterize $J_1$ and $J_2$.
	
	Also observe that $J_1$ and $J_2$ induce the same orientation. Indeed, representing the set of complex structures on $S$ as $C(S)\sqcup C(\overline S)$, the map 
	\begin{align*}
		C(S)\sqcup C(\overline S)  &\to \Proj_\C (\C T_p S)\setminus \Proj_\R (T_pS) \\
		J_0 &\mapsto \Proj_\C (V_i(J_0)_p)
	\end{align*}
	is continuous, hence it sends $C(S)$ and $C(\overline S)$ onto the two distinct connected component of the image.
	
	As a result, every bicomplex structure induces an orientation. An orientation of $S$ fixed, denote with $BC(S)$ the set of orientation-consistent bicomplex strucures. With the notations above, we therefore have a bijection 
	\begin{equation}
		\label{BC(S) to C(S)xC(S)}
		\begin{split}
			BC(S) &\xrightarrow\sim C(S)\times C(S)\\
			\mathbf {J} &\mapsto (J_1, J_2). 
		\end{split}
	\end{equation}
	Endow $BC(S)$ with the pull-back topology for which, in particular, it is connected.
\end{Remark}

\begin{Definition}
	\label{Def positive metrics}
	Let $g$ be a complex metric on $S$.
	\begin{itemize}
		\item	Denote the set of isotropic vectors of $g$ as
		\[
		ll(g):=\{v\in \C T S \ |\ g(v,v)=0 \}\setminus \{0_S\}
		\]
		and define $ll(g_p):=ll(g) \cap \C T_p S$. Notice that $ll(g_p)$ is the union of two complex lines with trivial intersection.
		\item
		A \emph{positive complex metric} on a surface $S$ is a complex metric $g$ such that the two points $\Proj_\C(ll(g_p))$ lie in different connected components of $\Proj_\C(\C T_p S)\setminus \Proj_\C (T_pS)$.
		\item Denote with $CM^+ (S)$ the space of positive complex metrics on $S$ with the usual $C^{\infty}$-topology for spaces of sections.
	\end{itemize}
\end{Definition}

For the rest of the chapter, we will mostly focus on positive complex metrics. The main interest lies in the fact that positive complex metrics locally determine a bicomplex structure.

\begin{Proposition}
	Let $g$ be a positive complex metric on a surface $S$.
	\begin{itemize}
		\item[1)] For all $p\in S$ there exist exactly two opposite endomorphisms $\textbf{J}_p, -\textbf{J}_p\in End(\C T_pS)$ such that 
		\begin{equation}
			\label{compatibility}
			\begin{cases}
				g_p(\textbf{J}_p\cdot,\textbf{J}_p \cdot)= g_p(\cdot, \cdot),\\
				\textbf{J}_p^2=-id_p
			\end{cases}.
		\end{equation}
		The two solutions $\textbf{J}_p, -\textbf{J}_p$ depend smoothly on $p$ and on $g_p$.
		\item[2)] For all $p_0\in S$, there exists a local neighbourhood $U=U(p_0)$ over which $\textbf{J}=\{\textbf{J}_p\}_{p\in U}$ and $-\textbf{J}=\{-\textbf{J}_p\}_{p\in U}$ define two bicomplex structures on $U$ and
		\[	
		ll(g_p)=V_i(\textbf{J}_p)\cup V_{-i} (\textbf{J}_p).
		\] 
		We will say that $\textbf{J}$ and $-\textbf{J}$ are (locally) \emph{compatible} with $g$.
		\item [3)] If $\textbf{J}$ is compatible with $g$ and $\textbf{J}\equiv (J_1,J_2)$ in correspondence $(\ref{BC(S) to C(S)xC(S)})$, then locally 
		\[
		g= f\ dw_1 \cdot d\overline w_2
		\]
		where $f$ is a $\C$-valued smooth function and $w_1$, $w_2$ are local holomorphic coordinates for the complex structures $J_1$ and $J_2$ respectively. 
		\item[4)]
		If $\textbf{J}$ is compatible with $g$ on $U$, for any norm-1 local vector field $X$ on $U$, a local $g$-orthonormal frame on $U$ is given by $\{X,\textbf JX\}$. 
	\end{itemize}
\end{Proposition}
\begin{proof}
	\begin{itemize}
		\item[1)]  With respect to two local $g$-orthonormal vector fields $X_1$ and $X_2$, $\textbf{J}_p$ is represented by a matrix $A_p\in Mat(2,\C)$ satisfying:
		\[
		\begin{cases}
			A_p^2=-I_2\\
			^tA_p A_p= I_2
		\end{cases},
		\]
		which admits exactly two opposite constant solutions for $A_p$, namely $A_p=\pm \begin{pmatrix} 0 &1\\-1 &0 \end{pmatrix}$, so the two solutions $\textbf J_p$ and $-\textbf J_p$ are smooth.
		\item[2)] The equation follows directly from $(\ref{compatibility})$. Since $g$ is positive, $\Proj_\C(V_{i}(\textbf{J}_p))$ and $\Proj_\C(V_{-i}(\textbf{J}_p))$ lie in different connected components of $\Proj_\C(\C T_p S)\setminus \Proj_\R (T_pS)$, so $\textbf{J}$ and $-\textbf{J}$ are bicomplex structures.
		
		\item[3)] A $\C$-bilinear form on a 2-dimensional $\C$-vector space is determined up to a constant by its isotropic directions. For local holomorphic coordinates $w_1,w_2$ for $J_1$ and $J_2$, we have
		\[
		V_i (\textbf{J})= V_i (J_2)= \C \frac{\partial}{\partial w_2} \qquad V_{-i} (\textbf{J})= V_{-i} (J_1)= \C \frac{\partial}{\partial \overline w_1}, 	
		\]
		hence the locally defined complex metrics $g$ and $ dw_1 \cdot d\overline  w_2= \frac 1 2 (dw_1 \otimes d\overline  w_2+ d\overline  w_2\otimes dw_1)$ are pointwise equal up to a constant since they have the same isotropic directions. Explicitly, we have 
		\[
		g= g\bigg(\frac {\partial}{\partial w_1}, \frac {\partial}{\partial \overline w_2} \bigg) dw_1 d\overline w_2
		\]
		\item[4)]  The thesis follows directly by the compatibility of $\textbf{J}$ with $g$.
	\end{itemize}
\end{proof}

We just showed that positive complex metrics locally define a bicomplex structure. In fact, the obstruction to the existence of a global compatible bicomplex structure is only a matter of orientability.

\begin{Theorem}
	\label{Teorema orientazione}
	Let $g$ be a positive complex metric on a surface $S$. 
	
	The following are equivalent:
	\begin{itemize}
		\item $S$ is orientable;
		\item there exists a global positive bicomplex structure $\textbf{J}$ on $S$ compatible with $g$;
		\item there exists a complex $1$-form $dA\in \Omega^2 (\C TS, \C)$, unique up to a sign, such that $\|dA\|^2_g=1$, i.e. for all local $g-$orthonormal frame $(X_1,X_2)$ one has $\nobreak{dA(X_1,X_2)\in\{-1,1\} }$.
		
		We will call $dA$ the \emph{area form} of $g$.
	\end{itemize}
\end{Theorem}
\begin{proof}
	As we showed, a bicomplex structure induces two complex structures which induce an orientation on $S$. Conversely, given an orientation on $S$, locally there exists a unique bicomplex structure compatible with $g$ and inducing the same orientation, hence there exists a unique global one.
	
	If $\textbf J$ is a bicomplex structure compatible with $g$, then $dA=g(\textbf J\cdot, \cdot)$ is a  2-form of norm 1. Conversely, if $\textbf J$ is a local bicomplex structure compatible with $g$, the local $2$-forms $g(\textbf J\cdot, \cdot)$ and $g(-\textbf J\cdot, \cdot)$ are opposite, have module 1 and locally they are the only $2$-forms of $g$-module 1, since $dim_\C\bigwedge^2 \C T_pS=1$: as a result, for a given $dA$, one has $dA=g(\textbf J \cdot, \cdot)$ for a unique global $\textbf J$.
	
\end{proof}
We will say that $\{X, Y \}$ is a \emph{positive orthonormal frame} for $\C TS$ with respect to the orientation induced by $\textbf{J}$ if $\textbf J (X)=Y$. 

With the language of fiber bundles, a positive complex metric on a surface $S$ provides a reduction of the frame bundle of $\C TS \to S$ to the structure group $O(2, \C)$, while the pair of a positive complex metric and an orientation on $S$ provides a reduction to $SO(2,\C)$.

\begin{Cor}
	Given an oriented surface $S$, the map that sends a positive complex metric $g$ to the orientation-consistent bicomplex structure $\textbf{J}$ compatible with $g$ induces a homeomorphism
	\begin{equation}
		\label{conformal structures and BC}
		\begin{split}
			\faktor{CM^+(S)}{C^{\infty} (S, \C^*)}&\xrightarrow{\sim}   BC(S) \cong C(S)\times C(S)\\
			C^{\infty} (S, \C^*)\cdot g &\mapsto \textbf{J}_g
		\end{split}
	\end{equation}
	Conformal structures of Riemannian metrics correspond to elements of the diagonal of $C(S)\times C(S)$.
\end{Cor}
\begin{proof}
	We showed that the map is continuous. Let us construct an inverse. Fix a real volume form $\alpha\in \Omega^2(TS)$, which extends by $\C$-bilinearity to a form in $\Omega^2(\C TS,\C)$. For any bicomplex structure $\textbf{J}$, define the symmetric tensor
	\[g:= \alpha(\textbf{J}\cdot, \cdot)+ \alpha(\cdot, \textbf{J} \cdot):\]
	$g$ is a complex metric as one can check on any basis $Y_i, Y_{-i}$ of eigenvectors for $\textbf{J}$ with different wigenvalues; moreover $g$ induces $\textbf{J}$ in turn, hence it is a positive complex metric.
	
	By choosing $\alpha$ as a real volume form for $S$, one gets an explicit correspondence between the diagonal of $C(S)\times C(S)$ and the set of conformal structures of Riemannian metrics.
\end{proof}

Finally, we remark that Proposition \ref{surfaces in G first part} can be restated in an interesting way for positive complex metrics in terms of developing maps for projective structures. 

\begin{Theorem}
	\label{Totally geodesic immersions}
	Let $S$ be an oriented surface and $g$ a positive complex metric on $S$. Denote with $(\widetilde S, \widetilde g)$ the universal cover. 
	
	$(S,g)$ has constant curvature $-1$ if and only if there exists a smooth map \[(f_1, f_2)\colon \widetilde S \to \GGG=\CP^1\times \CP^1 \setminus\Delta\] such that: 
	\begin{itemize}
		
		\item 
		$\widetilde g=(f_1, f_2)^* \inners=-\frac{4}{(f_1 - f_2)^2} df_1 \cdot df_2$
		\item $f_1$ and $f_2$ are local diffeomorphisms, respectively preserving and reversing the orientation.
		\item $f_1$ and $f_2$ are $(\pi_1 (S), \PSL)$-equivariant local diffeomorphisms inducing two projective structures with the same representation \[hol_{f_1}=hol_{f_2}=mon_{(f_1, f_2)}\colon \pi_1 (S)\to\PSL
		\]
		
	\end{itemize}
	Moreover, if $\textbf{J}$ is the bicomplex structure on $S$ induced by $g$, then, in identification (\ref{BC(S) to C(S)xC(S)}),
	\[
	\textbf{J} \equiv (f_1 ^* (\mathbb{J}^{\CP^1}), f_2 ^* (-\mathbb{J}^{\CP^1}) )
	\] 
	where $\mathbb{J}^{\CP^1}$ is the standard complex structure on $\CP^1$. 
	\begin{proof}
		As we observed in Proposition \ref{surfaces in G first part}, the fact that the complex metric $g$ is positive implies that $f_1$ and $f_2$ are local diffeomorphisms. 
		
		Let $w_k$ be a local holomorphic chart on $p\in S$ for the pull-back complex structure $f_k^* \mathbb{J}^{\CP^1}$. Then, by Cauchy-Riemann equations,
		\[
		Ker(df_k \colon \C T_pS \to \C)= Span_\C \Big(\frac{\partial}{\partial \overline {w_k} }_{|p} \Big)
		\]
		so by the explicit description of the metric one has
		\[ll(g_p)=\bigg(
		Span_\C \Big(\frac{\partial}{\partial \overline {w_1} }_{|p} \Big)\oplus Span_\C \Big(\frac{\partial}{\partial \overline {w_2} }_{|p} \Big) 
		\bigg) \setminus \{0_p\}
		\]
		and the two points of $\Proj_\C(ll(g_p))$ lie in two different components of $\Proj_\C (\C T_p S)\setminus \Proj_\R (T_pS)$ if and only if the two pull-back complex structures via $f_1$ and $f_2$ induce opposite orientations. 
		
		The choice of $(f_1, f_2)$ inducing the same $\widetilde g$ is unique up to post-composition with elements of $\PSL$, since the composition with the diagonal swap  $s:(x,y)\mapsto(y,x)$ on $\GGG$ would give an immersion with orientation-reversing first component.
		
		All the elements of $\pi_1(S)$ are orientation-preserving, hence $\sigma$ is $(\pi_1(S),\PSL)$-equivariant and, since $\PSL$ acts diagonally on $\GGG$, both $f_1$ and $f_2$ are equivariant with the same holonomy.
		
		From the description of $ll(g)$ above, one can conclude the last part of the theorem.
	\end{proof}
\end{Theorem}

\section{A Gauss-Bonnet Theorem for  positive complex metrics}

In this section, our aim is to prove a generalization of Gauss-Bonnet Theorem in the setting of positive complex metrics.

Let $g$ be a positive complex metric on an oriented surface $S$. We defined the area form $dA_g$ of $g$ as the unique 2-form with constant $g$-norm $1$ and compatible with the orientation on $S$, in the sense explained in Theorem \ref{Teorema orientazione}. Locally, 
\[
dA_g= \theta^1 \wedge \theta^2
\]
for a local positive $g$-coframe $\{\theta^1,\theta^2\}$.

\begin{Theorem}
	\label{Gauss Bonnet}
	Let $S$ be a closed oriented surface and $g$ a complex metric of constant curvature $K_g$.
	Then \[
	\int_S K_g dA_g = 2 \pi \chi(S)
	\]
	where $\chi(S)$ is the Euler-Poincaré characteristic of $S$.
\end{Theorem}
\begin{proof}
	Consider a real vector field $\XX\in \Gamma(TS)$ with finite set of zeros $\Lambda=\{x_1, \dots, x_n\}$. Since the complex metric is positive, $\| \XX\|^2\ne 0$ over $S\setminus \Lambda$. Let the index $k$ vary in $\{1, \dots, n\}$. 
	
	Consider for all $k$ two open disks $B_k$ and $\widetilde B_k$ in $S$, such that \[x_k\in B_k \subset \overline B_k \subset \widetilde B_k
	\]
	and such that $B_k \cap B_j\ne \emptyset$ iff $k=j$.
	
	Let $\textbf{J}$ be the bicomplex structure induced by $g$ and by the orientation on $S$. The $g$-orthonormal frame $\bigg(\frac \XX {\|\XX\|}, \textbf{J} \big(\frac \XX {\|\XX\|}\big) \bigg)$ can be defined locally on $S\setminus \Lambda$ but is defined globally on it only up to a sign, and the same holds for the induced dual coframe $\pm(\theta^1, \theta^2)$. Nevertheless, the area form $dA_g=\theta^1\wedge\theta^2$ is globally well-defined on $S\setminus \Lambda$ and so is the corresponding Levi-Civita connection form $\omega$ defined by 
	\begin{align*}
		d\theta^1 =\omega \wedge \theta^2 \\
		d\theta^2= -\omega \wedge \theta^1.
	\end{align*}
	On the whole $S\setminus \Lambda$ we have $d\omega=-KdA$ as in standard Riemannian Geometry (see \cite{Kobayashi-Nomizu1}).
	
	On each $\widetilde B_i$ we fix a norm-$1$ complex vector field $e_{i}$ and consider the coframe $(\theta^1_i, \theta^2_i)$ defined as the dual of $(e_i, \textbf{J}(e_i))$; let $\omega_i$ be the Levi-Civita connection form with respect to $(\theta^1_i, \theta^2_i)$. 
	Orient each $\partial {\widetilde B_i}$ with the orientation induced by $\widetilde B_i$. By Stokes formula,
	\begin{align*}                  
		\int_S K dA &= \int_{S\setminus \bigcup_i  B_i} K dA + \sum_{i=1}^n \int_{\overline B_i} KdA= \\
		=& -\int_{S\setminus \bigcup_i \widetilde B_i} d\omega - \sum_{i=1}^n \int_{\overline B_i} d\omega_i=\\
		=& \sum_{i=1}^n \int_{\partial B_i} \omega - \omega_i.
	\end{align*}

	We prove that $\int_S KdA$ is a conformal invariant. 
	
	For all smooth $f\colon S\to \C$, $\overline g= f g$ is a positive complex metric conformal to $g$. For some local choice of a square root of $f$ and of $g(\XX,\XX)$, one can define a local $\overline g$-orthonormal frame $\Big(\frac \XX { \sqrt{f \|\XX,\XX\|_g^2}  }, \frac {\textbf{J}\XX} { \sqrt{f \|\XX,\XX\|_g^2}  }\Big)$ with $\overline{g}$-dual given by the coframe $(\overline \theta^1, \overline \theta^2)$, which in turn defines a Levi-Civita connection form $\overline \omega$ for $\overline g$ independent of the choice of the root. On each 
	$\widetilde B_i$, fix a square root $\sqrt f$ of 
	$f$, so a $\overline g$-orthonormal frame is given by 
	$(\frac{e_1} {\sqrt f }, \frac{\textbf{J} (e_1)} {\sqrt f })$ with 
	associated $\overline g$-coframe $(\overline \theta_i^1, \overline \theta_i ^2)$ and Levi-Civita connection form $\overline \omega_i$. By an explicit computation, one can check that on each $\widetilde B_i \setminus \{x_i\} $ 
	\[
	\overline \omega - \omega= \overline \omega_i - \omega_i = -\frac1{2f}df\circ \textbf{J}.
	\]
	As a result, 
	\[
	\int_S K_{\overline g} dA_{\overline g} = \sum_{i=1}^n \int_{\partial B_i} \overline \omega -\overline \omega_i= \sum_{i=1}^n \int_{\partial B_i} \omega -\omega_i= \int_S K_g dA_g.
	\]

	Now, on each $\widetilde B_k \setminus \{x_k\}$, the coframes $(\theta^1_i, \theta^2_i)$ and $(\theta^1, \theta^2)$ (the last one is well-defined only locally) are both $g$-orthonormal and $\textbf{J}$-oriented, so we locally have 
	\[
	\begin{pmatrix}
		\theta^1\\
		\theta^2
	\end{pmatrix}=
	\begin{pmatrix}
		\cos(\alpha_k) & \sin(\alpha_k) \\
		-\sin(\alpha_k) & \cos(\alpha_k)
	\end{pmatrix} \begin{pmatrix}
		\theta_k ^1 \\
		\theta_k^2
	\end{pmatrix}
	\]
	where $e^{i\alpha_k} \in \C^*$ is a locally defined smooth function, defined only up to a sign on $\widetilde B_k \setminus \{x_k\}$. Nevertheless, $e^{2i\alpha_k}\colon \partial B_k \to S^1$ is a well defined smooth function and $d\alpha_k$ is a well-defined one form on $\partial B_k$.
	
	A standard calculation shows that $\omega - \omega_k =d \alpha_k$, hence
	
	\begin{align*}
		\int_S KdA &= \sum_{k=1}^n \int_{\partial B_k} \omega - \omega_k= \sum_{k=1}^n \int_{\partial B_k} d \alpha_k =\\
		&= \sum_{k=1}^n \int_{\partial B_k} \frac 1 {2i} \frac{ d (e^{2i\alpha_k} )}{e^{2i \alpha_k}}= \sum_{k=1}^n \pi deg(e^{2i\alpha_k}) \in \pi \Z.
	\end{align*}

	Recalling that $BC(S)$ is connected and that the quantity $\int_S K_g dA_g$ is a conformal invariant, the map 
	\[
	\int_S K \colon BC(S) \to \pi \Z
	\]
	is well-posed and continuous, hence it is constant. Since $\int_S K dA= 2 \pi \chi(S)$ for Riemannian metrics, the thesis follows.
\end{proof}

\section{A Uniformization Theorem through Bers Theorem}

The aim of this section is to prove a generalization of the classic Riemann's Uniformization Theorem for bicomplex structures and positive complex metrics.

Before stating the main theorem, we recall some notions about quasi-Fuchsian representations.

A representation $\rho\colon \pi_1(S)\to \PSL$ is \emph{quasi-Fuchsian} if its limit set $\Lambda_\rho\subset \CP^1$, i.e. the set of accumulation points of any $\rho(\pi_1(S))$-orbit in $\CP^1$, is a Jordan curve in $\CP^1$. The set of quasi-Fuchsian representations defines $QF(S)$ an open subset of the character variety $\mathcal X (S)$, defined in Equation \eqref{eq: def character variety}.

A major result in the setting of quasi-Fuchsian representations is the Bers Simultaneous Uniformization Theorem.

\begin{Theorem} [Bers Simultaneous Uniformization Theorem, \cite{Bers}]
	\label{Bers}
	Let $S$ be an oriented surface with $\chi (S)<0$. For all $J_1, J_2$ complex structures over $S$, there exists a unique quasi-Fuchsian representation $\rho=\rho(J_1,J_2) \colon \pi_1(S) \to \PSL$ such that, defined $\Lambda_\rho\subset \CP^1$ as the limit set of $\rho$ and $\Omega_+, \Omega_-$ as the connected components of $\CP^1 \setminus \Lambda_\rho$:
	\begin{itemize}
		\item there exists a unique diffeomorphism $f_1 \colon \widetilde S \to \Omega_+$, which is $J_1$-holomorphic and $\rho$-equivariant;
		\item there exists a unique diffeomorphism $f_2 \colon \widetilde S \to \Omega_-$, which is $J_2$-antiholomorphic and $\rho$-equivariant.
	\end{itemize} 
	Moreover, $(\rho, f_1, f_2)$ are continuous functions of $(J_1, J_2)\in BC(S)$.
	
	This correspondence determines a homeomorphism in the quotient
	\begin{equation}
		\label{Bers homeo}
		\mathfrak B\colon\Tau(S)\times \Tau(S)\xrightarrow{\sim}QF(S)
	\end{equation}
	where $\Tau(S)$ is the Teichm\"uller space of $S$.
\end{Theorem}
\vspace{5mm}

Recall that in Subsection \ref{subsec proj structures and holonomy} we considered the holonomy maps $\widetilde {Hol}$ and $Hol$ from the spaces $\widetilde P(S)$ and $P(S)$ respectively. Define the loci
\begin{align*}
	\widetilde{\mathcal P}_{QF}(S):=& \widetilde{Hol}^{-1}(QF(S)) \subset \widetilde {\mathcal P}(S), \\
	\mathcal P_{QF} (S):=& Hol^{-1} (QF(S) )\subset  \mathcal P(S)
\end{align*}
which, by continuity of the holonomy maps, are open subsets of $\widetilde{\mathcal P}(S)$ and of ${\mathcal P}(S)$ respectively. 

\begin{Theorem}[W. Goldman, \cite{Goldman:projectivestructures}]
	\label{Teo Goldman} 
	Let $S$ be a compact oriented surface of genus $g(S)\ge 2$.
	
	An open connected component of the space $\widetilde {\mathcal P}_{QF} (S)$ is
	\begin{align*}
		\widetilde {\mathcal P}_0  (S):=\{ &f\colon \widetilde S \to \CP^1,  \text{$\rho$-equivariant for some quasi-Fuchsian $\rho\colon \pi_1(S)\to \CP^1$,} \\
		&\text{ with } f(\widetilde S)\cap \Lambda_\rho=\emptyset   \}.
	\end{align*}
	Moreover, the restriction 
	\[
	Hol \colon \faktor{\widetilde {\mathcal P}_0 (S)}{\text{Diff}_0 (S)}=: {\mathcal P}_0 (S) \to QF(S)
	\]
	is a homeomorphism. 
	
	With the notations of Theorem \ref{Bers} and equation \eqref{Bers homeo}, the inverse is given by \[[\rho]\mapsto [f_1 (\rho)]=\Big[f_1 \Big(\mathfrak B^{-1} ([\rho])\Big)\Big].
	\]	
	
\end{Theorem}

\vspace{15pt}

We are now able to state a generalization of the Uniformization Theorem.

For an oriented surface $S$, define 
\[CM^+_{-1} (S):=\{\text{Positive complex metrics with constant curvature $-1$} \}\]
endowed with the subspace topology from $CM^+ (S)$.

\begin{Theorem} 
	\label{Uniformization 1}
	
	Let $S$ be an oriented surface with $\chi(S)<0$. 
	\begin{itemize}
		\item[1)] 
		For all $g\in CM^+(S)$, there exists a smooth $f\colon S\to \C^*$ such that $f\cdot g$ has constant curvature $-1$ and quasi-Fuchsian monodromy.
		
		More precisely, the map
		\begin{align*}
			\mathfrak{{J}} \colon C M^+(S) &\to BC (S) \\
			g &\mapsto [g]=\textbf J_g
		\end{align*}
		admits a continuous right inverse 
		\[
		\mathfrak{U}\colon  C(S)\times C(S)\cong BC(S) \to C M^+_{-1} (S)
		\]
		such that the diagram 
		\[
		\begin{tikzcd}
			C(S)\times C(S)\cong BC(S) \arrow[r, "\mathfrak U"] \arrow[d] &CM_{-1}^+(S) \arrow[d, "mon"]  \\
			\Tau(S)\times \Tau(S) \arrow[r, "\mathfrak B"] & \mathcal X(S)
		\end{tikzcd}
		\]
		commutes, where $\mathfrak B$ is defined as in (\ref{Bers homeo}).

		\item[2)]	If $S$ is closed, the image of $\mathfrak U$ is the connected component of $C M^+_{-1} (S)$ containing Riemannian metrics.
	\end{itemize}
\end{Theorem}

\begin{proof} [Proof of Theorem $\ref{Uniformization 1}$]
	\begin{itemize}
		\item[1)]
		Consider $\textbf{J}\in BC(S)$, which corresponds to two complex structures $(J_1, J_2)\in C(S)\times C(S)$.
		
		Applying Theorem \ref{Bers} to the couple $(J_1, J_2)$, there exist:
		\begin{itemize}
			\item $\rho=\rho(J_1,J_2)\colon \pi_1 (S)\to \PSL$ quasi-Fuchsian;
			\item $f_1 \colon \widetilde S \to \CP^1$ $J_1$-holomorphic and $\rho$-equivariant embedding;
			\item $f_2 \colon \widetilde S \to \CP^1$ $J_2$-antiholomorphic and $\rho$-equivariant embedding
		\end{itemize}
		with $Im(f_1)\cap Im(f_2)=\emptyset$. Hence, we have an admissible, $\rho$-equivariant embedding
		\[
		\sigma=(f_1, f_2) \colon \widetilde S \to {\GGG}= \CP^1 \times \CP^1 \setminus \Delta.
		\]
		By Theorem \ref{Totally geodesic immersions}, the pull-back complex metric on $\widetilde S$ is positive and projects to a positive complex metric $g$ on $S$. 
		
		After we show that the constructed $g$ is compatible with $\textbf{J}$, the thesis follows by defining $\mathfrak{U}(\textbf{J}):=g$. Let $w_1$ and $w_2$ be the complex coordinates on $\widetilde S$ induced by $J_1$ and $J_2$ respectively. Then,
		\begin{align*}
			g &= \frac{4}{(f_1  - f_2 )^2} f_1^* dz_1 \cdot f_2^* dz_2= \frac{4}{(f_1 - f_2 )^2}  d f_1 \cdot d f_2 = \\
			&= \frac{4}{(f_1  - f_2 )^2}  \frac{\partial f_1}{\partial  {w_1}} \frac{\partial f_2}{\partial\overline{ w_2}}  d {w_1} \cdot d \overline{w_2};
		\end{align*}
		hence $V_i(J_2)=V_i(\textbf{J})$ and $V_{-i} (J_1)=V_{-i}(\textbf{J})$ are the isotropic directions of $g$.
		\item[2)] 
		By regarding
		\[
		Im(\mathfrak U)=\{g\ |\ (\mathfrak U \circ \mathfrak J) (g)=g \}\subset C M_{-1}^+(S),
		\]
		we have that $Im(\mathfrak U)$ is a closed connected subset of $\C M_{-1}^+ (S)$. We need to prove that it is also open.
		
		As we stated in Theorem \ref{Totally geodesic immersions}, for any $g\in CM_{-1}^+$ there exist two developing maps for projective structures $f_1$ and $f_2$ on $S$ and $\overline S$ respectively, uniquely defined up to post-composition with elements in $\PSL$, such that $\sigma=(f_1,f_2)\colon (\widetilde S, \widetilde g)\to \GGG$ is an isometric immersion. We therefore have a map 
		\begin{align*}
			proj\colon CM_{-1} ^+ (S)&\to \mathcal {\widetilde P}(S)\times \mathcal{\widetilde  P}(\overline S)\\
			g=(f_1,f_2)^* \inners_\GGG &\mapsto ([f_1], [f_2]).
		\end{align*}
		
		The thesis follows after we show that \[Im (\mathfrak U )= proj^{-1}  (\mathcal{\widetilde P}_0(S)\times \mathcal{\widetilde  P}_0(\overline S) )\] since $\mathcal{\widetilde P}_0(S)$ is an open subset of $\mathcal{\widetilde P}_{QF}(S)$, hence of $\mathcal{\widetilde P}(S)$.
		
		By construction of $\mathfrak U$, clearly \[Im (\mathfrak U )\subseteq proj^{-1}  (\mathcal{\widetilde P}_0(S)\times \mathcal{\widetilde  P}_0(\overline S) ).\]
		
		Conversely, assume $g=(f_1,f_2)^* \inners_\GGG \in proj^{-1}  (\mathcal{\widetilde P}_0(S)\times \mathcal{\widetilde  P}_0(\overline S) )$. Since $hol_{f_1}=hol_{f_2}\colon \pi_1(S)\to \PSL$, $f_1$ and $f_2$, by Theorem \ref{Teo Goldman}, the maps $f_1$ and $f_2$ correspond exactly to the maps one gets through Theorem \ref{Bers} from the pair of complex structures $(f_1^* (\mathbb J^{\CP^1}), f_2^*(-\mathbb J^{\CP^1}))$, hence 
		\[
		(f_1,f_2)^*\inners_\GGG = \mathfrak U (f_1^* (\mathbb J^{\CP^1}), f_2^*(-\mathbb J^{\CP^1}) ). 
		\]

	\end{itemize}
\end{proof}

\part{Integrable immersions into $\G{n+1}$}
\label{parte Seppi}

\chapter{The Gauss map of hypersurfaces in $\Hyp^{n+1}$}\label{sec:gauss map}
In this Chapter we will focus on the construction of the Gauss map of an immersed hypersurface, its relation with the normal evolution and the geodesic flow action on the unit tangent bundle, and provide several examples of great importance for the rest of Part \ref{parte Seppi}.

After a discussion in the general setting, we will then focus on the case of immersions with small principal curvatures.

\section{Lift to the unit tangent bundle}
We defined the Gauss map associated to an immersion into $\Hyp^3$ in Section \ref{sec metrica complessa indotta in G}.  
Let us formally recall it now for hypersurfaces in general dimension, together with the notions of lift to the unit tangent bundle. 

Recall that in Chapter \ref{chapter spazio geodetiche} we defined a para-Sasaki metric $\gs{n+1}$ on $T^1\Hyp^{n+1}$ and a para-K\"ahler structure $(\GG, \JJ, \Omega)$ on the space $\G{n+1}$.

\begin{Def}\label{Def:lift and gauss}
	Let $M$ be an oriented $n$-dimensional manifold, let $\sigma:M\to\Hyp^{n+1}$ be an immersion, and let $\nu$ be the unit normal vector field of $\sigma$ compatible with the orientations of $M$ and $\Hyp^{n+1}$. We define the \emph{lift} of $\sigma$ as 
	$$\zeta_\sigma:M\to T^1\Hyp^{n+1}\qquad \zeta_\sigma(p)=(\sigma(p),\nu(p))~.$$
	The \emph{Gauss map} of $\sigma$ is then the map
	$$G_\sigma:M\to\G{n+1}\qquad G_\sigma=\mathrm{p}\circ \zeta_\sigma~.$$
\end{Def}

In other words, the Gauss map of $\sigma$ is the map which associates to $p\in M$ the geodesic $\ell$ of $\Hyp^{n+1}$ orthogonal to the image of $d_p\sigma$ at $\sigma(p)$, oriented compatibly with the orientations of $M$ and $\Hyp^{n+1}$.

Recall that the \emph{first fundamental form} $\I$ of $\sigma$ is defined as the pull-back Riemannian metric $\I=\sigma^*{\inners_{\Hyp^{n+1}}}$ on $M$.

Also {recall that} the \emph{shape operator} $B$ of $\sigma$ {can be defined} as the $(1,1)$-tensor on $M$ defined by 
\begin{equation}\label{eq:shape}
	d\sigma\circ B(W)=-D_{W} \nu~,
\end{equation} for $D$ the Levi-Civita connection of $\Hyp^{n+1}$.

\begin{prop}\label{prop:gauss immersion}
	Given an oriented manifold $M^n$ and an immersion $\sigma:M\to\Hyp^{n+1}$, the lift of $\sigma$ is an immersion orthogonal to the fibers of $\mathrm{p}:T^1\Hyp^{n+1}\to \G{n+1}$. As a consequence $G_\sigma$ is an immersion.
\end{prop}
\begin{proof}
	By a direct computation in the hyperboloid model, the differential of $\zeta_\sigma$ has the expression 
	\begin{equation}
		\label{eq: differential of sigma tilde}
		d_p\zeta_\sigma (W)=(d_p\sigma(W),d_p\nu(W))=(d_p\sigma(W),-d_p\sigma(B(W)))~,
	\end{equation} indeed, the ambient derivative  in $\R^{n+1,1}$ of the vector field $\nu$ equals the covariant derivative with respect to $D$, since {$d_p\nu(W)$} is orthogonal to $\sigma(p)$ as a consequence of the condition $\langle \sigma(p),\nu(p)\rangle=0$. 
	
	As both $d_p\sigma(W)$ and $d_p\sigma(B(W))$ are tangential to the image of $\sigma$ at $p$, hence orthogonal to $\nu(p)$, $d_p\zeta_\sigma (W)$ can be written as:
	\begin{equation}\label{eq:differential lift}
		d_p\zeta_\sigma(W)=d_p\sigma(W)^\HH-d_p\sigma(B(W))^\V~.
	\end{equation}
	Therefore, for every $W\neq 0$, $d_p \zeta_\sigma(W)$ is a non-zero vector orthogonal to $\chi_{\zeta(p)}$ by Definition \ref{Def:parasasaki}. Since the differential of $\mathrm p$ is a vector space isomorphism between $\chi_{\zeta(p)}^\perp$ and $T_{G_\sigma(p)}\G{n+1}$, the Gauss map $G_\sigma$ is also an immersion. 
\end{proof}

As a consequence of Proposition \ref{prop:gauss immersion}, we can compute the first fundamental form of the Gauss map $G_\sigma$, that is, the pull-back metric $G_\sigma^*\GG$, which we denote by $\overline\I$. Since the lift $\zeta_\sigma$ is orthogonal to $\chi$, it suffices to compute the pull-back metric of $\gs{n+1}$ by $\zeta_\sigma$. By Equation \eqref{eq:differential lift}, we obtain:
\begin{equation}\label{eq:fff gauss}
	\overline\I=\I - \III~,
\end{equation}
where $\III=\I(B\cdot,B\cdot)$ is the third fundamental form of $\sigma$.

Let us now see that the orthogonality to the generator of the geodesic flow essentially characterizes the lifts of immersed hypersurfaces in $\Hyp^{n+1}$, in the following sense.

\begin{prop}\label{prop:gauss immersion converse}
	Let $M^n$ be an orientable manifold and $\zeta:M\to T^1\Hyp^{n+1}$ be an immersion orthogonal to the fibers of $\mathrm{p}:T^1\Hyp^{n+1}\to \G{n+1}$. If $\sigma:=\pi\circ\zeta$ is an immersion, then $\zeta$ is the lift of $\sigma$ with respect to an orientation of $M$.
\end{prop}
\begin{proof}
	Let us write $\zeta=(\sigma,\nu)$. Choosing the orientation of $M$ suitably, we only need to show that the unit vector field $\nu$ is normal  to the immersion $\sigma$. By differentiating, $d\zeta=(d\sigma,\mathrm d\nu)$ and by \eqref{eq:modelTT} we obtain 
	$$\langle \nu(p),d\sigma(W)\rangle+\langle \sigma(p),\mathrm d\nu(W)\rangle=0$$ 
	for all  $W\in T_p M$. By the orthogonality hypothesis and the expression $\chi_{\zeta(p)}=(\nu(p),\sigma(p))$ (Lemma \ref{lemma:generator geoflow hor lift}) we obtain
	$$\langle \nu(p),d\sigma(W)\rangle-\langle \sigma(p),\mathrm d\nu(W)\rangle=0~,$$
	hence $\langle \nu(p),\mathrm d\sigma(W)\rangle=0$ for all $W$. Since by hypothesis the differential of $\sigma$ is injective, $\nu(p)$ is uniquely determined up to the sign, and is a unit normal vector to the immersion $\sigma$. 
\end{proof}

In relation with Proposition \ref{prop:gauss immersion converse}, it is important to remark that there are (plenty of) immersions in $T^1\Hyp^{n+1}$ which are orthogonal to $\chi$ but are \emph{not} the lifts of immersions in $\Hyp^{n+1}$, meaning that they become singular when post-composed with the projection $\pi:T^1\Hyp^{n+1}\to \Hyp^{n+1}$. 
{Some examples of this phenomenon (Example \ref{ex:spheres}), and more in general of the Gauss map construction, are presented in Subsection \ref{subsec:examples} below.}

\subsection{Geodesic flow and normal evolution} We develop here the construction of the normal evolution of hypersurfaces into $\Hyp^{n+1}$.

\begin{Def}\label{Def normal evolution}
	Given an oriented manifold $M^n$ and an immersion $\sigma:M\to\Hyp^{n+1}$, the \emph{normal evolution} of $\sigma$ is the map 
	$$\sigma_t:M\to\Hyp^{n+1}\qquad \sigma_t(p)=\exp_{\sigma(p)}(t\nu(p))~,$$
	where $\nu$ is the unit normal vector field of $\sigma$ compatible with the orientations of $M$ and $\Hyp^{n+1}$.
\end{Def}

The relation between the normal evolution and the geodesic flow is encoded in the following proposition. 

\begin{prop}
	\label{prop: flusso geod e flusso normale}
	Let $M^n$ be an orientable manifold  and $\sigma:M\to \Hyp^{n+1}$ be an immersion. Suppose $\sigma_t$ is an immersion for some $t\in\R$. Then $\zeta_{\sigma_t}=\varphi_t\circ \zeta_\sigma$. 
\end{prop}
\begin{proof}
	The claim is equivalent to showing that, if the differential of $\sigma_t$ is injective at $p$, then $(\sigma_t(p),\nu_t(p))=\varphi_t(\sigma(p),\nu(p))$, where $\nu_t$ is the normal vector of $\sigma_t$. The equality on the first coordinate holds by definition of the geodesic flow, since $t\mapsto \gamma(t)=\sigma_t(p)$ is precisely the parameterized geodesic such that $\gamma(0)=\sigma(p)$ with speed $\gamma'(0)=\nu(p)$. It thus remains to check that the normal vector in $\sigma_t(p)$ equals $\gamma'(t)$.
	
	By the usual expression of the exponential map in the hyperboloid model of $\Hyp^{n+1}$, the normal evolution is:
	\begin{equation}\label{eq:normal evolution hyperboloid}
		\sigma_t(p)=\cosh(t)\sigma(p)+\sinh(t)\nu(p)~,
	\end{equation}
	and therefore 
	\begin{equation}\label{eq:normal evolution hyperboloid2}
		d\sigma_t(V)=d\sigma\circ(\cosh(t)\mathrm{id}-\sinh(t)B)(V)
	\end{equation}
	for $V\in T_p M$.
	It is then immediate to check that
	$$\dot	\gamma(t)=\sinh(t)\sigma(p)+\cosh(t)\nu(p)$$
	is orthogonal to $d\sigma_t(V)$
	for all $V$.
	If $d\sigma_t$ is injective, this implies that $\gamma'(t)$ is the unique unit vector normal to the image of $\sigma$ and compatible with the orientations, hence it equals $\nu_t(p)$. This concludes the proof.
\end{proof}

A straightforward consequence, recalling that the Gauss map is defined as $G_\sigma=\mathrm{p}\circ\zeta_\sigma$ and that $\mathrm p\circ\varphi_t=\mathrm p$, is the following.

\begin{cor}\label{prop:gauss map invariant normal evo}
	The Gauss map of an immersion $\sigma:M^n\to \Hyp^{n+1}$ is invariant under the normal evolution, namely $G_{\sigma_t}=G_\sigma$, as long as $\sigma_t$ is an immersion.
\end{cor}

\begin{remark}
	It follows from Equation \eqref{eq:normal evolution hyperboloid2} that, for any immersion $\sigma:M^n\to \Hyp^{n+1}$, the differential of $d\sigma_t$ at a point $p$ is injective for small $t$. However, in general  $\sigma_t$ might fail to be a global immersion for all $t\neq 0$.
	In Section \ref{sec:small princ curv} we will discuss the condition of small principal curvatures for $\sigma$, which is a sufficient condition to ensure that $\sigma_t$ remains an immersion for all $t$.
	
	As a related phenomenon, it is possible to construct examples of immersions $\zeta:M^n\to T^1\Hyp^{n+1}$ {which are orthogonal to the fibers of $\mathrm p$ but} such that $\pi\circ\varphi_t\circ\zeta$ fails to be an immersion for all $t\in\R$. We will discuss this problem later on, see Example \ref{ex: Lagrangian not globally integrable}.
\end{remark}

\subsection{Fundamental examples}\label{subsec:examples}

It is now useful to describe several explicit examples. All of {them} will actually play a role in some of the proofs in the next sections.

\begin{example}[Totally geodesic hyperplanes] \label{ex:totally geodesic}
	Let us consider a totally geodesic hyperplane $\mathcal P$ in $\Hyp^{n+1}$, and let $\sigma:\mathcal P\to \Hyp^{n+1}$ be the inclusion map. Since in this case the shape operator vanishes everywhere, from Equation \eqref{eq:fff gauss} the Gauss map is an isometric immersion (actually, an embedding) into $\G{n+1}$ with respect to the first fundamental form of $\sigma$. Totally geodesic immersions are in fact the only cases for which this occurs.
	
	A remark that will be important in the following is that the lift $\zeta_\sigma$ is horizontal: that is, by Equation \eqref{eq:differential lift}, $d\zeta_\sigma(w)$ equals the horizontal lift of $w$ for every vector $w$ tangent to $\mathcal P$ at $x$. Therefore for every $x\in \mathcal P$, the image of $d\zeta_\sigma$ at $x$ is exactly the horizontal subspace $\HP_{(x,\nu(x))}$, for $\nu(x)$ the unit normal vector of $\mathcal P$. 
\end{example}

\begin{example}[Spheres in tangent space] \label{ex:spheres}
	A qualitatively opposite example is the following. Given a point $x\in\Hyp^{n+1}$, let us choose an isometric identification of $T_x \Hyp^{n+1}$ with the $(n+1)$-dimensional Euclidean space, and consider the $n$-sphere $\Sph^n$ as a hypersurface in $T_x \Hyp^{n+1}$ by means of this identification. Then we can define the map 
	$$\zeta:\Sph^n\to T^1 \Hyp^{n+1}\qquad \zeta(v)=(x,v)~.$$
	The differential of $\zeta$ reads $d\zeta_v(w)=(0,w)=w^\V$ for every $w\in T_v\Sph^n\cong v^\perp$, hence $\zeta$ is an immersion, which is orthogonal to the fibers of $\mathrm{p}$. Actually,  $\zeta$ is vertical: this means that $d\zeta_v(w)$ is the vertical lift of $w$ for every $w\in v^\perp$, and therefore  $d_v\zeta(T_v \Sph^n)$ is exactly the vertical subspace $\VP_{(x,v)}$. 
	
	Clearly we are not in the situation of Propositions \ref{prop:gauss immersion} and \ref{prop:gauss immersion converse}, as $\pi\circ \zeta$ is a constant map. On the other hand, $\mathrm{p}\circ \zeta$ has image in $\G{n+1}$ consisting of all the (oriented) geodesics $\ell$ going through $x$. {However, when post-composing $\zeta$ with the geodesic flow, $\varphi_t\circ \zeta$ projects to an immersion in $\Hyp^{n+1}$} {for all $t\ne 0$} {and is in fact an embedding with image a geodesic sphere of $\Hyp^{n+1}$ of radius $|t|$ centered at $x$.}
	As a final remark, the first fundamental form of $\mathrm{p}\circ \zeta$,  is \emph{minus} the spherical metric of $\Sph^n$, since by Equation \eqref{eq: Sas di T1Hn} $\gs{n+1}(w^\V,w^\V)=-\langle w,w\rangle$. 
\end{example}

The previous two examples can actually be seen as special cases of a more general construction, which will be very useful in the next sections. 
\begin{example}[A mixed hypersurface in the unit tangent bundle] \label{ex:mixed}
	Let us consider a totally geodesic $k$-dimensional submanifold $Q$ of $\Hyp^{n+1}$, for $0\leq k\leq n$. Consider the unit normal bundle 
	$$\mathrm N^1 Q=\{(x,v)\in T^1\Hyp^{n+1}\,|\,x\in Q,v\text{ orthogonal to }Q  \}~,$$
	which is an $n$-dimensional submanifold of $T^1\Hyp^{n+1}$, and 
	let $\iota$ be the obvious inclusion map of $\mathrm N^1 Q$ in $T^1\Hyp^{n+1}$. Observe that $\pi\circ \iota$ is nothing but the bundle map $\mathrm N^1 Q\to Q$, hence not an immersion unless $k=n$ which is the case we discussed in Example \ref{ex:totally geodesic}. The map $\mathrm{p}\circ\iota$ has instead image in $\G{n+1}$ which consists of all the oriented geodesics $\ell$ orthogonal to $Q$.  See Figure \ref{fig:cylinder}.
	Let us {focus on} its geometry in $\G{n+1}$.  
	
	Given  $(x,v)\in N^1 Q$, take an orthonormal basis $\{w_1,\ldots,w_k\}$ of $T_x Q$. Clearly the $w_i${'s} are orthogonal to $v$, and let us complete them to an orthonormal basis $\{w_1,\ldots,w_n\}$ of $v^\perp\subset T_x \Hyp^{n+1}$. Then $\{w_1,\ldots,w_n\}$ identifies to a basis of $T_{(x,v)}\mathrm N^1 Q$. By repeating the arguments of the previous two examples, $d\iota_{(x,v)}(w_i)$ is the \emph{horizontal} lift of $w_i$ at $(x,v)$ if $1\leq i\leq k$, and is the \emph{vertical} lift if $i>k$. In particular they are all orthogonal to $\chi_{(x,v)}$, and therefore the induced metric on $\mathrm N^1 Q$ by the metric $\gs{n+1}$ coincides with the first fundamental form of $\mathrm{p}\circ\iota$. This metric has signature $(k,n-k)$, and $\{w_1,\ldots,w_n\}$ is an orthonormal basis, for which $w_1,\ldots,w_k$ are positive directions and $w_{k+1},\ldots,w_n$ negative directions. 
	
	{Similarly to the previous example, for all $t\ne 0$ the map $\varphi_t\circ\iota$ has the property that, when post-composed with the projection $\pi$, it gives an embedding with image the equidistant ``cylinder'' around $Q$.}
\end{example}

\begin{figure}[htbp]
	\centering
	\includegraphics[height=4.5cm]{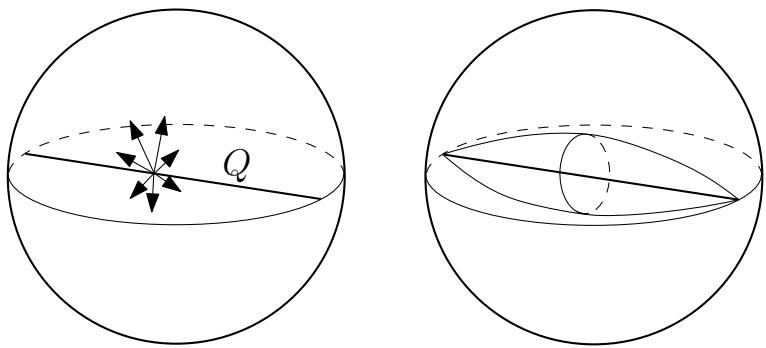} 
	
	\caption{The normal bundle $N^1Q$ of a $k$-dimensional totally geodesic submanifold $Q$ in $\Hyp^{n+1}$ (here $k=1$ and $n=2$). On the right: after composing with the geodesic flow $\varphi_t$ for $t\neq 0$, one obtains an equidistant cylinder.}\label{fig:cylinder}
\end{figure}

Let us now consider a final example, which allows also to {prove} the integrability of the almost para-complex structure $\JJ$ of $\G{n+1}$ we introduced in Lemma \ref{lemma:paracomplex structure}.

\begin{example}[Horospheres] \label{ex:horospheres}
	Let us consider a horosphere $H$ in $\Hyp^{n+1}$, and apply the Gauss map construction of Definition \ref{Def:lift and gauss} to the inclusion $\sigma:H\to \Hyp^{n+1}$.
	
	It is known that {the shape operator of $H$ is $\pm\mathrm{id}$ (the sign varies according to the choice of the normal vector field, or equivalently by the choice of orientation on $H$)}, a result we will also deduce later on from our arguments in Remark \ref{rmk princ curv horosphere}.  Define $\zeta_\pm$ as the lift of $\sigma$ induced by the choice of the normal vector field for which the shape operator is $\pm \mathrm{id}$.
	
	
	Now, by Proposition \ref{prop:gauss immersion}, the lift $\zeta_{\pm}$ is orthogonal to the fibers of $\mathrm{p}$, and moreover, by Equation \eqref{eq:differential lift}, $d\zeta_{\pm}(w)=w^\HH\pm w^\V$. As a result, by Equation \eqref{eq:defiJ2}, one has in fact that the image of $d_x\zeta_{\pm}$ is the whole $(\pm 1)$-eigenspace {of $\mathrm J$} in $T_{\zeta_\pm(x)}T^1 \Hyp^n$.
\end{example}


A direct application of Example \ref{ex:horospheres} shows that the almost para-complex structure $\JJ$ is integrable:

\begin{cor}
	The $(1,1)$-tensor $\JJ$ is a para-complex structure on $\G{n+1}$.
\end{cor}
\begin{proof}
	Given $(x,v)\in T^1\Hyp^{n+1}$, consider the two horospheres $H^\pm$ containing $x$ with normal vector $v$ at $x$. The vector $v$ points to the convex side of one of them, and to the concave side of the other. Let us orient them in such a way that $v$ is compatible with the ambient orientation. Then Example \ref{ex:horospheres} shows that the Gauss maps of the horospheres $H^\pm$ have image integral submanifolds for the distributions $T^\pm\G{n+1}$ associated to the almost para-complex structure $\JJ$, which is therefore integrable. 
\end{proof}

\section{A consequence: $(\G 3, \mathbbm g)\cong Re\big((\GGG, \inners_\GGG)\big)$}
\label{subsection: G(3) modello Bonsante-Christian}

We defined two different metrics on the space of geodesics of $\Hyp^3$, the pseudo-Riemannian metric $\GG$ and the holomorphic Riemannian metric $\inners=\inners_\GGG$. 

By Proposition \ref{Prop metrica in G} and Equation \ref{eq:fff gauss}, we can immediately observe that if $\sigma\colon S\to \Hyp^3$ is an immersion, then 
\begin{equation}
	\label{eq: Re gauss map}
G_\sigma^*(\GG)= Re\Big(G_\sigma^*(\inners_\GGG)\Big).
\end{equation}

We show that Equation \ref{eq: Re gauss map}, together with the examples considered in Subsection \ref{subsec:examples}, allows to prove the following

\begin{prop}
	\label{prop: pull-back Re hRm}
	On $\G{3}$,
	\[
	\GG= Re \Big( \inners_\GGG \Big)
	\]
\end{prop}

\begin{proof}
Let $\ell=\mathrm p (x,v)\in \G{3}$, then in the hyperboloid model $\chi_{(x,v)}=(v,0)$, and
	\[
	\chi^\bot_{(x,v)} =Span_\R ((w_1,0), (w_2,0), (0,w_1), (0,w_2)) \subset T_{(x,v)} T^1 \Hyp^3
	\]
	with $(w_1,w_2,v)$ being a orthonormal basis for $T_x M$. Then, \[d\mathrm p (w_1,0), d\mathrm p(w_2,0), d\mathrm p(0,w_1), d\mathrm p(0,w_2)\in T_{\ell}\G{3} \] generate the tangent space and are mutually $\GG$-orthonormal, with the former two ones having $\GG$-norm $1$ and the latter two ones having $\GG$-norm $-1$. We prove that $Re\inners$ coincides with $\GG$ on this basis.
	\begin{itemize}
		\item As in Example \ref{ex:totally geodesic}, let $\sigma_1\colon \Hyp^2\hookrightarrow\Hyp^3$ be the totally geodesic inclusion passing by $x$ and with normal vector $v$. Then $\zeta_{\sigma_1} (x)=\ell$ and \[d\zeta_{\sigma_1} (T_x\Hyp^2)= Span_{\R} ((w_1,0), (w_2,0)).\] 
		As a result, by Proposition \ref{prop: pull-back Re hRm}, the metrics $\GG$ and $Re \inners$ coincide on $Span_{\R} (d\mathrm p (w_1,0), d\mathrm p(w_2,0))$, hence $(w_1,0)$ and $(w_2,0)$ have norm $1$ and are orthogonal for both metrics.
		\item As in Example \ref{ex:spheres}, there exists an embedded sphere $\sigma_2\colon S^2 \hookrightarrow \Hyp^3$ passing by $x$ and with $\ell$ as normal geodesic at $x$ such that \[d\zeta_{\sigma_2} (T_x S^2)= Span_{\R} ((0,w_1), (0,w_2)),\]
		hence, similarly as above, $d\mathrm p (0,w_1)$ and $d\mathrm p(0,w_2)$ have norm $-1$ and are orthogonal for both metrics.
		\item Using Example \ref{ex:mixed}, one can take an embedded cylinder $\sigma_3\colon S^1 \times \R\hookrightarrow \Hyp^3$ passing by $x$ with normal geodesic $\ell$ such that
		\[
		d\zeta_{\sigma_3} (T_x (S^1 \times \R))= Span_{\R} ((w_1,0), (0,w_2)).
		\]
		Hence $d\mathrm p(w_1,0)$ and $d \mathrm p(0,w_2)$ are orthogonal for $Re \inners$ too, and similarly for $(w_2,0)$ and $(0, w_1)$.
		\item We are left to prove that $Re\inner{d\mathrm p(w_k,0),d \mathrm p(0,w_k)}_\GGG=0$, with $k=1,2$. 
		By Example \ref{ex:horospheres}, the horosphere $\sigma_4\colon H\hookrightarrow \Hyp^3$ passing by $x$ and orthogonal to $v$ can be oriented so that
		\[
		d\zeta_{\sigma_4}(T_x H)= Span_\R( (w_1,w_1), (w_2, w_2) ).
		\]
		Hence, clearly $d\zeta_{\sigma_4}(T_x H)$ is isotropic for $\gs 3$, so $dG_{\sigma_4}(T_x H)\subset T_{\ell}\G{3}$ is isotropic for both $\GG$ and $Re{\inners}$. As a result, for $k=1,2$,
		\begin{align*}
			0&= Re\inner{d\mathrm p(w_k, w_k), d\mathrm p(w_k, w_k)}=\\
			&=Re\inner{d\mathrm p(w_k, 0), d\mathrm p(w_k, 0)} +  Re\inner{d\mathrm p(0, w_k), d\mathrm p(0, w_k)}+ 2 Re\inner{d\mathrm p(w_k, 0), d\mathrm p(0,w_k)}=\\
			&= 2 Re\inner{d\mathrm p(w_k, 0), d\mathrm p(0,w_k)},
		\end{align*}
		and the proof follows.
	\end{itemize}
	
\end{proof}

\section{Immersions with small principal curvatures}\label{sec:small princ curv}

In this section we define and study the properties of immersed hypersurfaces in $\Hyp^{n+1}$ with small principal curvatures and their Gauss map.

\subsection{Extrinsic geometry of hypersurfaces}\label{sec:extr geom small princ curv}
Let us start by defining our condition of small principal curvatures. Recall that the principal curvatures of an immersion of a hypersurface in a Riemannian manifold (in our case the ambient manifold is $\Hyp^{n+1}$) are the eigenvalues of the shape operator, which was defined in \eqref{eq:shape}.

\begin{Def}
	An immersion $\sigma\colon M^n \to \Hyp^{n+1}$ has \emph{small principal curvatures} if its principal curvatures at every point {lie in $(-1,1)\subset \R$}. 
\end{Def}

As a consequence of Equation \eqref{eq:fff gauss}, we have a direct characterization of immersions with small principal curvatures in terms of their Gauss map:

\begin{prop}\label{prop: small curv sse riemannian}
	Given an oriented manifold $M^n$ and an immersion $\sigma:M\to\Hyp^{n+1}$, $\sigma$ has small principal curvatures if and only if its Gauss map $G_\sigma$ is a Riemannian immersion.
\end{prop}

We recall that an immersion into a pseudo-Riemannian manifold is \emph{Riemannian} if the pull-back of the ambient pseudo-Riemannian metric, which in our case is {the first fundamental form  $\overline\I=G_\sigma^*\GG$}, is a Riemannian metric.

\begin{proof}[Proof of Proposition $\ref{prop: small curv sse riemannian}$]
	The condition that the Gauss map is a Riemannian immersion is equivalent to ${\overline\I}(W,W)>0$ for every $W\neq 0$. By Equation \eqref{eq:fff gauss}, this is equivalent to $\| B(W)\|^2<\|W\|^2$ for the norm on $M$ induced by $\I$, and this is equivalent to the eigenvalues of $B$ (that is, the principal curvatures) being strictly smaller than 1 in absolute value. 
\end{proof}

\begin{remark} \label{rmk: negative curv}
	Let us observe that a consequence of the hypothesis of small principal curvatures is that the first fundamental form of $\sigma$ has negative sectional curvature. Indeed, if $V,W$ is a pair of orthonormal vectors on $T_p M$, then by the Gauss' equation the sectional curvature of the plane spanned by $V$ and $W$ is:
	$$K_{\mathrm{Span}(V,W)}=-1+\II(V,V)\II(W,W)-\II(V,W)^2~.$$
	Since the principal curvatures of $\sigma$ are less than one in absolute value, we have $\| B(V)\|<\|V\|$ and the same for $W$. Moreover $V$ and $W$ are unit vectors, hence both $|\II(V,V)|$ and $|\II(W,W)|$ are less than one and the sectional curvature is negative.
\end{remark}

Recall that we introduced in Definition \ref{Def normal evolution} the normal evolution $\sigma_t$ of an immersion $\sigma:M\to\Hyp^{n+1}$, for $M$ an oriented $n$-manifold. An immediate consequence of Proposition \ref{prop: small curv sse riemannian} is the following:

\begin{cor}\label{cor:equidistant is immersed}
	Given an oriented manifold $M^n$ and an immersion $\sigma:M\to\Hyp^{n+1}$ with small principal curvatures, for all $t\in\R$ the normal evolution $\sigma_t$ is an immersion with small principal curvatures.
\end{cor}
\begin{proof}
	It follows from Equation \eqref{eq:normal evolution hyperboloid2} that $\sigma_t$ is an immersion if the shape operator $B$ of $\sigma$ satisfies $\| B(W)\|^2<\|W\|^2$ for every $W\neq 0$, that is, if $\sigma$ has small principal curvatures. Since the Gauss map is invariant under the normal evolution by Corollary \ref{prop:gauss map invariant normal evo}, $\sigma_t$ has small principal curvatures for all $t$ as a consequence of Proposition \ref{prop: small curv sse riemannian}.
\end{proof}

It will be useful to describe more precisely, under the hypothesis of small principal curvatures, the behaviour of the principal curvatures under the normal evolution. 

\begin{lemma} \label{lemma:evolution fsigma}
	Let $M^n$ be an oriented manifold and let  $\sigma:M\to\Hyp^{n+1}$ be an immersion. If $\sigma_t$ is an immersion, the principal directions (i.e. the eigenvectors of the shape operator) of $\sigma_t$ are the same as those of $\sigma$ with principal curvatures
	\begin{equation}\label{eq della discordia}
		\lambda_{i;t}=	\frac{\lambda_i-\tanh(t)}{1-\tanh(t)\lambda_i}~.
	\end{equation}
	where $\lambda_1(p),\ldots,\lambda_n(p)$ denote the principal curvatures of $\sigma$ at $p$.
	
	As a consequence, if $\sigma$ has small principal curvatures, the function
	\begin{equation} \label{eq:aux function}
		\kappa_\sigma(p):=\frac{1}{n}\sum_{i=1}^n\arctanh(\lambda_i(p))~,
	\end{equation}
	satisfies 
	\[\kappa_{\sigma_t}=\kappa_\sigma-t\] for every $t\in \R$. 
\end{lemma}
\begin{proof}
	We showed in the proof of Proposition \ref{prop: flusso geod e flusso normale} that in the hyperboloid model of $\Hyp^{n+1}$ the normal vector of $\sigma_t$, compatible with the orientations of $M$ and $\Hyp^{n+1}$, has the expression 
	$$\nu_t(p)=\sinh(t)\sigma(p)+\cosh(t)\nu(p)~,$$
	where $\nu=\nu_0$ is the normal vector for $\sigma=\sigma_0$. Using also Equation \eqref{eq:normal evolution hyperboloid2}, the shape operator $B_t$ of $\sigma_t$, whose defining condition is 
	$d\sigma_t\circ B_t(W)=-D_{W} \nu_t$ as in Equation \eqref{eq:shape}, is:
	\begin{equation}\label{eq:shape normal evo}
		B_t=(\mathrm{id}-\tanh(t)B)^{-1}\circ(B-\tanh(t)\mathrm{id})~.
	\end{equation}
	First, Equation \eqref{eq:shape normal evo} shows that if $V$ is a principal direction (i.e. an eigenvalue of $B$) for $\sigma$, then it is also for $\sigma_t$. Second, if $\lambda_i$ is a principal curvature of $\sigma$, then the corresponding principal curvature for $\sigma_t$ is 
	{\begin{equation}\label{eq della discordia}
			\lambda_{i;t}=	\frac{\lambda_i-\tanh(t)}{1-\tanh(t)\lambda_i}=\tanh(\mu_i-t)~,
	\end{equation}} where $\mu_i=\arctanh(\lambda_i)$. 
	
	Observe that if the $\lambda_i's$ are lower than 1 in absolute value, then $\arctanh{\lambda_i}$ is well-posed and one has
	\[
	\lambda_{i;t}= \tanh(\arctanh(\lambda_i)-t).
	\]
	The formula $\kappa_{\sigma_t}=\kappa_\sigma-t$  follows. 
\end{proof}

\begin{remark}\label{rmk fsigma smooth}
	Assume $\sigma$ has small principal curvatures. Although the principal curvatures of $\sigma$ are not smooth functions, the function $\kappa_\sigma$ defined in \eqref{eq:aux function} is smooth as long as $\sigma$ has small principal curvatures. Indeed, using the expression of the inverse hyperbolic tangent in terms of the elementary functions, we may express:
	\begin{align*}
		\kappa_\sigma(p)&=\frac{1}{2n}\left(\sum_{i=1}^n\log\left(\frac{1+\lambda_i(p)}{1-\lambda_i(p)}\right)\right)=\frac{1}{2n}\log\left(\frac{\prod_{i=1}^n(1+\lambda_i(p))}{\prod_{i=1}^n(1-\lambda_i(p))}\right)=\\
		&=\frac{1}{2n}\log\left(\frac{\det(\mathrm {id}+B)}{\det(\mathrm{id}-B)}\right)~,
	\end{align*}
	where $B$ is the shape operator of $\sigma$ as usual.
	This proves the smoothness of $\kappa_\sigma$, which will be implicitly used in Proposition \ref{Prop: formula H in G}.
\end{remark}

\subsection{Comparison horospheres}


Our next goal is to discuss global injectivity of immersions with small principal curvatures (Proposition \ref{prop injectivity}) and {of} their Gauss maps (Proposition \ref{prop:gauss maps diffeo onto image}), {under the following completeness assumption.}
\begin{Def}
	An immersion $\sigma\colon M^n \to \Hyp^{n+1}$  is \emph{complete} if the first fundamental form $\I$ is a complete Riemannian metric.
\end{Def}

Here we provide some preliminary steps.

\begin{Def}\label{Def:convex immersion}
	{Given an oriented manifold $M^n$ and  an immersion $\sigma:M\to\Hyp^{n+1}$, let $B=-D\nu$ be its shape operator with respect to the unit normal vector field $\nu$ compatible with the orientations of $M$ and $\Hyp^{n+1}$. We say that $\sigma$} is (\emph{strictly}) \emph{convex}  if {$B$} is negative semi-definite (resp. definite), and, conversely, that it is (\emph{strictly}) \emph{concave} if {$B$} is positive semi-definite (resp. definite).
\end{Def}
{When $\sigma$ is an embedding, we refer to its image as a (strictly) convex/concave hypersurface.} Clearly reversing the orientation (and therefore the normal vector field) of a (strictly) convex hypersurface it becomes (strictly) concave, and viceversa. An example of convex hypersurface

A classical fact is 
that a properly embedded strictly convex hyperurface in $\Hyp^{n+1}$ disconnects it into two connected components and that exactly one of them is geodesically convex (the one towards which $-\nu$ is pointing): we denote the closure of this connected component as the \emph{convex side} of the hypersurface, and denote as the \emph{concave side} the closure of the other one.

{We need another definition before stating the next Lemma.} We say that a smooth curve $\gamma:[a,b]\to\Hyp^n$ parameterized by arclength has \emph{small acceleration} if $\| D_{\gamma'(t)}\gamma'(t)\|<1$ for all $t$, where $D$ denotes the Levi-Civita connection of $\Hyp^n$ as usual.

\begin{lemma}\label{lemma curve small acc} \label{rmk:strictly concave side} 
	Let $\gamma:[a,b]\to\Hyp^n$ be a smooth curve of small acceleration. Then the image of $\gamma$ lies on the concave side of any horosphere tangent to $\gamma$.
	{More precisely, $\gamma$ lies in the interior of the concave side except for the tangency point. }
\end{lemma}
\begin{proof}
	Up to reparametrization we can assume that the tangency point is $\gamma(0)$, and we shall prove that $\gamma(t)$ lies on the concave side of any horosphere tangent to $\gamma'(0)$ for every $t>0$. Recall that we are also assuming that $\gamma$ is parameterized by arclength. We will use the upper half-space model of $\Hyp^n$, namely, $\Hyp^n$ is the region $\{x_n>0\}$ in $\R^n$ endowed with the metric $(\frac 1 {x_n^2})(dx_1^2+\ldots+dx_n^2)$. Up to isometry, we can assume that $\gamma(0)=(0,\ldots,0,1)$, $\gamma'(0)=(1,0,\ldots,0)$ and that the tangent horosphere is $\{x_n=1\}$. 
	
	{Let us first} show that {$\gamma$ lies on the concave side of the horosphere for small $t$, namely, denoting $\gamma(t)=(\gamma_1(t),\ldots,\gamma_n(t))$, that} $\gamma_n(t)<1$ for small $t$. Since $\gamma_n(0)=1$ and $\gamma_n'(0)=0$, it will be sufficient to check that $\gamma_n''(0)<0$. Using the assumption on $\gamma'(0)$ and a direct computation of the Christoffel symbols $\Gamma_{11}^n=1$, we get
	$$(D_{\gamma'}\gamma')_n(0)=\gamma_n''(0)+1~.$$
	Since by hypothesis $\gamma$ has small acceleration, and at $\gamma(0)$ the metric of the upper half-space model coincides with the standard metric $dx_1^2+\ldots+dx_n^2$, $|(D_{\gamma'}\gamma')_n(0)|<1$ and therefore $\gamma_n''(0)<0$. {We conclude that, for suitable $\epsilon>0$, $\gamma_n(t)< 1$ for all $t\in (-\epsilon, \epsilon)\setminus \{0\}$.}
	
	Let us now show that $\gamma(t)$  {lies in the interior of} the concave side of the tangent horosphere $\{x_n=1\}$ for all {$t\ne 0$}, that is, that {$\gamma_n(t)< 1$ for all $t\ne 0$}. Suppose by contradiction that $\gamma_n(t_0)=1$ for some {$t_0\geq \epsilon$}. Then $\gamma_n$ has a minimum point {$t_{\min}$} in $(0,t_0)$, with minimum value $m<1$. The horosphere $\{x_n=m\}$ is then tangent to $\gamma$ at {$\gamma(t_{\min})$} and $\gamma_n(t)\geq m$ for $t$ in a neighbourhood of $t_{\min}$. By re-applying the argument of the previous part of the proof, this gives a contradiction. See Figure \ref{fig:tangenthorospheres}.
\end{proof}

\begin{figure}[htbp]
	\centering
	\includegraphics[height=5.5cm]{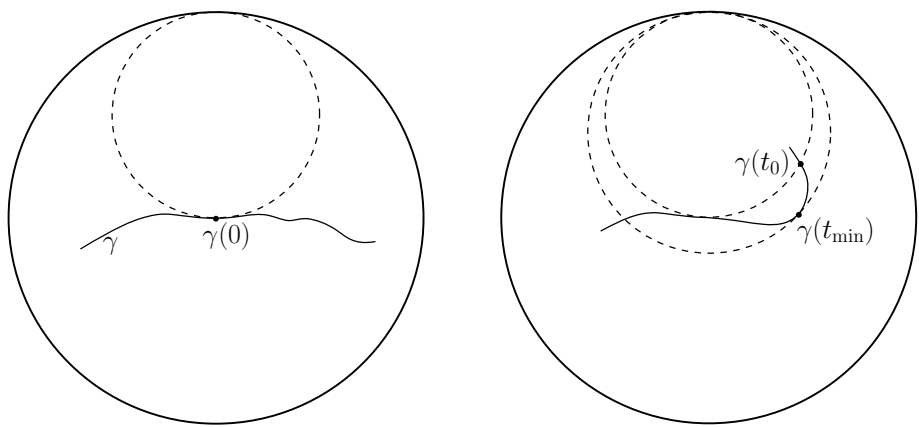} 
	
	\caption{A schematic picture of the argument in the proof of Lemma \ref{lemma curve small acc}. On the left, for $t\in(-\epsilon,\epsilon)$ the image of the curve $\gamma(t)$ lies in the concave side of the horosphere tangent to $\gamma$ at $t=0$. On the right, the same holds in fact for every $t$, for otherwise one would obtain a contradiction with the first part of the proof at the minimum point $t_{\min}$.}\label{fig:tangenthorospheres}
\end{figure}


\begin{remark}\label{rmk princ curv horosphere}
	Given an immersion $\sigma:M^n\to \Hyp^{n+1}$ (or in general into any Riemannian manifold), a curve $\gamma:[a,b]\to M$ is a geodesic for the first fundamental form {of $\sigma$} (in short, {it} is an \emph{intrinsic} geodesic) if and only if $D_{(\sigma\circ\gamma)'}(\sigma\circ\gamma)'$ is orthogonal to the image of $\sigma$. In this case we have indeed
	\begin{equation}\label{eq:acceleration intrinsic geodesic}
		D_{(\sigma\circ\gamma)'}(\sigma\circ\gamma)'=\II(\gamma'(t),\gamma'(t))\nu(\gamma(t))
	\end{equation}
	where $\nu$ is the unit normal vector of the immersion with respect to the chosen orientations.
	
	By applying this remark to an intrinsic geodesic for the horosphere $\{x_n=1\}$, which has the form $\gamma(t)=(a_1t,\ldots,a_{n-1}t,1)$ (here $\sigma$ is simply the inclusion), and repeating  the same computation of the proof of Lemma \ref{lemma curve small acc}, we see that the second fundamental form of a horosphere equals the first fundamental form. Hence the principal curvatures of a horosphere are all identically equal to $1$ for the choice of inward normal vector, and therefore the shape operator is the identity at every point, a fact we have already used in Example \ref{ex:horospheres}. 
\end{remark}

An immediate  consequence of Lemma \ref{lemma curve small acc} is the following:
\begin{lemma} \label{lemma concave side horospheres}
	Given a {complete}  immersion $\sigma:M^n\to\Hyp^{n+1}$ with small principal curvatures, the image of $\sigma$ lies strictly on the concave side of any tangent horosphere. That is, for every $p\in M$, $\sigma(M\setminus \{p\})$ lies  in the interior of the concave side of each of the two horospheres tangent to $\sigma$ at $\sigma(p)$.
\end{lemma}
\begin{proof}
	
	Let us fix $p\in M$ and let $q\in M$, with $p\ne q$. By completeness
	there exists an intrinsic geodesic $\gamma$ on $M$ joining $p$ and $q$, which we assume to be parameterized by arclength. Applying Equation \eqref{eq:acceleration intrinsic geodesic} as in Remark \ref{rmk princ curv horosphere}, we have 
	$$\| D_{(\sigma\circ\gamma)'}(\sigma\circ\gamma)'\|=|\II(\gamma'(t),\gamma'(t))|<\I(\gamma'(t),\gamma'(t))=\|(\sigma\circ\gamma)'(t)\|^2=1~,$$
	hence $\sigma\circ\gamma$ has small acceleration. The conclusion follows from Lemma \ref{lemma curve small acc}.
\end{proof}

\begin{remark}\label{rmk: tangent metric spheres}
	Observe that any metric sphere in $\Hyp^{n+1}$ is contained in the convex side of any tangent horosphere. As a result, a hypersurface with small principal curvatures lies in the complementary of any metric ball of {$\Hyp^{n+1}$} whose boundary is tangent to the hypersurface. See Figure \ref{fig:tangentsurfaces}.
\end{remark}

\begin{remark}\label{rmk:convex side caps}
	A $r$-\emph{cap} in the hyperbolic space is the hypersurface at (signed) distance $r$ from a totally geodesic plane. By a simple computation (for instance using Equation \eqref{eq:shape normal evo}), $r$-caps are umbilical hypersurfaces with principal curvatures  {identically equal to $-\tanh(r)$, computed with respect to the unit normal vector pointing to the side where $r$ is increasing}. Now, if $\sigma:M\to\Hyp^{n+1}$ is  an immersion with principal curvatures smaller that $\epsilon=\tanh(r){\in (0,1)}$ in absolute value, then one can repeat wordly the proofs of Lemma \ref{lemma curve small acc} and Lemma \ref{lemma concave side horospheres}, by replacing horospheres with $r$-caps, and conclude that the image of $\sigma$ lies strictly on the {concave} side of every tangent $r$-cap {for $r=\arctanh(\epsilon)$}.  See Figure \ref{fig:tangentsurfaces}.
	A similar conclusion (which is however not interesting for the purpose of this paper) could of course be obtained  under the assumption that $\sigma$ has principal curvatures bounded by some constant $\epsilon>1$, in terms of tangent {metric} spheres with curvature greater than $\epsilon$ in {absolute value}.
\end{remark}

\begin{figure}[htbp]
	\centering
	\includegraphics[height=5.5cm]{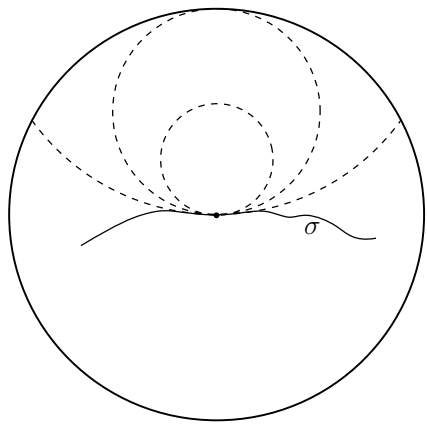} 
	
	\caption{Schematically, an immersion $\sigma$ tangent at one point to a metric sphere (whose principal curvatures are larger than $1$), a horosphere (equal to $1$) and a $r$-cap (smaller than $1$). The image of $\sigma$ is contained in the concave side of the three of them.}\label{fig:tangentsurfaces}
\end{figure}

\subsection{Injectivity results}
Having established these preliminary results, let us finally discuss the global injectivity of $\sigma$ and $G_\sigma$ under the hypothesis of completeness.
{Before that, we relate the completeness assumption for $\sigma$ to some topological conditions}.


\begin{remark} \label{rmk:proper implies complete}
	Let us observe that proper immersions $\sigma\colon M\to \Hyp^{n+1}$ are complete.  
	{Indeed, i}f $p,q\in M$ have distance {at most} $r$ for the first fundamental form $\I$, then, by definition of distance on a Riemannian manifold,  $dist_{\Hyp^{n+1}}(\sigma(p),\sigma(q))\le r$: as a result, 
	\[
	\sigma(B_{\I}(x,r)) \subset B_{\Hyp^{n+1}}(x,r).
	\]
	Assuming $\sigma$ is proper, $\sigma^{-1}(\overline{B_{\Hyp^{n+1}}(x,r)})$ is a compact subspace of $M$ containing $B_{\I}(x,r)$, therefore $\overline {B_{\I}(x,r)}$ is compact. We conclude that $\I$ is complete by Hopf-Rinow Theorem.
\end{remark}

{A less trivial result is that Remark \ref{rmk:proper implies complete} can be reversed for immersions with small principal curvatures: in fact,} for immersions with small principal curvatures, being properly immersed, properly embedded and complete are all equivalent conditions

\begin{prop} \label{prop injectivity}
	Let $M^n$ be a manifold and $\sigma:M\to\Hyp^{n+1}$ be a {complete}  immersion  with small principal curvatures. Then $\sigma$ is a proper embedding and $M$ is diffeomorphic to $\R^n$.
\end{prop}
\begin{proof}
	To show that $\sigma$ is injective, let us suppose {by contradiction} that $\sigma(p)=\sigma(q)=y_0$ for $p\neq q$. Let $\gamma:[a,b]\to M$ be an intrinsic {$\I$-}geodesic joining $p$ and $q$ {parametrized by arclength}, which exists because $\I$ is complete. As in Lemma \ref{lemma concave side horospheres}, $\sigma\circ\gamma$ has small acceleration. Let 
	$$r_0:=\max_{t\in [a,b]}d_{\Hyp^{n+1}}(y_0,\sigma\circ\gamma(t))~.$$
	Then $\sigma\circ\gamma$ is tangent at some point $\sigma\circ\gamma(t_0)$ to the metric sphere in {$\Hyp^{n+1}$} centered at $y_0$ of radius $r_0$, and contained in its convex side. {By Remark \ref{rmk: tangent metric spheres}, $\sigma\circ\gamma$ lies in the convex side of the horosphere tangent to the hypersurface at $\sigma\circ\gamma(t_0)$.} This contradicts Lemma \ref{lemma curve small acc} and shows that $\sigma$ is an injective immersion. See Figure \ref{fig:tangenthorospheres2}.

	\begin{figure}[htbp]
		\centering
		\includegraphics[height=5.5cm]{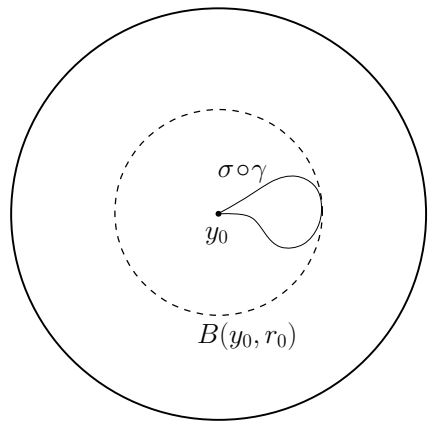} 
		
		\caption{A sketch of the proof of the first part of Proposition \ref{prop injectivity}, namely the injectivity of $\sigma$. If $\sigma(p)=\sigma(q)=y_0$ for $p\neq q$, then the image $\sigma\circ\gamma$ of a $\I$-geodesic connecting $p$ and $q$ would be tangent to a metric ball centered at $y_0$, which contradicts the assumption that $\sigma$ has small principal curvatures.}\label{fig:tangenthorospheres2}
	\end{figure}

	It follows that $M$ is simply connected. Indeed, let $u:\widetilde M\to M$ be a universal covering. If $M$ were not simply connected, then $u$ would not be injective, hence $\sigma\circ u$ would give a non-injective immersion in $\Hyp^{n+1}$ with small principal curvatures, contradicting the above part of the proof. Since the first fundamental form is a {complete} negatively curved Riemannian metric on $M$ (Remark \ref{rmk: negative curv}), $M$ is {diffeomorphic to $\R^n$} {by the Cartan-Hadamard Theorem}. 
	
	Let us now show that $\sigma$ is proper, which also implies that it is a homeomorphism onto its image and thus an embedding. As a first step, suppose $y_0\in\Hyp^{n+1}$ is in the closure of the image of $\sigma$. We claim that the normal direction of $\sigma$ extends to $y_0$, meaning that there exists a vector $\nu_0\in T^1_{x_0}\Hyp^{n+1}$ such that  $[\nu(p_n)]\to [\nu_0]$ for every sequence $p_n\in M$ satisfying $\sigma(p_n)\to y_0$, where $\nu(p)$ denotes the unit normal vector of $\sigma$ at $p$ and $[\cdot]$ denotes the equivalence class up to multiplication by $\pm 1$. By compactness of unit tangent spheres, if $\sigma(p_n)\to y_0$ then one can extract a  subsequence $\nu(p_n)$ converging to (say) $\nu_0$. Observe that by Lemma \ref{lemma concave side horospheres}, the image of $\sigma$ lies in the concave side of any horosphere orthogonal to $\nu(p_n)$ at $\sigma(p_n)$. By a continuity argument, it lies also on the concave side of each of the two horospheres orthogonal to $\nu_0$ at $y_0$. The claim follows by a standard subsequence argument once we show that there can be no limit other than $\pm\nu_0$ along any subsequence. 
	
	{We will assume hereafter, in the upper half-space model 
		$$(\{x_{n+1}>0\},\frac 1 {x_{n+1}^2}(dx_1^2+\ldots+dx_{n+1}^2))~,$$ 
		that $y_0=(0,\ldots,0,1)$ and $\nu_0=(0,\ldots,0,1)$. {See Figure \ref{fig:graphs} on the left.} In this model, horospheres are either horizontal hyperplanes $\{x_{n+1}=c\}$ or spheres with south pole on $\{x_{n+1}=0\}$. By Lemma \ref{lemma concave side horospheres}, the image of $\sigma$ is contained in the concave side of both horospheres orthogonal to $\nu_0$, hence it lies in the region defined by $0<x_{n+1}\leq 1$ and $x_1^2+\ldots+x_{n}^2+(x_{n+1}-\frac{1}{2})^2\geq \frac{1}{4}$. Now, if $\nu_1\neq \pm\nu_0$ were a subsequential limit of $\nu(q_n)$ for some sequence $q_n$ with $\sigma(q_n)\to y_0$, then the image of $\sigma$ would lie on the concave side of some sphere with south pole on $\{x_{n+1}=0,(x_1,\ldots,x_{n})\neq(0,\ldots,0)\}$. But then $\sigma$ would either enter the region $x_{n+1}>1$ or the region $x_1^2+\ldots+x_{n}^2+(x_{n+1}-\frac{1}{2})^2< \frac{1}{4}$ in a neighbourhood of $y_0$, which gives a contradiction.}
	
	Having established the convergence of the normal direction to $[\nu_0]$, we can now find a neighbourhood $U$ of $y_0$ of the form $B(0,\epsilon)\times(\frac{1}{2},\frac{3}{2})$, where $B(0,\epsilon)$ is the ball of Euclidean radius $\epsilon$ centered at the origin in $\{x_{n+1}=0\}$, such that if $\sigma(p)\in U$, then the vertical projection from the tangent space of $\sigma$ at $\sigma(p)$ to $\{x_{n+1}=0\}$ is a linear isomorphism. By the implicit function theorem, $\sigma(M)\cap U$ is locally a graph over $\R^n$. Up to taking a smaller $\epsilon$, we can {arrange $U$ so that $\sigma(M)\cap U$ is a global graph over some open set of  $B(0,\epsilon)\subset\R^n$. Indeed as long as the normal vector $\nu$ is in a small neighbourhood of $\pm\nu_0$, the vertical lines over points in $B(0,\epsilon)$ may intersect the image of $\sigma$ in at most one point as a consequence of Lemma \ref{lemma concave side horospheres}. Let us denote $V\subseteq B(0,\epsilon)$ the image of the vertical projection from $\sigma(M)\cap U$ to $\R^n$, so that $\sigma(M)\cap U$ is 
		the graph of some function 
		$h:V\to(\frac{1}{2},\frac{3}{2})$ satisfying $h(0)=1$}.  Since the gradient of $h$ converges to $0$ at $0$, {up to restricting $U$ again, there is a constant $C>0$ such that the Euclidean norm of the gradient of $h$ is bounded by $C$.}
	
	{We shall now apply again the hypothesis that $\sigma$ is complete {to show that in fact $V=B(0,\epsilon)$. For this purpose, we assume that $V$ is a proper (open) subset of $B(0,\epsilon)$ and we will derive a contradiction. Under the assumption $V\neq B(0,\epsilon)$ we would find a Euclidean segment $c:[0,1]\to\R^n$ such that $c(s)\in V$ for $s\in[0,1)$ and $c(1)\in B(0,\epsilon)\setminus V$. The path $s\mapsto (c(s),h\circ c(s))$ is contained in $\sigma(M)$;  
			using $h\geq\frac{1}{2}$ and the bound on the gradient, we obtain that its hyperbolic length is less than $2\sqrt{1+C^2}$ times the Euclidean length of $c$, hence is finite. This contradicts completeness of $\sigma$. 
			In summary},  $\sigma(M)\cap U$ is the graph of a function globally defined on $B(0,\epsilon)$, and clearly contains the point $y_0$.}	{See Figure \ref{fig:graphs} on the right.}
	
	{We are now ready to complete the proof of the fact that $\sigma$ is proper. Indeed, let $p_n\in M$ be a sequence such that $\sigma(p_n)\to y_0$. We showed above that $y_0$ is in the image of $\sigma$ (say $y_0=\sigma(p_0)$) and that $p_n$ is definitively in $\sigma^{-1}(U)$, whose image is a graph over $B(0,\epsilon)$. Hence $p_n$ is at bounded distance from $p_0$ for the first fundamental form of $\sigma$, and therefore admits a subsequence $p_{n_k}$ converging to $p_0$. In conclusion, $\sigma$ is a proper embedding.}
\end{proof}

\begin{figure}[htbp]
	\centering
	\includegraphics[height=5.5cm]{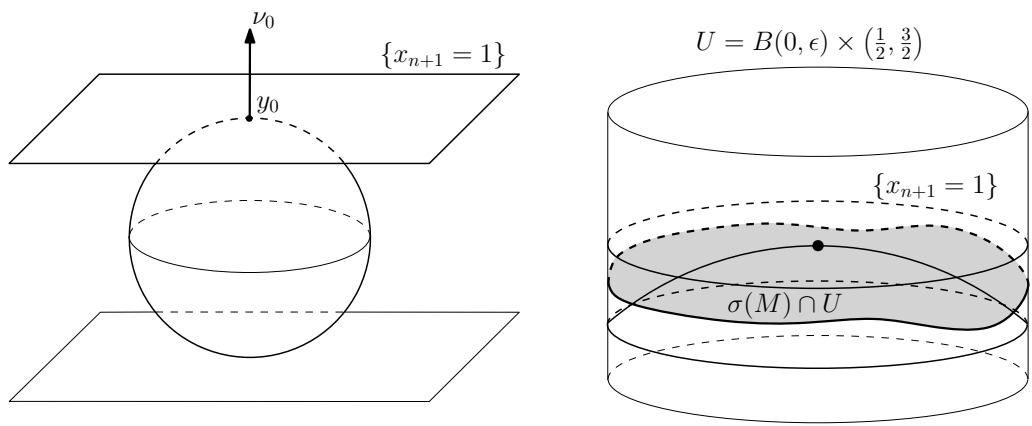} 
	
	\caption{The setting of the proof that $\sigma$ is proper in Proposition \ref{prop injectivity}: in the upper half-plane model, the image of $\sigma$ is contained below the horosphere $\{x_{n+1}=1\}$ and in the outer side of the horosphere $x_1^2+\ldots+x_{n}^2+(x_{n+1}-\frac{1}{2})^2= \frac{1}{4}$. On the right, the neighbourhood $U$ of $y_0$, where the image of $\sigma$ is proved to be the graph of a function $h:B(0,\epsilon)\to\R$.}\label{fig:graphs}
\end{figure}


{By an application of Lemma \ref{lemma concave side horospheres} one can easily show that the Gauss map $G_\sigma:M\to \G{n+1}$ is injective as well {if $\sigma$ is complete and has small principal curvatures}. However, we will prove here (Proposition \ref{prop:gauss maps diffeo onto image}) a stronger property of the Gauss map.}
Recall that the space of oriented geodesics of $\Hyp^{n+1}$ has the natural identification 
$$\G{n+1}\cong \partial\Hyp^{n+1}\times\partial\Hyp^{n+1}\setminus \Delta~,$$
for $\Delta$ the diagonal, given by mapping an oriented geodesic $\ell$ to its endpoints at infinity according to the orientation: {as a consequence, the map $G_\sigma$ can be seen as a pair of maps with values in the boundary of $\Hyp^n$}. More precisely, if we denote by $\gamma:\R\to\Hyp^{n+1}$ a parameterized geodesic, then the above identification reads:
\begin{equation}\label{eq:identification space geodesics}
	\gamma\mapsto\left(\lim_{t\to+\infty}\gamma(t),\lim_{t\to-\infty}\gamma(t)\right)~.
\end{equation}

Given an immersion of an oriented manifold $M^n$
into $\Hyp^{n+1}$, composing the Gauss map $G_\sigma:M\to\G{n+1}$ with the above map \eqref{eq:identification space geodesics} with values in $\partial\Hyp^{n+1}\times\partial\Hyp^{n+1}$ and projecting on each factor, we obtain the so-called \emph{hyperbolic Gauss maps} $G^\pm_\sigma:M\to\partial\Hyp^{n+1}$. They are explicitely expressed by
$$G^\pm_\sigma(p)=\lim_{t\to\pm\infty}\exp_{\sigma(p)}(t\nu(p))\in \partial\Hyp^{n+1}~.$$
The following proposition states their injectivity property under the small principal curvatures assumption, which will be applied in Proposition \ref{prop:action free prop disc},

\begin{prop}\label{prop:gauss maps diffeo onto image}
	Let $M^n$ be an oriented manifold and $\sigma:M\to\Hyp^{n+1}$ be a complete  immersion  with small principal curvatures. Then both hyperbolic Gauss maps $G^\pm_\sigma:M\to\partial\Hyp^{n+1}$ are diffeomorphisms onto their images. {In particular, the Gauss map $G_\sigma$ is an embedding.}
\end{prop}
\begin{proof}
	Let us first show that $G_\sigma^\pm$ are local diffeomorphisms. Recalling the definition of $\sigma_t$ (Definition \ref{Def normal evolution}) and its expression in the hyperboloid model of $\Hyp^{n+1}$ (Equation \eqref{eq:normal evolution hyperboloid}), $G_\sigma^\pm$ is the limit for $t\to \pm\infty$ in $\partial\Hyp^{n+1}$ of 
	$$\sigma_t(p)=\cosh(t)\sigma(p)\pm\sinh(t)\nu(p)~.$$
	
	{Recalling the definition of the boundary of $\Hyp^{n+1}$ as the projectivization of the null-cone (Equation \eqref{eq:bdy hyperboloid model})}, we will consider the boundary at infinity of $\Hyp^{n+1}$ {as the slice of the null-cone defined by $\{x_{n+2}=1\}$}. Given $p_0\in M$, up to isometries we can assume that $\sigma(p_0)=(0,\ldots,0,1)$ and $\nu(p_0)=(1,0,\ldots,0)$, 
	so that $G_\sigma^\pm(p_0)=(\pm 1,0,\ldots,0,1)$ and the tangent spaces to the image of $\sigma$ at $p_0$ and of $\partial\Hyp^{n+1}$ at $G_\sigma^\pm(p)$ are identified to the same subspace {$\{x_1=x_{n+2}=0\}$ in $\R^{n,1}$}.
	
	To compute the differential of $G_\sigma^\pm$ at $p_0$, we must differentiate the maps
	$$p\mapsto \lim_{t\to \pm\infty}\frac{\sigma_t(p)}{|\langle \sigma_t(p),\sigma(p_0)\rangle|}=\frac{\sigma(p)\pm\nu(p)}{|\langle \sigma(p)\pm\nu(p),\sigma(p_0)\rangle|}$$
	at $p=p_0$. Under these identifications, a direct computation for $V\in T_{p_0}M$ gives:
	$$dG_\sigma^\pm(V)=d\sigma\circ (\mathrm{id}\mp B)(V)~.$$
	Hence both differentials of $G_\sigma^\pm$ are invertible at $p_0$ if the eigenvalues of $B$ are always different from $1$ and $-1$, as in our hypothesis. This shows that $G_\sigma^+$ and $G_\sigma^-$ are local diffeomorphisms. 
	
	To see that $G_\sigma^\pm$ is injective, suppose that $G_\sigma^\pm(p)=G_\sigma^\pm(q)$. This means that $\sigma$ is orthogonal at $p$ and $q$ to two geodesics having a common point at infinity. Hence $\sigma$ is tangent at $p$ and $q$ to two horospheres $H_p$ and $H_q$ having the same point at infinity. By Lemma \ref{lemma concave side horospheres} the image of $\sigma$ must lie in the concave side of both $H_p$ and $H_q$, hence the two horospheres must coincide. But by Lemma \ref{lemma concave side horospheres} again, $\sigma(M\setminus \{p\})$ lies strictly in the concave side of $H_p$, hence necessarily $p=q$. See Figure \ref{fig:tangenthorospheres3}.
	
	{By the invariance of the domain, $G_\sigma^\pm$ are diffeomorphisms onto their images. Under the identification between $\G{n+1}$ and $\partial\Hyp^{n+1}\times\partial\Hyp^{n+1}\setminus \Delta$ the Gauss map $G_\sigma$ corresponds to $(G_\sigma^+,G_\sigma^-)$, and it follows that $G_\sigma$ is an embedding.}
\end{proof}

\begin{figure}[htbp]
	\centering
	\includegraphics[height=5.5cm]{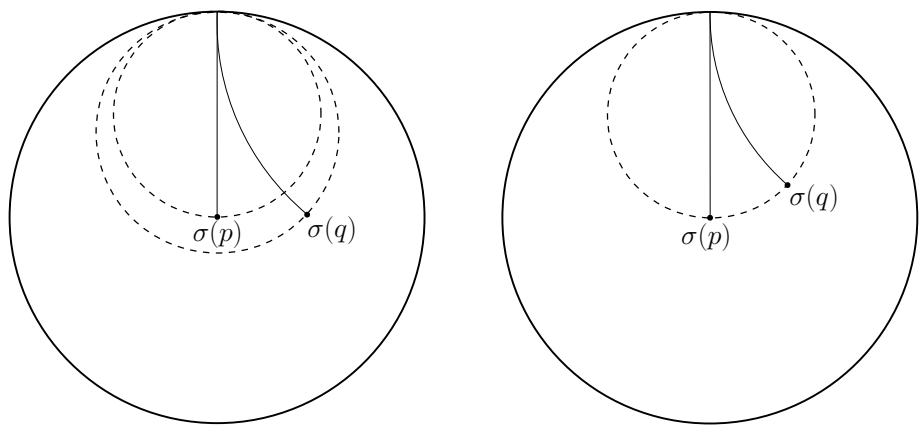} 
	
	\caption{The proof of the injectivity of $G_\sigma^+$. Suppose two orthogonal geodesics share the final point, hence the image of $\sigma$ is tangent at $\sigma(p)$ and $\sigma(q)$ have the same point at infinity. As a consequence of Lemma \ref{lemma concave side horospheres}, this is only possible if the two tangent horospheres coincide, and therefore if $p=q$. Replacing horospheres by metric spheres, the same argument proves that the orthogonal geodesics at different points are disjoint (See Proposition \ref{prop:action free prop disc0}).}\label{fig:tangenthorospheres3}
\end{figure}

\section{Nearly-Fuchsian manifolds and the space $\GGG_\rho$}\label{sec:nearly fuchsian}
Taking advantage of the results so far established in this chapter, we now introduce nearly-Fuchsian representations and manifolds. These will appear again in Chapter \ref{sec:hamiltonian}.

\begin{Def}\label{Def:nearly fuch rep}
	Let $M^n$ be a closed orientable manifold. A representation $\rho:\pi_1(M)\to\Isom_0(\Hyp^{n+1})$ is called \emph{nearly-Fuchsian}  if there exists a $\rho$-equivariant immersion $\widetilde\sigma:\widetilde M\to\Hyp^{n+1}$ with small principal curvatures. 
\end{Def}

We recall that an immersion $\widetilde\sigma\colon \widetilde M \to \Isom_0(\Hyp^{n+1})$ is \emph{$\rho$-equivariant} if
\begin{equation}\label{eq:equivariant immersion}
	\widetilde\sigma\circ\alpha= \rho(\alpha)\circ\widetilde \sigma~.\end{equation}
for all $\alpha\in \pi_1(M)$. Let us show that the action of nearly-Fuchsian representations is ``good'' on $\Hyp^{n+1}$ {(Proposition \ref{prop:action free prop disc0})} and also on a  region in $\partial\Hyp^{n+1}$ which is the disjoint union of two topological discs {(Proposition \ref{prop:action free prop disc})}.

\begin{prop}\label{prop:action free prop disc0}
	Let $M^n$ be a closed orientable manifold and $\rho:\pi_1(M)\to\Isom_0(\Hyp^{n+1})$ be a nearly-Fuchsian representation. Then $\rho$ gives a free and properly discontinous action of $\pi_1(M)$ on $\Hyp^{n+1}$. Moreover $\rho$ is convex cocompact, namely there exists a $\rho$-invariant geodesically convex subset $\mathcal C\subset\Hyp^{n+1}$ such that the quotient $\faktor{\mathcal C}{\rho(\pi_1(M))}$ is compact.
\end{prop}
\begin{proof}
	{Let $\widetilde\sigma$ be an equivariant immersion as in Definition \ref{Def:nearly fuch rep}. }We claim that the family of geodesics orthogonal to $\widetilde\sigma(\widetilde M)$ gives a foliation of $\Hyp^{n+1}$.
	Observing that the action of $\pi_1(M)$ on $\widetilde M$ is free and properly discontinous, this immediately implies that the action of $\pi_1(M)$ on $\Hyp^{n+1}$ induced by $\rho$ is free and properly discontinous.
	
	By repeating the same argument that shows, in the proof of  Proposition \ref{prop:gauss maps diffeo onto image}, the injectivity of $G_{\widetilde\sigma}^\pm$, replacing horospheres with {metric spheres of $\Hyp^{n+1}$} {and using Remark \ref{rmk: tangent metric spheres}}, one can prove that  two geodesics orthogonal to $\widetilde\sigma(\widetilde M)$ at different points are disjoint. 	
	
	
	To show that the orthogonal geodesics give a foliation of $\Hyp^{n+1}$, it remains to show that  every point $x\in\Hyp^{n+1}$ is contained in a geodesic of this family (which is necessarily unique). Of course we can assume $x\notin \widetilde\sigma(\widetilde M)$. {By cocompactness, $\widetilde\sigma$ is complete, hence it is a proper embedding by Proposition \ref{prop injectivity}. Then} 
	the map that associates to each element of $\widetilde\sigma(\widetilde M)$ its distance from $x$ attains its minimum: this implies that there exists $r>0$ such that the metric sphere of radius $r$ centered at $x$ is tangent to $\widetilde\sigma(\widetilde M)$ at some point $p$. Hence $x$ is on the geodesic through $p$. See Figure \ref{fig:limit2}.
	
	Let us now prove that $\rho$ is also convex-cocompact. To show this, we claim that there exists $t_+,t_-\in\R$ such that $\widetilde\sigma_{t_+}$ is a convex embedding, and $\widetilde\sigma_{t_-}$ a concave one.  
	Indeed in the proof of Lemma \ref{lemma:evolution fsigma} we showed that the principal curvatures of the normal evolution $\widetilde\sigma_t$ are equal to $\tanh(\mu_i-t)$, where $\mu_i$ is the hyperbolic arctangent of the corresponding principal curvature of $\widetilde\sigma$. Hence taking $t\ll0$ (resp. $t\gg 0$) one can make sure that the principal curvatures of $\widetilde\sigma_t$ are all negative (resp. positive), hence $\widetilde\sigma_t$ is convex (resp. concave). The region bounded by the images of $\widetilde\sigma_{t_+}$ and $\widetilde\sigma_{t_-}$ is then {geodesically} convex and diffeomorphic to $\widetilde M\times [t_-,t_+]$. Under this diffeomorphism, the action of $\pi_1(M)$ corresponds to the action by deck transformations on $\widetilde M$ and the trivial action on the second factor. Hence its  quotient is compact, being diffeomorphic to $M\times [t_-,t_+]$.
\end{proof}

{This implies that in dimension three, nearly-Fuchsian manifolds are quasi-Fuchsian. }

\begin{remark}
	There is another important consequence of Proposition \ref{prop:action free prop disc0}.
	Given $\widetilde\sigma$ an equivariant immersion as in Definition \ref{Def:nearly fuch rep}, it follows from the cocompactness of the action of $\rho$ on the geodesically convex region $\mathcal C$ that $\widetilde \sigma$ is a quasi-isometric embedding {in the sense of metric spaces}. By cocompactness and Remark \ref{rmk: negative curv}, $\widetilde M$ is a complete simply connected Riemannian manifold of negative sectional curvature, hence its visual boundary $\partial\widetilde M$ in the sense of Gromov is homeomorphic to $S^{n-1}$.
	By \cite[Proposition 6.3,][]{zbMATH01496599}, $\widetilde\sigma$  extends to a continuous injective map $\partial\widetilde\sigma$ from the visual boundary $\partial\widetilde M$ of $\widetilde M$ to $\partial\Hyp^{n+1}$. By compactness of $\partial\widetilde M$,
	the extension of $\widetilde\sigma$  is a homeomorphism onto its image.
	
	Since any two $\rho$-equivariant embeddings $\widetilde\sigma_1,\widetilde\sigma_2:\widetilde M\to\Hyp^{n+1}$ are at bounded distance from  each other by cocompactness, the extension $\partial\widetilde\sigma$ does not depend on $\widetilde\sigma$, but only on the representation $\rho$. In conclusion, the image of $\partial\widetilde\sigma$ is a topological $(n-1)$-sphere $\Lambda_\rho$ in $\partial\Hyp^{n+1}$, called the \emph{limit set} of the representation $\rho$. See Figure \ref{fig:limit}.
\end{remark}

\begin{figure}[htbp]
	\centering
	\includegraphics[height=4.8cm]{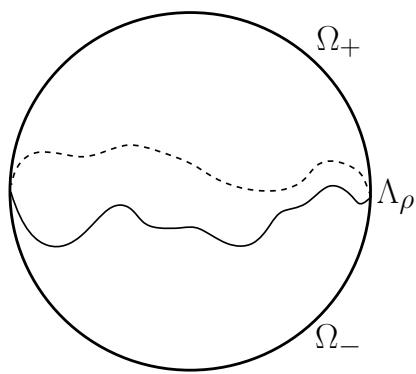} 
	
	\caption{A picture of the limit set $\Lambda_\rho$, which is a topological $(n-1)$-sphere  and disconnects $\partial\Hyp^{n+1}$ in two connected components $\Omega_+$ and $\Omega_-$, which are homeomorphic to $n$-discs.}\label{fig:limit}
\end{figure}

\begin{prop} \label{prop:action free prop disc}
	Let $M^n$ be a closed orientable manifold, $\rho:\pi_1(M)\to\Isom_0(\Hyp^{n+1})$ be a nearly-Fuchsian representation and $\Lambda_\rho$ be its limit set. Then {the action of $\rho$ extends to} a free and properly discontinous action on $\partial\Hyp^{n+1}\setminus\Lambda_\rho$, which is the disjoint union of two topological $n$-discs.
\end{prop}
\begin{proof}
	{Since the action of $\pi_1(M)$ on $\widetilde M$ is free and properly discontinuous, and $G_{\widetilde\sigma}^\pm$ are diffeomorphisms onto their image by Proposition \ref{prop:gauss maps diffeo onto image}, it follows that the action of $\rho(\pi_1(M))$} is free and properly discontinuous on $G_{\widetilde\sigma}^+(\widetilde M)$ and $G_{\widetilde\sigma}^-(\widetilde M)$, which are topological discs in $\Hyp^{n+1}$ since $\widetilde M$ is diffeomorphic to $\R^n$. We claim that 
	{$$G_{\widetilde\sigma}^+(\widetilde M)\cup G_{\widetilde\sigma}^-(\widetilde M)=\partial\Hyp^{n+1}\setminus\Lambda_\rho~.$$} 
	Observe that, by the Jordan-Brouwer separation Theorem, the complement of $\Lambda_\rho$ has two connected components, hence the claim will also imply that $G_{\widetilde\sigma}^+(\widetilde M)$ and $G_{\widetilde\sigma}^-(\widetilde M)$ are disjoint because they are both connected.


	In order to show that $\partial\Hyp^{n+1}\setminus\Lambda_\rho\subseteq G_{\widetilde\sigma}^+(\widetilde M)\cup G_{\widetilde\sigma}^-(\widetilde 
	M)$, one can repeat the same argument as Proposition \ref{prop:action free prop disc0}, now using horospheres, to see that every $x$ in the complement of $\Lambda_\rho$ is the endpoint of some geodesic orthogonal to ${\widetilde\sigma}(\widetilde M)$. See Figure \ref{fig:limit2}.
	

	{It only remains to show the other inclusion. By continuity, it suffices to show that every $x\in\Lambda_\rho$ is not on the image of $G_{\widetilde\sigma}^\pm$. Observe that by cocompactness the principal curvatures of $\widetilde\sigma$ are bounded by some constant $\epsilon<1$ in absolute value. Now,} if $x\in \partial \Hyp^{n+1}$ is the endpoint of an orthogonal line $\ell$, then, for all $r$, one would be able to construct a $r$-cap tangent to $\ell \cap \widetilde \sigma(\widetilde M)$ such that $x$ lies in the convex side of the $r$-cap: since, by Remark \ref{rmk:convex side caps}, for some $r$ the image of $\widetilde \sigma$ lies in the concave side of the $r$-cap, $x$ cannot lie in $\partial \widetilde \sigma(\widetilde M)= \Lambda_\rho$.
	See Figure \ref{fig:limit2} again.
\end{proof}

\begin{figure}[htbp]
	\centering
	\includegraphics[height=4.8cm]{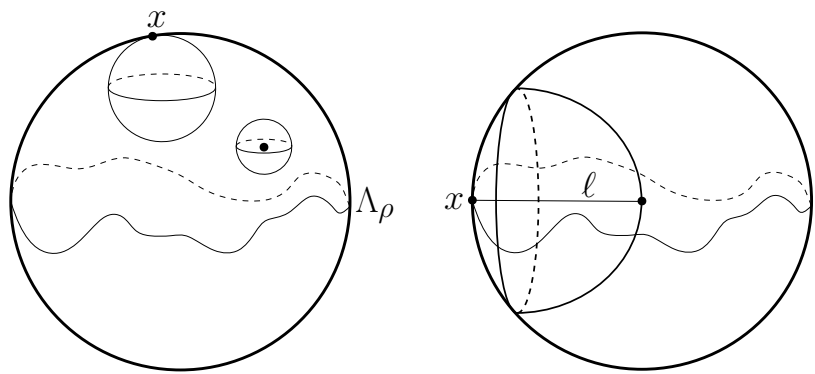} 
	
	\caption{The arguments in the proof of Proposition \ref{prop:action free prop disc}. On the left, since $\widetilde\sigma$ is proper and extends to $\Lambda_\rho$ in $\partial\Hyp^{n+1}$, from every point $x\notin\Lambda_\rho$ one can find a horosphere with point at infinity $x$ tangent to the image of $\widetilde\sigma$. The same argument works for an interior point $x$, using metric balls, which is the argument of Proposition \ref{prop:action free prop disc0}. On the right, a $r$-cap orthogonal to a geodesic $\ell$ with endpoint $x$. Since $\widetilde\sigma$ lies on the concave side of tangent $r$-caps for large $r$, $x$ cannot be in the image of $G_{\widetilde\sigma}^\pm$. The same argument is used in Lemma \ref{lemma:convex gauss map diffeo}, under the convexity assumption, in which case it suffices to use tangent hyperplanes instead of $r$-caps.}\label{fig:limit2}
\end{figure}

As a consequence of Proposition \ref{prop:action free prop disc0}, if $\rho:\pi_1(M)\to \Isom_0(\Hyp^{n+1})$ is a nearly-Fuchsian representation, then the quotient $\faktor{\Hyp^{n+1}}{\rho(\pi_1(M))}$ is a complete hyperbolic manifold {diffeomorphic to $M\times\R$}. This motivates the following definition.

\begin{Def}
	A hyperbolic manifold of dimension $n+1$ is \emph{nearly-Fuchsian}  if it is isometric to the quotient $\faktor{\Hyp^{n+1}}{\rho(\pi_1(M))}$, for $M$ a closed orientable $n$-manifold and $\rho:\pi_1(M)\to\Isom_0(\Hyp^{n+1})$ a nearly-Fuchsian representation.  
\end{Def}

\begin{remark}\label{rmk:embedding in the quotient}

	If $\widetilde\sigma:\widetilde M\to\Hyp^{n+1}$ is a $\rho$-equivariant embedding with small principal curvatures, then $\widetilde \sigma$ descends to the quotient defining a smooth injective map $\sigma\colon M \to  \faktor{\Hyp^{n+1}}{\rho(\pi_1(M))}$. Moreover, since $\widetilde\sigma$ is a $\rho$-equivariant homeomorphism with its image, $\sigma$ is a homeomorphism with its image as well {hence its image is an embedded hypersurface}.
\end{remark}

We conclude this section with a final definition which appears in the statement of Theorem \ref{thm:second char ham}. As a preliminary remark, recall from Propositions \ref{prop:gauss maps diffeo onto image} and \ref{prop:action free prop disc}  that if $\widetilde\sigma$ is a $\rho$-equivariant embedding with small principal curvatures, then each of the Gauss maps $G_{\widetilde\sigma}^\pm$ of ${\widetilde\sigma}$ is a diffeomorphism between $\widetilde M$ and a connected component of $\partial\Hyp^{n+1}\setminus \Lambda_\rho$.  Let us denote these connected components by $\Omega_\pm$ {as in Figure \ref{fig:limit}}, so that: 
$$\partial\Hyp^{n+1}\setminus \Lambda_\rho=\Omega_+\sqcup \Omega_-\qquad G_{\widetilde\sigma}^+(\widetilde M)=\Omega_+\qquad G_{\widetilde\sigma}^-(\widetilde M)=\Omega_-~.$$
{with the representation $\rho$ inducing an action of $\pi_1(M)$ on both $\Omega_+$ and $\Omega_-$.  Recalling the identification 
	$$\G{n+1}\cong \partial\Hyp^{n+1}\times\partial\Hyp^{n+1}\setminus \Delta~,$$
	given by 
	$$\gamma\mapsto\left(\lim_{t\to+\infty}\gamma(t),\lim_{t\to-\infty}\gamma(t)\right)~,$$
	{the following definition is well-posed}.
	
	\begin{Def}\label{Def quotient Grho}
		Given a closed oriented $n$-manifold $M$ and a nearly-Fuchsian representation $\rho:\pi_1(M)\to\Isom_0(\Hyp^{n+1})$, we define
		$\GGG_\rho$ as the quotient:
		$$\faktor{\{\gamma\in \G{n+1}\,|\,\lim_{t\to+\infty}\gamma(t)\in\Omega_+\text{ or }\lim_{t\to-\infty}\gamma(t)\in\Omega_-\}}{\rho(\pi_1(M))}~.$$
	\end{Def}

	Observe that, since the  action of $\rho(\pi_1(M))$ on $\partial\Hyp^{n+1}$ is free and properly discontinuous on both $\Omega_+$ and $\Omega_-$, it is also free and properly discontinuous on the region of $\G{n+1}$ consisting of geodesics having either final point in $\Omega_+$ or initial point in $\Omega_-$. Moreover, such region is simply connected, because it is the union of $\Omega_+\times \partial\Hyp^{n+1}\setminus \Delta$ and $\partial\Hyp^{n+1}\times\Omega_-\setminus \Delta$, both of which are simply connected (they retract on $\Omega_+\times\{\star\}$ and $\{\star\}\times\Omega_-$ respectively) and whose intersection $\Omega_+\times\Omega_-$ is again simply connected. 
	We conclude that $\GGG_\rho$ is a $2n$-manifold {whose fundamental group is isomorphic to $\pi_1(M)$}, and is endowed with a natural para-K\"ahler structure induced from that of  $\G{n+1}$ (which is preserved by the action of $\Isom_0(\Hyp^{n+1})$). 
	
	It is important to stress once more that $\Omega_+$ and $\Omega_-$ only depend on $\rho$ and not on ${\widetilde\sigma}$. We made here a choice in the labelling of $\Omega_+$ and $\Omega_-$ which only depends on the orientation of $M$. The other choice of labelling would result in a ``twin'' region, which differs from $\GGG_\rho$ by switching the roles of initial and final endpoints.
	
	A consequence of this construction, which is implicit in the statement of Theorem \ref{thm:second char ham}, is the following:
	
	\begin{cor} \label{cor:embedding in Grho}
		Let $M^n$ be a closed orientable manifold, $\rho:\pi_1(M)\to\Isom_0(\Hyp^{n+1})$ be a nearly-Fuchsian representation and ${\widetilde\sigma}:\widetilde M\to\Hyp^{n+1}$ be a $\rho$-equivariant embedding  of small principal curvatures. For a suitable choice of an orientation on $M$,  the Gauss map of $\widetilde \sigma$ takes values in $\Omega_+\times\Omega_-$, and induces an embedding of $M$ in $\GGG_\rho$ 
	\end{cor}
	\begin{proof}
		The only part of the statement which is left to prove is that the induced immersion of $M$ in $\GGG_\rho$ is an embedding, but the proof follows with the same argument as in Remark \ref{rmk:embedding in the quotient}. 
	\end{proof}

\chapter{Local, global and equivariant integrability of immersions into \\ {$\G{n+1}$}}
\label{sec:local int} 

In this chapter we introduce a connection on the principal $\R$-bundle $\mathrm{p}$ and relate its curvature with the symplectic geometry of the space of geodesics. 

As a consequence, we characterize, in terms of the Lagrangian condition, the immersions in the space of geodesics which can be locally seen as Gauss maps of immersions into $\Hyp^{n+1}$. After that, we focus on the problem of global integrability for immersions into $\G{n+1}$, looking at the obstructions and then focusing on Riemannian immersions. The final sections of this chapter are devoted to characterizing equivariantly integrable Riemannian immersions into $\G{n+1}$ in terms of their extrinsic geometry, namely in terms of their Maslov class.

\section{Connection on the bundle $T^1\Hyp^n\to \G{n}$}
Recall that a \emph{connection} on a principal $G$-bundle $P$ is a $\mathfrak g$-valued 1-form $\omega$ on $P$ such that:
\begin{itemize}
	\item $\mathrm{Ad}_g(R_g^*\omega)=\omega$;
	\item for every $\xi\in G$, $\omega(X_\xi)=\xi$ where $X_\xi$ is the vector field associated to $\xi$ by differentiating the action of $G$.
\end{itemize}
In the special case $G=\R$ which we consider in this paper, {$\omega$ is a real-valued 1-form and }the first property simply means that $\omega$ is invariant under the action of $\R$.

Let us now introduce the connection that we will use in the concrete.
\begin{Def}\label{Def connection form}
	We define the connection form on $\mathrm{p}:T^1\Hyp^{n}\to \G{n}$ as 
	$$\omega(X)=\gs{n}(X,\chi_{(x,v)})$$
	for $X\in T_{(x,v)}T^1\Hyp^{n}$, where $\chi$ is the infinitesimal generator of the geodesic flow.
\end{Def}

The 1-form $\omega$ indeed satisfies the two properties of a connection: the invariance  under the $\R$-action follows immediately from the invariance of $\gs{n}$ (Lemma \ref{lemma:geodflow isometric}) and of $\chi$ (Equation \eqref{eq:diff geoflow3}); the second property follows from Equation \eqref{eq:generator geoflow is unit}, namely $\omega(\chi_{(x,v)})=\gs{n}(\chi_{(x,v)},\chi_{(x,v)})=1$.

The connection $\omega$ is defined in such a way that the associated \emph{Ehresmann connection}, which we recall being a\begin{flushleft}
	\begin{flushleft}
		
	\end{flushleft}
\end{flushleft} distribution of $T^1\Hyp^n$ in direct sum with the tangent space of the fibers of $\mathrm{p}$, is simply the distribution orthogonal to $\chi$. Indeed the Ehresmann connection associated to $\omega$ is the kernel of $\omega$. The subspaces in the Ehresmann connections defined by the kernel of $\omega$ are usually called \emph{horizontal}; we will avoid this term here since it might be confused with the horizontal distribution $\HH$ with respect to the other bundle structure of $T^1\Hyp^n$, namely the unit tangent bundle $\pi:T^1\Hyp^n\to\Hyp^n$. 

Now, a connection on a principal $G$-bundle is \emph{flat} if the Ehresmann distribution is integrable, namely, every point admits a local section tangent to the kernel of $\omega$. We will simply refer to such a section as a \emph{flat} section. The bundle is a \emph{trivial flat} principal $G$-bundle if it has a global flat section.

Having introduced the necessary language, the following statement is a direct reformulation of Proposition \ref{prop:gauss immersion}.

\begin{Prop}\label{prop pullback flat}
	Given an oriented manifold $M^n$ and an immersion $\sigma:M\to\Hyp^{n+1}$, the $G_\sigma$-pull-back of $\mathrm{p}:T^1\Hyp^{n+1}\to\G{n+1}$ is a trivial flat bundle.
\end{Prop}
\begin{proof}
	The lift $\zeta_\sigma:M\to T^1\Hyp^{n+1}$ is orthogonal to $\chi$ by Proposition \ref{prop:gauss immersion} and therefore induces a global flat section of the pull-back bundle {via $G_\sigma= \mathrm p \circ \zeta_\sigma$}.
\end{proof}

\section{Curvature of the connection}
The purpose of this section is to compute the curvature of the connection $\omega$, which is simply the $\R$-valued 2-form $d\omega$. (The general formula for the curvature of a connection on a principal $G$-bundle is $d\omega+ \frac1 2 \omega\wedge\omega$, but the last term vanishes in our case $G=\R$.)

\begin{Remark}\label{rmk:curvature and bracket}
	It is known that the curvature of $\omega$ is the obstruction to the existence of local flat sections. In particular in the next proposition we will use extensively that, given $X,Y\in \chi_{(x,v)}^\perp\subset T_{(x,v)}T^1\Hyp^{n+1}$, if there exists an embedding {$ \zeta:M\to T^1\Hyp^{n+1}$ such that $\zeta(p)=(x,v)$, that $X,Y\in d\zeta(T_p M)$ and that $d\zeta(T_pM ) \subset \chi^\perp$,} then $d\omega(X,Y)=0$.

	This can be easily seen by a direct argument: if we now denote by $X$ and $Y$ some extensions tangential to the image of $\zeta$, one has 
	$$
	d\omega(X,Y)=\partial_X(\omega(Y))-\partial_Y(\omega(X))-\omega([X,Y])=0~,
	$$
	since $\omega(X)=\omega(Y)=0$ by the hypothesis that $X$ and $Y$ are orthogonal to $\chi$, whence $\omega(X)=\gs{n}(X,\chi)=0$, and moreover $\omega([X,Y])=0$ since $[X,Y]$ remains tangential to the image of $\sigma$. 
	
	The argument can in fact be reversed to see that $d\omega$ is exactly the obstruction to the existence of a flat section, by the Frobenius theorem. 
\end{Remark}

The following proposition represents an essential step to relate the curvature of $\omega$ and the symplectic geometry of the space of geodesics.

\begin{Prop}\label{prop:identity curvature form}
	The following identity holds for the connection form $\omega$ on the principal $\R$-bundle $\mathrm{p}:T^1\Hyp^{n+1}\to\G{n+1}$ and the symplectic form $\Omega$ of $\G{n+1}$:
	$$d\omega=\mathrm{p}^*\Omega~.$$
	
\end{Prop}
\begin{proof}
	We shall divide the proof in several steps. 
	
	First, let us show that $d\omega$ is the pull-back of some 2-form on $\G{n+1}$. Since $d\omega$ is obviously invariant under the geodesic flow, we only need to show that $d\omega(X,\chi_{(x,v)})=0$ for all $X\in T_{(x,v)}T^1\Hyp^{n+1}$. Clearly, it suffices to check this for $X\in \chi^\perp$. To apply the exterior derivative formula for $d\omega$, we consider $\chi$ as a globally defined vector field and we shall extend $X$ {around} $(x,v)$. For this purpose, take a curve $c:(-\epsilon,\epsilon)\to T^1\Hyp^{n+1}$ such that $c'(0)=X$ and $c'(s)$ is tangent to $\chi^\perp$ for every $s$. Then define the map $f(s,t)=\varphi_t(c(s))$ and observe that $\chi=\partial_t f$. We thus set $X=\partial_s f$, which is the desired extension {along a two-dimensional submanifold}. Then we have:
	$$d\omega(X,\chi)=\partial_X(\omega(\chi))-\partial_\chi(\omega(X))-\omega[X,\chi]=0~,$$
	where we have used that $\omega(\chi)\equiv 1$, that $\omega(X)=0$ along the curves $t\mapsto \varphi_t(x,v)$ (since the curve $c$ is taken to be in the distribution $\chi^\perp$ and the geodesic flow preserves both $\chi$ and its orthogonal complement), and finally that $[X,\chi]=0$ since {$\chi=\partial_t f$ and $X=\partial_s f$ are coordinate vector fields} for a submanifold in neighborhood of $(x,v)$. 
	
	Having proved this, it is now sufficient to show that $d\omega$ and $\mathrm{p}^*\Omega$ agree when restricted to $\chi^\perp$. Recall that $\Omega$ is defined as $\GG(\cdot,\JJ\cdot)$, where $\GG$ and $\JJ$ are the push-forward to $T_{\mathrm{p}(x,v)}\G{n+1}$, by means of the differential of $\mathrm{p}$, of the metric $\gs{n+1}$ and of the para-complex structure $\mathrm J$ on $\chi^\perp$. Thus we must equivalently  show that 
	\begin{equation}\label{eq:sufficient}
		d\omega(X,Y)=\gs{n+1}(X,\mathrm J Y)\qquad\text{for all }X,Y\in \chi_{(x,v)}^\perp~.
	\end{equation}
	
	To see this, take an orthonormal basis $\{w_1,\ldots,w_n\}$ for $v^\perp\subset T_x\Hyp^{n+1}$, and observe that $\{w_1^\HH,\ldots,w_n^\HH,w_1^\V,\ldots,w_n^\V\}$ is a $\gs{n+1}$-orthonormal basis of $\chi^\perp$. It is sufficient to check \eqref{eq:sufficient} for $X,Y$ {distinct elements of} this basis. We distinguish several cases. 
	
	\begin{itemize}
		\item First, consider the case $X=w_i^\HH$ and $Y=w_j^\HH$, for $i\neq j$. Then $\gs{n+1}(X,{\mathrm J}Y)=\gs{n+1}(w_i^\HH,w_j^\V)=0$ by Definition \ref{Def:parasasaki}. On the other hand, by Example \ref{ex:totally geodesic}, if $\sigma$ is the inclusion of the totally geodesic hyperplane orthogonal to $v$ at $x$, then its lift $\zeta_\sigma$ is a submanifold in $T^1\Hyp^{n+1}$ orthogonal to $\chi$ at every point and tangent to $X$ and $Y$. Then $d\omega(X,Y)=0$ by Remark \ref{rmk:curvature and bracket}.
		
		\item Second, consider $X=w_i^\V$ and $Y=w_j^\V$, for $i\neq j$. Then again $\gs{n+1}(X,{\mathrm J}Y)=\gs{n+1}(w_i^\V,w_j^\HH)=0$. Here we can apply Example \ref{ex:spheres} instead, and see that there is a $n$-dimensional sphere in $\VP_{(x,v)}$ orthogonal to the fibers of $\mathrm{p}$ and tangent to $X$ and $Y$, whence $d\omega(X,Y)=0$ by Remark \ref{rmk:curvature and bracket}.
		
		\item Third, consider $X=w_i^\HH$ and $Y=w_j^\V$, for $i\neq j$. Then $\gs{n+1}(X,{\mathrm J}Y)=\gs{n+1}(w_i^\HH,w_j^\HH)=0$ since $w_i$ and $w_j$ are orthogonal. Let us now apply Example \ref{ex:mixed}, for instance by taking $Q$ the geodesic going through $x$ with speed $w_i$. The normal bundle $\mathrm N^1 Q$ is a submanifold orthogonal to the fibers of $\mathrm{p}$ and tangent to $w_i^\HH$ and $w_j^\V$. So $d\omega(X,Y)=0$ by the usual argument.
		
		\item Finally, we have to deal with the case $X=w_i^\HH$ and $Y=w_i^\V$. Here $\gs{n+1}(X,{\mathrm J}Y)=\gs{n+1}(w_i^\HH,w_i^\HH)=1$. For this computation, we may assume $n=1$, up to restricting to the totally geodesic 2-plane spanned by $v$ and $w$, which is a copy of $\Hyp^2$. Hence we will simply denote $w_i=w$, and moreover we can assume (up to changing the sign) that $w=x\boxtimes v$, where $\boxtimes$ denotes the Lorentzian cross product in $\R^{2,1}$. In other words, $(x,v,w)$ forms an oriented orthonormal triple in $\R^{2,1}$. 
		
		Now, let us extend $X$ and $Y$ to globally defined vector fields on $T^1\Hyp^2$, by means of the assignment $(x,v)\mapsto (x\boxtimes v)^\HH$ and $(x,v)\mapsto (x\boxtimes v)^\V$. By this definition, $X$ and $Y$ are orthogonal to $\chi$; we claim that $[X,Y]=-\chi$. This will conclude the proof, since
		\begin{equation*}
			d\omega(X,Y)=\partial_X(\omega(Y))-\partial_Y(\omega(X))-\omega[X,Y]=-\omega[X,Y]=\gs{2}(\chi,\chi)=1~.
		\end{equation*}
		For the claim about the Lie bracket, let us use the hyperboloid model. Then $X=(x\boxtimes v,0)$ and $Y=(0,x\boxtimes v)$. We can consider $X$ and $Y$ as globally defined (by the same expressions) in the ambient space $\R^{2,1}\times\R^{2,1}$, so as to compute the Lie bracket in $\R^{2,1}\times\R^{2,1}$, which will remain tangential to $T^1\Hyp^2$ since $T^1\Hyp^2$ is a submanifold. Using the formula $[X,Y]=Jac_XY-Jac_YX$, where $Jac_X$ denotes the Jacobian of the vector field thought as a map from $\R^3$ to $\R^3$, we obtain
		$$[X,Y]=(x\boxtimes(x\boxtimes v),0)-(0,(x\boxtimes v)\boxtimes v)=-(v,x)=-\chi_{(x,v)}$$
		by the standard properties of the cross-product.
	\end{itemize}
	In summary, we have shown that $d\omega$ and $\mathrm{p}^*\Omega$ coincide on the basis $\{\chi,w_1^\HH,\ldots,w_n^\HH,w_1^\V,\ldots,w_n^\V\}$ of $T_{(x,v)}T^1\Hyp^{n+1}$, and therefore the desired identity holds.
\end{proof}

{We get as an immediate consequence the closedness of the fundamental form $\Omega$, a fact whose proof has been deferred from Section \ref{sec:parakahler metric GG}.
	\begin{Cor}\label{cor:omega closed}
		The fundamental form $\Omega=\GG(\cdot,\JJ\cdot)$ is closed.
	\end{Cor}
	\begin{proof}
		Using Proposition \ref{prop:identity curvature form} we have
		$$\mathrm{p}^*(d\Omega)=d(\mathrm{p}^*\Omega)=d(d\omega)=0~.$$
		Since $\mathrm{p}$ is surjective, it follows that $d\Omega=0$.
\end{proof}}

\section{{Lagrangian immersions and local integrability}} We have now all the ingredients to relate the Gauss maps of immersed hypersurfaces in $\Hyp^{n+1}$ with the Lagrangian condition for the symplectic geometry of $\G{n+1}$. 

\begin{Cor}\label{cor:lagrangian1}
	Given an oriented manifold $M^n$ and an immersion $\sigma:M\to\Hyp^{n+1}$, $\nobreak{G_\sigma:M\to \G{n+1}}$ is a Lagrangian immersion. 
\end{Cor}
\begin{proof}
	By Proposition \ref{prop pullback flat}, the pull-back by $G_\sigma$ of the principal $\R$-bundle
	$\mathrm{p}$ is flat {because} there exists {$\widehat G_\sigma=\zeta_\sigma:M\to T^1\Hyp^{n+1}$ orthogonal to $\chi$ such that $G_\sigma=\mathrm p\circ\widehat G_\sigma$. Hence $(\widehat G_\sigma)^*d\omega$ vanishes identically and $\widehat G_\sigma$ induces a flat section of $G_\sigma^*\mathrm p$.}  But by Proposition \ref{prop:identity curvature form}, $(\widehat G_\sigma)^*d\omega=(\widehat G_\sigma)^*(\mathrm{p}^*\Omega)=(G_\sigma)^*\Omega$, therefore $G_\sigma$ is Lagrangian.
\end{proof}

Observe that in Corollary \ref{cor:lagrangian1} we only use the flatness property of Proposition \ref{prop pullback flat}, and not the triviality of the  pull-back principal bundle. When $M$ is simply connected, we can partially reverse Corollary \ref{cor:lagrangian1}, showing that the Lagrangian condition is essentially the only local obstruction.

\begin{Cor}\label{cor:lagrangian2}
	Given an {orientable} simply connected manifold $M^n$ and a Lagrangian immersion $G:M\to \G{n+1}$, there exists an immersion $\zeta:M\to T^1\Hyp^{n+1}$ orthogonal to the fibers of $\mathrm{p}$ such that $G=\mathrm{p}\circ \zeta$. Moreover, if $\pi\circ \zeta:M\to \Hyp^{n+1}$ is an immersion, then $G$ coincides with its Gauss map.
\end{Cor}
\begin{proof}
	Since $G$ is Lagrangian, by Proposition \ref{prop pullback flat} the $G$-pull-back bundle of $\mathrm{p}$ is flat, and it is therefore a trivial flat bundle since $M$ is simply connected. Hence it admits a global flat section, which provides the map $\zeta:M\to T^1\Hyp^{n+1}$ orthogonal to $\chi$. Assuming moreover that $\sigma:=\pi\circ \zeta$ is an immersion, by Proposition \ref{prop:gauss immersion converse}, $G=\mathrm{p}\circ \zeta$ is the Gauss map of $\sigma$.
\end{proof}

Clearly the map $\zeta$ is not uniquely determined, and the different choices differ by the action of $\varphi_t$. {Lemma \ref{lemma:desingularize for small t}, and the following corollary, show that (by post-composing with $\varphi_t$ if necessary) one can always find $\zeta$ such that $\pi\circ\zeta$ is \emph{locally} an immersion.

	\begin{Lemma} \label{lemma:desingularize for small t}
		Let $M$ be a $n$-manifold and $\zeta:M\to T^1\Hyp^{n+1}$ be an immersion orthogonal to $\chi$.
		Suppose that the differential of $\pi\circ\zeta$ is singular at $p\in M$. Then there exists $\epsilon>0$ such that the differential of $\pi\circ\varphi_t\circ\zeta$ at $p$ is non-singular for all $t\in(-\epsilon,\epsilon)\setminus\{0\}$.
	\end{Lemma}
	\begin{proof}
		Let us denote $\zeta_t:=\varphi_t\circ\zeta$, $\sigma:=\pi\circ\zeta$, and $\sigma_t:=\pi\circ\zeta_t$. Assume also $\zeta(p)=(x,v)$. Let $\{V_1,\ldots, V_k\}$   be a basis of the kernel of $d_p\sigma$ and let us complete it to a basis $\{V_1,\ldots,V_n\}$ of $T_p M$. Hence if we denote $w_j:=d_p\sigma(V_j)$ for $j>k$, $\{w_{k+1},\ldots,w_n\}$ is a basis of the image of $d_p\sigma$. Exactly as in the proof of Proposition \ref{prop:gauss immersion converse}, we have $w_{k+1},\ldots,w_n\in v^\perp\subset T_x\Hyp^{n+1}$. Hence we have:
		$$d\zeta(V_1)=u_1^\V,\ldots,d\zeta(V_k)=u_k^\V,d\zeta(V_{k+1})=w_{k+1}^\HH+u_{k+1}^\V,\ldots,d\zeta(V_n)=w_n^\HH+u_n^\V$$
		for some $u_1,\ldots,u_n\in v^\perp$. 
		
		On the one hand, since $\zeta$ is an immersion, $u_1,\ldots,u_k$ are linearly independent. On the other hand, $\zeta$ is orthogonal to $\chi$, hence by Remark \ref{rmk:curvature and bracket} we have $\zeta^*d\omega=0$. Using Equation \eqref{eq:sufficient}, it follows that:
		$$\gs{n+1}(d\zeta(V_i),{\mathrm J}\circ d\zeta(V_j))=0$$
		for all {$i,j=1,\dots n$}. Applying this to any choice of {$i\leq k$ and $j>k$}, we find $\langle u_i,w_j\rangle=0$. Hence $\{u_1,\ldots,u_k,w_{k+1},\ldots,w_n\}$ is a basis of $v^\perp$. 
		
		We are now ready to prove the statement. By Equations \eqref{eq:diff geoflow1} and \eqref{eq:diff geoflow2}, we have
		$$d\sigma_t(V_i)=d\pi\circ d\varphi_t(u_i^\V)=\sinh(t)u_i$$
		for $1\leq i\leq k$, while
		$$d\sigma_t(V_j)=d\pi\circ d\varphi_t(w_j^\HH+u_j^\V)=\cosh(t)w_j+\sinh(t)u_j$$
		for $k+1\leq j\leq n$. The proof will be over if we show that $\{d\sigma(V_1),\ldots,d\sigma(V_n)\}$ are linearly independent for $t\in(-\epsilon,\epsilon)$, $t\neq 0$. In light of the above expressions, dividing by $\sinh(t)$ (which is not zero if $t\neq 0$) or $\cosh(t)$, this is equivalent to showing that 
		$$\{u_1,\ldots,u_k,w_{k+1}+\tanh(t)u_{k+1},\ldots,w_{n}+\tanh(t)u_{n}\}$$
		are linearly independent for small $t$. This is true because we have proved above that $\{u_1,\ldots,u_k,w_{k+1},\ldots,w_n\}$ is a basis, and linear independence is an open condition.
	\end{proof}
	
	\begin{Theorem}
		\label{cor: local integrability}
		{Let $G\colon M^n \to \G{n+1}$ be an immersion. Then $G$ is Lagrangian if and only if for all $p\in M$ there exists a neighbourhood $U$ {of $p$} and an immersion $\sigma\colon U\to \Hyp^{n+1}$ such that $G_\sigma = G|_{U}$.}
	\end{Theorem}
	\begin{proof}
		{The ``if'' part follows from Corollary \ref{cor:lagrangian1}. Conversely,} let $U$ be a simply connected neighbourhood of $p$. By Corollary \ref{cor:lagrangian2}, there exists an immersion $\zeta\colon U \to T^1\Hyp^{n+1}$ orthogonal to the fibers of $\mathrm{p}$ such that $G=\mathrm{p}\circ \zeta$. If the differential of $\pi\circ\zeta$ is non-singular at $p$, then, up to restricting $U$, we can assume $\sigma:=\pi\circ\zeta$ is an immersion of $U$ into $\Hyp^{n+1}$. By the second part of Corollary \ref{cor:lagrangian2}, $G|_U$ is the Gauss map of $\sigma$. If the differential of $\pi\circ\zeta$ is instead singular at $p$, by Lemma \ref{lemma:desingularize for small t} it suffices to replace $\zeta$ by $\zeta_t$ for small $t$ and we obtain the same conclusion.
	\end{proof}

\subsection{Lagrangian condition in $\GGG$ in terms of complex metrics}	
	In Proposition \ref{prop: pull-back Re hRm} we showed that the pseudo-Riemannian metric of $\GG$ on $\G 3$ is the real part of the hRm $\inners_\GGG$. Moreover, by Theorem \ref{teoremone stessa dimensione}, we know that a complex metric $g$ with constant curvature $-1$ on a surface $S$ corresponds to an equivariant immersion $\widetilde S\to \GGG$, unique up to post-composition with $\Isom(\Hyp^3)$.

	We show that there is a characterization of Lagrangian (i.e. locally integrable) immersions $G\colon S\to \G 3$ in terms of the induced complex metric.
	
	Before that, let us show that the symplectic form on $\G 3$ can be expressed in terms of the complex structure of $\GGG$.
	
	\begin{Prop}
		\label{prop Omega e Re forma d area in G}
		Let $(U,z)$ be an affine chart on $\CP^1$. Then, on the local chart $(U\times U\setminus \Delta, z\times z)$ for $\G 3$, the symplectic form $\Omega$ satisfies
		\begin{equation}
			\label{eq Omega e Re forma d area in G}
		\Omega= - \frac 1 2 Re \bigg( \frac{dz_1 \wedge dz_2}{(z_1-z_2)^2} \bigg).
		\end{equation}
	\end{Prop}
\begin{proof}
	Let us first note that, using exactly the same techniques as in the proof of Theorem \ref{teo: metrica su G}, one can prove that the RHS of Equation \ref{eq Omega e Re forma d area in G} does not depend on the local affine chart and that it is $\PSL$-invariant. Hence it is sufficient to check that Equation \ref{eq Omega e Re forma d area in G} holds in a point of $\GGG$
	
	Let us recall that in Equation \eqref{eq mappa differenziale di d} we computed the map $d\mathrm p\colon \chi_{(x,v)}^\perp \to T_\ell \GGG$, where, in the hyperplane model $\C\times \R^+$ for $\Hyp^3$, we denote $x=(0,1)$, $v=(1,0)$, $\ell=(1, -1)\in \C\times \C\setminus\Delta$.
	
	With notations as in the proof of Proposition \ref{prop:identity curvature form}, define $w_1^{\mathcal H}, w_2^{\mathcal H}, w_1^{\mathcal V}, w_2^{\mathcal V}\in v^\perp\oplus v^\perp= \chi_{(x,v)}^\bot$ as
	\begin{align*}
		w_1^{\mathcal H}&=((i, 0),(0,0))\ , \\
		w_2^{\mathcal H}&=((0,1), (0,0))\ , \\
		w_1^{\mathcal V}&=((0,0), (i,0))\ , \\
		w_2^{\mathcal V}&=((0,0),(0,1))\ .
	\end{align*}
The images via $d\mathrm p$ of these vectors define a orthonormal basis for $T_\ell \GGG$: explicitly via Equation \eqref{eq mappa differenziale di d} we have 
\begin{align*}
	d\mathrm p{(w_1^{\mathcal H})}&= (i,i)\ ,\\
	d\mathrm p{(w_2^{\mathcal H})}&=(-1,1)\ ,\\
	d\mathrm p{(w_1^{\mathcal V})}&=(-i,i)\ ,\\
	d\mathrm p{(w_2^{\mathcal V})}&=(1,1)\ .
\end{align*}

We computed in the proof of Proposition \ref{prop:identity curvature form} that $\Omega(	d\mathrm p(w_1^{\mathcal H}), d\mathrm p( 	w_2^{\mathcal H}))=\Omega(	d\mathrm p(w_1^{\mathcal V}), d\mathrm p(	w_2^{\mathcal V}))=0$, and that $\Omega(	d\mathrm p (w_j^{\mathcal H}), d \mathrm p(	w_k^{\mathcal V}) )=\delta_{j,k}$. Using the explicit descriptions above, one gets directly that the RHS of Equation \ref{eq Omega e Re forma d area in G} coincides on this basis of $T_\ell \GGG$.
\end{proof}

	\begin{Cor}
		\label{cor integrabile sse forma area reale}
		An immersion $G\colon S\to \inners \GGG$ is Lagrangian if and only if the (local) area form $dA_g$ of $g=G^*\inners_\GGG$ is real, i.e. it restricts to a real-valued $2$-form on $TS$.
	\end{Cor}
	\begin{proof}
		First of all, observe that clearly the statement does not depend on the sign of the area form.
		
		Let $G=(f_1,f_2)$, hence $g=\frac{-4}{(f_1-f_2)^2} df_1 \cdot df_2$, with $f_1,f_2\colon S\to \CP^1$.
		
		Let us show that
		
		\begin{equation}
			\label{eq forma d area di g}
		dA_g= 8 i \frac{1}{(f_1- f_2)^2} df_1 \wedge df_2 \ .
		\end{equation}
	
		In order to prove it, take a basis $(V_1,V_2)$ for $\C T_p S$ such that each $V_k$ is a non-zero isotropic vector for $g$ with $V_k\in Ker(d_p f_k\colon \C T_pS \to \C)$. Then, the vectors 
		\[W_1=\frac{V_1 +i V_2}{\sqrt{2i g(V_1, V_2)}}\ ,\quad  W_2= \frac{V_1 -i V_2}{i\sqrt{2i g(V_1, V_2)}}\] define a $g$-orthonormal basis, and $d A_g(W_1,W_2)=1$ leads to 
		\[
		dA_g(V_1,V_2)= 2i g(V_1,V_2):
		\] 
		as a result, one gets that Equation $\eqref{eq forma d area di g}$ holds since it holds in the basis $(V_1,V_2)$.
		
		Finally, by Proposition \ref{prop Omega e Re forma d area in G} and Equation \eqref{eq forma d area di g}, one gets
		\begin{align*}
			G^* \Omega&= -\frac 1 2 Re\ \Big( G^*\Big( \frac{dz_1 \wedge dz_2}{(z_1-z_2)^2}  \Big) \Big)=\\
			&=-\frac 1 2 Re \Big(\frac{df_1 \wedge df_2}{(f_1-f_2)^2} \Big)=\\
			&=-\frac 1 2 Im \Big(i \frac{df_1 \wedge df_2}{(f_1-f_2)^2} \Big)=\\
			&= \frac 1 {16} Im \Big(dA_g \Big) \ ,
		\end{align*}
	and the thesis follows.
	\end{proof}

	\section{Global integrability}

	Let us now approach the problem of global integrability. We provide an example to show that in general {$\pi\circ\varphi_t\circ\zeta$ might fail to be \emph{globally} an immersion for \emph{all} $t\in\R$}, as we already mentioned after Proposition \ref{prop:gauss immersion converse}. 
	
	\begin{Example}
		\label{ex: Lagrangian not globally integrable}
		
		
		Let us construct a curve {$G:(-T,T)\to\G 2$} with the property of being locally integrable\footnote{As a matter of fact, any curve in $\G 2$ is locally integrable by Theorem \ref{cor: local integrability}, since the domain is simply connected and it is trivially Lagrangian. However, in this example, we will see by construction that $G$ is locally integrable.} but not globally integrable. 
		
		Fix $r>0$ {and a maximal geodesic $\ell$ in $\Hyp^2$}.
		{Let us consider a smooth curve $\sigma_+\colon (-\varepsilon,T)\to \Hyp^2$, for some small enough $\varepsilon$ and big enough $T$, so that: \begin{itemize}
				\item {$\sigma_+$} is an immersion and is parameterized by arclength;
				\item $(\sigma_+)|_{(-\varepsilon, \varepsilon)}$ lies on the $r$-cap equidistant from $\ell$, oriented in such a way that the induced unit normal vector field $(\nu_+)|_{(-\varepsilon, \varepsilon)}$ is directed towards $\ell$;
				\item $(\sigma_+)|_{(T_0, T)}$ lies on the metric circle $\{x\in\Hyp^2\,|\,d_{\Hyp^2}(x,x_0)=r\}$ for some $x_0\in\Hyp^2$ and some $\varepsilon<T_0<T$, oriented in such a way that the induced unit normal vector field $(\nu_+)|_{(T_0, T)}$ is directed towards $x_0$.
			\end{itemize}
		} 
		
		\begin{figure}[htbp]
			\centering
			\includegraphics[height=5.5cm]{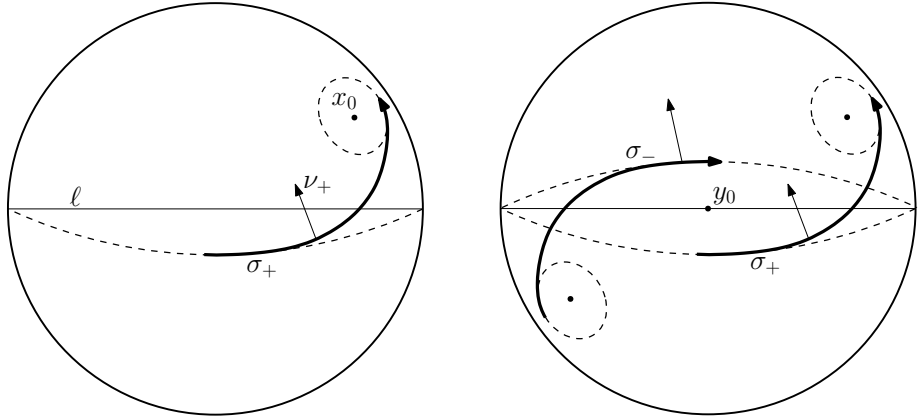} 
			
			\caption{On the left, the construction of the curve $\sigma_+$, which is an arclenght parameterization of a portion of $r$-cap equidistant from $\ell$ for $s\in(-\varepsilon,\varepsilon)$, and of a circle of radius $r$ for $s\in (T_0,T)$. On the right, the curve $\sigma_-$ whose image is obtained by an order-two elliptic isometry from the image of $\sigma_+$.}\label{fig:counterexample1}
		\end{figure}

		See Figure \ref{fig:counterexample1}. 	By a simple computation, for instance using Equation \eqref{eq della discordia}, the curvature of $\sigma_+$ equals $\tanh(r)$ on $(-\varepsilon, \varepsilon)$ and $\frac 1 {\tanh(r)}$ on $(T_0, T)$. Hence by the intermediate value theorem, {the image of the curvature function $k\colon (-\varepsilon, T)\to \R$ contains the interval $[\tanh(r),\frac 1 {\tanh(r)}]$.}   An important consequence of this observation is that $(\sigma_+)_t$ fails to be an immersion when $t\geq r$. More precisely, if we denote $\zeta_{\sigma_+}$ the lift to $T^1\Hyp^2$ as usual, using Equation \eqref{eq:geodflow} analogously as for Equation \eqref{eq:normal evolution hyperboloid2}, we obtain that:
		$$d(\pi\circ\varphi_t\circ\zeta_{\sigma_+})(V)=(\cosh(t)-\sinh(t)k)d\sigma(V)~.$$
		This shows that the differential of $\pi\circ\varphi_t\circ\zeta_{\sigma_+}$ at a point $s\in (-\varepsilon,T)$ vanishes if and only if 
		\begin{equation}\label{eq:curvature and differential}
			\tanh(t)=\frac 1 {k(s)}~.
		\end{equation}
		Since the image of the function $k$ contains the interval $[\tanh(r),\frac 1 {\tanh(r)}]$, if $t\geq r$ then there exists $s$ such that \eqref{eq:curvature and differential} is satisfied and therefore $\pi\circ\varphi_t\circ\zeta_{\sigma_+}$ is not an immersion at $s$.

		Now, let $y_0\in \ell$ be the point at distance $r$ from $\sigma_+(0)$ and let ${R}_0\colon \Hyp^2\to \Hyp^2$ be the symmetry at $y_0$, i.e. ${R}_0$ is the isometry of $\Hyp^2$ such that ${R}_0(y_0)=y_0$ and $d_{y_0} {R}_0=-\mathrm{id}$. Define $\sigma_-\colon (-T, \varepsilon)\to \Hyp^2$ as 
		\[
		\sigma_-(s):=  {R}_0 (\sigma_+ (-s)).
		\] 
		As a result, the normal vector field $\nu_-$ of $\sigma_-$  is such that $d R_0 ({\nu_+} (s))= {-}{\nu_-} (-s)$ and the curvature of $\sigma_-$ takes any value in the interval $[-\frac 1 {\tanh(r)}, -{\tanh(r)}]$. Hence $(\sigma_-)_t$ fails to be an immersion for $t\leq -r$.
		
		Finally, consider the two lifts $\zeta_{\sigma_+}$ and $\zeta_{\sigma_-}$ in $T^1\Hyp^2$. By construction one has that for all $s\in (-\varepsilon,\varepsilon)$
		\[
		\varphi_r \circ \zeta_{\sigma_+} (s)= \varphi_{-r} \circ \zeta_{\sigma_-} (s).
		\]
		As a result, we can define our counterexample $\zeta\colon (-T, T)\to T^1\Hyp^2$ as
		\[
		\zeta(s)
		= \begin{cases}
			\varphi_r \circ \zeta_{\sigma_+} (s)\quad &\text{ if $s>-\varepsilon$}\\
			\varphi_{-r} \circ \zeta_{\sigma_-} (s) \quad &\text{ if $s <\varepsilon$}
		\end{cases}.\]
		By construction, we have that $\mathrm p \circ \zeta_{\sigma_+}= \mathrm p \circ \zeta|_{(-\varepsilon, T)}$ and $\mathrm p \circ \zeta_{\sigma_-}= \mathrm p \circ \zeta|_{(-T, \varepsilon)}$, therefore $\mathrm p \circ \zeta$ is an immersion into $\G{2}$ and clearly it is locally integrable. 
		However, by the above discussion, $\pi\circ\varphi_t\circ\zeta$ fails to be an immersion for every $t\in\R$: for $t\geq 0$ because $\pi\circ\varphi_t\circ\zeta_{\sigma_+}$ has vanishing differential at some  $s\geq -\varepsilon$, and for $t\leq 0$ because the differential of $\pi\circ\varphi_t\circ\zeta_{\sigma_-}$ vanishes  at some  $s\leq \varepsilon$. See Figure \ref{fig:counterexample2}.
		\begin{figure}[htbp]
			\centering
			\includegraphics[height=5.5cm]{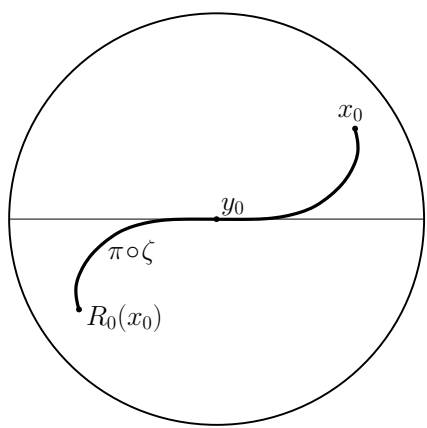} 
			\caption{The curve $\zeta:(-T,T)\to T^1\Hyp^2$ projects to a map $\pi\circ\zeta$ into $\Hyp^2$ which is not an immersion, because it is constantly equal to $x_0$ for $s\in(T_0,T)$ and to $R_0(x_0)$ for $s\in(-T,-T_0)$. Moreover the curvature of the regular part takes all the values in $(-\infty,+\infty)$. For this reason, each immersion $\varphi_t\circ\zeta$ into $T^1\Hyp^2$ does not project to an immersion in $\Hyp^2$.}\label{fig:counterexample2}
		\end{figure}	
	\end{Example}

	
	
	Corollary \ref{cor:lagrangian2} and {Theorem \ref{cor: local integrability}} can be improved under the additional assumption that the immersion $G$ is Riemannian. More precisely, we provide an improved characterization of {immersions into $\G{n+1}$ that are} Gauss maps of immersions with small principal curvatures in terms of the Lagrangian and Riemannian condition, again for {when} $M$ is simply connected. 
	
	
	\begin{Theorem}
		\label{prop: riemannian global integrability}
		Given a simply connected manifold $M^n$ and an {immersion} $G:M^n\to \G{n+1}$, $G$ is Riemannian and Lagrangian if and only if there exists an {immersion} $\sigma\colon M\to \Hyp^{n+1}$ with small principal curvatures such that $G_\sigma=G$, i.e. if and only if $G$ is globally integrable. 
		
		{If in addition $\sigma$ is complete, then it is a proper embedding, $G_\sigma$ is an embedding and $M$ is diffeomorphic to $\R^n$.}
	\end{Theorem}
	\begin{proof}
		We know from Corollary \ref{cor:lagrangian1} and Proposition \ref{prop: small curv sse riemannian} that the Riemannian and Lagrangian conditions {on $G$} are necessary.
		To see that they are also sufficient, by Corollary \ref{cor:lagrangian2} there exists $\zeta\colon M \to \G{n+1}$ orthogonal to the fibers of $\mathrm{p}$ such that $\mathrm{p}\circ \zeta=G$. We claim that $\pi\circ\zeta$ is an immersion, which implies that $G=G_\sigma$ for $\sigma=\pi\circ\zeta$ by the second part of Corollary \ref{cor:lagrangian2}. Indeed, if $X\in T_p M$ is such that $d\zeta (X) \in \mathcal V_{\zeta(x)}=\ker (d\pi_{\zeta(x)}) $, then $d\zeta (X)=w^\V$ for some $w\in T_{\sigma(p)}\Hyp^{n+1}$. Hence by Definition \ref{Def:parasasaki} and the construction of the metric $\GG$,  $\GG (X,X)= -\langle w,w\rangle \le 0$: since $\GG$ is Riemannian, necessarily $w=0$ and therefore $X=0$. 
		
		{By Proposition \ref{prop: small curv sse riemannian} $\sigma$ has small principal curvatures.
			The ``in addition'' part follows by Proposition \ref{prop injectivity} and Proposition \ref{prop:gauss maps diffeo onto image}.}
	\end{proof}
	
	As another consequence of Proposition \ref{prop injectivity} and Proposition \ref{prop:gauss maps diffeo onto image}, we obtain the following result.
	
	\begin{Theorem}\label{Cor G complete}
		Let $M^n$ be a manifold. If $G\colon M\to \G {n+1}$ is a complete Riemannian and Lagrangian immersion, then $M$ is diffeomorphic to $\R^n$ and there exists a proper embedding $\sigma:M\to\Hyp^{n+1}$ with small principal curvatures such that $G=G_\sigma$.
	\end{Theorem}
	\begin{proof}
		Let us lift $G$ to the universal cover $\widetilde M$, obtaining a Riemannian and Lagrangian immersion $\widetilde G:\widetilde M\to\G{n+1}$ which is still complete. By Theorem \ref{prop: riemannian global integrability}, $\widetilde G$ is the Gauss map of an immersion $\sigma$ with small principal curvatures. We claim that $\sigma$ is complete. Indeed by Equation \eqref{eq:fff gauss} the first fundamental form of $\widetilde G$, which is complete by hypothesis, equals $\I-\III$, hence it is complete since $\III$ is positive semi-definite.
		
		It follows from Proposition \ref{prop:gauss maps diffeo onto image} that $\widetilde G$ is injective. Hence $\widetilde M=M$ and $\widetilde G=G$, and therefore $G$ is the Gauss map of $\sigma$, which is complete. 
		By the ``in addition'' part of Theorem \ref{prop: riemannian global integrability}, $\sigma$ is properly embedded and $M$ is diffeomorphic to $\R^n$.
	\end{proof}
	
	In summary, the Lagrangian condition is essentially the only \emph{local} obstruction to realizing an immersion $G:M\to \G{n+1}$ as the Gauss map  of an immersion into $\Hyp^{n+1}$, up to the subtlety that this might be an immersion only when lifted to $T^1\Hyp^{n+1}$. {This subtlety however never occurs in the small principal curvatures case.} 
	{In the remainder of the paper, we will discuss the problem of characterizing immersions into $\G{n+1}$ which are Gauss maps of \emph{equivariant} immersions into $\Hyp^{n+1}$ with small principal curvatures.}


	\section{Equivariant integrability of Riemannian immersions}
	
	In this section, we provide the first characterization of \emph{equivariant} immersions in $\G{n+1}$ which arise as the Gauss maps of \emph{equivariant} immersions in $\Hyp^{n+1}$, in the Riemannian case. This is the content of Theorem \ref{teorema hol H baby}. We first try to motivate the problem,  introduce  the obstruction, namely the Maslov class, and study some of its properties. See for instance \cite{zbMATH01523513} for a discussion on the Maslov class in more general settings.

	
	\subsection{Motivating examples}

	{Given an $n$-manifold $M$, a representation $\rho:\pi_1(M)\to\Isom(\Hyp^{n+1})$, and a $\rho$-equivariant immersion $\widetilde\sigma:\widetilde M\to \Hyp^{n+1}$, it is immediate to see that the Gauss map $G_{\widetilde\sigma}:\widetilde M\to \G{n+1}$ is $\rho$-equivariant} {(recall also Remark \ref{rmk: Isom preserva Omega e J})}.
	Moreover, if ${\widetilde\sigma}$ has small principal curvatures, it follows from the discussion of the previous sections that $G_{\widetilde\sigma}$ is a Lagrangian and Riemannian {immersion}. 
	
	
	However, a $\rho$-equivariant Lagrangian immersion (even with the additional assumptions of being Riemannian and being an embedding) does not always arise as the Gauss map associated to a \emph{$\rho$-equivariant} immersion in $\Hyp^{n+1}$, as the following simple example shows for $n=1$.
	
	\begin{Example}\label{ex: global integrable non equivariant}
		Let us {construct} a coordinate system for a portion of $\G{2}$. Let $\gamma:\R\to\Hyp^2$ be a geodesic parameterized by arclength, and let us define a map $\eta:\R\times (0,\pi)\to\G{2}$ by defining $\eta(t,\theta)$ as the oriented geodesic that intersects $\gamma$ at $\gamma(t)$ with an angle $\theta$ ({measured counterclockwise} with respect to the standard orientation of $\Hyp^2$). 
		We can lift $\eta$ to a map $\widehat\eta:\R\times (0,\pi)\to T^1\Hyp^2$, which will however \emph{not} be orthogonal to the fibers of the projection $T^1\Hyp^2\to\G{2}$. The lift is simply defined as 
		$$\widehat \eta(t,\theta)=(\gamma(t),\cos(\theta)\gamma'(t)+\sin(\theta) w)~,$$ where $w$ is the vector {forming an angle $\frac \pi 2$ with $\gamma'(t)$}. We omitted the dependence of $w$ on $t$ since, in the hyperboloid model, $w$ is actually a constant vector in $\R^{2,1}$.  
		
		Let us compute the pull-back of the metric $\GG$ on $\G{2}$ by the map $\eta$. We have already observed in Example \ref{ex:spheres} that the restriction of $\GG$ on the image of $\theta\mapsto\eta(t_0,\theta)$ is minus the standard metric of $\Sph^1$. Indeed in this simple case, $d\widehat\eta_{(t,\theta)}(\partial_\theta)=(0,w)$ is in the vertical subspace 
		$\VP$
		and by Definition \ref{Def:parasasaki} its squared norm is $-1$. On the other hand, since the vector field $\cos(\theta)\gamma'(t)+\sin(\theta) w$ is parallel along $\gamma$, when we differentiate in $t$ we obtain, by applying the definition of horizontal lift:
		\begin{equation}\label{eq:lift example n=1}
			d\widehat\eta_{(t,\theta)}(\partial_t)=\cos(\theta)(\gamma'(t))^\HH+\sin(\theta)w^\HH~.\end{equation}
		Moreover, Equation \eqref{eq:lift example n=1} gives the decomposition of $d\widehat\eta_{(t,\theta)}(\partial_t)$ in $T_{\widehat\eta_{(t,\theta)}}T^1\Hyp^2={\mathrm{Span}(\chi)\oplus\chi^\perp}$ as in Equation \eqref{eq:direct sum}, since the first term is a multiple of the generator of the geodesic flow, and the second term is in $\HP$.
		This shows, by definition of the metric $\GG$, that $d\eta_{(t,\theta)}(\partial_t)$ has squared norm $\sin^2(\theta)$ and that $d\eta_{(t,\theta)}(\partial_t)$ and $d\eta_{(t,\theta)}(\partial_\theta)$ are orthogonal. In conclusion, we have showed:
		$$\eta^*\GG=-d\theta^2+\sin^2(\theta)dt^2~.$$
		
		We are now ready to produce our example of $\rho$-equivariant embedding $G:\widetilde M\to\G{2}$ which is not $\rho$-integrable. Consider $M=\Sph^1$, $\widetilde M=\R$, and the representation $\rho:\Z\to\Isom(\Hyp^2)$ which is a hyperbolic translation along $\gamma$. The induced action on $\G{2}$ preserves the image of $\eta$ and its generator acts on the $(t,\theta)$-coordinates simply by $(t,\theta)\mapsto (t+c,\theta)$. Hence the map
		$$G:\R\to\G{2}\qquad G(s)=\eta(s,\theta_0)$$
		for some constant $\theta_0\in (0,\pi)$ is a $\rho$-equivariant Riemannian embedding by the above expression of $\eta^*\GG$. Of course the Lagrangian condition is {trivially satisfied since $n=1$}. By Theorem \ref{prop: riemannian global integrability} $G$ coincides with the Gauss map $G_\sigma$ associated to some embedding $\sigma:\R\to\Hyp^2$ with small curvature. It is easy to see that any such embedding $\sigma$ is not $\rho$-equivariant unless $\theta_0=\frac \pi 2$, see Figure \ref{fig:example_dim2}.
	\end{Example}
	
	\begin{figure}[htbp]
		\centering
		\begin{minipage}[c]{.5\textwidth}
			\centering
			\includegraphics[width=.6\textwidth]{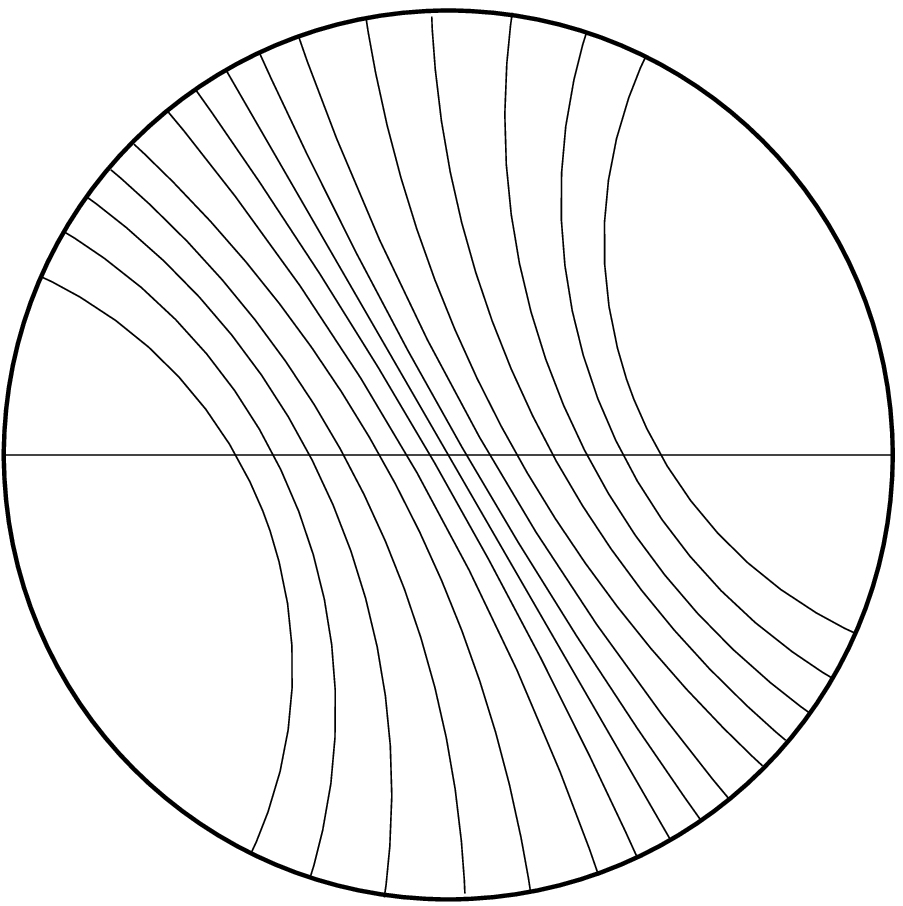} 
		\end{minipage}%
		\begin{minipage}[c]{.5\textwidth}
			\centering
			\includegraphics[width=.7\textwidth]{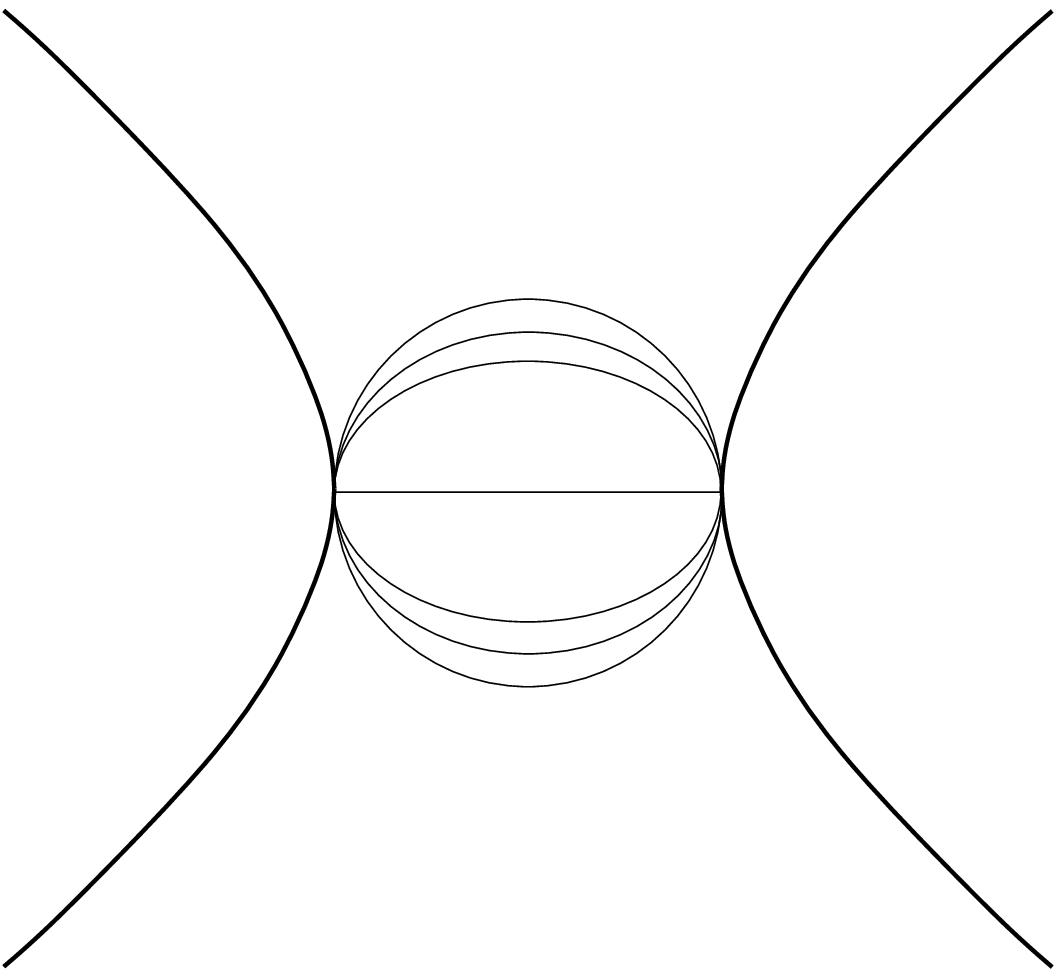} 
		\end{minipage}%
		\caption{On the left, the family of geodesics forming a fixed angle $\theta$ with the horizontal geodesic (in the Poincar\'e disc model of $\Hyp^2$). On the right, the metric $-d\theta^2+\sin^2(\theta)dt^2$ represents a portion of Anti-de Sitter space of dimension 2. Here are pictured some lines defined by $\theta=c$, in the projective model of Anti-de Sitter space.}\label{fig:example_dim2}
	\end{figure}

	{We also briefly provide an example of a locally integrable, but not globally integrable immersion in $\G{3}$ for $M$ not simply connected (lifting to the universal cover $\widetilde M$ this corresponds to $\rho$ being the trivial representation). This example in particular motivates Corollary \ref{cor hol H baby}, which is a direct consequence of Theorem \ref{teorema hol H baby}.}
	
	\begin{Example}\label{ex: global integrable non equivariant2}
		First, let us consider a totally geodesic plane $\mathcal P$ in $\Hyp^3$ and an annulus $\mathcal A$ contained in $\mathcal P$. Let $c:\widetilde {\mathcal A}\to \mathcal A$ be the universal covering. Then $c$ is an immersion in $\Hyp^3$ with small principal curvatures (in fact, totally geodesic), and is clearly not injective. See Figure \ref{fig:annulus} on the left. Of course, in light of Proposition \ref{prop injectivity}, this is possible because the immersion $c$ is not complete.
		
		Now, let us deform $\mathcal A$ in the following way. We cut $\mathcal A$ along a geodesic segment $s\subset\mathcal P$ to obtain a rectangle $\mathcal R$ having two (opposite) geodesic sides, say $r_1$ and $r_2$. Then we deform such rectangle to get an immersion $c':\mathcal R\to\Hyp^3$, so that one geodesic side remains unchanged (say $c'(r_1)=s$), while the other side $r_2$ is mapped to an $r$-cap equidistant from $\mathcal P$, for small $r$, in such a way that it projects to $s$ under the normal evolution. We can also arrange $c$ so that a neighbourhood of $r_1$ is mapped to $\mathcal P$, while a neighbourhood of $r_2$ is mapped to the $r$-cap equidistant from $\mathcal P$. See Figure \ref{fig:annulus} on the right.
		
		By virtue of this construction, the Gauss map of $c'$ coincides on the edges $r_1$ and $r_2$ of $\mathcal R$, and therefore induces an immersion $G':\mathcal A\to\G{3}$. Clearly $G'$ is locally integrable, but not globally integrable.  In other words, the lift to the universal cover of $G'$ is a $\rho$-equivariant immersion of $\widetilde A$ to $\G 3$ which is not $\rho$-integrable, for $\rho$ the trivial representation.

		\begin{figure}[htbp]
			\centering
			\includegraphics[height=4.5cm]{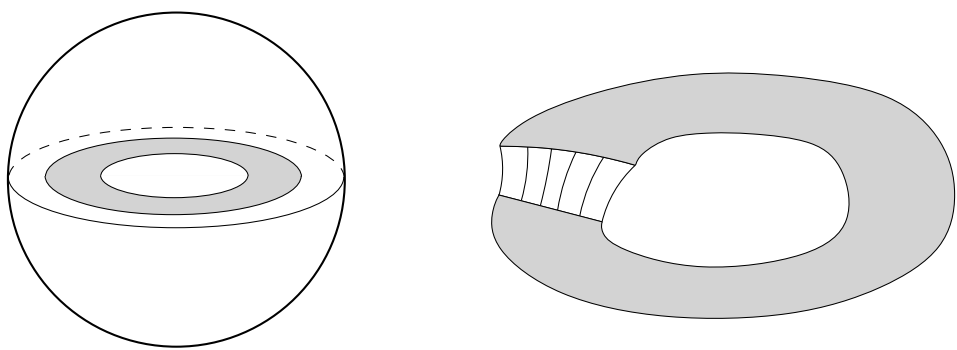} 
			
			\caption{On the left, a totally geodesic annulus $\mathcal A$ in a plane $\mathcal P$. On the right, an embedded rectangle with the property that a neighbourhood of one side lies in $\mathcal P$, while a neighbourhood of the opposite side lies on an $r$-cap equidistant from  $\mathcal P$. Such rectangle induces an embedding of $\mathcal A$ in $\G{3}$ which is locally, but not globally, integrable.}\label{fig:annulus}
		\end{figure}
		
	\end{Example}
	
	Motivated by the previous examples, we introduce the relevant definition for our problem.

	\begin{Def}\label{Def rhointegrable}
		Given {an $n$-manifold $M$ and a} representation $\rho\colon \pi_1(M) \to \Isom(\Hyp^{n+1})$, a $\rho$-equivariant immersion $G\colon \widetilde M \to \G{n+1}$ is \emph{$\rho$-integrable} if there exists a $\rho$-equivariant immersion ${\widetilde\sigma}\colon \widetilde M \to \Hyp^{n+1}$ whose Gauss map is $G$.
	\end{Def}

	\subsection{The Maslov class and the characterization}
	
	Let us now introduce the obstruction which will permit us to classify $\rho$-integrable Lagrangian immersions under the Riemannian assumption, namely the Maslov class. For this purpose, let $G:\widetilde M\to\G{n+1}$ be a Riemannian immersion. The
	\emph{second fundamental form} of $G$ is a symmetric bilinear form on $M$ with values in the normal bundle of $G$, defined as
	$$\overline\II(V,W)=({\mathbbm D}_{dG(V)}(dG(W)))^\perp$$
	for vector fields $V,W$, where {$\mathbbm D$} denotes the ambient Levi-Civita connection {of $\GG$} and $\perp$ the projection to the normal subspace of $G$. {One can prove that $\overline\II(V,W)$ is a tensor, i.e. that it}  depends on the value of $V$ and $W$ {pointwise.} 
	The
	\emph{mean curvature} is then 
	$$\overline{\mathrm H}=\frac{1}{n}\mathrm{tr}_{\overline\I}\overline\II~,$$
	that is, it is the trace of $\overline\II$ with respect to the first fundamental form $\overline\I$ of $G$, and is therefore a section of the normal bundle of $G$.

	Consider now the 1-form on $\widetilde M$ given by $G^*(\Omega (\overline {\mathrm H}, \cdot ) )$. It will follow from Proposition \ref{Prop: formula H in G} {(see Corollary \ref{cor:maslov closed})} that this is a closed 1-form. Since $\Isom(\Hyp^{n+1})$ acts by automorphisms of the para-K\"ahler manifold $(\G{n+1}, \GG, \mathbb J, \Omega)$, if $G$ is $\rho$-equivariant, then the form $G^*(\Omega (\overline {\mathrm H}, \cdot ) )$ is $\pi_1(M)$-invariant: as a result, it defines a well-posed closed 1-form on $M$. Its cohomology class is the so-called Maslov class:
	
	\begin{Def}\label{Def Maslov class}
		Given an $n$-manifold $M$, a representation $\rho\colon \pi_1(M) \to \Isom(\Hyp^{n+1})$ and a $\rho$-equivariant Lagrangian and Riemannian {immersion} $G\colon \widetilde M \to \G{n+1}$, the
		\emph{Maslov class} of $G$ is the cohomology class 
		\[
		\mu_G:=[G^*(\Omega (\overline {\mathrm H}, \cdot ) )]\in H_{dR}^1(M)~.
		\]
	\end{Def}
	
	The main result of this section is the following, and it will be deduced as a {consequence of} Theorem \ref{Teorema hol H}.
	
	\begin{Theorem}
		\label{teorema hol H baby}
		Given an orientable $n$-manifold  $M$ and a representation $\rho\colon \pi_1(M) \to \Isom(\Hyp^{n+1})$, a $\rho$-equivariant  Riemannian and Lagrangian {immersion} $G\colon \widetilde M \to \G{n+1}$ is
		$\rho$-integrable if and only if $\mu_G=0$ in $H_{dR}^1(M)$.
	\end{Theorem}
	
	We immediately obtain the following characterization of global integrability {for $\pi_1(M)\ne \{1\}$}.
	
	\begin{Cor}
		\label{cor hol H baby}
		Given an orientable $n$-manifold  $M$ and an {immersion} $G\colon M \to \G{n+1}$, $G$ is the Gauss map of an immersion $\sigma:M\to\Hyp^{n+1}$ of small principal curvatures if and only
		if $G$ is Riemannian and Lagrangian and $\mu_G=0$ in $H_{dR}^1(M)$.
	\end{Cor}
	\begin{proof}
		{Denote $\rho$ the trivial representation. Given $G:M\to\G{n+1}$, precomposing with the universal covering map we obtain an immersion $\widetilde G:\widetilde M\to\G{n+1}$ which is $\rho$-equivariant. Observe that $\widetilde G$ is the Gauss map of some immersion $\widetilde\sigma:\widetilde M\to \Hyp^{n+1}$ by Theorem \ref{prop: riemannian global integrability}. Then $G$ is the Gauss map of some immersion in $\Hyp^{n+1}$ if and only if $\widetilde\sigma$ descends to the quotient $M$, i.e. it is $\rho$-integrable. Hence this is equivalent to the vanishing of the Maslov class by Theorem \ref{teorema hol H baby}.}
	\end{proof}
	
	\subsection{Mean curvature of Gauss maps}
	
	Recall that, given an embedding $\sigma:M\to\Hyp^{n+1}$ with small principal curvatures, we introduced in \eqref{eq:aux function} the function $\kappa_\sigma:M\to\R$ which is the mean of the hyperbolic arctangents of the principal curvatures of $\sigma$. This function is strictly related to the mean curvature of the Gauss map of $\sigma$, as in the following proposition.
	
	\begin{Prop}
		\label{Prop: formula H in G}
		Let $M^n$ be an oriented manifold, $\sigma:M\to\Hyp^{n+1}$ an embedding with small principal curvatures, and $G_\sigma:M\to\G{n+1}$ its Gauss map. Then 
		$$G_\sigma^*(\Omega (\overline {\mathrm H}, \cdot ) )=d (\kappa_\sigma)= d\left({\frac{1}{n}}\sum_{i=1}^n\arctanh\lambda_i\right)~,$$
		where $\lambda_1,\ldots,\lambda_n$ denote the principal curvatures of $\sigma$.
	\end{Prop}

	The essential step in the proof of Proposition \ref{Prop: formula H in G} is the following computation for the mean curvature vector of the Gauss map $G_\sigma$:
	\begin{equation} \label{eq:formula H in G}
		\overline{\mathrm H}=-\JJ(dG_\sigma(\overline \nabla \kappa_\sigma))~,
	\end{equation}
	where $\overline \nabla$ denotes the gradient with respect to the first fundamental form $\overline \I$ of $G_\sigma$. 
	Indeed, once Equation \eqref{eq:formula H in G} is established, Proposition \ref{Prop: formula H in G} follows immediately since
	$$\Omega(\overline{\mathrm H},dG_\sigma(V))=-\GG(\JJ(\overline{\mathrm H}),dG_\sigma(V))=\GG(dG_\sigma(\overline \nabla \kappa_\sigma),dG_\sigma(V))=d\kappa_\sigma(V)~.$$
	
	{The computations leading to Equation \eqref{eq:formula H in G} will be done in $T^1 \Hyp^{n+1}$ equipped with the metric $\gstar{n+1}$ defined in Remark \ref{rmk other metric2}, which is the restriction of the flat pseudo-Riemannian metric \eqref{eq:metric minkxmink} of $\R^{n+1,1}\times \R^{n+1,1}$ to $T^1\Hyp^{n+1}$, seen as a submanifold as in \eqref{eq:modelT}. This approach is actually very useful: the Levi-Civita connection of $\gstar{n+1}$ on $T^1\Hyp^{n+1}$, that we denote by
		$\widehat D$}, {will be} just the normal projection of the flat connection $\mathrm d$ of $\R^{n+1,1}\times \R^{n+1,1}$ to $T^1\Hyp^{n+1}$.
	
	{Indeed}, the following lemma will be useful to compute the Levi-Civita connection {$\mathbbm D$} of $\G{n+1}$. Given a vector $X\in T_\ell\G{n+1}$ and $(x,v)\in \mathrm{p}^{-1}(\ell)$, we define the \emph{horizontal lift} of $X$ at   $(x,v)$ as the unique vector $\widetilde X\in T_{(x,v)}T^1\Hyp^{n+1}$ such that 
	\begin{equation}\label{eq:hor lift}
		\widetilde X\in \chi_{(x,v)}^\perp \qquad\text{and}\qquad d\mathrm p(\widetilde X)=X~.\end{equation}
	For a vector field $X$ on an open set $U$ of $\G{n+1}$, we will also refer to the  vector field $\widetilde X$ on $\mathrm{p}^{-1}(U)$, defined by the conditions  \eqref{eq:hor lift}, as the \emph{horizontal lift} of the vector field $X$.


	\begin{Lemma}
		\label{Lemma: curvatura media in G 1}
		Given two  vector fields $X,Y$ on $\GG$, 
		\[
		\mathbbm D_XY=d\mathrm p(\widehat D_{\widetilde X}\widetilde Y)
		\]
	\end{Lemma}
	\begin{proof}
		{By the well-known characterization of the Levi-Civita connection, }it is sufficient to prove that the expression
		{$\mathbbm A_XY:=d\mathrm p(\widehat D_{\widetilde X}\widetilde Y)$}
		is a {well-posed linear connection which is torsion-free and compatible with the metric of $\GG$. We remark that this is not obvious because, although the metric of $\GG$ is the restriction of the metric $\gstar{n+1}$ to $\chi^\perp$, there is no {flat} section of the bundle projection $p:T^1\Hyp^{n+1}\to\G{n+1}$, hence $\G{n+1}$ cannot be seen as an isometrically embedded submanifold of $T^1 \Hyp^{n+1}$.}
		
		First, {observe that} the expression $(\mathbbm A_XY)_{|\ell}=(d_{(x,v)}\mathrm p)(\widehat D_{\widetilde X}\widetilde Y)$ does not depend on the choice of the point $(x,v)\in \mathrm{p}^{-1}(\ell)$. Indeed, given two points $(x_1,v_1)$ and $(x_2,v_2)$ in $\mathrm{p}^{-1}(\ell)$, there exists $t$ such that $(x_2,v_2)=\varphi_t(x_1,v_1)$. By a small adaptation of Lemma \ref{lemma:geodflow isometric}, the geodesic flow $\varphi_t$ acts by isometries of the metric $\gstar{n+1}$ (see also Remarks \ref{rmk other metric1} and \ref{rmk other metric2}), hence it also preserves the horizontal lifts $\widetilde X$ and $\widetilde Y$ and the Levi-Civita connection $\widehat D$. Hence $d\mathrm p ( \widehat D_{\widetilde X} \widetilde Y)= \mathbbm A_X Y$ is a well-defined vector field on $\G{n+1}$ whose horizontal lift is the projection of $\widehat D_{\widetilde X} \widetilde Y $ to $\chi^\bot$.
		
		{We} check that $\mathbbm A$ is a linear connection{. It} is immediate to check the additivity in $X$ and $Y$. Moreover we have the {$C^\infty$-}linearity in $X$ since:
		$$\mathbbm A_{fX} Y= d \mathrm{p} \left( \widehat D_{(f\circ \mathrm{p})\widetilde X} \widetilde Y \right) = d\mathrm p \left( (f\circ \mathrm{p}) \hat D_{\widetilde X} \widetilde Y \right)= f \cdot (\mathbbm A_X Y)~,$$
		and the Leibnitz rule in $Y$, for:
		$$\mathbbm A_X (fY)= d\mathrm p \left( \partial_{\widetilde X} (f\circ \mathrm{p})\ \widetilde Y+ (f\circ \mathrm{p}) \widehat D_{\widetilde X} \widetilde Y\right)= \partial_X f\ Y + f\cdot (\mathbbm A_X Y)~.$$
		The connection $\mathbbm A$ is torsion-free: 
		$$	
		\mathbbm A_X Y - \mathbbm A_Y X =d \mathrm p (\widehat D_ {\widetilde X} \widetilde Y) - d\mathrm  p( \widehat D_{\widetilde Y} \widetilde X)= d\mathrm p ([\widetilde X, \widetilde Y])= [X,Y]~.
		$$
		Finally, we show that $\mathbbm A$ is compatible with the metric $\GG$:
		\begin{align*}
			\GG(\mathbbm A_X Y, Z) + \GG(Y, \mathbbm A_X Z)  &= \gstar{n+1} (\widetilde{\mathbbm A_X Y}, \widetilde Z) + \gstar{n+1 } (\widetilde Y, \widetilde{\mathbbm A_X Z})=\\
			&= \gstar{n+1} ( \widehat D_{\widetilde X} \widetilde Y, \widetilde Z) + \gstar{n+1} (\widetilde Y, \widehat D_{\widetilde X} \widetilde Z) =\\
			&= \partial_{\widetilde X} (\gstar{n+1} (\widetilde Y, \widetilde Z)  )= \partial_X (\GG(Y,Z))~,    \end{align*}
		where in the first line we used the definition of $\GG$, and in the second line the fact that the horizontal lift of $\mathbbm A_X Y$ is the orthogonal projection (with kernel spanned by $\chi$) of $\widehat D_{\widetilde X}\widetilde Y$.
	\end{proof}
	
	We are now ready to provide the proof of Proposition \ref{Prop: formula H in G}.
	
	\begin{proof}[Proof of Proposition $\ref{Prop: formula H in G}$]
		As already observed after Equation \eqref{eq:formula H in G}, it suffices to prove that 
		$\overline{\mathrm H}=-\JJ(dG_\sigma(\overline\nabla \kappa_\sigma))$. So we shall compute the mean curvature vector of $G_\sigma$ in $\G{n+1}$. For this purpose, let $\{e_1,\ldots,e_n\}$ be a local frame on $M$ which is orthonormal with respect to the first fundamental form $\overline \I=G_\sigma^*\GG$. To simplify the notation, let us denote $\epsilon_i:=dG_\sigma(e_i)$.
		Then $\{\JJ\epsilon_1,\ldots,\JJ\epsilon_n\}$ is an orthonormal basis for the normal bundle of $G_\sigma$, on which the metric $\GG$ is negative definite since $G_\sigma$ is Riemannian. The mean curvature vector can be computed as:
		$$\overline{\mathrm H}=\frac{1}{n}\sum_{i=1}^n \overline \II(\epsilon_i,\epsilon_i)=-\frac{1}{n}\sum_{i=1}^n\sum_{k=1}^n \GG(\overline \II(\epsilon_i,\epsilon_i),\JJ\epsilon_k)\JJ\epsilon_k=-\frac{1}{n}\sum_{i=1}^n\sum_{k=1}^n \GG(\mathbbm D_{\epsilon_i}\epsilon_i,\JJ\epsilon_k)\JJ\epsilon_k~,$$
		where in the last equality we used that $\overline \II(V,W)$ is the normal projection of $\mathbbm D_V W$. 
		
		Let us now apply this expression to a particular {$\overline \I$-}orthonormal frame $\{e_1,\ldots,e_n\}$ obtained in the following way. Pick a local {$\I$-}orthogonal frame on $M$ of eigenvectors for the shape operator $B$ of $\sigma$, and normalize each of them so as to have unit norm for $\overline\I$. Hence each $e_i$ is an eigenvector of $B$, whose corresponding eigenvalue $\lambda_i$ are the principal curvatures of $\sigma$. We claim that, with this choice, $\GG(\mathbbm D_{\epsilon_i}\epsilon_i,\JJ\epsilon_k)=d(\arctanh\lambda_i)(e_k)$. This will conclude the proof, for
		$$\overline{\mathrm H}=\frac{1}{n}\JJ\left(\sum_{i=1}^n\sum_{k=1}^n d(\arctanh\lambda_i)(e_k)\epsilon_k\right)=\JJ\left(\sum_{k=1}^n \partial_{e_k}\kappa_\sigma\ dG_\sigma(e_k)\right)=\JJ(dG_\sigma(\overline \nabla \kappa_\sigma)))~.$$
		
		To show the claim, we will first use Lemma \ref{Lemma: curvatura media in G 1} to get 
		$$\GG(\mathbbm D_{\epsilon_i}\epsilon_i,\JJ\epsilon_k)=\gstar{n+1}(\widehat D_{\widetilde\epsilon_i}\widetilde\epsilon_i,J\widetilde \epsilon_k)~,$$
		where $\widehat D$ is the Levi-Civita connection of $\gstar{n+1}$ {and $\widetilde\epsilon_i$ is the horizontal lift of $\epsilon_i$}. 
		As in Equation \eqref{eq: differential of sigma tilde}, we can write $$d\zeta_\sigma(e_i)=(d\sigma(e_i),-\lambda_id\sigma(e_i))$$
		and the Levi-Civita connection $\widehat D$ is the normal projection with respect to the metric \eqref{eq:metric minkxmink} of the flat connection $\mathrm d$ of $\R^{n+1,1}\times\R^{n+1,1}$. Hence we can compute: 
		\begin{align*}
			\GG(\mathbbm D_{\epsilon_i}\epsilon_i,\JJ\epsilon_k)&=-\lambda_k\langle \mathrm d_{d\sigma(e_i)}d\sigma(e_i),d\sigma(e_k)\rangle+\lambda_i\langle \mathrm d_{d\sigma(e_i)}d\sigma(e_i),d\sigma(e_k)\rangle+\partial_{e_i} \lambda_i\langle d\sigma(e_i),d\sigma(e_k)\rangle\\
			&=(\lambda_i-\lambda_k)\ \I(\nabla_{e_i}e_i,e_k)+(\partial_{e_i} \lambda_i)\ \I(e_i,e_k)~.
		\end{align*}
		We recall that $g$ denotes the first fundamental form of $\sigma$, and $\nabla$ its Levi-Civita connection, and in the last equality we used that the Levi-Civita connection of $\Hyp^{n+1}$ is the  projection to the hyperboloid in Minkowski space $\R^{n+1,1}$ of the ambient flat connection.
		
		Now, when $i=k$ we obtained the desired result:
		$$\GG(\mathbbm D_{\epsilon_i}\epsilon_i,\JJ\epsilon_i)=\frac{\partial_{e_i} \lambda_i}{1-\lambda_i^2}=d(\arctanh\lambda_i)(e_i)$$
		since $e_i$ is a unit vector for the metric $\overline\I$, hence using the expression $\overline\I=\I-\III$ from Equation \eqref{eq:fff gauss} its {squared} norm for the metric $\I$ is $(1-\lambda_i^2)^{-1}$. When $i\neq k$, the latter term disappears since $\{e_1,\ldots,e_n\}$ is an orthogonal frame for $g$, and we are thus left with showing that 
		$$(\lambda_i-\lambda_k)\ \I(\nabla_{e_i}e_i,e_k)=d(\arctanh\lambda_i)(e_k)~.$$
		
		For this purpose, using the compatibility of $\nabla$ with the metric, namely $\partial_{ e_i} (\I( e_i,  e_k))=\I(\nabla_{ e_i} e_i,  e_k )+\I( e_i,  \nabla_{ e_i} e_k)$, that $\I(e_i,e_k)=0$, and that $\nabla$ is torsion-free, we get:
		$$(\lambda_i-\lambda_k)\ \I(\nabla_{ e_i}  e_i,  e_k) = - (\lambda_i-\lambda_k)\ \I( e_i, \nabla_{ e_i}  e_k)
		= \lambda_k\ \I( e_i,  \nabla_{ e_i} e_k) - \lambda_i\ \I( e_i,  \nabla_{ e_k} e_i) -\lambda_i\ \I( e_i, [ e_i, e_k ])~.$$
		
		Now, 
		recall that the Codazzi equation for $\sigma$ is $d^\nabla B=0$. 
		Applying it to the vector fields $e_i$ and $e_k$, we obtain
		$$d^\nabla B( e_i,  e_k)= \nabla_{ e_i} (\lambda_k  e_k) - \nabla_{ e_k} (\lambda_i  e_i)- B([ e_i,  e_k])=0,$$
		from which we derive
		\begin{equation}
			\label{eq: lemma2 curvatura media}
			\lambda_k \nabla_{e_i}e_k - \lambda_i \nabla_{e_k}e_i= (\partial_{ e_k} \lambda_i) e_i -(\partial_{ e_i} \lambda_k) e_k  + B([ e_i,  e_k])~.
		\end{equation}
		Using Equation \eqref{eq: lemma2 curvatura media} in the previous expression, we finally obtain:
		
		\begin{align*}
			(\lambda_i-\lambda_k)\ \I(\nabla_{ e_i}  e_i,  e_k) =
			& (\partial_{ e_k}\lambda_i)\ \I( e_i,  e_i) - (\partial_{ e_i}\lambda_k)\ \I( e_i,  e_k)+ \I( e_i, B[ e_i,  e_k] )- \I(B( e_i), [ e_i,  e_k])\\
			= &  \frac{\partial_{ e_k} \lambda_i}{1- \lambda_i^2}=d(\arctanh \lambda_i)(e_k)
		\end{align*}
		where the cancellations from the first to the second line are due to the fact that $B$ is $\I$-self adjoint and that $\I(e_i,e_k)=0$. This concludes the proof.
	\end{proof}
	
	\begin{Cor}\label{cor:maslov closed}
		Given an $n$-manifold $M$, a representation $\rho\colon \pi_1(M) \to \Isom(\Hyp^{n+1})$ and a $\rho$-equivariant Lagrangian and Riemannian {immersion} $G\colon \widetilde M \to \G{n+1}$, the Maslov class $\mu_G$ is a well-defined cohomology class in $H^1_{dR}(M,\R)$.
	\end{Cor}
	\begin{proof}
		{By Theorem \ref{prop: riemannian global integrability}, $G$ is the Gauss map of a (in general non equivariant) immersion $\sigma:\widetilde M\to\Hyp^{n+1}$. By Proposition \ref{Prop: formula H in G}, the 1-form $G^*(\Omega (\overline {\mathrm H}, \cdot ) )$ on $\widetilde M$ is exact, and $\rho$-equivariant, hence it induces a closed 1-form on $M$ whose cohomology class is $\mu_G$ as in Definition \ref{Def Maslov class}.}
	\end{proof}
	
	\subsection{Holonomy of the pull-back bundle and proof of Theorem \ref{teorema hol H baby} }\label{sec:hol flat princ bdles}
	
	An immediate consequence of Proposition \ref{Prop: formula H in G} is that the vanishing of the Maslov class is a necessary condition for a $\rho$-equivariant Lagrangian and Riemannian embedding $G:\widetilde M\to \G{n+1}$ to be $\rho$-integrable (Definition \ref{Def rhointegrable}). Indeed, if $\widetilde\sigma:\widetilde M\to\Hyp^{n+1}$ is a $\rho$-equivariant embedding with $G_{\widetilde\sigma}=G$ (hence necessarily with small principal curvatures), then the function $f_{\widetilde\sigma}$ descends to a well-defined function on $M$, hence by Proposition \ref{Prop: formula H in G} {$G^*(\Omega (\overline {\mathrm H}, \cdot ) )$} is an exact 1-form, i.e. the Maslov class $\mu_{G_{\widetilde\sigma}}$ vanishes. We will now see that this condition is also sufficient, which will be a consequence of a more general result, {Theorem \ref{Teorema hol H}}.
	
	Let $G:\widetilde M\to\G{n+1}$ be a $\rho$-equivariant Lagrangian and Riemannian embedding. We have already used that the $G$-pull-back bundle $\widetilde {\mathrm{p}}_{G}\colon \widetilde P \to \widetilde M$ of $\mathrm p:T^1\Hyp^{n+1}\to\G{n+1}$ is a flat trivial $\R$-principal bundle over $\widetilde M$, {namely, it is isomorphic, as a flat principal bundle, to the trivial bundle $\widetilde M\times \R\to \widetilde M$ with flat sections $\widetilde M\times \{\ast\}$}. Moreover, {$G$ being $\rho$-equivariant,} the fundamental group $\pi_1(M)$ acts freely and properly discontinously on $\widetilde P$, thus inducing a flat $\R$-principal bundle structure $\mathrm p_G\colon P \to M$, where $P$ is the quotient of $\widetilde P$ by the action of $\pi_1(M)$. However the bundle $\mathrm p_G\colon P\to M$ is not trivial in general. The obstruction to triviality is represented by the \emph{holonomy} of the bundle, which can be defined, in our setting, as follows.

	\begin{Def}\label{Def holonomy}
		Let $P\to M$ be a flat principal $\R$-bundle that is isomorphic to the quotient of the trivial bundle $\widetilde M\times\R\to\widetilde M$ by an equivariant (left) action of $\pi_1(M)$.  The \emph{holonomy representation} is the representation $\hol:\pi_1(M)\to\R$ such that  the action of $\pi_1(M)$ is expressed by:
		$$\alpha\cdot(m,s)=(\alpha\cdot m,\hol(\alpha)+ s)~.$$
	\end{Def}
	
	\begin{Remark}\label{rmk: alternative definition of hol}
		Fix $p\in M$ and $\alpha$ a closed $C^1$ loop based at $p$. Then pick a horizontal lift $\widehat\alpha$ to the total space of $\mathrm p_G$, namely with $\frac{d\widehat \alpha}{dt}$ orthogonal to the fibers, so that $\mathrm p\circ  \widehat \alpha=\alpha$. (The lift is uniquely determined by its initial point in $\mathrm{p}_G^{-1}(p)$.)  It follows from Definition \ref{Def holonomy} that
		\[
		\hol_G(\alpha) \cdot\widehat\alpha(1)=  \widehat \alpha(0) .
		\]
		In the identification $\pi_1(M)=\pi_1(M, [p])$, this allows to give an alternative definition of $\hol_G$ through homotopy classes of closed paths in $M$.
	\end{Remark}
	
	\begin{Remark}
		We remark that in general, for flat principal $G$-bundles, the holonomy representation is only defined up to conjugacy, but in out case $G=\R$ is abelian and therefore $\hol$ is uniquely determined by the isomorphism class of the flat principal bundle.
		
		Also observe that, since $\R$ is {abelian}, $\hol_G$ induces a map
		from $H_1(M, \Z)$ to $\R$,
		where $H_1(M,\Z)$ is the first homology group of $M$ and we are using that there is a canonical isomorphism between $H_1(M, \Z)$ and the abelianization of the fundamental group of $M$ in a point. Equivalently, $\hol_G$ is identified to an element of $H^1(M, \R)$.
	\end{Remark}
	
	We can interpret the holonomy of the principal bundle $\mathrm p_G$ in terms of the geometry of $\Hyp^{n+1}$. 
	{Global flat sections of the trivial bundle  $\widetilde{\mathrm{p}}_{G}\colon \widetilde P \to \widetilde M$ correspond to {Riemannian} embeddings $\zeta:\widetilde M\to T^1\Hyp^{n+1}$ as in Corollaries \ref{cor:lagrangian1} and \ref{cor:lagrangian2}.}
	By Theorem \ref{prop: riemannian global integrability}, such a $\zeta$ is the lift to $T^1\Hyp^{n+1}$ of an embedding $\sigma:\widetilde M\to\Hyp^{n+1}$ with small principal curvatures. 
	
	Now, let $\alpha\in\pi_1(M)$. {By equivariance of $G$, namely $G\circ\alpha=\rho(\alpha)\circ G$, it follows that $\mathrm p \circ\zeta\circ \alpha= \mathrm p \circ \rho(\alpha)\circ \zeta$, hence $\rho(\alpha)\circ\zeta\circ\alpha^{-1}:\widetilde M\to T^1\Hyp^{n+1}$ provides another flat section of the pull-back bundle $\widetilde {\mathrm{p}}_{G}$. Therefore} there exists $t_\alpha\in\R$ such that
	\begin{equation}\label{eq:equivariance lifts}
		\varphi_{t_\alpha}\circ\zeta=\rho(\alpha)\circ\zeta\circ\alpha^{-1}~.
	\end{equation}
	Then the value $t_\alpha$ is precisely the holonomy of the quotient bundle $\mathrm p_G$, namely the group representation
	\[
	\hol_G\colon \pi_1(M)\to \R\qquad \hol_G(\alpha)=t_\alpha
	\]
	A direct consequence of this discussion is the following:
	
	\begin{Lemma}\label{lemma:integrable iff trivial bundle}
		Given an $n$-manifold $M$ and a representation $\rho\colon \pi_1(M) \to \Isom(\Hyp^{n+1})$, a $\rho$-equivariant Lagrangian and Riemannian embedding $G\colon \widetilde M \to \G{n+1}$ is $\rho$-integrable if and only if the $\R$-principal flat bundle $\mathrm p_G$ is trivial.
	\end{Lemma}
	\begin{proof}
		The bundle $\mathrm p_G$ is trivial if and only if its holonomy $\hol_G$ vanishes identically, that is, if and only if $t_\alpha=0$ for every $\alpha\in\pi_1(M)$. By the above construction, this is equivalent to the condition that $\zeta\circ\alpha=\rho(\alpha)\circ\zeta$ for all $\alpha$, which {is equivalent to} $\sigma\circ\alpha=\rho(\alpha)\circ\sigma$, {namely} that $\sigma$ is $\rho$-equivariant.
	\end{proof}


	We are ready to prove the following.

	\begin{Theorem}
		\label{Teorema hol H}
		Given an $n$-manifold $M$, a representation $\rho\colon \pi_1(M) \to \Isom(\Hyp^{n+1})$ and a $\rho$-equivariant Lagrangian and Riemannian embedding $G\colon \widetilde M \to \G{n+1}$, the holonomy of $\mathrm p_G$ is given by
		\[
		\hol_G(\alpha)=\int_{\alpha} \mu_G.
		\]
		for all $\alpha \in \pi_1(M)$.
	\end{Theorem}
	
	Observe that Theorem \ref{teorema hol H baby} follows immediately from Theorem \ref{Teorema hol H} since, by the standard de Rham Theorem, there exists an isomorphism 
	\begin{equation*}
		\begin{split}
			H_{dR}^1(M, \R) &\xrightarrow{\sim} H^1(M, \R) \\
			\eta &\mapsto \bigg( \xi \mapsto \int_\xi \eta \bigg),
		\end{split}
	\end{equation*}
	hence $\hol_G\equiv 0$ if and only if $\mu_G=0$. 
	
	\begin{proof}[Proof of Theorem $\ref{Teorema hol H}$]
		Let $\zeta:\widetilde M\to T^1\Hyp^{n+1}$ be a map {such} that $\mathrm{p}\circ \zeta=G$, so as to induce a global section of the pull-back bundle $\widetilde {\mathrm p}_G$. Then by Equation \eqref{eq:equivariance lifts} the holonomy $t_\alpha=\hol_G(\alpha)$ satisfies $\varphi_{t_\alpha}\circ\zeta\circ\alpha=\rho(\alpha)\circ\zeta$. By Proposition \ref{prop: flusso geod e flusso normale}, this gives the following equivariance relation for $\sigma=\pi\circ\zeta$:
		$$(\sigma\circ\alpha)_{t_\alpha}=\rho(\alpha)\circ\sigma~.$$
		{Let now $\kappa_\sigma$ denote the mean of the hyperbolic arctangents of the principal curvatures, as in Equation \eqref{eq:aux function}}. Lemma \ref{lemma:evolution fsigma} and the fact that $\rho(\alpha)$ acts isometrically imply:
		$$f_{\sigma\circ\alpha}=\kappa_\sigma+t_\alpha~.$$
		Now, by Proposition \ref{Prop: formula H in G} and the definition of the Maslov class, we have:
		$$\int_\alpha \mu_G=\int_\alpha d\kappa_\sigma=\kappa_\sigma(\alpha(p))-\kappa_\sigma(p)=t_\alpha$$
		for any point $p\in M$. This concludes the proof.
	\end{proof}

	\section{Minimal Lagrangian immersions}
	
	We prove here two direct corollaries of Theorem \ref{teorema hol H baby}. Let us first recall the definition of minimal Lagrangian (Riemannian) immersions.
	
	\begin{Def}
		A {Riemannian} immersion of an $n$-manifold into $\G{n+1}$ is \emph{minimal Lagrangian} if:
		\begin{itemize}
			\item  its mean curvature vector vanishes identically;
			\item it is Lagrangian with respect to the symplectic form $\Omega$.
		\end{itemize}
	\end{Def}
	
	Our first corollary is essentially a consequence of Theorem \ref{teorema hol H baby}.
	
	\begin{Cor} \label{cor:minimal is rho integrable}
		Let $M^n$ be a closed {orientable} manifold  and  $\rho:\pi_1(M)\to\Isom(\Hyp^{n+1})$ a representation. If $G:\widetilde M\to\G{n+1}$ is a $\rho$-equivariant Riemannian minimal Lagrangian immersion, then $G$ is the Gauss map of a $\rho$-equivariant embedding $\widetilde\sigma:\widetilde M\to \Hyp^{n+1}$ with small principal curvatures such that
		$$\kappa_{\widetilde\sigma}={\frac{1}{n}}\sum_{i=1}^n\arctanh\lambda_i=0~.$$
		In particular, $\rho$ is a nearly-Fuchsian representation and $G$ is an embedding.
	\end{Cor}
	We remark that if $n=2$, then the condition $\kappa_{\widetilde\sigma}=0$ is equivalent to $\lambda_1+\lambda_2=0$ since $\arctanh$ is an odd {and injective} function. That is, in this case $\widetilde\sigma$ is a \emph{minimal} embedding in $\Hyp^3$. 
	\begin{proof}
		Suppose $G$ is a $\rho$-equivariant minimal Lagrangian immersion. Since its mean curvature vector vanishes identically, we have $\mu_G=0$ and therefore $G$ is $\rho$-integrable by Theorem \ref{teorema hol H baby}. That is, there exists a $\rho$-equivariant immersion $\widetilde\sigma:\widetilde M\to\Hyp^{n+1}$ such that $G=G_{\widetilde\sigma}$. By Proposition \ref{prop: small curv sse riemannian}, $\widetilde\sigma$ has small principal curvatures, hence $\rho$ is nearly-Fuchsian. By Proposition \ref{Prop: formula H in G}, we have that $\kappa_{\widetilde\sigma}$ is constant. By Lemma \ref{lemma:evolution fsigma}, up to taking the normal evolution, we can find $\widetilde\sigma$ such that $\kappa_{\widetilde\sigma}$ vanishes identically.
		
		Finally, $\widetilde \sigma$ is complete by cocompactness, and therefore both $\widetilde\sigma$ and $G$ are embeddings by Proposition \ref{prop injectivity} and Proposition \ref{prop:gauss maps diffeo onto image}.  
	\end{proof}
	
	The following is a uniqueness result for $\rho$-equivariant minimal Lagrangian immersions.
	
	\begin{Cor}\label{cor:uniqueness min lag}
		Given a closed {orientable} manifold $M^n$ and a representation $\rho:\pi_1(M)\to\Isom(\Hyp^{n+1})$, there exists at most one $\rho$-equivariant Riemannian minimal Lagrangian immersion $G:\widetilde M\to\G{n+1}$ up to reparametrization. If such a $G$ exists, then $\rho$ is nearly-Fuchsian and $G$ {induces a {minimal Lagrangian} embedding of $M$ in $\GGG_\rho$}.
	\end{Cor}
	\begin{proof}
		Suppose that $G$ and $G'$ are $\rho$-equivariant minimal Lagrangian immersions. By Corollary \ref{cor:minimal is rho integrable}, there exist $\rho$-equivariant embeddings $\widetilde\sigma,\widetilde\sigma':\widetilde M\to\Hyp^{n+1}$ {with small principal curvatures} such that $G=G_{\widetilde\sigma}$ and $G'=G_{\widetilde\sigma'}$, with $\kappa{\widetilde\sigma}=\kappa{\widetilde\sigma'}=0$.  {Moreover, $G$ and $G'$ induce embeddings in  $\GGG_\rho$ by Corollary \ref{cor:embedding in Grho}.}
		
		By Remark \ref{rmk:embedding in the quotient}, both $\sigma$ and $\sigma'$ induce embeddings of {$M$ in the nearly-Fuchsian manifold $\faktor{\Hyp^{n+1}}{\rho(\pi_1(M))}$; let us denote with $\Sigma$ and $\Sigma'$ the corresponding images}. We claim that $\Sigma=\Sigma'$, which implies the uniqueness in the statement. 
		
		To see this, consider the signed distance from $\Sigma$, which is a proper function 
		$$r:\faktor{\Hyp^{n+1}}{\rho(\pi_1(M))}\to\R~.$$
		Since $\Sigma'$ is closed, $r|_{\Sigma'}$ admits a maximum value $r_{\max}$ achieved at some point $x_{\max}\in\Sigma'$.   This means that at the point $x_{\max}$, $\Sigma'$ is tangent to a hypersurface $\Sigma_{r_{\max}}$ at signed distance $r_{\max}$ from $\Sigma$, and $\Sigma'$ is contained in the side of $\Sigma_{r_{\max}}$ where $r$ is decreasing. This implies that, if $B'$ denotes the shape operator of $\Sigma'$ and $B_{r_{\max}}$ that of $\Sigma_{r_{\max}}$, both computed with respect to the unit normal vector pointing to the side of increasing $r$, then $B_{r_{\max}}-B'$ is positive semi-definite at $x_{\max}$.
		
		Let us now denote by $\lambda_1,\ldots,\lambda_n$ the eigenvalues of $B_{r_{\max}}$ and $\lambda'_1,\ldots,\lambda_n'$ those of $B'$. Let us moreover assume that $\lambda_1\leq \ldots\leq\lambda_n$ and similarly for the $\lambda_i'$. By Weyl's monotonicity theorem, $\lambda_i\geq \lambda_i'$ at $x_{\max}$ for $i=1,\ldots,n$. Since $\arctanh$ is a monotone increasing function, this implies that
		$$\sum_{i=1}^n\arctanh\lambda_i(x_{\max})\geq \sum_{i=1}^n\arctanh\lambda_i'(x_{\max})~.$$
		Since $\kappa{\widetilde\sigma'}=0$, the right-hand side vanishes. On the other hand, since $\kappa{\widetilde\sigma}=0$, from Lemma \ref{lemma:evolution fsigma} the left-hand side is identically equal to $-r_{\max}$. Hence $r_{\max}\leq 0$. Repeating the same argument replacing the maximum point of $r$ on $\Sigma'$ by the minimum point, one shows $r_{\min}\geq 0$. Hence $r|_{\Sigma'}$ vanishes identically, which proves that $\Sigma=\Sigma'$ and thus concludes the proof.
	\end{proof}


\chapter{Hamiltonian symplectomorphisms of $\G{\rho}$ and deformations of immersions}\label{sec:hamiltonian}

In this chapter we will provide the second characterization of $\rho$-integrability, in the case of a nearly-Fuchsian representation $\rho:\pi_1(M)\to\Isom(\Hyp^{n+1})$. We first introduce the terminology and state the result (Theorem \ref{thm:second char ham}); then we introduce the so-called \emph{Lagrangian Flux map} which will play a central role in the proof of Theorem \ref{thm:second char ham}.

\section{Hamiltonian group and nearly-Fuchsian manifolds}
We will restrict hereafter to the case of nearly-Fuchsian representations $\rho:\pi_1(M)\to\Isom(\Hyp^{n+1})$. 
Let $G:\widetilde M\to\G{n+1}$ be a $\rho$-integrable immersion as in Definition \ref{Def rhointegrable}. Since $\rho$ is nearly-Fuchsian, we showed in Corollary \ref{cor:embedding in Grho} that $G$ induces an embedded submanifold in the para-K\"ahler manifold $\GGG_\rho$, defined in Definition \ref{Def quotient Grho}. This motivates the following definition in the spirit of Definition \ref{Def rhointegrable}.

\begin{Def}\label{Def rhointegrable submfd}
	Given a closed {orientable} $n$-manifold $M$ and a nearly-Fuchsian representation $\rho:\pi_1(M)\to\Isom(\Hyp^{n+1})$,
	an embedding 
	$M\to \GGG_\rho$ is $\rho$-\emph{integrable} if it is induced in the quotient from {a $\rho$-integrable} embedding ${G}\colon \widetilde M\to \G{n+1}$.
	Similarly, an embedded submanifold $\mathcal L\subset \GGG_\rho$ is $\rho$-\emph{integrable} if it is the image of a $\rho$-integrable embedding.
\end{Def}

Theorem \ref{thm:second char ham} below gives a description of the set of $\rho$-integrable submanifolds $\mathcal L\subset \GGG_\rho$ which are induced by immersions $G$ with small principal curvatures. Clearly, {as we have previously shown}, a necessary condition on $\mathcal L$ is that of being Lagrangian and Riemannian. To state the theorem, we need to recall the notion of Hamiltonian symplectomorphism.

\begin{Def}\label{Def Hamc}
	Given a symplectic manifold $(\mathcal X,\Omega)$, a compactly supported symplectomorphism $\Phi$ is \emph{Hamiltonian} if there exists a compactly supported smooth function $F_\bullet:\mathcal X\times[0,1]\to\R$
	such that $\Phi=\Phi_1$, where $\Phi_{s_0}$ is the flow at time $s_0$ of the (time-dependent) vector field $X_s$ defined by:
	\begin{equation}\label{eq:hamiltonian function}
		dF_s=\Omega(X_s,\cdot)~.
	\end{equation}
	The isotopy  $\Phi_\bullet:\mathcal X\times [0,1]\to \mathcal X$ is called \emph{Hamiltonian isotopy}.
\end{Def} 

\begin{Remark}\label{rmk:hamiltonian and cartan}
	If $\Phi_\bullet$ is a Hamiltonian isotopy as in Definition \ref{Def Hamc}, then $\Phi_s$ is a symplectomorphism for every $s\in[0,1]$. Indeed
	$$\mathcal L_{X_s}\Omega=\iota_{X_s}d\Omega+d(\iota_{X_s}\Omega)=0$$
	as a consequence of Cartan's formula and Equation \eqref{eq:hamiltonian function}, and $\Phi_s$ is clearly Hamiltonian. 
\end{Remark}

Compactly supported Hamiltonian symplectomorphisms form a group which we will denote by $\Ham_c(\mathcal X,\Omega)$.

{The aim of this chapter is to prove the following result.}

\begin{Theorem}\label{thm:second char ham}
	Let $M$ be a closed {orientable} $n$-manifold, $\rho:\pi_1(M)\to\Isom(\Hyp^{n+1})$ be a nearly-Fuchsian representation and $\mathcal L\subset\GGG_\rho$ a Riemannian $\rho$-integrable submanifold. Then a Riemannian submanifold $\mathcal L'$ is $\rho$-integrable if and only if there exists $\Phi\in \Ham_c(\GGG_\rho,\Omega)$ such that $\Phi(\mathcal L)=\mathcal L'$.
\end{Theorem}

Of course, although not stated in Theorem \ref{thm:second char ham}, both $\mathcal L$ and $\mathcal L'$ are necessarily Lagrangian as a consequence of Corollary \ref{cor:lagrangian1}.

\section{The Lagrangian Flux}

We shall now define the Flux map for Lagrangian submanifolds, which was introduced in \cite{solomon}, and relate it to the holonomy of $\R$-principal bundles.

\begin{Def}\label{Def:flux}
	Let $(\mathcal X,\Omega)$ be a symplectic manifold and let $\Upsilon_\bullet:M\times[0,1]\to\mathcal X$ be a smooth map such that each $\Upsilon_t$ is a Lagrangian embedding of $M$. Then we define:
	$$\Flux(\Upsilon_\bullet)=\int_0^1 \Upsilon_s^*(\Omega(X_s,\cdot))ds\in H^1_{dR}(M,\R)~,$$
	where
	$$X_{s_0}(\Upsilon_{s_0}(p))=\left.\frac{d}{ds}\right|_{s=s_0}\Upsilon_s(p)\in T_{\Upsilon_{s_0}(p)}\mathcal X~.$$
\end{Def}

Observe that by Cartan's formula the integrand $\Upsilon_s^*(\Omega(X_s,\cdot))$ is a closed 1-form for every $s$, hence $\Flux(\Upsilon_\bullet)$ is well-defined as a cohomology class in $H^1_{dR}(M,\R)$. 

Now, let $\mathcal L$ be a Lagrangian embedded submanifold in $\GGG_\rho$, which is induced by a $\rho$-equivariant immersion $G:\widetilde M\to\G{n+1}$. Recall that in Subsection \ref{sec:hol flat princ bdles} we defined the principal $\R$-bundle $\mathrm p_G$ as the quotient of $\widetilde{\mathrm p}_G=G^*\mathrm p$ by the action of $\pi_1(M)$. Moreover in Theorem \ref{Teorema hol H} we computed the holonomy 
$$\hol_G:\pi_1(M)\to\R$$
of $\mathrm p_G$. 
The key relation between Lagrangian flux and $\hol_{G}$ is stated in the following proposition.

\begin{Prop} 
	\label{prop: flux = diff holonomy}
	Let $M$ be a closed {orientable} $n$-manifold and $\rho:\pi_1(M)\to\Isom(\Hyp^{n+1})$ be a nearly-Fuchsian representation. If $\Upsilon_s$ is as in Definition \ref{Def:flux} and $\Upsilon_0(M)=\mathcal L$, $\Upsilon_1(M)=\mathcal L'$, then
	\[
	\hol_{\Upsilon_1} (\alpha) - \hol_{\Upsilon_0} (\alpha)=\int_\alpha\Flux (\Upsilon_\bullet)~.
	\]
	In particular, $\Flux (\Upsilon_\bullet)  (\alpha)$ depends uniquely on the endpoints of $\Upsilon_\bullet$.
\end{Prop}

To prove Proposition \ref{prop: flux = diff holonomy}, we will make use of the following expression for the holonomy representation.

\begin{Prop}
	\label{prop: holonomy as integral of loops}
	Let $G\colon \widetilde M\to\G{n+1}$ be a $\rho$-equivariant Lagrangian embedding and $\mathrm p_G$ be the associated $\R$-principal bundle over $M$. If $\alpha:[0,1]\to M$ is a smooth loop and $\overline\alpha$ a smooth loop in the total space of $\mathrm p_G$ such that $\alpha= \mathrm p_G (\overline \alpha)$, then
	\[
	\hol_G(\alpha)= \int_{\overline \alpha} \omega
	\]
	where $\omega$ is the principal connection of $\mathrm p_G$.
\end{Prop}

\begin{proof}
	Say $\alpha(0)=\alpha(1)= x_0$. {Recalling Remark \ref{rmk: alternative definition of hol}},
	let $\widehat \alpha$ be the horizontal lift of $\alpha$ starting at $\overline \alpha(0)$. We apply Stokes' theorem. Define {a smooth map $f$ from $[0,1]\times [0,1]$ to the total space of $\mathrm p_G$} so that 
	\begin{itemize}
		\item $f(x,0)= \overline \alpha(x)$,
		\item $f(x,1)= \widehat \alpha(x)$, 
		\item {$f(0, y)\equiv \overline \alpha(0)=\overline\alpha(1)$},
		\item {$y\mapsto f(1, y)$  parametrizes the interval from $\overline \alpha(1)$ to $\widehat \alpha(1)$ in $\mathrm p_G^{-1} (x_0)\approx \R$.}
	\end{itemize}
	{By Stokes' Theorem and the flatness of $\mathrm p_G$}, one gets that  
	\begin{align*}
		0= \int_{[0,1]\times [0,1]} f^* d\omega&=
		\int_{\overline \alpha} \omega + \int_{f(1, \cdot)} \omega - \int_{\widehat \alpha} \omega - \int_{f(0, \cdot)} \omega\\
		&=\int_{\overline \alpha} \omega + \int_{f(1, \cdot)} \omega.
	\end{align*}
	
	By Remark \ref{rmk: alternative definition of hol},  $\widehat \alpha(1)= (-\hol_G(\alpha)) \cdot \overline\alpha(1)$. 
	Since $\omega= \gs{n+1}(\chi, \cdot)$ and $f(1, \cdot)$ is contained in {$\mathrm p_G^{-1} (x_0)$}, one gets that
	\[
	\int_{f(1, \cdot)} \omega=- \hol_G(\alpha)
	\]
	and the proof follows.
\end{proof}

\begin{proof}[Proof of Proposition $\ref{prop: flux = diff holonomy}$]
	Define 
	$\Theta:[0,1]\times S^1\to M$ by
	$\Theta(s,t)=\Upsilon_s(\alpha(t))$. Since the bundle $\mathrm p_G$ has contractible fibre, there always exists a smooth global section. In particular, there exists $\overline\Theta$ such that $\Theta=\mathrm p_G\circ \overline\Theta$. 
	By Proposition \ref{prop: holonomy as integral of loops}, recalling that $d\omega=\mathrm{p}^* \Omega$, and applying Stokes' Theorem, we obtain:
	$$\hol_{\Upsilon_1} (\alpha) - \hol_{\Upsilon_0}(\alpha)=\int_{\overline\Theta(1, \cdot)} \omega - \int_{\overline\Theta(0, \cdot)} \omega=\int_{[0,1]\times S^1}\overline\Theta^*d\omega=\int_{[0,1]\times S^1}\Theta^*\Omega$$
	and the last term equals  $\int_\alpha\Flux(\Upsilon_\bullet)$.
\end{proof}

We conclude this section by proving one (easy) implication of Theorem \ref{thm:second char ham}. As mentioned in the introduction, this implication does not need the hypothesis that $\mathcal L$ and $\mathcal L'$ are Riemannian. 

\begin{proof}[Proof of the ``if'' part of Theorem $\ref{thm:second char ham}$]
	Suppose there exists a Hamiltonian symplectomorphism $\Phi=\Phi_1$, endpoint of a Hamiltonian isotopy $\Phi_\bullet$, such that $\Phi(\mathcal L)=\mathcal L'$. Then define the map $\Upsilon_\bullet:M\times[0,1]\to\GGG_\rho$ in such a way that $\Upsilon_0:M\to\mathcal G$ is an embedding with image $\mathcal L$ and
	$$\Upsilon_s=\Phi_s\circ \Upsilon_0~.$$
	By Remark \ref{rmk:hamiltonian and cartan}, $\Phi_s$ is a (Hamiltonian) symplectomorphism for every $s\in[0,1]$, hence $\Upsilon_s$ is a Lagrangian embedding for all $s$. We claim that $\Flux(\Upsilon_\bullet)$ vanishes in $H^1_{dR}(M,\R)$. Indeed, for every $s$ we have
	$$\Upsilon_s^*(\Omega(X_s,\cdot))=\Upsilon_s^*dF_s=df_s$$
	by Equation \eqref{eq:hamiltonian function}, where $X_s$ is the vector field generating the Hamiltonian isotopy (and hence $\Upsilon_\bullet$) and $f_s=F_s\circ\Upsilon_s$. Therefore 
	$\int_0^1\Upsilon_s^*(\Omega(X_s,\cdot))ds$ is exact, namely $\Flux(\Upsilon_\bullet)=0$.
	
	Using Proposition \ref{prop: flux = diff holonomy}, we have $\hol_{\Upsilon_0}  = \hol_{\Upsilon_1}$. By Lemma \ref{lemma:integrable iff trivial bundle}, this shows that $\mathcal L$ is $\rho$-integrable if and only if $\mathcal L'$ is $\rho$-integrable, and this concludes the proof of the first implication in Theorem \ref{thm:second char ham}.
\end{proof}

\section{Conclusion of Theorem \ref{thm:second char ham}} 

We are left with the other implication in  Theorem \ref{thm:second char ham}.  Given two Riemannian $\rho$-integrable submanifolds $\mathcal L,\mathcal L'\subset \GGG_\rho$, we shall produce $\Phi\in\Ham_c(\GGG_\rho,\Omega)$ mapping $\mathcal L$ to $\mathcal L'$. We remark here that the results and methods of \cite{solomon} use stronger topological hypothesis, hence do not apply under our assumptions.

Roughly speaking, the idea is to reduce the problem to finding a deformation in the nearly-Fuchsian manifold $\faktor{\Hyp^{n+1}}{\rho(\pi_1(M))}$ which interpolates between two hypersurfaces of small principal curvatures corresponding to $\mathcal L$ to $\mathcal L'$. For technical reasons, it will be easier to deal with {convex} hypersurfaces  that we defined in Definition \ref{Def:convex immersion}}

\begin{Lemma} \label{lemma:convex gauss map diffeo}
Let $M^n$ be a closed oriented manifold, $\rho:\pi_1(M)\to\Isom(\Hyp^{n+1})$ be a nearly-Fuchsian representation and $\widetilde\sigma:\widetilde M\to\Hyp^{n+1}$ be a $\rho$-equivariant embedding. If $\widetilde\sigma$ is convex, then the Gauss map $G_{\widetilde\sigma}^+$ is an {equivariant} diffeomorphism between $\widetilde M$ and the connected component $\Omega_+$ of $\partial\Hyp^{n+1}\setminus\Lambda_\rho$.
\end{Lemma}
\begin{proof}
By the same argument as in Section \ref{sec:nearly fuchsian} (see the discussion between Proposition \ref{prop:action free prop disc0} and Proposition \ref{prop:action free prop disc}), $\widetilde\sigma$ extends to a continuous injective map of the visual boundary of $\widetilde M $ with image $\Lambda_\rho$. We can now repeat wordly the argument of Proposition \ref{prop:gauss maps diffeo onto image} to show that, if $B$ is negative semi-definite, then $G_{\widetilde\sigma}^+$ is a diffeomorphism onto its image. To show that {$G_{\widetilde\sigma}^+(\widetilde M)=\Omega_+$}, we repeat instead the proof of Proposition \ref{prop:action free prop disc}. More precisely, one first shows (using tangent horospheres) that every $x\in\Omega_+$ is in the image of $G_{\widetilde\sigma}^+$. Then, by continuity, it suffices to show that every $x\in\Lambda_\rho$ is not on the image of $G_{\widetilde\sigma}^+$. To see this, the last paragraph of the proof of Proposition \ref{prop:action free prop disc} applies unchanged, and when considering tangent $r$-caps we can even take $r=0$, that is, replace $r$-caps by totally geodesic hyperplanes. See Figure \ref{fig:limit2}.
\end{proof}

\begin{Lemma}\label{lemma:convex interpolation}
Let $M^n$ be a closed oriented manifold and $\rho:\pi_1(M)\to\Isom(\Hyp^{n+1})$ be a nearly-Fuchsian representation. Given two closed hypersurfaces $\Sigma_0$ and $\Sigma_1$ of small principal curvatures in the nearly-Fuchsian manifold $\faktor{\Hyp^{n+1}}{\rho(\pi_1(M))}$, there exists an isotopy $$\upsilon_\bullet\colon M\times[0,1]\to\faktor{\Hyp^{n+1}}{\rho(\pi_1(M))}$$ such that:
\begin{itemize}
	\item $\upsilon_s$ is a convex embedding for all $s\in[0,1]$;
	\item $\upsilon_0(M)$ is a hypersurface equidistant from $\Sigma_0$;
	\item $\upsilon_1(M)$ is a hypersurface equidistant from $\Sigma_1$.
\end{itemize}
\end{Lemma} 
\begin{proof}
First of all, let us observe that we can find hypersurfaces equidistant from $\Sigma_0$ and $\Sigma_1$ which are convex. Indeed, by Corollary \ref{cor:equidistant is immersed} and Remark \ref{rmk:embedding in the quotient}, $t$-equidistant hypersurfaces are embedded for all $t\in\R$. Moreover, by compactness, the principal curvatures of $\Sigma_0$ and $\Sigma_1$ are in $(-\epsilon,\epsilon)$ for some $0<\epsilon<1$, and applying Equation \eqref{eq della discordia} we may find  $t_0$ such that the principal curvatures of the $t$-equidistant hypersurfaces are negative for $t\geq t_0$ -- namely, the equidistant hypersurfaces are convex.

Abusing notation, up to taking equidistant hypersurfaces as explained above, we will now assume that $\Sigma_0$ and $\Sigma_1$ are convex, and our goal is to produce $\upsilon_\bullet$ such that $\upsilon_s$ is a convex embedding for all $s\in[0,1]$, $\upsilon_0(M)=\Sigma_0$ and $\upsilon_1(M)=\Sigma_1$. Up to replacing $\Sigma_0$ and $\Sigma_1$ again {with equidistant hypersurfaces}, we {can} also assume that $\Sigma_0\cap\Sigma_1=\emptyset$, that $\Sigma_1$ is in the concave side of $\Sigma_0$, and that the equidistant surfaces from $\Sigma_1$ which intersect $\Sigma_0$ are all convex. We call $\mathcal A$ the region of $\faktor{\Hyp^{n+1}}{\rho(\pi_1(M))}$ bounded by $\Sigma_0$ and containing $\Sigma_1$.

Let us now consider the (signed) distance functions $r_0$ and $r_1$ from $\Sigma_0$ and $\Sigma_1$ respectively, chosen in such a way that both $r_0$ and $r_1$ are positive functions on the concave side of $\Sigma_0$ and $\Sigma_1$ respectively. Again by Corollary \ref{cor:equidistant is immersed} and Remark \ref{rmk:embedding in the quotient}, these functions are smooth and have nonsingular differential everywhere. Let us denote by $\nu_i$ the gradient of $r_i$. {The vector field $\nu_i$} has unit norm and is tangent to the orthogonal foliations of $\Sigma_i$ which have been described in the proof of Proposition \ref{prop:action free prop disc}. (Proposition \ref{prop:action free prop disc} describes the foliation in the universal cover, but it clearly descends to the quotient $\faktor{\Hyp^{n+1}}{\rho(\pi_1(M))}$.)

We claim that both $r_i$'s are convex functions in the region $\mathcal A$,  i.e. that their Hessians are positive semi-definite, as a consequence of the fact that the level sets of $r_i$ in $\mathcal A$ are all convex. Recall that the Riemannian Hessian of a smooth function $f:\mathcal A\to\R$ is the symmetric 2-tensor defined as 
\begin{equation}
	\label{eq: def hessiano}
	\nabla^2f(X,Y)=\partial_X (\partial_Y f)-\partial_{D_XY}f~,
\end{equation}
where $X,Y$ are local vector fields and $D$ is the ambient Levi-Civita connection as usual. Clearly $\nabla^2r_i(\nu_i,\nu_i)=0$ since $r_i$ is linear along the integral curves of $\nu_i$ and such integral curves, which are the leaves of the orthogonal foliation described above, are geodesics. Moroever, if $X$ is a vector field tangent to the level sets of $r_i$, then $\nabla^2r_i(X,\nu_i)=0$: indeed the first term {in the RHS of Equation \eqref{eq: def hessiano}} vanishes  because $r_i$ is linear along the integral curves, and the second term as well, because $D_{X}\nu_i=-B_i(X)$ is tangential to the level sets of $r_i$ and thus $\partial_{D_{X}\nu_i}r_i=0$.

To conclude that $\nabla^2r_i$ is positive semi-definite, it remains to show that $\nabla^2r_i(X,X)\geq 0$ for all $X$ tangent to the level sets. 
It is more instructive to perform this computation in the general setting of a smooth function  $f:\mathcal A\to\R$. Since the unit normal vector field to the level set of $f$ is $\nu=\frac{Df}{\|Df\|}$,{with $Df$ being the gradient of $f$, for all $X,Y$ vector fields tangent to the fibers, we} get that:
\begin{equation}\label{eq:hessian and II}
	\nabla^2f(X,Y)=-\partial_{D_XY}f=-\langle D_XY,\nu\rangle\partial_\nu f=-\|Df\|\II(X,Y)~,
\end{equation}
where in the last step we used that $\partial_\nu f=\langle Df,\nu\rangle=\|Df\|$, and $\II$ denotes the second fundamental form of the level sets of $f$. When $f=r_i$, in the region $\mathcal A$ the level sets of $r_i$ are convex, hence {$\II$ is negative semi-definite and} $\nabla^2r_i(X,X)\geq 0$.

We remark that Equation \eqref{eq:hessian and II} also shows that, if $f$ is a convex function, then its level sets are convex hypersurfaces as long as $Df\neq 0$. We shall now apply this remark to the zero set of the function $f_s=(1-s)r_0+sr_1$ for $s\in[0,1]$. The differential of $f_s$ never vanishes, for $\|Dr_0\|=\|Dr_1\|=1$, hence $Df_s=0$ is only possible for $s=\frac 1 2$ if $Dr_0=-Dr_1$: nevertheless, this cannot happen  since the geodesics with initial vector $Dr_0=\nu_0$ and $Dr_1=\nu_1$ both have final endpoint in $\Omega_+$ and initial endpoint in $\Omega_-$ by (the proof of) Proposition \ref{prop:action free prop disc}. Hence $\{f_s=0\}$ is an embedded hypersurface for all $s$. Observe moreover that
$$\{f_s=0\}=\bigg\{\frac{r_0}{r_0-r_1}=s\bigg\}~.$$
Since $\Sigma_0\cap \Sigma_1=\emptyset$, $r_0-r_1$ never vanishes, and this shows that the hypersurfaces $\{f_s=0\}$ provide a foliation of the region between $\Sigma_0$ and $\Sigma_1$, which is contained in $\mathcal A$. Since both $r_0$ and $r_1$ are convex functions in $\mathcal A$, 
$$\nabla^2f_s(X,X)=(1-s)\nabla^2r_0(X,X)+s\nabla^2r_1(X,X)\geq 0$$
for every $X$, hence $f_s$ is convex. As remarked just after Equation \eqref{eq:hessian and II}, since $\|Df_s\|\neq 0$, $\{f_s=0\}$ is a convex hypersurface.

It is not hard now to produce $\upsilon_\bullet:M\times[0,1]\to\faktor{\Hyp^{n+1}}{\rho(\pi_1(M))}$ such that $\upsilon_s(M)=\{f_s=0\}$. For instance one can flow along the vector field $\frac{DF}{\|DF\|^2}$ where $F=\frac{r_0}{r_0-r_1}$. Alternatively one can apply Lemma \ref{lemma:convex gauss map diffeo} to infer that the Gauss maps $G^+_{\widetilde\sigma}$ in the universal cover  induce diffeomorphisms of each hypersurface $\{f_s=0\}$ with $\faktor{\Omega_+}{\rho(\pi_1(M))}\cong M$, and define $\upsilon_s$ as the inverse map.
\end{proof}

\begin{proof}[Proof of the ``only if'' part of Theorem $\ref{thm:second char ham}$]
Suppose $\mathcal L$ and $\mathcal L'$ are $\rho$-integrable Riemannian submanifolds in $\GGG_\rho$. Then there exists hypersurfaces $\Sigma$ and $\Sigma'$ in $\faktor{\Hyp^{n+1}}{\rho(\pi_1(M))}$ whose Gauss map image induce $\mathcal L$ and $\mathcal L'$ respectively. We now apply Lemma \ref{lemma:convex interpolation} and find $\upsilon_\bullet$ such that $\upsilon_s$ is convex for every $s$, and the images of $\upsilon_0$ and $\upsilon_1$ are equidistant hypersurfaces from $\Sigma$ and $\Sigma'$ respectively. Define 
$\Upsilon_\bullet:M\times[0,1]\to\GGG_\rho$ so that $\Upsilon_s$ is the map into $\GGG_\rho$ induced  by the Gauss map of the lifts on the universal cover $\widetilde\upsilon_s:\widetilde M\to\Hyp^{n+1}$. As a consequence of Lemma \ref{lemma:convex gauss map diffeo} the Gauss map of each $\widetilde\upsilon_s$ is an embedding with image in $\Omega_+\times\partial\Hyp^{n+1}\setminus\Delta$ in $\G{n+1}$. 
Repeating the same argument of Remark \ref{rmk:embedding in the quotient}, $\Upsilon_s:M\to\GGG_\rho$ is an embedding for every $s$. As a particular case, by Lemma \ref{prop:gauss map invariant normal evo}, $\Upsilon_0(M)=\mathcal L$ and $\Upsilon_1(M)=\mathcal L'$.

By construction for every $s\in[0,1]$ the image of $\Upsilon_s$ is a $\rho$-integrable embedded submanifold in $\GGG_\rho$. Let us denote by $\mathcal L_s$ such submanifold. It follows (Lemma \ref{lemma:integrable iff trivial bundle}) that {$\hol_{\Upsilon_s}$} is trivial for all $s$.
By Proposition \ref{prop: flux = diff holonomy}, together with the definition of $\Flux$, we have that
$$\int_0^{s}\Upsilon_r^*(\Omega(X_r,\cdot))dr=0\in H^1_{dR}(M,\R)$$
for all $s$, where $X_s$ is the vector field generating $\Upsilon_s$. Hence necessarily the cohomology class of $\Upsilon_s^*(\Omega(X_s,\cdot))$  in $H^1_{dR}(M,\R)$ is trivial for all $s$. We can therefore find a smooth function $f_\bullet:M\times[0,1]\to\R$ such that $df_s=\Upsilon_s^*(\Omega(X_s,\cdot))$ for all $s$. Pushing forward $f_s$ by means of $\Upsilon_s$, we have defined smooth functions on $\mathcal L_s$ whose differential equals $\Omega(X_s,\cdot)$. Let us extend them to $F_\bullet:\GGG_\rho\to\R$ so that $F_\bullet$ is compactly supported and $F_s\circ \Upsilon_s=f_s$. 

Let $\widehat X_s$ be the symplectic gradient of $F_s$, namely
$$dF_s=\Omega(\widehat X_s,\cdot)~,$$
and let $\Phi_s$ be the flow generated by $\widehat X_s$. From
$d(F_s\circ \Upsilon_s)=df_s$
we see that 
$$\Omega(\widehat X_s,d\Upsilon_s(V))=\Omega(X_s,d\Upsilon_s(V))$$
for all $V\in T_p M$. This implies that $\Omega(\widehat X_s-X_s,\cdot)$ vanishes identically along the Lagrangian submanifold $\mathcal L_s$. By non-degeneracy of the symplectic form $\Omega$, $\widehat X_s-X_s$ is tangential to  $\mathcal L_s$. Therefore 
$\Phi_s\circ\Upsilon_0$ and $\Upsilon_s$ differ by pre-composition with a diffeomorphism $\phi_s$ of $M$ (which is indeed obtained as the flow on $M$ of the vector field $\Upsilon_s^*(\widehat X_s-X_s)$). This shows that $\Phi_s(\mathcal L)=\mathcal L_s$. In particular, $\Phi=\Phi_1$ is the desired compactly supported Hamiltonian symplectomorphism of $\GGG_\rho$ such that $\Phi(\mathcal L)=\mathcal L'$. 
\end{proof}

\section{Evolution by geometric flows}\label{app:geometric flows}

The aim of this last section is to provide {a relationship between certain geometric flows for hypersurfaces in $\Hyp^{n+1}$ and their induced flows in $T^1 \Hyp^{n+1}$ and in $\G{n+1}$.}

Let $M=M^n$ be an oriented manifold. Let $\sigma \colon M \times (-\varepsilon,\varepsilon) \to \Hyp^{n+1}$ be a smooth map such that $\sigma_t= \sigma(\cdot, t)$ is an immersion with small principal curvatures for all $t$, and let $\nu=\nu(x,t)$ be the normal vector field.



\begin{Prop} 
\label{Prop: flows on G and T1Hn}
Let $f\colon M\times(-\varepsilon, \varepsilon)\to \R$ be a smooth map such that 
\[
\frac{d}{dt} \sigma_t  = f_t \nu_t,
\]
{and let $\zeta_t:=\zeta_{\sigma_t}:M\to T^1\Hyp^{n+1}$ be the lift to $T^1\Hyp^{n+1}$, $G_t:=G_{\sigma_t}:M\to\G{n+1}$ be the Gauss map.} Then,
\begin{align}
	\label{eq flusso in T1}
	\frac {d}{dt} \zeta_t &= - d\zeta_t ( B_t ( \overline \nabla^t f_t) ) - {\mathrm J}  ( d\zeta_t (\overline \nabla^t f_t)) + f_t \chi \\
	\label{eq flusso in G}
	\frac {d}{dt} G_{t} &= - d G_{t} ( B_t ( \overline \nabla^t f_t) ) - \mathbb J   (d G_{t} (\overline \nabla^t f_t) )
\end{align}
where $\overline \nabla^t f_t$ is the gradient of $f_t$ with respect to the first fundamental form $\overline\I_t=G_t^*\GG$ {and $B_t$ is the shape operator of $\sigma_t$.}
\end{Prop}

{As a preliminary step to prove Proposition \ref{Prop: flows on G and T1Hn}, we  compute} the variation in time of the normal vector field. Recalling that $D$ denotes the Levi-Civita connection on $\Hyp^{n+1}$, we show:

\begin{equation}\label{eq appendix 1}
D_{\frac{d\sigma_t}{dt}  } {\nu_t}= - d\sigma( \nabla^t f_t)~,
\end{equation}
where now $\nabla^t$ denotes the gradient with respect to the first fundamental form $\I_t$ of $\sigma_t$.
On the one hand, by metric compatibility,
\[
\inner{D_{\frac{d\sigma_t}{dt} }\nu_t, \nu_t }= \frac 1 2 \partial_{\frac{d\sigma_t}{dt} } \inner{\nu_t,\nu_t}= 0
\]
hence $D_{\frac{d\sigma_t}{dt} }\nu_t$ is tangent to the hypersurface.

On the other hand, let $X$ be any vector field over $M$. Since $X$ and $\frac{\partial}{\partial t}$ commute on $M\times(-\varepsilon,\varepsilon)$,
\begin{align*}
\inner{D_{\frac{d\sigma_t}{dt} } \nu_t, d\sigma_t(X)}&= \partial_{\frac{d\sigma_t}{dt} } \inner{\nu_t, d\sigma_t (X)} - \inner{\nu_t, D_{ \frac{d\sigma_t}{dt}  } (d\sigma_t (X) )}\\
&= 0 - \inner{\nu_t, D_{ \frac{d\sigma_t}{dt}  } (d\sigma_t (X) )}\\
&=-\inner{\nu_t, D_{d\sigma_t(X)} (f_t \nu_t) }\\
&=- X( f_t) -f_t \inner{\nu_t, D_{d\sigma_t(X)} \nu_t}\\
&=-X(f_t) =- \I_t(\nabla^t f_t, X)~.
\end{align*}
This shows Equation \eqref{eq appendix 1}. As a result, in the hyperboloid model \eqref{eq:modelT} we have:
\begin{equation}\label{eq appendix 2}
\frac{d}{dt} \zeta_t = \bigg(\frac d {dt} \sigma_t, D_{\frac{d \sigma_t}{dt}} \nu\bigg)= (f_t \nu_t, - d\sigma_t (\nabla^t f_t))
\end{equation}

\begin{proof}[Proof of Proposition $\ref{Prop: flows on G and T1Hn}$]
Let $e_{1;t}, \dots, e_{n;t}$ be a local $\overline{\I}_t$-orthonormal frame diagonalizing $B_t$, so $B_t(e_{k;t})= \lambda_{k;t} e_{k;t}$.
By definition of $\gs{n+1}$ and of ${\mathrm J}$, 
\begin{equation}
	\label{eq: base T1 Hn}
	\big(d\zeta_{t}( e_{1;t}), \dots, d\zeta_{t}(e_{n;t}), \chi, {\mathrm J} d\zeta_{t}( e_{1;t}), \dots, {\mathrm J} d \zeta_{t}( e_{n;t})\big)
\end{equation}
defines at each point of the image an orthonormal basis for the tangent space of $T^1 \Hyp^{n+1}$, with the former $n+1$ vectors having norm $1$ and the latter $n$ vectors having norm $-1$.

We prove Equation \eqref{eq flusso in T1}, then Equation \eqref{eq flusso in G} follows after observing that
\[
\frac {d}{dt} G_{t} = (dG_{t})\bigg(\frac \partial {\partial t}\bigg)= (d\mathrm p \circ d\zeta_t)\bigg(\frac {\partial} {\partial t} \bigg)= d\mathrm p \bigg(\frac{d}{dt} \zeta_t\bigg).
\]
We show that LHS and RHS of \eqref{eq flusso in T1} have the same coordinates with respect to the basis \eqref{eq: base T1 Hn}. By  Equations \eqref{eq: differential of sigma tilde} and \eqref{eq appendix 2},
\begin{align*}
	\gs{n+1} \bigg( \frac d {dt} \zeta_t, {\mathrm J} d\zeta_t ( e_{k;t}) \bigg) &= f_t \inner{\nu_t, - d\sigma_t (B_t (e_{k;t})) } - \inner{- d\sigma_t (\nabla^t f_t), d\sigma_t (e_{k;t})} \\
	&= \inner{  d\sigma_t (\nabla^t f_t), d\sigma_t (e_{k;t})}= \partial_{e_{k;t}} f_t \\
	& 
	= \overline \I_t (\overline {\nabla}^t f_t,  e_{k;t})
	= \gs{n+1} (- {\mathrm J} d\zeta_t (\overline {\nabla}^t f_t), {\mathrm J} d\zeta_t (e_{k;t}) )~.
\end{align*}
Similarly, recalling that $B_t$ is self-adjoint with respect to both $\I_t$ and $\overline \I_t$, one has 
\begin{align*}
	\gs{n+1} \bigg( \frac d {dt}\zeta_t, d\zeta_t (e_{k;t}) \bigg)&= \inner{f_t \nu_t, d\sigma_t (e_{k;t})} - \inner{- d\sigma_t(\nabla^t f_t), - d\sigma_t (B_t ( e_{k;t} )) }  \\
	&= - \inner{d\sigma_t (\nabla^t f_t),  d\sigma_t (B_t(e_{k;t}))}\\
	&= - \I_t (\nabla^t f_t,  B_t(e_{k;t}))
	= - \overline{\I}_t (\overline{\nabla}^t f_t,  B_t(e_{k;t}))\\
	&= -\overline{\I}_t (B_t( \overline {\nabla}^t f_t), e_{k;t})
	= \gs{n+1} (- d\zeta_t (B_t (\overline {\nabla}^t f_t) ), d\zeta_t (e_{k;t}) )~.
\end{align*}
Finally, 
\begin{align*}
	\gs{n+1} \bigg(\frac d {dt} \zeta_t, \chi\bigg)= f_t\inner{\nu_t,\nu_t}= f_t  = \gs{n+1}(f_t\chi, \chi)
\end{align*}
and the proof follows.
\end{proof}

An interesting corollary of Proposition \ref{Prop: flows on G and T1Hn} involves mean curvature {flow}. Directly by Proposition \ref{Prop: formula H in G}, one has the following.

\begin{Cor}
	\label{cor: mean curvature flow}
The  flow in $\Hyp^{n+1}$ defined by
\[
\frac d {dt} \sigma_t = \frac 1 n \sum_{k=1}^n \arctanh(\lambda_{k;t}),
\]
on hypersurfaces of small principal curvatures, induces in $\G{n+1}$ the mean curvature flow up to a tangent factor, namely
\[
\frac d {dt} G_{t} =\mathrm{\overline H_t} +   B_t (\mathbb J (\mathrm{\overline H_t})) .
\]
\end{Cor}

\printbibliography[heading=bibintoc]

\end{document}